\documentclass[11pt,bull-l]{amsart}
\usepackage[top = 0.9in, bottom = 0.9in, left =1in, right = 1in]{geometry}
\usepackage{geometry}
\usepackage{amsfonts,amscd,amssymb,amsmath,mathrsfs}
\usepackage{graphicx,bbm,bm}
\usepackage{xparse}
\usepackage{enumerate,subfigure}
  
\usepackage{tikz,cite}
\usepackage{float, pgfplots, overpic}
\usetikzlibrary{arrows, calc, decorations.markings, positioning, fixedpointarithmetic}
\usepackage{xcolor}

\usepackage{url}

\usepackage[notocbasic]{nomencl}

\makenomenclature

\definecolor{MichiganBlue}{HTML}{00274C}
\definecolor{MichiganYellow}{HTML}{FFCB05}  
\definecolor{NicePurple}{RGB}{75,56,76} %PrincePurple
\definecolor{NiceRed}{RGB}{230,37,52}
\definecolor{MidnightBlue}{rgb}{0.1, 0.1, 0.44}
  
\renewcommand{\Re}{\mathrm{Re}\,}
\renewcommand{\Im}{\mathrm{Im}\,}
\newcommand{\Db}{\mathbf{D}}
\newcommand{\Vb}{\mathbf{V}} 
\newcommand{\Hb}{\mathbf{H}}
\newcommand{\Lb}{\mathbf{L}}
\newcommand{\Tr}{\operatorname{Tr}}
\newcommand{\OO}{\mathrm{O}}  
\newcommand{\oo}{\mathrm{o}}  
\newcommand{\ri}{\mathrm{i}}
\newcommand{\dd}{\mathrm{d}}

\newcommand{\ub}{\mathbf{u}}
\newcommand{\vb}{\mathbf{v}}
\newcommand{\xb}{\mathbf{x}}
\newcommand{\yb}{\mathbf{y}}
\newcommand{\bb}{\mathbf{b}}
\newcommand{\eb}{\mathbf{e}}
\newcommand{\Kb}{\mathbf{K}}
\newcommand{\Ab}{\mathbf{A}}  
\newcommand{\sG}{\mathsf{G}}  

\renewcommand{\vec}{\mathbf}

\newtheorem{theorem}{Theorem}[section]
\newtheorem{definition}{Definition}
\newtheorem{remark}{Remark}[section]
\newtheorem{lemma}[theorem]{Lemma}
\newtheorem{proposition}[theorem]{Proposition}
\newtheorem{corollary}[theorem]{Corollary}
\newtheorem{assum}{Assumption}

\numberwithin{equation}{section}

\newcommand{\LP}{\operatorname{L}}

\DeclareMathOperator{\e}{e}
\newcommand{\I}{\mathrm{i}}
\newcommand{\sd}{\mathrm{d}}

\newcommand{\prob}{\mathbb{P}}

\renewcommand{\Pr}{\prob}
\DeclareDocumentCommand \one { o }
{%
\ensuremath
\IfNoValueTF {#1}
{\mathbf{1}  }
{\ensuremath{\mathbf{1}\left\{ {#1} \right\} }}%
}

\newcommand{\lawequals}{\overset{\mathscr{L}}{=}}
\DeclareDocumentCommand{\Prto} {o} {
\IfNoValueTF {#1}
 {\overset{\Pr}{\longrightarrow}}
 { \xrightarrow[ #1 \to \infty]{\Pr }}
}
\DeclareDocumentCommand{\Asto} {o} {
\IfNoValueTF {#1}
 {\overset{\operatorname{a.s.}}{\longrightarrow}}
 {
 \xrightarrow[ #1 \to \infty]{\operatorname{a.s.} }
 }
}
\DeclareDocumentCommand{\Mgfto} {o} {
\IfNoValueTF {#1}
{\overset{\operatorname{mgf}}{\longrightarrow}}
{ \xrightarrow[ #1 \to \infty]{\operatorname{mgf} }}
}

\DeclareDocumentCommand{\Wkto} {o} {
\IfNoValueTF {#1}
 {\overset{(d)}{\longrightarrow}}
 { \xrightarrow[ #1 \to \infty]{(d) }}
}

\DeclareDocumentCommand{\To} {o} {
\IfNoValueTF {#1}
 {\rightarrow}
 { \xrightarrow[]{#1 \to \infty }}
}

\DeclareDocumentCommand \LPto { O{1} }
{\overset{\operatorname{\LP^{#1}}}{\longrightarrow}}

\title[A Riemann--Hilbert approach to the perturbation of orthogonal polynomials]{A Riemann--Hilbert approach to the perturbation theory for orthogonal polynomials: Applications to numerical linear algebra and random matrix theory}

\author{Xiucai Ding}
\address{University of California, Davis}
\email{xcading@ucdavis.edu}
  
\author{Thomas Trogdon}
\address{University of Washington, Seattle, WA}
\email{trogdon@uw.edu}

\thanks{The authors would like to thank Percy Deift for many useful discussions, Peter Miller for pointing us to references related to Nuttall's theorem and Deniz Bilman for his superior Tikz abilities. XCD is partially supported by NSF DMS-2113489 and TT is partially supported by NSF DMS-1945652.  This work was facilitated through the use of advanced computational, storage, and networking infrastructure provided by the Hyak supercomputer system at the University of Washington.}    

\usepackage{hyperref}
\hypersetup{colorlinks,linkcolor={blue},citecolor={blue},urlcolor={red}}  
\begin{document}

\maketitle

\begin{abstract}
We establish a new perturbation theory for orthogonal polynomials using a Riemann--Hilbert approach and consider applications in numerical linear algebra and random matrix theory. This new approach shows that the orthogonal polynomials with respect to two measures can be effectively compared using the difference of their Stieltjes transforms on a suitably chosen contour. Moreover, when two measures are close and satisfy some regularity conditions,  we use the theta functions of a hyperelliptic Riemann surface to derive explicit and accurate expansion formulae for the perturbed orthogonal polynomials. %The leading error terms can be fully characterized by the difference of the Stieltjes transforms on the contour. 

In contrast to other approaches, a key strength of the methodology is that estimates can remain valid as the degree of the polynomial grows.  The results are applied to analyze several numerical algorithms from linear algebra, including the Lanczos tridiagonalization procedure, the Cholesky factorization and the conjugate gradient algorithm. As a case study, we investigate these algorithms applied to a general spiked sample covariance matrix model by considering the eigenvector empirical spectral distribution and its limits.  For the first time, we give precise estimates on the output of the algorithms, applied to this wide class of random matrices, as the number of iterations diverges. In this setting, beyond the first order expansion, we also derive a new mesoscopic central limit theorem for the associated orthogonal polynomials and other quantities relevant to numerical algorithms.   
\end{abstract}

% \printnomenclature
% \nomnopage{$\mathbb C$}{The field of complex numbers}  % To put an entry in the table without a page number.  Use
% % \nomnom command
% \nomnopage{$\mathbb F^{N}$}{The space of all $N$-dimensional vectors with entries in $\mathbb F$}
% \nomnopage{$\mathbb F^{N \times M}$}{The space of all $N \times M$ matrices with entries in $\mathbb F$}
% \nomnopage{$A_{j:k,\ell:p}$}{The submatrix of $A$ consisting of all entries in rows $j$ through $k$ and columns $\ell$ through $p$.}
% \nomnopage{$A_{j,\ell:p}$}{$A_{j:j,\ell:p}$}
% \nomnopage{$A_{j:k,\ell}$}{$A_{j:k,\ell:\ell}$}  

{ \hypersetup{linkcolor=black}
%\setcounter{tocdepth}{2}
%\setcounter{secnumdepth}{2}
%\tableofcontents
}
%\addcontentsline{toc}{chapter}

\section{Introduction}
  
We consider a Riemann--Hilbert approach to the perturbation of orthogonal polynomials. More specifically, we present an approach to compare the orthogonal polynomials with respect to two compactly supported measures on $\mathbb R$ by comparing their Stieltjes transforms on a contour that encircles and contracts to the union of the supports.  The approach uses and generalizes the Fokas--Its--Kitaev reformulation of orthogonal polynomials \cite{FokasOP} as the solution of a Riemann--Hilbert problem.  This approach is especially powerful when the orthogonal polynomials with respect to one of the measures has known asymptotics.  And in particular, it allows one to compare, in a convenient framework, polynomials orthogonal to a discrete empirical measure, i.e. discrete orthogonal polynomials, to the polynomials orthogonal with respect to a limiting measure.  We refer the reader to \cite{Baik2007} for many related details concerning discrete orthogonal polynomials.

Measures are often compared rather effectively using their moments.  But even measures that are rather close in a variety of senses can have vastly different moments of high order. For this reason many studies of the perturbations of orthogonal polynomials are not infintesimal in nature, see \cite{Garza2016} and the references therein, particularly \cite{Zhedanov1997}. One construction of orthogonal polynomials uses their representation in terms of determinants of Hankel moment matrices (see \cite{DeiftOrthogonalPolynomials} and \cite{GautschiOP}, for example).  This fact was recently exploited in \cite{DT1,Paquette2020} to compare two sequences of orthogonal polynomials when the degree is bounded. But as the degree increases, this approach fails because two sequences of orthogonal polynomials with respect to two similar measures typically deviate exponentially, see \cite[Section 2.1.6]{GautschiOP}. But the Fokas--Its--Kitaev Riemann--Hilbert problem gives a mechanism to make sense of the behavior of one sequence of orthogonal polynomials relative to another, giving a sense in which the mapping from a Stieltjes transform of a measure to the associated orthogonal polynomials (and their weighted Cauchy integrals) is well conditioned.  %Futhermore, if the convergence of the Stieltjes transforms of two measures is sufficienlty strong, and the limiting measure is sufficiently regular, this exponential deviation can be captured exactly.

Comparing sequences of orthogonal polynomials via their Stieltjes transforms lends itself directly to estimates from random matrix theory.  For example, the well-known local laws for Wigner, generalized Wigner and (spiked) sample covariance matrices are precisely comparisons of Stieltjes transforms of measures on contours approaching the supports on small scales; see the monograph \cite{MR3699468} for more details.  Importantly, the standard empirical spectral distributions associated with these matrices, measures that weight each eigenvalue equally, are not as likely to arise in applications from computational mathematics.  So one, in turn, looks to the so-called anisotropic local laws \cite{Knowles2017} which gives, in particular, the comparison of the Stieltjes transform of the eigenvector empirical spectral distribution (VESD) which, for an $N \times N$ symmetric matrix $W$ and vector $\vec b$, is given by \cite{BMP},
\begin{align}\label{eq:intro_VESD}
 \nu = \sum_{j=1}^N |\langle \vec q_j, \vec b \rangle|^2 \delta_{\lambda_j(W)},
\end{align}
where $\vec q_j$ is a normalized eigenvector associated with eigenvalue $\lambda_j(W)$ of $W$.  For the sake of completeness, we note that if the weights $|\langle \vec q_j, \vec b \rangle|^2$ are each replaced with $1/N$ the resulting measure is called the empirical spectral distribution (ESD).

Our main application of the estimates for random polynomials orthogonal to the VESD concerns the (bi/tri)diagonalization of random matrices and, as a consequence, applications to other critically important numerical algorithms acting on random matrices, see Section \ref{sec:matfac} for more details. Here we take the tridiagonalization as an example.  Going back to the work of Silverstein \cite{Silverstein1985}, and the subsequent work of Dumitriu and Edelman \cite{Dumitriu2002}, it is well-known that the tridiagonalization $T$ of a Wishart matrix $W = XX^*$, where $X_{ij} \lawequals \mathcal N(0,M^{-1})$, and $X$ is $N \times M,$ and has independent entries, has an explicit distributional description in terms of independent $\chi$-distributed random variables (see \eqref{eq:chol_wishart} below).  But this description is actually derived first from a distributional description of the Cholesky decomposition\footnote{We discuss tridiagonalization and the Cholesky decomposition in Section~\ref{sec:matfac} below.} 
\begin{equation}\label{defn_L}
T = L L^*, \quad L = (\ell_{i,j}) .  
\end{equation}
The Cholesky factorization in this context is a lower-bidiagonal factorization of the tridiagonalization.   An immediate consequence of this bidiagonalization is that $\ell_{n,n} - \sqrt{\frac{M - n + 1}{M}}$ and $\ell_{n+1,n} - \sqrt{\frac{N - n}{M}}$ tend to zero and have Gaussian fluctuations provided $M-n$ and $N -n$, respectively, tend to $\infty$.  It is therefore natural to ask if this behavior persists for both non-Gaussian entries (universality) and if it persists for sample covariance matrices with non-trivial covariance. It was recently proved in \cite{Paquette2020} that for non-Gaussian entries with trivial covariance,  if $N/M \to c \in (0,1]$ and $n$ is fixed one sees that the upper-left $n\times n$ subblock of $L$ tends to the Cholesky factorization of the three-term recurrence Jacobi matrix for the orthogonal polynomials with respect to the Marchenko--Pastur law with parameter $c$. These arguments do not apply if either $n$ diverges and the entries $X_{ij}$ are non-Gaussian or if the covariance is non-trivial. Our Riemann--Hilbert approach extends these results, and the results of \cite{DT1}, to non-trivial covariance and unbounded $n$. 

We summarize related results in Sections \ref{sec_RHtheory} and \ref{sec_somerelatedwork} and provide an overview of our results and key innovations in Section \ref{sec_overview}.

%It is well-known that the entries in the Jacobi matrix associated to a measure can be constructed using the Fokas--It--Kitaev Riemann--Hilbert problem \cite{DeiftOrthogonalPolynomials}.  What is perhaps less well known, is that the Cholesky factorization of this matrix, when the support of the measure is a subset of $(0,\infty)$, is also simple to build directly using this Riemann--Hilbert formulation.

% -----------------
%Points to hit:
%\begin{itemize}
%\item Riemann--Hilbert framework provides a mechanism to compare a sequence of OPs to another sequence when the asymptotics of the second sequence are known.
%\item The focus of the current work is to compare polynomials orthogonal wrt the VESD for spiked sample covariance matrices to those orthogonal to the limiting empirical %spectral distribution.
%\item Applications include the tridiagonalization, the Cholesky factorization of this tridiagonalization and the performance of the conjugate gradient algorithm (here we will %note that, in principle, the latter two applications can be derived from the first, but the RH framework is better)
%\item should say something that this framework is easily to be connected with RMT since the $\mu$ can be regarded as the deterministic limit whereas $\nu$ is the empirical %distribution. In this sense, we completely switch the role of RMT and RHP.
%\end{itemize}

\subsection{A new application of Riemann--Hilbert analysis in random matrix theory}\label{sec_RHtheory}  
%
%Riemann--Hilbert problem is a powerful tool in formulating many problems in pure and applied mathematics, for example,  integrable systems  \cite{MR2011605}, orthogonal polynomials and random matrix theory \cite{DeiftOrthogonalPolynomials} and  numerical analysis \cite{TrogdonSOBook}. 

In this section, we summarize some related results on the Riemann--Hilbert approach to orthogonal polynomials and various related applications and demonstrate how our approach differs. It is known  from the celebrated work  of Fokas, Its and Kitaev \cite{FokasOP} that orthogonal polynomials can be characterized  as the solution of a $2 \times 2$ matrix Riemann--Hilbert problem with jump on the real line. Later on, a remarkable steepest descent method was proposed by Deift and Zhou in \cite{DZS} to study the asymptotics of the modified Korteweg-de Vries equation. Since then, various extensions have been made, including to the asymptotics of orthogonal polynomials. More specifically, the extension on the unit circle was studied in \cite{MR1682248}, general measures and universality were studied in \cite{BI,MR2377682,MR1702716,MR1469319,MR1680380, MR2087231}, the biorthogonal polynomial problem was studied in \cite{MR2021905, MR1985213, MR2127887, WZ}  and multiple orthogonal polynomials were studied in \cite{MR2006283}. For a more comprehensive review, we refer the reader to \cite{Baik2007,Bleher2011, DeiftOrthogonalPolynomials,MR2307753, MR2022855}.  Of particular relevance is the monograph \cite{Baik2007}.  In a slightly different form, this text contains the transformation \eqref{eq:c0transform} and the hyperelliptic Riemann surface theory employed in Appendix~\ref{app:OPs}.

Classically, the way in which Riemann--Hilbert problems and orthogonal polynomial theory connect to random matrix theory is very different from the framework we propose here. More precisely, Riemann--Hilbert problems historically enter random matrix theory via the analysis of orthogonal polynomials because the eigenvalues of many random matrix ensembles can be viewed as a determinantal point processes and the correlation functions have a determinantal kernel function that can be expressed as a sum of orthogonal polynomials. Consequently, using the Christoffel--Darboux formula, the eigenvalue correlation functions can be expressed in terms of the solution of a Riemann--Hilbert problem; see \cite{mehta2004random,DeiftOrthogonalPolynomials} for a review. On the other hand, the gap probabilities can be represented as a Fredholm determinant and the limiting expressions themselves can be expressed in terms of the solution of a Riemann--Hilbert problem; see the monographs \cite{MR2581882, Largegap} for a review. This approach, combined  with the steepest descent method, allows for the large $N$  asymptotics to be determined explicitly for various random matrix models leading to the determination of explicit limiting kernels. For example, for the Gaussian Unitary Ensemble (GUE), the correlation function for the bulk eigenvalues converge to the sine kernel \cite{MR278668,mehta2004random, MR1469319} and the large gap probability of the edge eigenvalues converge to the Airy kernel \cite{MR1257246}. We refer the readers to \cite{Bleher2011, DeiftOrthogonalPolynomials,MR2022855} for a more exhaustive discussion. The methodology has also been applied to various other random matrix models, see \cite{MR2486670,MR2103904,MR2531553, MR2306224,MR2363388, MR3257662,MR1912278,MR3459158,WZ}, to name but a few.

%We also point out that in \cite{MR3459158} the authors studied the linear statistics of orthogonal polynomial ensembles in random matrix theory. To obtain the CLT (see Theorem 1.2 therein), they need to assume that the three-term recurrence coefficients of the two measures are close enough. As we will see in Theorem \ref{thm_pertubed}, this is a consequence of our result. 
%  and gap probabilities of the eigenvalues of many random matrix ensembles can be  \cite{Largegap} 
%{\color{blue}[connections with random matrix theory]}
% Since then, various extensions  
%Riemann--Hilbert problems entered random matrix theory through their effectiveness in analyzing invariant ensembles \cite{DeiftOrthogonalPolynomials}.
% {\color{blue}  In the literature, computational tools, not a framework. How people use RHP to solve RMT problem. Some papers needed to be cited and discussed:        \cite{} }
 
In the current paper, we do not study orthogonal polynomials and random matrices by following the classic research line above. In contrast, we apply a Riemann--Hilbert approach to study the behavior of orthogonal polynomials with respect to perturbations of the orthogonality measure.   We then apply the theory to polynomials orthogonal with respect to the VESD \eqref{eq:intro_VESD} when $W$ is random.  The perturbations we consider are quantified by the closeness of their Stieltjes transforms. Such a setting is general.  A wide class of (random) measures that can be thought of as appropriate perturbations of a deterministic measure are measures arising from widely studied random matrix models, where the local laws \cite{MR3699468} guarantee the closeness of the limiting and empirical measures.  Our new approach, and its generality, can best be summarized by the fact that while some random matrix ensembles have eigenvalue statistics that can be analyzed by orthogonal polynomial theory, all random matrices generate measures (again, see \eqref{eq:intro_VESD}), and the analysis of the orthogonal polynomials with respect to such a measure are important.  We show exactly how this analysis can be accomplished using Riemann--Hilbert analysis.

\subsection{Some related work on numerical algorithms}\label{sec_somerelatedwork} 
 
Our motivation comes from the analysis of various iterative numerical algorithms in linear algebra (see Section \ref{sec:matfac} for a review), especially when the inputs are random matrices. A common feature for these algorithms is that their analysis can be reduced to understanding certain (discrete) orthogonal polynomials and their associated Cauchy transforms (see (\ref{eq_bnmuexpansion}), (\ref{eq_enw}), (\ref{eq_rn}) and (\ref{eq_alphabetarepresentation}) for illustrations).  By establishing a perturbation theory for orthogonal polynomials, we are able to provide the first-order limits and asymptotic distributions\footnote{We determine distributions when the inputs are random.} related to these algorithms.

In the literature,  various numerical algorithms have been studied when the inputs are random matrices. The tridiagonalization of Wishart matrix (i.e., sample covariance matrix with standard Gaussian entries) has been analyzed in \cite{Dumitriu2002, Silverstein1985}, the finite iterations of CGA for a sample covariance matrix with trivial covariance  was analyzed in \cite{MR4188626,Paquette2020}, the Toda algorithm on  Wishart matrices was analyzed in \cite{DT17, DT18}. These analyses rely on either a Gaussian assumption or the trivial covariance assumption. The finite iterations of CGA with general covariance structure was analyzed in \cite{DT1}. The general phenomenon that some algorithms have, in an appropriate sense, high concentration in their outputs even when the inputs are random data can be seen in each of these works.  And quite often the performance of the algorithms under consideration is universal. We refer to the readers to  \cite{Deift14973,DiagonalRMT,Deift2019a,Sagun2015} for further discussions.

There has also been significant developments in the area of smoothed analysis of algorithms \cite{Sankar2006,Spielman2004}.  More closely related to the current work is \cite{Menon2016}.  We leave the problem of using the current results in this context as future work.

% {\color{blue}[add summary here:]} {\color{blue}[add more here \cite{DT18} \cite{DT17}] Make a summary of the limitations of these algorithms}.  
%  {\color{red} literature on the numerical part}    
% 
%
% 
% 
% {\color{red} list some other works say potential}  

\subsection{An overview of main results}\label{sec_overview} Given a probability measure $\mu$ with finite moments, we apply the Gram-Schmidt orthogonalization process to the monomials $\{1,\lambda,\lambda^2, \cdots\}$ to obtain  the monic orthogonal polynomials $\pi_n(\lambda;\mu)$, $n = 0,1,2,\ldots$, which can be defined by
%for a probability measure $\mu$ on $\mathbb R$, with finite moments, can be defined by
\begin{align}\label{eq_MOP}
  \pi_n(\lambda;\mu) = \lambda^n + \OO(\lambda^{n-1}), \quad \lambda \to \infty, \quad \int_{\mathbb R} \pi_n(\lambda;\mu) \pi_m(\lambda;\mu) \mu( \sd \lambda) = 0, \quad n \neq m.
\end{align}
Given two measures $\mu$ and $\nu,$ where $\nu$ can be regarded as a perturbed or empirical version of $\mu$, we aim to study how $\pi_n(\lambda; \mu)$ and $\pi_n(\lambda;\nu)$ relate asymptotically, both as $n$ increases and as $\nu \to \mu$.

The starting point of our analysis is the quantity  
$X_n(z;\mu, \nu)$ introduced in (\ref{eq_novelquantity}). The motivation to use $X_n(z; \mu,\nu)$ is threefold. First, it naturally connects $\pi_n(\lambda;\mu)$ and $\pi_n(\lambda;\nu)$ and their associated Cauchy transforms (c.f.~(\ref{eq_defncauchytransform})). Second, $X_n$ is the solution of a matrix Riemann--Hilbert problem that can be explicitly formulated using the Fokas--Its--Kitaev approach. Third, the relevant quantities associated to the numerical algorithms we consider can be expressed in terms of the entries of $X_n(z;\mu, \nu)$.  The Riemann--Hilbert problem for $X_n(z;\mu,\nu)$ can be solved asymptotically, and this result is recorded in Proposition \ref{prop:perturb}. Equivalently, it establishes a new perturbation result for orthogonal polynomials. Heuristically, it states that for two compactly supported measures $\mu, \nu$ on $\mathbb R$ such that
  \begin{align}\label{eq_differencestieltjesdefn}
    \int \frac{\nu(\sd \lambda) - \mu(\sd \lambda)}{\lambda - z},  
  \end{align}
  is sufficiently small on a contour which encircles, and is sufficiently close to $\mathrm{supp}(\mu) \cup \mathrm{supp}(\nu)$, one has for the monic polynomials $\pi_n$,
  \begin{align}\label{eq:pin_summary}
    \pi_n(z;\nu) = \pi_n(z;\mu) ( 1 + f_1(z;\mu,\nu) ) + f_2(z;\mu,\nu) \pi_{n-1}(z, \mu)\frac{\mathfrak c^{2(p-n)}}{\|\pi_{n-1}(\cdot;\mu)\|_{L^2(\mu)}^2},
  \end{align}
  for functions $f_1,f_2 = \oo(1)$ depending on the size of (\ref{eq_differencestieltjesdefn}) and some constant $\mathfrak c$.    Here $p$ is the number of spikes  (i.e., point masses, see c.f.~(\ref{eq:mu})); see (\ref{eq_pertubationtheoryofOP}) for more details.   A further expansion of the functions $f_1, f_2$ determine the next order correction, which we, in view of our primary application to random matrices, call the fluctuation term.

  Then, assuming that $\mu$ satisfies some regularity conditions (c.f.~Assumption \ref{assum_measure}), we first derive some accurate and uniformly valid asymptotic formulae for the unperturbed orthogonal polynomials utilizing theta functions on a hyperelliptic Riemann surface (c.f.~(\ref{eq_thetafunction})). The results are stated in Theorem \ref{lem_deterministicexpansion}. By controlling  a key auxiliary quantity (c.f.~(\ref{eq:Mn})) in Lemma \ref{lem_steponecontrol}, we are able use Proposition \ref{prop:perturb} and Theorem \ref{lem_deterministicexpansion} to provide asymptotic formulae for the perturbed orthogonal polynomials and their Cauchy transforms as in Theorem \ref{t:main_OP} and Remark \ref{rem_explicitexpansionformula}. These formulae give explicitly how some critical exponential prefactors are arranged. Moreover, the leading error terms can be fully characterized by a variant of (\ref{eq_differencestieltjesdefn}). Thus, the calculation of the fluctuations of $\pi_n(z;\nu)$ reduces to the analysis of (\ref{eq_differencestieltjesdefn}).

We mention several points related to random matrix theory here. First, Assumption \ref{assum_measure} is satisfied by the limiting eigenvalue or eigenvector empirical spectral distributions of many classically studied random matrix models. In this context, $\nu$ can be the eigenvalue or eigenvector empirical spectral distribution. Second, the degree $n$ is allowed to be unbounded (with respect to some divergent parameter) and it depends on the closeness of the Stieltjes transforms of the measures $\mu$ and $\nu$. For example, in the random matrix model setting regarding an $N \times N$ matrix, as will be discussed in  Remark \ref{rmk_divergent}, $n$ can be as large as $\OO(N^{1/4-\epsilon}),$ for some arbitrarily small constant $\epsilon>0$ for ESD, and $\OO(N^{1/6-\epsilon})$ for VESD. To our best knowledge, this is the first such asymptotic result allowing $n$ to diverge.

Motivated by several important applications in numerical linear algebra, we apply Theorems \ref{lem_deterministicexpansion} and \ref{t:main_OP} to analyze iterative numerical algorithms, including Lanczos tridiagonalization, the Cholesky factorization and conjugate gradient algorithm (CGA); see Section \ref{sec:matfac} for a brief summary of these algorithms.     
First, we apply Theorem \ref{lem_deterministicexpansion} to these algorithms and obtain accurate asymptotic formulae for the key quantities. For Lanczos, it is equivalent to the study of the asymptotics of the three-term recurrence coefficients of the (discrete) orthogonal polynomials. The results are recorded in Corollary \ref{cor_deterthreeterm}. The Cholesky factorization of the Lanczos Jacobi matrix (c.f.~(\ref{eq:jacobi_def})) can also be analyzed similarly as in Corollary \ref{cor_choleskeylimit}.  This Cholesky factorization coincides with the well-known Golub-Kahan bidiagonalization procedure which, as pointed out previously, has a full distributional characterization in the isotropic Gaussian case.   But our results hold for non-Gaussian samples with non-trivial covariance. CGA is analyzed in Corollary \ref{cor_cgadeterasymp}. Based on the unperturbed asymptotics for $\mu$ we establish the perturbed asymptotics for these algorithms and the results are reported in Theorem \ref{thm_pertubed}. Again, the leading errors can be fully expressed in terms of (\ref{eq_differencestieltjesdefn}) and the associated theta functions.

As mentioned earlier, the fluctuations of the perturbed orthogonal polynomials and related quantities of the numerical algorithms depend on (\ref{eq_differencestieltjesdefn}) which should be expected to have a problem-specific form. In Section \ref{sec_spikedcov}, we consider a concrete case study, in the random matrix context, using a general spiked sample covariance matrix model. 
More specifically,  $\nu$ is the VESD of the $N \times N$ sample covariance matrix whose deterministic equivalent $\mu$ can be characterized using the anisotropic local laws as discussed in Section \ref{sec_subsectionvesd}.
 The methodology we propose here shows how Riemann--Hilbert problems can assist yet again, later in the analysis of a random matrix ensemble, once one has some knowledge of the local law. The main result is Theorem \ref{thm_mainclt} which establishes a general mesoscopic-type central limit theorem (CLT) by analyzing a functional version of (\ref{eq_differencestieltjesdefn}). We mention that the CLT is mesoscopic as its scaling also depends $n$. Informally, we prove that for  $z \in \mathbb{R},$ when $n \ll N^{1/6}$ 
\begin{equation*}
\frac{\sqrt{N/n^2}}{\mathsf{Z}(z; \mu)} (\pi_n(z;\mu)-\pi_n(z;\nu)) \Wkto[N] \mathcal{N}(0,\mathsf d(z)(\mathsf{V}_1+\mathsf{V}_2)),
\end{equation*}
where $\mathsf{Z}(z;\mu)$ is a normalization constant that depends on $z$ and $\mu$, $\mathsf{V}_1$ depends on $\mu$ and is independent of $n$, $\mathsf{V}_2$  depends on both $n$ and the fourth moments of the entries of the matrix and $\Wkto[N]$ indicates convergence in law. Moreover, as long as  $n \to \infty,$ $\mathsf{V}_2 \to 0$ so that the CLT only depends on the first two moments. Finally,  $\mathsf d(z)$ is a deterministic function depending on the application under consideration. For example, for the various aforementioned numerical algorithms, $\mathsf{d}(z)$ can be found explicitly is summarized in Corollary \ref{coro_explicitdistribution}. Nevertheless, we mention that even though we work on the spiked sample covariance matrix model in the current paper, our methods can be easily applied to other random matrix models once the local laws are established. 

We emphasize that our results of the case study generalize many existing results in numerical linear algebra and random matrix theory. First, we show that for a general class of spiked sample covariance matrices, if $n \ll N^{1/6}$ then the upper-left $n \times n$ subblock of $L$ in (\ref{defn_L}) tends to the upper-left subblock of the Cholesky factorization of the three-term recurrence Jacobi matrix  for the orthogonal polynomials with respect to the limiting VESD, with universal Gaussian fluctuations.   We also establish that the dependence on the fourth moment diminishes as $n$ increases, a phenomenon that was empirically observed in \cite{Paquette2020}. Second, we establish precise convergence statistics for CGA when the matrix is a general spiked sample covariance matrix model. We allow $n,$ which here is taken to be the number of iterations in CGA, to be divergent with $N.$ In particular, we show that the residuals always have Gaussian fluctuations and become more universal (i.e., only depend on the first two moments) as more iterations are run. Comparable results have only been previously established for fixed $n$ and trivial covariance case in \cite{Paquette2020} for the case of Wishart matrices.

Finally, we highlight an  open question. In the current paper, the breakthrough allows $n$ to increase with $N$ in a moderate way, i.e., $1 \leq n  \leq N^{\alpha}, 0 \leq \alpha<1/6$.  It is interesting to consider the regime $1/6 \leq \alpha \leq 1.$ Based on our numerical simulations, we conjecture that our results still hold for all $0 \leq \alpha<1.$ However, when $\alpha=1,$ our current results clearly fail to hold (see Figure \ref{fig:linear}) and we need to develop entirely new tools to handle this regime. We will pursue this direction in the future.

\vspace{2pt}

\noindent {\bf Conventions. } For two sequences of real values $\{a_N\}$ and $\{b_N\},$ we write $a_N=\OO(b_N)$ if $|a_N| \leq C|b_N|$ for some constant $C>0,$ and $a_N=\oo(b_N)$ if $|a_N| \leq c_N | b_N|$ for some positive sequence $c_N \downarrow 0.$ Moreover, we write $a_N \asymp b_N$ if $a_N=\OO(b_N)$ and $b_N=\OO(a_N).$   The notation $\langle \vec b , \vec a \rangle$ is used for the standard $\ell^2$ inner product and $\|\vec b\|_2^2 = \langle \vec b ,\vec{b} \rangle$.  We use $\vec{f}_k$ to denote the $k$th standard Euclidean basis vector.

%
%For any complex random variable $y,$ we say $y \sim \mathcal{N}_{\mathbb{C}}(0, \sigma^2)$ is a circularly symmetric complex Gaussian random variable if the real and complex parts of $y$ are independent $\mathcal{N}(0, \sigma^2/2)$ random variables. 

%{\color{red}[ ]}

\section{The Riemann--Hilbert problem for orthogonal polynomials and their perturbations}\label{sec_RHPframework}

Consider a probability measure $\mu$ without a singular continuous part. We suppose its absolute continuous density $\rho$ is supported on a finite number of disjoint intervals $[\texttt{a}_j, \texttt{b}_j], 1 \leq j \leq g+1.$ We also allow $\mu$ having a finite number of spikes, i.e., point masses at $\texttt{c}_i , 1 \leq i \leq p$, with masses $w_j$.% where either $w_j=0$ or $|w_j| \asymp  1.$ Moreover, we suppose that the distance of any two points in the set $\{\texttt{a}_j\}\cup \{\texttt{b}_j\} \cup \{\texttt{c}_j\}$ are bounded from below and above. 

In \cite{FokasOP}, the authors found a characterization of orthogonal polynomials in terms of a matrix Riemann--Hilbert problem. We now review such a formulation.  Define the Cauchy transforms of the monic polynomials
\begin{align}\label{eq_defncauchytransform}
  c_n(z;\mu) = \frac{1}{2 \pi \I} \int_{\mathbb R} \frac{\pi_n(\lambda;\mu)}{\lambda - z} \mu(\sd \lambda),
 %  \quad x \not \in \mathbb R, \ {\color{red}[z??]}
\end{align} 
and the matrix-valued function
  \begin{align} \label{eq:def_Y}
    Y_n(z;\mu) &= \begin{bmatrix} \pi_n(z;\mu) & c_n(z;\mu) \\
      \gamma_{n-1}(\mu) \pi_{n-1}(z;\mu) & \gamma_{n-1}(\mu) c_{n-1}(z;\mu) \end{bmatrix}, \quad z \not \in \mathrm{supp}(\mu),
      \end{align}
      where we used the notation
      \begin{equation}\label{eq_gammangammmu}
    \gamma_{n}(\mu) = - 2 \pi \I \|\pi_n(\cdot;\mu)\|_{L^2(\mu)}^{-2}.
    \end{equation}
  It then follows that (see \cite{FokasOP} or \cite{MR2022855})
  \begin{align}\label{eq:def_Y_jump}
    Y_n^{+}(z;\mu)&=  Y_n^{-}(z;\mu) \begin{bmatrix} 1 & \rho(z) \\ 0 & 1 \end{bmatrix}, \ \  Y_n^{\pm}(z;\mu):=\lim_{\epsilon \to 0^+} Y_n(z \pm \I \epsilon;\mu),
  \end{align}
  at all points $z \in \mathbb R$ where $\mu$ has a continuous density $\rho$.  Additionally,
  \begin{align}\label{eq:def_Y_inf}
    Y_n(z;\mu) \begin{bmatrix} z^{-n} & 0 \\ 0 & z^n \end{bmatrix} &= I + \OO(1/z), \quad z \to \infty.
  \end{align}

  Due to the discrete contributions to $\mu$, this does not fully characterize $Y_n$.  We compute
  \begin{align} \label{eq:def_Y_res}
    \mathrm{Res}_{z = \mathtt c_j} Y_n(z;\mu) & = \begin{bmatrix}  
      0 & \frac{1}{2 \pi \I} w_j \pi_n(\mathtt c_j; \mu) \\
      0 & \frac{\gamma_{n-1}}{2 \pi \I} w_j \pi_{n-1}(\mathtt c_j; \mu) \\
    \end{bmatrix}   \\
 &   = \lim_{z \to \mathtt c_j} Y_n(z;\mu) \begin{bmatrix} 0 & \frac{w_j}{2 \pi \I} \\
      0 & 0 \end{bmatrix}, \quad j = 1,2,\ldots,p. \notag
  \end{align}
Conditions \eqref{eq:def_Y_jump},  \eqref{eq:def_Y_inf} and  \eqref{eq:def_Y_res} constitute a Riemann--Hilbert problem for $Y_n(z;\mu)$ and $Y_n(z;\mu)$ is the unique solution of this problem if one requires continuous boundary values.  %This observation was first made in \cite{FokasItsKitaev}. 

% and the Riemann--Hilbert problem was formulated explicitly later in \cite{FokasOP}. 

%{\color{red}[XCD: maybe we need to add a little more arguments on the above equations (maybe in the appendix). It is straightforward but  nontrivial to people who has no background in RHP and matrix model. ]}
  
  \begin{remark}
    At points where $\mu$ has a density, but it fails to be continuous, one may have to impose additional conditions to uniquely characterize $Y_n$.  The assumptions we impose on $\mu$ in the current work allow us to ignore such complications.
  \end{remark}

\subsection{Perturbation theory for orthogonal polynomials} 
Let $\nu$ be a perturbed (and potentially random) version of $\mu.$ Suppose $\mu$ and $\nu$ are both measures supported on a finite number (i.e. $g+1$) of intervals with a finite number (i.e. $p$) of spikes
%of the form \eqref{eq:mu}
for (potentially) different choices of $\mathtt a_j,\mathtt b_j,w_j,\mathtt c_j,h_j$ and $g,p$.  Define
  \begin{align}\label{eq:c0transform}
    \tilde Y_n(z;\mu) = \begin{cases} Y_n(z;\mu) \begin{bmatrix} 1 & -c_0(z;\mu) \\ 0 & 1 \end{bmatrix} & z \text{ inside } \Gamma,\\
      Y_n(z;\mu) & \text{otherwise},
    \end{cases}
  \end{align}
  where $\Gamma$ is a simple curve with counter-clockwise orientation that encloses the support of $\mu$.  Using (\ref{eq:def_Y_jump}),  we then compute the jumps of $\tilde Y_n$ on $\cup_j (\mathtt a_j,\mathtt b_j)$:
  \begin{align*}
    \tilde Y^+_n(z;\mu) &=  Y^+_n(z;\mu) \begin{bmatrix} 1 & -c_0^+(z;\mu) \\ 0 & 1 \end{bmatrix} = Y_n^-(z;\mu) \begin{bmatrix} 1 & \rho(z) \\ 0 & 1 \end{bmatrix}\begin{bmatrix} 1 & -c_0^+(z;\mu) \\ 0 & 1 \end{bmatrix} \\
    &= Y_n^-(z;\mu) \begin{bmatrix} 1 & \rho(z) - c_0^+(z;\mu) \\ 0 & 1 \end{bmatrix} = \tilde Y_n^-(z;\mu) \begin{bmatrix} 1 & c_0^-(z;\mu) + \rho(z) - c_0^+(z;\mu) \\ 0 & 1 \end{bmatrix}.
  \end{align*}
  For $z \in \cup_j [\mathtt a_j, \mathtt b_j]$ {the inversion formula holds \cite{MR2567175},} i.e., 
  \begin{align*}
    c_0^+(z;\mu) - c_0^-(z;\mu) = \rho(z),
  \end{align*}
  and therefore $\tilde Y_n$ has a trivial jump on $\cup_j (\mathtt a_j, \mathtt b_j)$.  Next, using (\ref{eq:def_Y}) and residue theorem, we check the residues of $\tilde Y_n(z;\mu)$
  \begin{align*}
    \mathrm{Res}_{z = \mathtt c_j} \tilde Y_n(z;\mu) &= \mathrm{Res}_{z = \mathtt c_j} Y_n(z;\mu)\begin{bmatrix} 1 & -c_0(z;\mu) \\ 0 & 1 \end{bmatrix} \\
&     = \begin{bmatrix} 0 & \mathrm{Res}_{z = \mathtt c_j} ( -c_0(z;\mu) (Y_n(z;\mu))_{11}  + (Y_n(z;\mu))_{12})\\
      0 & \mathrm{Res}_{z = \mathtt c_j} ( -c_0(z;\mu) (Y_n(z;\mu))_{21}  + (Y_n(z;\mu))_{22})
    \end{bmatrix}\\
    & = 0.
  \end{align*}
  We conclude that $\tilde Y_n(z;\mu)$ must be analytic inside $\Gamma$ and satisfies
  \begin{align}
    \tilde Y^+_n(z;\mu) &=  \tilde Y_n^-(z;\mu)\begin{bmatrix} 1 & -c_0(z;\mu) \\ 0 & 1 \end{bmatrix}, \quad z \in \Gamma, \label{eq_tildeyone} \\
    \tilde Y_n(z;\mu) &\begin{bmatrix} z^{-n} & 0 \\ 0 & z^{n} \end{bmatrix} = I + \OO(1/z), \quad z \to \infty. \label{eq_tildeytwo}
  \end{align}

  As we will see in the next section, it is convenient to consider
  \begin{align}\label{eq_frackcdefinition}
    \check Y_n(z;\mu) = \mathfrak c^{(n-p) \sigma_3} \tilde Y_n(z;\mu), \quad \mathfrak c \in \mathbb C \setminus\{0\}, \quad \sigma_3 = \begin{bmatrix} 1 & 0 \\ 0 & -1 \end{bmatrix},
  \end{align}
where $\mathfrak{c}$ is closely related to the capacity of $\cup_i[\mathtt a_i, \mathtt b_i]$ and formally defined in (\ref{eq_capcity}) after necessary notation is introduced. Note that the above modification does not affect the jump satisfied by $\check Y_n$, only its asymptotics. 

To connect the two measures, $\mu$ and $\nu,$ we consider
  \begin{align}\label{eq_novelquantity}
    X_n(z;\mu,\nu) =  \check Y_n(z;\nu)  \check Y_n(z;\mu)^{-1},
  \end{align}
  where we note that $\det \check Y_n(z;\mu) \equiv 1$.  Using (\ref{eq_tildeyone}) and (\ref{eq_tildeytwo}), by an elementary calculation,
    \begin{align*}
    X^+_n(z;\mu,\nu) &=  X_n^-(z;\mu,\nu)J_n(z;\mu,\nu), \ z \in \Gamma; \ \text{and} \ \ X_n(z;\mu,\nu)  = I + \OO(1/z), \ z \to \infty,
  \end{align*}
  where $J_n(z; \mu, \nu)$ is defined as 
  \begin{equation*}
  J_n(z; \mu, \nu):=\left[ I + c_0(z,\mu - \nu)  \check Y_n^-(z;\mu) \begin{bmatrix} 0 & 1\\ 0 & 0 \end{bmatrix} \check Y_n^-(z;\mu)^{-1}\right]. 
  \end{equation*}
%  \begin{align*}
%    X^+_n(z;\mu,\nu) &=  X_n^-(z;\mu,\nu)\underbrace{\left[ I + \left( c_0(z,\mu) - c_0(z;\nu) \right)  \check Y_n^-(z;\mu) \begin{bmatrix} 0 & 1\\ 0 & 0 \end{bmatrix} \check Y_n^-(z;\mu)^{-1}\right]}_{J_n(z;\mu,\nu)}, \quad z \in \Gamma,\\
%    X_n(z;\mu,\nu) & = I + O(1/z), \quad z \to \infty.
%  \end{align*}

  %\begin{align*}
  %  \check Y_n(z;\nu)  \check Y_n(z;\mu)^{-1} &= \mathfrak c^{(p-n) \sigma_3} \tilde Y_n(z;\nu)  \tilde Y_n(z;\mu)^{-1}\mathfrak c^{(n-p) \sigma_3}\\
%                                              & = \mathfrak c^{(p-n) \sigma_3} (I + O(1/z)) \begin{bmatrix} z^n & 0 \\ 0 & z^{-n} \end{bmatrix}   \tilde Y_n(z;\mu)^{-1}\mathfrak c^{(n-p) \sigma_3}\\
%    &= \mathfrak c^{(p-n) \sigma_3} (I + O(1/z))\mathfrak c^{(n-p) \sigma_3} = I + O(1/z).
%  \end{align*}

  Now, suppose that $\Gamma = \Gamma(N)$, $\nu = \nu(N)$ and $n = n(N)$ depend on a common asymptotic parameter $N$.  The Riemann--Hilbert problem for $X_n$ can be reformulated as a singular integral equation for a new unknown $U_n$ defined on $\Gamma$ using the representation
  \begin{align*}
    X_n(z;\mu,\nu) = I + \mathcal C_{\Gamma} U_n(z; \mu, \nu), \quad \mathcal C_\Gamma U (z) := \frac{1}{2 \pi \I} \int_{\Gamma} \frac{U(z')}{z' -z} \sd z'.
  \end{align*}
  
  \begin{proposition}\label{prop:perturb}
    For an integer $N$, suppose $\Gamma = \Gamma(N)$ is a piecewise smooth, simple, closed curve that encricles $\mathrm{supp}(\mu) \cup \mathrm{supp}(\nu)$ such that the operator norm of $\mathcal C^-_\Gamma$ on $L^2(\Gamma)$ is bounded by $C_N$.  Suppose $n = n(N)$ and $\nu = \nu(N)$ are functions of $N$ such that as $N\to \infty$, $C_N\|J_n - I\|_{L^\infty(\Gamma)} \to 0$.  Then we have 
    \begin{align}
    X_n(z;\mu,\nu) &= I + \frac{1}{2 \pi \I} \int_{\Gamma} \frac{c_0(z';\mu-\nu) M_n(z';\mu)}{z' - z} \sd z' + \OO\left(C_N\frac{\|J_n-I\|_{L^\infty(\Gamma)}^2}{1 + |z|} \right),\label{eq:Xn}\\
    M_n(z;\mu) &= \check Y_n^-(z;\mu) \begin{bmatrix} 0 & 1\\ 0 & 0 \end{bmatrix} \check Y_n^-(z;\mu)^{-1},\label{eq:Mn}
  \end{align}
  uniformly on subsets of $\mathbb C$ bounded uniformly away from $\Gamma$.
  \end{proposition}
  \begin{proof}
Define the boundary-value operator,
  \begin{align*}
    \mathcal C_\Gamma^\pm U (z) = \lim_{z' \to z} \mathcal C_{\Gamma} U(z'),
  \end{align*}
  where the limit is taken non-tangentially within the interior ($+$) or exterior ($-$) of $\Gamma$.  Then $U_n$ must satisfy
  \begin{align*}
    U_n - \mathcal C^-_{\Gamma} U_n (J_n - I) = J_n - I.
  \end{align*}
   This is a near-identity operator equation for $N$ sufficiently large and it can therefore be solved by a Neumann series.  In particular,
  \begin{align*}
    \|U_n - (J_n - I)\|_{L^2(\Gamma)} = \OO(C_N\|J_n-I\|_{L^\infty(\Gamma)}^2),
  \end{align*}
  which implies the conclusion.
  \end{proof}

  \begin{remark}
Proposition \ref{prop:perturb} establishes the perturbation for orthogonal polynomials generated by two close measures using the quantity (\ref{eq_novelquantity}). In particular, let
\begin{align}\label{eq_p(zndefn)}
  P(z;n) = X_n(z;\mu,\nu) - I.%= \frac{1}{2 \pi \I} \int_{\Gamma} \frac{c_0(z';\mu-\nu) M_n(z';\mu)}{z' - z} \sd z' + \OO\left(\frac{\|J_n-I\|_{L^\infty(\Gamma)}^2}{1 + |z|} \right).
\end{align}
Using (\ref{eq:Xn}) and the definition (\ref{eq_novelquantity}), we readily see that 
\begin{align}\label{eq_pertubationtheoryofOP}
   \pi_n(z;\nu) = \pi_n(z;\mu) ( 1 + P_{11}(z;n)) +  \mathfrak c^{2(p-n)} \gamma_{n-1}(\mu) \pi_{n-1}(z;\mu) P_{12}(z;n),\\
   c_n(z;\nu) = c_n(z;\mu) ( 1 + P_{11}(z;n)) + \mathfrak c^{2(p-n)} \gamma_{n-1}(\mu) c_{n-1}(z;\mu) P_{12}(z;n), \notag
\end{align}
where $P_{ij}$ is the $(i,j)$ entry of $P.$  If the two measures are close, the functions $P_{ij}$ will decay so that, to leading order, $\pi_n(z;\nu)$ and $c_n(z;\nu)$ are given by $\pi_n(z;\mu)$ and $c_n(z;\mu)$, as expected. The above results may depend on the choice of the contour $\Gamma.$ In the current paper, we will choose $\Gamma$ to be the boundary of a rectangle and $\|\mathcal C^-_\Gamma\|_{L^2(\Gamma)}$ is bounded by an absolute constant \cite{Bottcher1997}.
  \end{remark}

%{\color{red}
%\begin{remark}
%Let
%
%\end{remark}
%}

 % Check:
 % \begin{align*}
 %   X_n^+(z;\mu,\nu) &= \check Y_n(z;\nu)^+ [Y_n(z;\mu)^+]^{-1} = \check Y_n(z;\nu)^- \begin{bmatrix} 1 & - c_0(z;\nu) \\ 0 & 1 \end{bmatrix}  \begin{bmatrix} 1 & c_0(z;\mu) \\ 0 & 1 \end{bmatrix} [Y_n(z;\mu)^-]^{-1}\\
  %  &= \check Y_n(z;\nu)^- [\check Y_n(z;\mu)^-]^{-1} \check Y_n(z;\mu)^- \begin{bmatrix} 1 & c_0(z;\mu) - c_0(z;\nu)\\  0 & 1 \end{bmatrix}[\check Y_n(z;\mu)^-]^{-1}.
  %\End{Align*}

\subsection{Large $n$ asymptotics of polynomials orthogonal with respect to measures supported on multiple intervals}

Recall (\ref{eq:def_Y}).  In order to directly compare the orthogonal polynomial $\pi_n(x;\nu)$ to $\pi_n(x;\mu)$ one needs (1) an estimate on $M_n(z;\mu)$ in  \eqref{eq:Mn}.  Furthermore, supposing that $J_n - I \to 0$, one is left with  
\begin{align*}
  \check Y_n(z;\nu) = X_n(z;\mu,\nu) \check Y_n(z;\mu)  = (I + \oo(1))\check Y_n(z;\mu) .
\end{align*}
And so, one needs (2) some information about $\check Y_n(z;\mu)$ to make conclusions about $\check Y_n(z;\nu)$.  One such way to accomplish (1) and (2) is to compute the large $n$ asymptotics of $Y_n(z;\mu)$.  The calculations rely on solving another Riemann--Hilbert problem and this is accomplished in Appendix~\ref{app:OPs}. We summarize the results in Theorem \ref{lem_deterministicexpansion} below. The result relies on the following regularity assumption. 

\begin{assum}\label{assum_measure}
Consider a probability measure $\mu$ that satisfies the following assumptions. 
\begin{enumerate}
\item Square-root behavior with spikes: The measure $\mu$ is of the form\footnote{One can include inverse square-roots if needed, but this requires incorporating additional conditions into the Riemann--Hilbert problem to ensure unique solvability.}
  \begin{align}\label{eq:mu}
     \mu(\sd \lambda) = \underbrace{\sum_{j=1}^{g+1} h_j(\lambda) \one_{[\mathtt a_j,\mathtt b_j]}(\lambda) \sqrt{(\mathtt b_j - \lambda)(\lambda - \mathtt a_j)}}_{\rho(\lambda)} \sd \lambda + \sum_{j=1}^p w_j \delta_{\mathtt c_j}(\sd \lambda),
  \end{align}
  for disjoint intervals $[\mathtt a_j,\mathtt b_j]$ and points $\mathtt c_j$ located away from these intervals. 
\item Uniformity (1):  We allow $\mu$ to depend implicitly on a parameter $N$ but require that $g,p$ be non-negative, constant (for sufficiently large $N$) and require that the distance between any two points in the set $\{\mathtt c_j\} \cup \{\mathtt a_j\} \cup\{\mathtt b_j\}$ is bounded above and below.
\item Analyticity:  To each interval $[\mathtt a_j,\mathtt b_j]$ we associate a bounded open set $\Omega_j$ (independent of $N$) containing $[\mathtt a_j,\mathtt b_j]$ for all $N$ such that $h_j$ has an analytic continuation to $\Omega_j$.
\item Uniformity (2):  We suppose there is an absolute constant $D \geq 1$ such that
  \begin{align*}
    \sup_{z \in \Omega_j} \max\{ |h_j(z)|, |h_j(z)|^{-1}\} \leq D, 
  \end{align*}
  for every $1 \leq j \leq g+1$.
\item Uniformity (3): For every $j$, we assume that either $N^{-\sigma}/D \leq |w_j| \leq D$, $0 \leq \sigma < \infty$ or $w_j=0.$ 
\end{enumerate}
\end{assum}

We point out that the limiting ESDs and VESDs for many commonly studied random matrix models satisfy Assumption \ref{assum_measure}. We refer the readers to Lemma \ref{lem_property} and the discussion below for more details on this. Now we state the results. Let $D_j$ be a small region containing $[\mathtt a_j, \mathtt b_j]$ and let $\mathring \Sigma_j$ be a small ball that has $\mathtt c_j$ as its center. Then we define a function $f$ as follows,
\begin{align*}
  f(z) = \begin{cases} \pm 1/\check \rho_j(z) & z \in D_j \cap \{\pm \Im z > 0\},\\
      \frac{\tilde w_j}{z - \mathtt c_j} & z \in \mathring \Sigma_j,\\
      0 & \text{otherwise}, \end{cases} 
\end{align*}
where $\tilde w_j$ is defined in (\ref{eq_tildewjdefn}) and $\check \rho_j$  is defined in Section \ref{subsubsec_lenproblem} after necessary notation is introduced. Since $D_j$ and  $\mathring \Sigma_j$ can be chosen to be well separated according to Assumption \ref{assum_measure},  we will see in Section \ref{subsubsec_lenproblem} that their choices will not influence our results much.   The function $f$ here captures the fact that the asymptotics for orthogonal polynomials away from the support of $\mu$ is different from the asymptotics on or near the support.

\begin{theorem}\label{lem_deterministicexpansion} Suppose Assumption \ref{assum_measure} holds for $\mu = \mu(N)$ for sufficiently large $N.$ Let $Y_n(z;\mu)$ be as (\ref{eq:def_Y}) and recall $\mathfrak{c}, \sigma_3$ in (\ref{eq_frackcdefinition}).  Then for some constant $c>0$ 
\begin{align}\label{eq:Yn}
     Y_n(z;\mu)& = \mathfrak c^{(p-n) \sigma_3}\left( I + \OO \left( \frac{\e^{-cn}}{1 + |z|}\right) \right)K_n(z,\mu)  \e^{\varphi_n(z) \sigma_3} \left( \prod_{j=1}^p (z- \mathtt c_j) \right)^{\sigma_3} \notag \\
               &+ f(z) \mathfrak c^{(p-n) \sigma_3}\left( I + \OO \left( \frac{\e^{-cn}}{1 + |z|} \right) \right)K_n(z,\mu)  \e^{\varphi_n(z) \sigma_3} \begin{bmatrix} 0 & 0 \\ 1 & 0 \end{bmatrix} \left( \prod_{j=1}^p (z- \mathtt c_j) \right)^{\sigma_3}. 
\end{align}
Here we used the notation
\begin{align}
  K_n(z,\mu) &= \e^{-\sigma_3 \sG(\infty)} L_n(\infty)^{-1} L_n(z), \label{eq_knz}\\
  \varphi_n(z) &= \sG(z) + (n-p) \mathfrak g(z), \label{eq_varphiz}
\end{align}
where $\sG(z)$ is defined (\ref{eq_definG}), $L_n(z)$ is defined in (\ref{eq:L}) and $\mathfrak{g}(z)$ is defined in Section \ref{subsubsec_differential}, after some necessary notation is introduced. 
\end{theorem}
\begin{proof}
See Appendix \ref{app:OPs}. 
\end{proof}

The function $\mathfrak g(z)$, as defined in Section~\ref{subsubsec_differential}, is classically known as the exterior Green's function with pole at $\infty$, see \cite{Peherstorfer}, for example.  It expresses the global distribution of the zeros of the orthogonal polynomials.  The function $\sG(z)$ is an instance of a so-called Szeg\H{o} function \cite{MR2087231}.  For the definition of $L_n$ see \eqref{eq_LninLn}.  
  
For the reader's convenience, in Appendix \ref{sec_detailedexpression}, we provide more detailed expressions for the entries of $Y_n(z;\mu).$ Theorem \ref{lem_deterministicexpansion} has many important consequences. For example, it can be used to study the asymptotics of the three-term recurrence coefficients of the orthogonal polynomials (see Section \ref{sec_asymptoticsthreeterm}), the residuals and errors of conjugate gradient algorithm (see Section \ref{sec_pertubationcga}) and the Cholesky factorization of the tridiagonalization (see Section \ref{sec_choselec}). We will discuss these applications and provide explicit formulae in Section \ref{sec_asymptotics}.    

Equipped with the above theorem, we now proceed to accomplish the aforementioned goals (1) and (2) on some specifically chosen contour $\Gamma$. In sequel, unless otherwise specified, we will consistently use the following contour. For some small constant $\eta>0,$ let $\Gamma_j$ be the rectangle that is a distance $\eta$ from $[\mathtt a_j,\mathtt b_j]$, i.e., 
\begin{align}\label{eq:Sigmaj}
  \Gamma_j = \Gamma_j(\eta) &= \left( [\mathtt a_j-\eta,\mathtt b_j +\eta ] + \I \eta \right) \cup \left( [\mathtt a_j- \eta,\mathtt b_j + \eta ] - \I \eta\right) \\
  & \cup \left( \mathtt b_j + \eta + \I [ -\eta, \eta]\right)  \cup \left( \mathtt a_j - \eta + \I [ -\eta, \eta]\right). \notag 
\end{align}
The following lemma accomplishes (1) by providing an estimate on $M_n(z;\mu)$ in (\ref{eq:Mn}). For definiteness, we consider the matrix norm $\|A \|_{\infty}=\max_{ij}|A_{ij}|.$ 
%{\color{blue}[start from here]}
\begin{lemma}\label{lem_steponecontrol}
Suppose Assumption \ref{assum_measure} holds. On $\Gamma_j$ in (\ref{eq:Sigmaj}), we have
\begin{align}\label{eq_MN(zmu)}
  \|M_n(z;\mu)\|_{\infty} \leq C \eta^{-1} \e^{C' n \eta^{1/2}},
\end{align}
for constants $C,C'>0$.
\end{lemma}
\begin{proof}
We start by preparing some basic estimates. 
First, on $\Gamma_j,$ according to Assumption \ref{assum_measure}, we have 
\begin{align*}
C^{-1} \eta \leq  \prod_{j=1}^{g+1} |z - \mathtt a_j| \leq C, \quad C^{-1} \eta \leq  \prod_{j=1}^{g+1} |z - \mathtt b_j| \leq C,
\end{align*}
for an absolute constant $C > 0$. Second, using (\ref{eq:L}) and (\ref{eq_gammaz}) together with (\ref{eq_G(z)}), we see from (\ref{eq_knz}) that for $z \in \Gamma_j$,
\begin{align*}
  \|K_n(z;\mu)\|_{\infty} &\leq C(|z - \mathtt a_j|^{-1/4} + |z - \mathtt b_j|^{-1/4}), \\
  \|K_n(z;\mu)^{-1}\|_{\infty} &\leq C(|z - \mathtt a_j|^{-1/4} + |z - \mathtt b_j|^{-1/4}), 
\end{align*}
for some absolute constant $C>0$. Third, to estimate $\mathfrak g(z)$ in the upper-half plane, we first note that\footnote{Here $^+$ denotes the limit from within the upper-half plane.} $\Re \mathfrak g^+(z) = 0$ for $z \in [\mathtt a_j,\mathtt b_j]$ for any $j$. According to the arguments of Section \ref{subsubsec_differential}, we find that there exists some $D > 0$ such that $|Q_g(z)| \leq D$ (recall (\ref{eq_gprimedefinition})) on $\cup_j \Gamma_j$ which implies that for $z \in \Gamma_j$
\begin{align*}
  \Re \mathfrak g(z) \leq D' \mathrm{dist}(z,[\mathtt a_j,\mathtt b_j])^{1/2} \leq 2^{1/4} D' \eta^{1/2},
\end{align*}
for a new absolute constant $D'>0$.

% We calculate
% \begin{align*}
%   M_n(z;\mu) = \begin{bmatrix} -\check Y_{11}(z;\mu) \check Y_{21}(z;\mu) & \check Y_{11}(z;\mu)^2 \\ \check Y_{21}(z;\mu) & \check Y_{11}(z;\mu) \check Y_{21}(z;\mu) \end{bmatrix}.
% \end{align*}
% Then we use \eqref{eq_Yn11} and \eqref{eq_Yn21} to compute
% \begin{align*}
% \end{align*}

Next we estimate $M_n(z;\mu).$ Inserting (\ref{eq:Yn}) into (\ref{eq:Mn}), we obtain
\begin{align*}
  M_n(z;\mu) &= %K_n(z;\mu) \left( I + O \left( {e^{-cn}}\right) \right) \e^{\varphi_n(z) \sigma_3} \prod_{j=1}^p (z-c_j)^{\sigma_3} \begin{bmatrix} 0 & 1 \\ 0 & 0 \end{bmatrix} \prod_{j=1}^p (z-c_j)^{-\sigma_3}\e^{-\varphi_n(z) \sigma_3} \left( I + O \left( {e^{-cn}}\right) \right)K_n(z;\mu)^{-1} \\
 \left[ \e^{2 \varphi_n(z)} \prod_{j=1}^p (z-\mathtt c_j)^{2} \right] \left( I + \OO \left( {e^{-cn}}\right)\right) K_n(z;\mu)  \begin{bmatrix} 0 & 1 \\ 0 & 0 \end{bmatrix}  K_n(z;\mu)^{-1}\left( I + \OO \left( {e^{-cn}}\right)\right) \\
             %& +  f(z) \left[ \prod_{j=1}^p (z-c_j)^{2} \right] K_n(z;\mu) \left( I + O \left( {e^{-cn}}\right) \right) \e^{\varphi_n(z) \sigma_3} \begin{bmatrix} 0 & 0 \\ 1 & 0 \end{bmatrix} \begin{bmatrix} 0 & 1 \\ 0 & 0 \end{bmatrix} \e^{-\varphi_n(z) \sigma_3}\left( I + O \left( {e^{-cn}}\right) \right)K_n(z;\mu)^{-1} \\
             & +  f(z) \left[ \prod_{j=1}^p (z-\mathtt c_j)^{2} \right] \left( I + \OO \left( {e^{-cn}}\right) \right)K_n(z;\mu)   \begin{bmatrix} 0 & 0 \\ 0 & 1 \end{bmatrix} K_n(z;\mu)^{-1}\left( I + \OO \left( {e^{-cn}}\right) \right) \\
             & -  f(z) \left[ \prod_{j=1}^p (z-\mathtt c_j)^{2} \right] \left( I + \OO \left( {e^{-cn}}\right) \right)K_n(z;\mu)  \begin{bmatrix} 1 & 0 \\ 0 & 0 \end{bmatrix} K_n(z;\mu)^{-1} \left( I + \OO \left( {e^{-cn}}\right) \right)\\
  & -  f(z)^2 \left[  \e^{-2 \varphi_n(z)}\prod_{j=1}^p (z-\mathtt c_j)^{2} \right]\left( I + \OO \left( {e^{-cn}}\right) \right)  K_n(z;\mu)  \begin{bmatrix} 0 & 0 \\ 1 & 0 \end{bmatrix} K_n(z;\mu)^{-1}\left( I + \OO \left( {e^{-cn}}\right) \right). 
\end{align*}
Using Lemma~\ref{l:Gsing} we estimate for $z \in \Gamma_j$
\begin{align*}
  |\e^{2 \varphi_n(z)}| \|K_n(z;\mu)\|_{\infty}\|K_n(z;\mu)^{-1}\|_{\infty} &\leq C |z - \mathtt b_j|^{-1}|z - \mathtt a_j|^{-1},\\
  |f(z)| \|K_n(z;\mu)\|_{\infty}\|K_n(z;\mu)^{-1}\|_{\infty} &\leq C |z - \mathtt b_j|^{-1}|z - \mathtt a_j|^{-1},\\
  |f(z)|^2|\e^{-2 \varphi_n(z)}| \|K_n(z;\mu)\|_{\infty}\|K_n(z;\mu)^{-1}\|_{\infty}&\leq C |z - \mathtt b_j|^{-1}|z - \mathtt a_j|^{-1},
\end{align*}
for a new constant $C$.  The lemma follows.

%This completes our proof using the established bounds. 
%Based on the definitions of $M_n(z;\mu)$ and the contour $\Gamma,$ we readily obtain the following lemma.
\end{proof}

Armed with Lemma \ref{lem_steponecontrol}, we are ready to state a more detailed asymptotic result on the perturbation of orthogonal polynomials when Assumption \ref{assum_measure} holds.

\begin{theorem}\label{t:main_OP}
  Let $N$ be a positive integer and suppose $\mu = \mu(N)$ satisfies Assumption~\ref{assum_measure} for sufficiently large $N$.  Suppose further that a measure $\nu = \nu(N)$ is such that
  \begin{align*}
    \nu - \sum_{j=1}^p w_j \delta_{\mathtt c_j},
  \end{align*}
  has its support inside $\Gamma = \Gamma(\eta) = \bigcup_j \Gamma_j(\eta)$, as defined in \eqref{eq:Sigmaj}, and $\|c_0(\cdot, \mu - \nu)\|_{L^\infty(\Gamma)} \leq E(N,\eta)$.   If $n \leq C\eta^{-1/2}$, $C > 0$, and $\eta = \eta(N)$ is such that $E(N,\eta) \eta^{-1/2} \to 0$ as $N \to \infty$ then  Proposition \ref{prop:perturb} holds. In particular, we have  
%  \begin{align*}
%    Y_n(z;\nu) = X_n(z;\mu,\nu) Y_n(z;\mu),
%  \end{align*}
%  where $Y_n(z;\mu)$ is given by \eqref{eq:Yn} and
  \begin{align*}
    X_n(z;\mu,\nu) &= I + \frac{1}{2 \pi \I} \int_{\Gamma} \frac{c_0(z';\mu-\nu) M_n(z';\mu)}{z' - z} \sd z' +\OO\left(\frac{E(N,\eta)^2 \eta^{-1}}{1 + |z|} \right),
  \end{align*}
  uniformly for $z$ in sets bounded away from $\Gamma$.
\end{theorem}
\begin{proof}
The proof follows directly from Theorem \ref{lem_deterministicexpansion}, Lemma \ref{lem_steponecontrol} and Proposition \ref{prop:perturb}. 
\end{proof}

\begin{remark}\label{rmk_divergent}
Theorem \ref{t:main_OP} makes precise the fact that in order to let $X_n(z;\mu, \nu)$ be close to $I,$ we will need $c_0(z;\mu-\nu)$ to be small.  The sense in which this occurs depends on each specific problem and the related application. In applications of random matrix theory, for most of the commonly encountered models, when $\mu$ is the limiting ESD or VESD and $\nu$ is the ESD or VESD,  one typically has\footnote{$X_n = \OO_{\mathbb P}(g(n))$ as $n \to \infty$ if $|c_nX_n/g(n)| \to 0$ in probability for any sequence $c_n \to 0$.}
\begin{align*}
| c_0(z;\mu-\nu) |=\OO_{\mathbb{P}}\left(\frac{1}{N \eta} \right), \ \ \text{or} \ 
  | c_0(z;\mu-\nu) |=\OO_{\mathbb{P}}\left(\frac{1}{\sqrt{N \eta}} \right),
\end{align*}
on the entirety of $\cup_j \Sigma_j$ and this will dictate what $\eta,$ or equivalently $n,$ can be. Consequently, we have
\begin{equation}\label{eq_etachoice}
  n = \OO(\eta^{-1/2}), \ \text{where} \ n \ll N^{1/4} \ \text{for ESD and} \ n \ll N^{1/6} \ \text{for VESD},
\end{equation}
is required to be able to apply Theorem~\ref{t:main_OP}. We also point out that if $\mu$ has spikes, then $\nu$ will have spikes near the spikes of $\mu$.  Instead of directly considering $\mu - \nu$ we apply Theorem~\ref{t:main_OP} to $\tilde \mu - \nu$ where the limiting spikes of $\mu$ are replaced with the nearby random spikes of $\nu$.  Despite the fact that $\tilde \mu$ is then random, it satisfies Assumption~\ref{assum_measure} with high probability and the asymptotics of the associated orthogonal polynomials follow the same form, see Remark~\ref{rem_explicitexpansionformula} below. 
\end{remark}

\begin{remark}\label{rem_explicitexpansionformula}
Combining Theorems \ref{lem_deterministicexpansion} and \ref{t:main_OP}, we can provide a more detailed perturbation formulae for the orthogonal polynomials compared to (\ref{eq_pertubationtheoryofOP}). In particular, inserting (\ref{eq:Yn}) (or equivalently the expressions in Appendix \ref{sec_detailedexpression}) into (\ref{eq_pertubationtheoryofOP}), we obtain that for 
$z$ bounded away from $\Gamma,$
\begin{align}\label{eq:explicit_pin}
  \pi_n(z;\nu) & = \mathfrak c^{(p-n)} \e^{(n-p) \mathfrak g(z) + \sG(z) - \sG(\infty)} \\ \notag
  & \times \left[ \prod_{j=1}^p (z - \mathtt c_j) \right] \left[( 1 + P_{11}(z;n)) E_{11}(z;n) + P_{12}(z;n) \e^{ 2 \sG(\infty)} E_{21}(z;n) \right],\\\notag
  c_n(z;\nu) & = \mathfrak c^{(p-n)} \e^{-(n-p) \mathfrak g(z) - \sG(z) - \sG(\infty)} \\\notag
  & \times \left[ \prod_{j=1}^p (z - \mathtt c_j)^{-1} \right] \left[( 1 + P_{11}(z;n)) E_{12}(z;n) + P_{12}(z;n) \e^{ 2 \sG(\infty)} E_{22}(z;n) \right],
\end{align} 
where $E_{ij}(z;n), 1 \leq i,j \leq 2,$  defined in Appendix \ref{sec_detailedexpression} only depend on $\mu$.  Compared to (\ref{eq_pertubationtheoryofOP}), the above expressions give much more information as they give explicitly how the exponential prefactors are arranged.  
\end{remark}

\begin{remark}\label{rmk_cltdiscussion}
As can be seen from the above discussion, if $\nu$ is random then the main random quantity to be understood is the entries of $P(z;n)$ as defined in (\ref{eq_p(zndefn)}). Consequently, in order to understand the second-order fluctuation of the concerned quantities, it suffices to derive a CLT for $P(z;n).$ The main task is to understand the asymptotics of $c_0(z';\mu-\nu)$ on the contour $\Gamma$. This is usually problem-specific and depends on the measures $\mu$ and $\nu.$ Considering applications in random matrix theory where $\mu$ is the limiting distribution and $\nu$ is the empirical distribution, the distribution of $c_0(z';\mu-\nu)$, of course, depends on the underlying random matrix model. In Section \ref{sec_spikedcov}, we consider the spiked sample covariance matrix model and establish a general CLT which can be used to understand the distribution of the related quantities. 
% after being properly scaled,  has been studied for some random matrix models   
\end{remark}

\section{Algorithmic applications: Asymptotic formulae for numerical algorithms}\label{sec_algorithmicapp}
In this section, we apply the results of Section \ref{sec_RHPframework} to study several important numerical algorithms.

\subsection{A high level discussion of matrix factorizations and algorithms}\label{sec:matfac}

We briefly discuss background for the numerical algorithms under consideration.

%whose action on random matrices will be analyzed below.

%\subsubsection{Householder tridiagonalizaton} 

\subsubsection{Lanczos tridiagonalization}
We first introduce the Householder tridiagonalizaton procedure. It is the process by which a real symmetric or complex Hermitian matrix $W$ is transformed to a real symmetric tridiagonal matrix using Householder reflectors.  Householder reflectors can be written in the form
\begin{align*}
  U_k = \begin{bmatrix} I_k & 0 \\ 0 & I_{N -k} - \vec u \vec u^* \end{bmatrix},              
\end{align*}
where $I_k$ is the $k \times k$ identity matrix and $\vec u \in \mathbb C^{(N-k) \times (N-k)}$ is a unit vector.  By selecting $\vec u$ correctly for each $k$
\begin{align*}
  U_N U_{N-1} \cdots U_{1} W U_{1}^* U_{2}^* \cdots U_N^*,
\end{align*}
is a real symmetric tridiagonal matrix.  See \cite{TrefethenBau}, for example.

  The Lanczos tridiagonalization algorithm applied to a real symmetric or complex Hermitian matrix $W$ and vector $\vec b$ accomplishes the same goal as the Householder tridiagonalization algorithm with some added flexibility.  Run to completion, in exact arithmetic, the Lanczos algorithm performs Gram-Schmidt on the vectors $\{\vec b, W \vec b, \ldots , W^{N-1} \vec b\}$ constructing an orthogonal or unitary matrix
\begin{align}\label{Qformaldefinition}
  Q = \begin{bmatrix} \vec q_1 & \vec q_2  & \cdots & \vec q_N \end{bmatrix},
\end{align}
and necessarily $T = Q^* W Q$ is a tridiagonal matrix.  Note that $\vec q_1 = \vec b/\|\vec b\|_2$. It is well-known \cite{TrefethenBau} the entries in the Lanczos matrix $T$ coincides with  the three-term recurrence coefficients for the discrete orthogonal polynomials with respect to the VESD generated by $\vec b$ and $W$ (c.f.~(\ref{eq:jacobi_def})).   

\subsubsection{Cholesky factorization}  The Cholesky factorization of a positive definite matrix $W$ is a factorization $W = L L^*$ where $L$ is lower-triangular with positive diagonal entries.  When applied to a tridiagonal matrix $T$, $L$ is lower-bidiagonal and has non-negative entries if $T$ has non-negative entries. The Cholesky factorization is a special case of Gaussian elimination.

\subsubsection{The conjugate gradient algorithm}  The conjugate gradient algorithm (CGA) is an iterative method to solve the linear system $W \vec x = \vec b$.  The method begins with an initial guess $\vec x_0$ and in the current work we always take $\vec x_0 = \vec 0$.  The algorithm is mathematically described by the solution of a sequence of minimization problems:
\begin{align}\label{eq_cgamatheformulation}
  \vec x_k = \mathrm{argmin}_{\vec y \in \mathcal K_k} \| \vec y - \vec x\|_W, \quad \mathcal K_k = \mathrm{span}\{\vec b, W \vec b, \ldots , W^{k-1} \vec b\}, \quad \|\vec y \|_W^2 = \langle \vec y , W \vec y \rangle.
\end{align}
While one has the expression,
\begin{align*}
  \vec x_k = Q_k (Q_k^* W Q_k)^{-1} \vec f_1,
\end{align*}
it is quite remarkable that an extremely efficient iteration process is possible \cite{Hestenes1952}. Here $Q_k:=[\vec q_1, \cdots, \vec q_k ]$ as in (\ref{Qformaldefinition}). It is also of intrinsic mathematical interest that this process makes sense for bounded positive-definite operators on a Hilbert space.

\subsection{Unperturbed asymptotics: Applications of Theorem \ref{lem_deterministicexpansion}}\label{sec_asymptotics}
In this subsection, we consider several important consequences of Theorem \ref{lem_deterministicexpansion} when applied to the numerical algorithms in Section \ref{sec:matfac}. As we will see later,  a common feature is that the analysis of these algorithms boil down to the analysis of some functionals of the orthogonal polynomials and Cauchy transforms evaluated at either $z=0$ or $z=\infty$.  The main theorem is now stated and its consequences follow.

Based on $\{\pi_n(\lambda;\mu)\}$ in (\ref{eq_MOP}),  the orthonormal polynomials $p_n(\lambda;\mu)$, $n = 0,1,2,\ldots$, are the defined by
\begin{align*}
  p_n(\lambda;\mu) = \frac{\pi_n(\lambda;\mu)}{\| \pi_n(\cdot;\mu)\|_{L^2(\mu)}}, \quad \| \pi_n(\cdot;\mu)\|_{L^2(\mu)}^2 = \int_{\mathbb R} \pi_{n}(\lambda;\mu)^2 \mu(\sd \lambda).
\end{align*}
We write $p_n(z;\mu) = \ell_n z^n + s_n z^{n-1} + \cdots = \ell_n \pi_n(z;\mu)$ where $\ell_n = \ell_n(\mu)$ satisfies
\begin{align}\label{eq_ellndefn}
  \ell_n^{-2} = \int_{\mathbb R} \pi_n(z;\mu)^2 \mu(\sd z) = \int_{\mathbb R} \pi_n(z;\mu) z^n \mu(\sd z).
\end{align}

\begin{theorem}\label{thm_main_asympt}
  Suppose Assumption \ref{assum_measure} holds for $\mu = \mu(N)$ for sufficiently large $N$ and $n \to \infty$ as $N \to \infty$. Then for some $c > 0$ we have the following. \begin{enumerate}
    \item If $z = 0$ is bounded away from $\left( \cup_j \Omega_j\right) \cup \left( \cup_j \mathtt c_j \right)$ then\footnote{This result can be stated appropriately for any $z$ but for simplicity we just take $z = 0$ because that is all that is needed in the sequel.}
\begin{align*}
  Y_n(0;\mu)_{11} &= \mathfrak c^{(p-n)} \e^{-\sG(\infty)} \e^{\sG(0)} \e^{(n-p) \mathfrak g(0)} \left[\prod_{j=1}^p(-\mathtt c_j) \right] E_{11}(0;n),\\
  Y_n(0;\mu)_{12} &= \mathfrak c^{(p-n)} \e^{-\sG(\infty)} \e^{-\sG(0)} \e^{-(n-p) \mathfrak g(0)} \left[\prod_{j=1}^p(-\mathtt c_j)^{-1} \right] E_{12}(0;n),
\end{align*}
where
\begin{align*}
  E_{11}(0;n) &= \frac{1}{2} \left( \prod_{j=1}^{g+1} \left( \frac{ \mathtt b_j}{\mathtt a_j} \right)^{1/4} + \prod_{j=1}^{g+1} \left( \frac{ \mathtt a_j}{\mathtt b_j} \right)^{1/4}\right) \frac{\Theta_1(0;\vec d_2;(n-p)\bm{\Delta} + \bm{\zeta})}{\Theta_1(\infty;\vec d_2;(n-p)\bm{\Delta} + \bm{\zeta})} + \OO(\e^{-cn}),\\
  E_{12}(0;n) &= \frac{1}{2\I} \left( \prod_{j=1}^{g+1} \left( \frac{ \mathtt b_j}{\mathtt a_j} \right)^{1/4} - \prod_{j=1}^{g+1} \left( \frac{ \mathtt a_j}{\mathtt b_j} \right)^{1/4}\right) \frac{\Theta_2(0;\vec d_2;(n-p)\bm{\Delta} + \bm{\zeta})}{\Theta_1(\infty;\vec d_2;(n-p)\bm{\Delta} + \bm{\zeta})} + \OO(\e^{-cn}).
\end{align*}

\item And   
\begin{align*}
  \ell_n^{-2}(\mu) &=- 2 \pi \I \lim_{z \to \infty} z^{n+1} Y_n(z;\mu)_{12} \\
  &= e^{-2\sG(\infty)} \mathfrak{c}^{2(p-n)} \frac{\pi}{2} \sum_{j=1}^{g+1}(\mathtt{b}_j-\mathtt{a}_j) \frac{\Theta_2(\infty;\vec d_2;(n-p)\bm{\Delta} + \bm{\zeta})}{\Theta_1(\infty;\vec d_2;(n-p)\bm{\Delta} + \bm{\zeta})}+\OO(\e^{-cn}),\\  
  \frac{s_n(\mu)}{\ell_n(\mu)} &= \lim_{z \to \infty} z \left(z^{-n} Y_n(z;\mu)_{11} - 1 \right) \\
  &= \frac{m_{g+1}}{2 \pi \I} -  \frac{m_g}{2 \pi \I} \sum_{j=1}^{g+1} (\mathtt a_j + \mathtt b_j)  +  (n-p)\mathfrak g_1 - \sum_{j=1}^p \mathtt c_j + \frac{\Theta_1^{(1)}(\vec d_2; (n-p)\bm{\Delta} + \bm{\zeta})}{ \Theta_1(\infty;\vec d_2;(n-p)\bm{\Delta} + \bm{\zeta})} + \OO(\e^{-cn}).
\end{align*}
\end{enumerate}
Here $\mathfrak{c}$ is defined in (\ref{eq_frackcdefinition}) and $\mathfrak g_1$ is the coefficient of the $\OO(1/z)$ term in the expansion of $\mathfrak g(z)$ at $\infty$. The other quantities will be made explicit in the proof after some necessary notation is introduced. In particular, $\Theta=(\Theta_1, \Theta_2)$ is  a vector-valued function defined in (\ref{eq_vectortheta}) using the Riemann theta function (c.f.~(\ref{eq_thetafunction})), $\vec d_2$ is defined in (\ref{eq_d2}), $\bm{\Delta}$ is defined in (\ref{eq_defndelta}), the entries of $\bm{\zeta}$ are defined via (\ref{eq_zetaequation}) and $\Theta^{(1)}$ is defined in (\ref{eq_thetaone}).  
\end{theorem}
\begin{proof}
See Appendix~\ref{sec_detailedexpression}. 
\end{proof}

% in Section \ref{sec:matfac}. 
%We now derive some useful formulae under the assumption that $\mu$ satisfies Assumption~\ref{assum_measure} which are consequences of \eqref{eq:Yn}.
\subsubsection{Asymptotics of the three-term recurrence coefficients}\label{sec_asymptoticsthreeterm}
%We now use the asymptotic formulae for $Y_n(z;\mu)$ to derive some important formulae that will be of great use in what follows. 
The three-term recurrence coefficients $a_n(\mu),b_n(\mu)$, $n \geq 0$, for $(p_n(x;\mu))_{n\geq 0}$ satisfy
\begin{align}\label{eq:three-term}
  a_n(\mu) p_n(x;\mu) + b_n(\mu) p_{n+1}(x;\mu) + b_{n-1}(\mu) p_{n-1}(x;\mu) = x p_n(x;\mu), \quad n \geq 0,
\end{align}
are often organized into a Jacobi matrix:
\begin{align}\label{eq:jacobi_def}
  \mathcal J(\mu) = \begin{bmatrix} a_0 & b_0 \\
    b_0 & a_1 & b_1 \\
    & b_1 & a_2 & b_2\\
    && b_2 & a_3 & \ddots \\
    &&& \ddots & \ddots \end{bmatrix}, \ a_n = a_n(\mu), \ b_n = b_n(\mu).
\end{align}
We let $\mathcal J_n(\mu)$ denote the upper-left $n \times n$ subblock of $\mathcal J(\mu)$. The following theorem establishes the asymptotics of these coefficients. 
\begin{corollary}\label{cor_deterthreeterm}
Suppose Assumption \ref{assum_measure} holds for $\mu = \mu(N)$ for sufficiently large $N$. Then in the notation of Theorem~\ref{thm_main_asympt} we have that
\begin{align*}
  b_n(\mu)^2 &= \frac{1}{\mathfrak c^2} \frac{\displaystyle \frac{\Theta_2(\infty;\vec d_2;(n+1) \bm{\Delta} + \bm{\zeta})}{\Theta_1(\infty;\vec d_2;(n+1) \bm{\Delta} + \bm{\zeta})} + \OO(\e^{-cn})}{\displaystyle \frac{\Theta_2(\infty;\vec d_2;(n-p)\bm{\Delta} + \bm{\zeta})}{\Theta_1(\infty;\vec d_2;(n-p)\bm{\Delta} + \bm{\zeta})} + \OO(\e^{-cn})},\\
a_n(\mu) &= \frac{\Theta_1^{(1)}(\vec d_2; (n-p)\bm{\Delta} + \bm{\zeta})}{ \Theta_1(\infty;\vec d_2;(n-p)\bm{\Delta} + \bm{\zeta})} - \frac{\Theta_1^{(1)}(\vec d_2; (n+1)\bm{\Delta} + \bm{\zeta})}{ \Theta_1(\infty;\vec d_2;(n+1) \bm{\Delta} + \bm{\zeta})} + \mathfrak g_1 + \OO(\e^{-cn}).
\end{align*} 
%where $\mathfrak{c}$ is defined in (\ref{eq_frackcdefinition}) and $\mathfrak g_1$ is the coefficient of the $\OO(1/z)$ term in the expansion of $\mathfrak g(z)$ at $\infty$. The other quantities will be made explicit in the proof after some necessary notation is introduced. In particular, $\Theta=(\Theta_1, \Theta_2)$ is  a vector-valued function defined in (\ref{eq_vectortheta}) using the Riemann theta function (c.f.~(\ref{eq_thetafunction})), $\vec d_2$ is defined in (\ref{eq_d2}), $\bm{\Delta}$ is defined in (\ref{eq_defndelta}), the entries of $\bm{\zeta}$ are defined via (\ref{eq_zetaequation}) and $\Theta^{(1)}$ is defined in (\ref{eq_thetaone}).  
\end{corollary}
\begin{proof}
See Appendix \ref{sec_limitformula}. 
\end{proof}
\begin{remark}\label{eq_remarkcalculationone}
We provide a single interval example to illustrate how different quantities in the above theorem can be calculated. In the general setting, these quantities can be calculated numerically, as will be discussed in Section \ref{sec_calculationofkeyparameters}.  

 Consider that $g=0$ and $p=0$ in (\ref{eq:mu}). When $\mathtt{b}_1=1$ and $\mathtt{a}_1=-1,$ one can check from (\ref{eq_vectortheta}) that $\Theta_1=\Theta_2=1$ and $\mathfrak g_1  = 0$.  Following \cite{Peherstorfer}, $\mathfrak{c}^{-2}=\frac{1}{4}$ so that 
 \begin{equation*}
 a_n=\OO(\e^{-cn}), \ b_n=\frac{1}{2}+\OO(\e^{-cn}),
 \end{equation*}
which recovers the result of \cite{MR2087231}. For general $\mathtt{a}_1$ and $\mathtt{b}_1,$
 \begin{equation*}
 a_n=\frac{\mathtt{b}_1+\mathtt{a}_1}{2}+\OO(\e^{-cn}), \ b_n=\frac{\mathtt{b}_1-\mathtt{a}_1}{4}+\OO(\e^{-cn}),
 \end{equation*}
which matches the result of \cite{MR2087231} (see also \cite[Theorem 5.2]{DT1}).

%{\color{blue}[are there any examples we can used to give some more explicit formulas? it is completely fine for the basic MP law.]}
\end{remark}

%where the last conclusion follows from the periodicity of Riemann's theta function.

\subsubsection{Asymptotics of CGA in infinite dimensions}\label{sec_pertubationcga}
 With the help of Corollary \ref{cor_deterthreeterm}, we proceed to understand the performance of CGA (c.f.~(\ref{eq_cgamatheformulation})) to solve $W \vec x = \vec b$ with $\vec x_0 = \vec 0$, producing iterates $\vec x_n$, $n = 1,2,\ldots,$ and $\langle \vec b, (W - z)^{-1} \vec b \rangle = 2 \pi \I c_0(z;\mu)$ for a measure $\mu.$ The residual and error vectors are defined as 
\begin{equation*}
\vec r_n = \vec b - W \vec x_n, \  \vec e_n = \vec x - \vec x_n.
\end{equation*}
Then we have the following formulae, where we note that for the assumptions of the theorem to hold, $W$ must be an infinite-dimensional operator.
%{\color{red}[the proof is not completely revised]}
%If one applies the conjugate gradient algorithm then it follows that \cite{Paquette2020} the residual vector  and error vector  satisfy
\begin{corollary}\label{cor_cgadeterasymp} Suppose Assumption \ref{assum_measure} holds for $\mu = \mu(N)$ for sufficiently large $N$ and $c_0(z;\mu) = 2 \pi \I \langle \vec b, (W - z)^{-1} \vec b \rangle$. Then  
\begin{align*}
 \|\vec e_n\|_W^2 = \e^{- 2 \sG(0)} \e^{-2 (n-p) \mathfrak g(0)} \left[\prod_{j=1}^p \mathtt c_j^{-2}\right] \frac{E_{12}(0;n)}{E_{11}(0;n)},
\end{align*}
and
%\footnote{Here $\|\vec x\|_W^2 := \vec x^* W \vec x$.}
\begin{align*}
  \| \vec r_n \|_2^2 = \frac{\displaystyle \frac{\pi}{2} \sum_{j=1}^{g+1} (\mathtt b_j - \mathtt a_j) \frac{\Theta_2(\infty;\vec d_2;(n-p)\bm{\Delta} + \bm{\zeta})}{\Theta_1(\infty;\vec d_2;(n-p)\bm{\Delta} + \bm{\zeta})} + \OO(\e^{-cn})}{\displaystyle  \e^{2 (n-p) \mathfrak g(0)+2\sG(0)} \left[\prod_{j=1}^p \mathtt c_j^2\right] E_{11}(0;n)^2}.
\end{align*}
Here we recall the definitions of $G, \mathfrak{g}$ in (\ref{eq_knz}) and (\ref{eq_varphiz}), $\Theta$ in Theorem \ref{thm_main_asympt}, and $E_{11}, E_{12}$ are defined in Appendix \ref{sec_detailedexpression} after some necessary notation is introduced. 
\end{corollary}
\begin{proof}
See Appendix \ref{sec_limitformula}.
\end{proof}
\begin{remark}\label{remark_formulaone}
As in Remark \ref{eq_remarkcalculationone}, the parameters of the above formulae can be calculated numerically as in Section \ref{sec_calculationofkeyparameters}. In the single interval case, together with (\ref{eq_LninLn}) and (\ref{eq_e11}), it is remarkable to see that 
\begin{equation*}
\frac{\| \vec r_{n} \|_2^2}{\| \vec r_{n-1} \|_2^2}=\e^{-2 \mathfrak{g}(0)}+\OO(\e^{-cn}), \ \frac{\| \vec e_{n} \|_W^2}{\| \vec e_{n-1} \|_W^2}=\e^{-2 \mathfrak{g}(0)}+\OO(\e^{-cn}).
\end{equation*}
This implies that the ratios of the errors and residuals stay constant and are independent of the spikes. In fact, following the calculations in Section \ref{subsubsec_differential}, when $\mathtt{a}_1>0$ and $g = 0$ it is easy to see that $\e^{-\mathfrak{g}(0)}=(\sqrt{\mathtt{b}_1}-\sqrt{\mathtt{a}_1})/(\sqrt{\mathtt{b}_1}+\sqrt{\mathtt{a}_1}),$ which matches \cite[Theorem 3.3]{DT1}. And in comparing with \cite{Paquette2020,MR4188626} using the support $[(1-\sqrt{d})^2, (1 + \sqrt{d})^2]$ of the Marchenko--Pastur distribution with parameter $d$, $0 < d \leq 1$, one obtains, for example,
\begin{align*}
  \frac{\| \vec r_{n} \|_2^2}{\| \vec r_{n-1} \|_2^2}= d +\OO(\e^{-cn}).
\end{align*}
\end{remark}

\subsubsection{Asymptotics of the Cholesky factorization}\label{sec_choselec}

It is well known that in the case where $\mathrm{supp}(\mu) \subset (0,\infty),$ the matrix $\mathcal J(\mu)$ in (\ref{eq:jacobi_def}) has a Cholesky factorization
\begin{align}\label{eq:L_def}
  \mathcal J(\mu) = \mathcal L(\mu) \mathcal L(\mu)^*, \quad \mathcal L(\mu) = \begin{bmatrix} \alpha_0 & \\
    \beta_0 & \alpha_1  \\
    & \beta_1 & \alpha_2 \\
    && \beta_2 & \alpha_3 \\
    &&& \ddots & \ddots \end{bmatrix}, \quad \alpha_j = \alpha_j(\mu), \ \beta_j = \beta_j(\mu).
\end{align}
Let $\mathcal L_n(\mu)$ be the upper-left $n \times n$ subblock of $\mathcal L_n(\mu)$ and it is important that
\begin{align*}
  \mathcal J_n(\mu) = \mathcal L_n(\mu) \mathcal L_n(\mu)^*.
\end{align*}
The following holds. 
\begin{corollary}\label{cor_choleskeylimit}
Suppose the assumptions of Theorem \ref{lem_deterministicexpansion} hold, then we have that
\begin{align*}
  \alpha_n(\mu)^2 = - \mathfrak c^{-1} \e^{\mathfrak g(0)} \frac{E_{11}(0;n+1)}{E_{11}(0;n)},\\
  \beta_n(\mu)^2 
  %= \frac{\beta_n(\mu)^2}{b_n(\mu)^2} 
  = - \frac{\mathfrak{c}  b_n(\mu)^2}{\e^{\mathfrak g(0)}}  \frac{E_{11}(0;n)}{E_{11}(0;n+1)},
\end{align*}
where the expansion of $b_n(\mu)$ can be found in Corollary \ref{cor_deterthreeterm}. 
\end{corollary}
\begin{proof}
See Appendix \ref{sec_limitformula}.   
\end{proof}

\begin{remark}
First, as in Remark \ref{remark_formulaone}, in the single interval case $g = 0$, we can provide a more explicit formula. In this context, we have that
\begin{equation*}
\alpha_n=\frac{\sqrt{\mathtt{a}_1}+\sqrt{\mathtt{b}_1}}{2}+\OO(\e^{-cn}), \ b_n= \frac{\sqrt{\mathtt{b}_1}-\sqrt{\mathtt{a}_1}}{2}+\OO(\e^{-cn}). 
\end{equation*}
Second, according to \cite[Section 6]{Paquette2020}, we can also write 
\begin{align*}
  \frac{\| \vec r_n\|_2}{\|\vec r_{n-1}\|_2} = \frac{\beta_{n-1}}{\alpha_{n-1}}.
\end{align*}
Combining the above two formulae will recover the arguments in Remark \ref{remark_formulaone}. %{\color{blue}[maybe add a little bit more detail if necessary]} 
\end{remark}

\subsection{Perturbed formulae and perturbed asymptotics: Applications of Theorem \ref{t:main_OP}}\label{sec_subasymptotics}
In this subsection, we consider several important consequences of Theorem \ref{t:main_OP} when applied to the aforementioned numerical algorithms. In what follows, we use $\nu$ as a perturbation of the measure $\mu$ and suppose that they satisfy the assumptions of Theorem \ref{t:main_OP}.  We first state how all the quantities that are analyzed in Theorem~\ref{thm_main_asympt} are perturbed. %Each corresponding to non-perturbed asymptotics, Theorems \ref{thm_deterthreeterm}--\ref{thm_choleskeylimit}, respectively, we have the following main theorem

\begin{theorem}  For measures $\mu,\nu$ satisfying the hypotheses of Proposition~\ref{prop:perturb}
  \begin{align*}
    % Y_n(0;\nu)_{11} &= Y_n(0;\mu)_{11} ( 1 + P_{11}(0;n)) - 2 \pi \I  \frac{\mathfrak c^{2(p-n)}}{\ell_{n-1}^2(\mu)}  Y_{n-1}(0;\mu)_{11} P_{12}(0;n),\\
    Y_n(0;\nu)_{11} &= Y_n(0;\mu)_{11} ( 1 + P_{11}(0;n)) + Y_{n}(0;\mu)_{21} P_{12}(0;n),\\
    % Y_n(0;\nu)_{12} &= Y_n(0;\mu)_{12} ( 1 + P_{11}(0;n)) - 2 \pi \I  \frac{\mathfrak c^{2(p-n)}}{\ell_{n-1}^2(\mu)} Y_{n-1}(z;\mu)_{12} P_{12}(0;n),\\
    Y_n(0;\nu)_{12} &= Y_n(0;\mu)_{12} ( 1 + P_{11}(0;n)) - 2 \pi \I  \frac{\mathfrak c^{2(p-n)}}{\ell_{n-1}^2(\mu)} Y_{n}(z;\mu)_{22} P_{12}(0;n),\\
    \ell_n^{-2}(\nu) &= \ell_n^{-2}(\mu) - 2 \pi \I \mathfrak c^{2(p-n)} P_{12}^{(1)}(n),\\
    \frac{\ell_n(\nu)}{s_n(\nu)} &= \frac{\ell_n(\mu)}{s_n(\mu)} +  P_{11}^{(1)}(n),
  \end{align*}
  where the matrix $P(z;n) = P(z;n,\mu,\nu)$ is defined in \eqref{eq_p(zndefn)} and $P^{(1)}(n) = P^{(1)}(n;\mu,\nu)$ is defined by
  \begin{align}\label{eq_defnp1}
    P^{(1)}(n) = \lim_{z \to \infty}z P(z;n).
  \end{align}
\end{theorem}
\begin{proof}
  This is a direct calculation first using
  \begin{align*}
    Y_n(z;\nu) = \mathfrak c^{(n-p)\sigma_3} ( I + P(z;n))\mathfrak c^{(p-n)\sigma_3} Y_n(z;\mu),
  \end{align*}
  and expanding
  \begin{align*}
    Y_n(z;\nu)z^{-n\sigma_3} = \mathfrak c^{(n-p)\sigma_3} ( I + P(z;n))\mathfrak c^{(p-n)\sigma_3} Y_n(z;\mu)z^{-n\sigma_3},
  \end{align*}
  in a series at infinity.
\end{proof}

Since these are exact formulae, one can easily add the asymptotics of Theorem~\ref{thm_main_asympt} (adding in the formulae \eqref{eq_Yn21} and \eqref{eq_Yn22}) to create perturbed versions of Corollaries~\ref{cor_deterthreeterm}, \ref{cor_cgadeterasymp}, and \ref{cor_choleskeylimit}.  We summarize this in the following theorem.

% {\color{blue}[all these results are new. Please check. ]} 
\begin{theorem}\label{thm_pertubed} Suppose the assumptions of Theorem \ref{t:main_OP} hold.

 \begin{enumerate}
 \item For the three-term recurrence coefficients, corresponding to Corollary~\ref{cor_deterthreeterm}, we have 
 \begin{align*}
  b_n(\nu)^2= \frac{1}{\mathfrak c^2} \frac{\displaystyle \frac{\pi}{2} \sum_{j=1}^{g+1}(\mathtt{b}_j-\mathtt{a}_j)\frac{\Theta_2(\infty;\vec d_2;(n+1) \bm{\Delta} + \bm{\zeta})}{\Theta_1(\infty;\vec d_2;(n+1) \bm{\Delta} + \bm{\zeta})}+P_{12}^{(1)}(n+1) \e^{2 G(\infty)} + \OO(\e^{-cn})}{\displaystyle \frac{\pi}{2} \sum_{j=1}^{g+1}(\mathtt{b}_j-\mathtt{a}_j)\frac{\Theta_2(\infty;\vec d_2;(n-p)\bm{\Delta} + \bm{\zeta})}{\Theta_1(\infty;\vec d_2;(n-p)\bm{\Delta} + \bm{\zeta})} + P_{12}^{(1)}(n) \e^{2 G(\infty)}+\OO(\e^{-cn})},
\end{align*}
and 
\begin{align*}
  a_n(\nu) &=\frac{\Theta_1^{(1)}(\vec d_2; (n-p)\bm{\Delta} + \bm{\zeta})}{ \Theta_1(\infty;\vec d_2;(n-p)\bm{\Delta} + \bm{\zeta})} - \frac{\Theta_1^{(1)}(\vec d_2; (n+1)\bm{\Delta} + \bm{\zeta})}{ \Theta_1(\infty;\vec d_2;(n+1) \bm{\Delta} + \bm{\zeta})}\\
  &+\mathfrak g_1+P_{11}^{(1)}(n)-P_{11}^{(1)}(n+1) + \OO(\e^{-cn}),
\end{align*}
where the matrix $P^{(1)}$ is defined in (\ref{eq_defnp1}).% after the necessary notation is introduced. 
\item For CGA, corresponding to Corollary \ref{cor_cgadeterasymp}, we have 
\begin{align}
& \| \mathbf{e}_n \|_W^2 = \e^{-2 (n-p) \mathfrak g(0)- 2 \sG(0)} \left[\prod_{j=1}^p \mathtt c_j^{-2}\right] \frac{(1 + P_{11}(0;n))E_{12}(0;n) + P_{12}(0;n) \e^{2 \sG(\infty)} E_{22}(0;n)}{( 1 + P_{11}(0;n)) E_{11}(0;n) + P_{12}(0;n) \e^{ 2 G(\infty)} E_{21}(0;n)}, \label{eq:res-formula} \\
& \|\mathbf{r}_n \|_2^2=\frac{\displaystyle \frac{\pi}{2} \sum_{j=1}^{g+1} (\mathtt b_j - \mathtt a_j) \frac{\Theta_2(\infty;\vec d_2;(n-p)\bm{\Delta} + \bm{\zeta})}{\Theta_1(\infty;\vec d_2;(n-p)\bm{\Delta} + \bm{\zeta})} +\frac{2 \pi}{\I} P_{12}^{(1)}(n) \e^{2\sG(\infty)}+ \OO(\e^{-cn})}{\displaystyle  \e^{2 (n-p) \mathfrak g(0)+2\sG(0)} \left[\prod_{j=1}^p \mathtt c_j^2\right]\left[( 1 + P_{11}(0;n)) E_{11}(0;n) + P_{12}(0;n) \e^{ 2 \sG(\infty)} E_{21}(0;n) \right]^2 } .\label{eq:error-formula}
\end{align}
\item For the Cholesky factorization, corresponding to Corollary \ref{cor_choleskeylimit}, we have 
 \begin{align*}
  \alpha_n(\nu)^2=- \mathfrak c^{-1} \e^{\mathfrak g(0)} \frac{( 1 + P_{11}(0;n+1)) E_{11}(0;n+1) + P_{12}(0;n+1) \e^{ 2 G(\infty)} E_{21}(0;n+1) }{( 1 + P_{11}(0;n)) E_{11}(0;n) + P_{12}(0;n) \e^{ 2 G(\infty)} E_{21}(0;n) },\\
  \beta_n(\nu)^2=- \mathfrak c \e^{-\mathfrak g(0)} b_n(\nu)^2 \frac{( 1 + P_{11}(0;n)) E_{11}(0;n) + P_{12}(0;n) \e^{ 2 G(\infty)} E_{21}(0;n) }{( 1 + P_{11}(0;n+1)) E_{11}(0;n+1) + P_{12}(0;n+1) \e^{ 2 G(\infty)} E_{21}(0;n+1) }. 
\end{align*}
\end{enumerate}
\end{theorem}
%\begin{proof}
%See Appendix \ref{sec_proofpertubed}. 
%\end{proof}
We do not present the formulae for $E_{12}(0;n)$ and $E_{22}(0;n)$ explicitly, but these can be found in Section \ref{sec_detailedexpression}.

\begin{remark}
This theorem is particularly important because our asymptotic formulae in the previous section only hold when $\mu$ satisfies Assumption~\ref{assum_measure} which corresponds to running CGA on an infinite-dimensional system.  But $\nu$ can arise as a VESD of a finite-dimensional system which allows Theorem~\ref{thm_pertubed} to apply to (large) finite-dimensional linear algebra computations.
  
Also, as in Remark \ref{rmk_cltdiscussion}, we can see, from Theorem \ref{thm_pertubed}, that to obtain the fluctuations of quantities related to the numerical algorithms, it suffices to focus on the matrix $P(z)$ either at $z=0$ or $z=\infty.$ In Section \ref{sec_spikedcov}, we will focus on the spiked sample covariance matrix model and study these fluctuations, i.e., the limiting behavior of $P(z)$.
\end{remark}

%\begin{remark}
%{\color{blue}
%add some concrete examples here. }
%\end{remark}

%\begin{remark}
%{\color{blue} add more on fluctuation. }
%\end{remark}

\section{Case study: Spiked sample covariance matrix model}\label{sec_spikedcov}
In this section, we focus our discussion on a concrete random matrix model, the celebrated spiked sample covariance matrix model, to illustrate how to conduct the analysis. Motivated by the applications in applied mathematics, we focus on the analysis of its limiting VESD; see Section \ref{sec_subsectionvesd} for more details. For any probability measure $\mu,$ its Stieltjes transform is defined as 
\begin{equation*}
m_{\mu}(z)= 2 \pi \I c_0(z;\mu) = \int \frac{1}{x-z} \mu(\dd x),  \ z \in \mathbb{C}_+.
\end{equation*}

\subsection{The deformed Marchenko--Pastur law}
We first introduce the celebrated deformed Marchenko--Pastur (MP) law. Let $X$ be an $N \times M$ random matrix with independent and identically distributed (iid) centered entries with variance $M^{-1}$ and $\Sigma_0$ be a positive definite deterministic matrix satisfying some regularity conditions (c.f.~Assumption \ref{assum_summary}). Denote the sample covariance matrix and its companion as follows 
\begin{equation}\label{eq_definitioncovariance}
\mathcal{Q}_1=\Sigma_0^{1/2} XX^* \Sigma_0^{1/2}, \ \mathcal{Q}_2= X^*\Sigma_0 X.  
\end{equation}
In the sequel, we assume that for some small constant $0<\tau<1$
\begin{equation}\label{eq_dimensionality}
\tau \leq c_N:=\frac{N}{M} \leq \tau^{-1}.
\end{equation}
% We then introduce the the so-called deformed MP law.

Denote the spectral decomposition of $\Sigma_0$ as 
\begin{equation*}
\Sigma_0=\sum_{k=1}^N \sigma_i \mathbf{v}_i \mathbf{v}_i^*, \ \ 0<\sigma_N \leq \sigma_{N-1} \leq \cdots \leq \sigma_1<\infty. 
\end{equation*} 
The Stieltjes transform $m(z)$ of the deformed MP law can be characterized as the unique solution of the following equation \cite[Lemma 2.2]{Knowles2017}
\begin{equation*}
z=f(m), \ \Im m(z) \geq 0, 
\end{equation*}
where $f(x)$ is defined as
\begin{equation}\label{eq_defnstitlesjtransform}
f(x)=-\frac{1}{x}+\frac{1}{M} \sum_{k=1}^N \frac{1}{x+\sigma_k^{-1}}.
\end{equation}
%More specifically, since we assume that $0<\sigma_N \leq \cdots \leq \sigma_1<\infty,$ by  or \cite{baibook}, we have that there exists a unique solution $m_c \equiv m_c(z) \in \mathbb{C}_+$ satisfying 
%\begin{equation}\label{eq_defnmc}
%z=f(m_c), \ \operatorname{Im} m_c>0. 
%\end{equation}  
Denote $\varrho = \varrho_{\Sigma_0, N}$ as the probability measure associated with $m$. Then $\varrho$ is referred to as the \emph{deformed MP law}, whose properties are summarized as follows; see Lemmas 2.5 and 2.6 of \cite{Knowles2017} for more details. 

\begin{lemma}\label{lem_property} The support of $\varrho$ is a union of connected components on $\mathbb{R}_+:$
\begin{equation}\label{eq_supportdmp}
\operatorname{supp} \varrho =\bigcup_{k=1}^q [\texttt{e}_{2k}, \texttt{e}_{2k-1}] \subset (0, \infty),
\end{equation}
where $q$ depends on the ESD of $\Sigma_0.$ Here $\texttt  e_k$'s  can be characterized as follows: There exists a real sequence $\{\texttt t_k\}_{k=1}^{2q}$ such that $(x,m)=(\texttt e_k, \texttt t_k)$ are real solutions to the equations 
\begin{equation*}
x=f(m), \ \text{and} \ f'(m)=0.
\end{equation*}
\end{lemma}  
Based on Lemma \ref{lem_property}, we shall call the sequence of $\texttt e_{k}, k=1,2,\cdots, 2q,$ as the edges of the deformed MP law $\varrho.$ 
%More specifically, following the conventions in random matrix theory, we denote $\lambda_+:=a_1,$ as the right-most edge of $\varrho.$
To avoid repetition, we summarize the assumptions which will be used in the current paper. These assumptions are standard and commonly used in the random matrix theory literature; see Definition 2.7 of \cite{Knowles2017} for more details. 
%We will need the following mild assumptions \cite{MR3704770} to rule out the existence of spikes in $\Sigma_0$ and guarantee the regularity behavior of $\varrho.$  
%{\color{red}[change notations here]}

\begin{assum}\label{assum_summary} 
%We assume that the following assumptions hold throughout the paper:
%\begin{enumerate}
%\item{\bf On dimensionality}
We assume that (\ref{eq_dimensionality}) holds and $|c_N-1| \geq \tau$. Moreover,
%\item{\bf On $X$ in (\ref{eq_definitioncovariance}).} 
for $X=(X_{ij}),$ we assume that $X_{ij} , 1 \leq i \leq N, \ 1 \leq j \leq M,$ are iid random variables such that
\begin{equation*}
\mathbb{E} X_{ij}=0, \ \mathbb{E} X_{ij}^2=\frac{1}{M}. 
\end{equation*}
Moreover, we assume that for all $k \in \mathbb{N},$ there exists some constant $C_k$ such that 
\begin{equation}\label{eq_momentassumption}
\mathbb{E} |\sqrt{M} X_{ij}|^k \leq C_k. 
\end{equation}
%\item{\bf On $\Sigma_0$.}
For $\Sigma_0,$ we assume that for some small constant $0<\tau_1<1,$ the following holds
\begin{equation*}
\tau_1 \leq \sigma_N \leq \sigma_{N-1} \leq \cdots \leq \sigma_1 \leq \tau_1^{-1}. 
\end{equation*}
Additionally, for the two sequences of $\{\texttt e_k\}$ and $\{\texttt c_k\}$ in Lemma \ref{lem_property}, we assume that 
\begin{equation*}
\texttt e_k \geq \tau_1,  \ \min_{l \neq k} | \texttt e_k-\texttt e_l|\geq \tau_1, \ \min_{i}|\sigma_i^{-1}+\texttt t_k| \geq \tau_1. 
\end{equation*}
Finally, for any fixed small constant $\tau_2,$ there exists some constant $\varsigma = \varsigma_{\tau_1, \tau_2}>0$ such that the density of $\varrho$ in $[\texttt e_{2k}+\tau_2, \texttt e_{2k-1}-\tau_2]$ is bounded from below by $\varsigma.$
%\end{enumerate}
\end{assum} 

\begin{remark}\label{eq_esdsquare}   
We make a remark on the deformed MP law. Even though we will not study $\varrho$ and its perturbation (i.e., the empirical spectral distribution (ESD)), we point out that $\varrho$ satisfies Assumption \ref{assum_measure}. According to \cite[Section A.2]{Knowles2017} (or Lemma 3.6 of \cite{DY}, or Proposition 2.6 of \cite{FI}), under Assumption \ref{assum_summary}, we have that $\varrho(x) \sim \sqrt{\texttt e_{k}-x}, \ x \in [\texttt e_k-\tau, \texttt e_k]$ for some small constant $\tau>0.$ Consequently, we can conclude that  $\varrho$ satisfies (\ref{eq:mu})  by setting $\mathtt{a}_j=\texttt e_{2j}, \mathtt{b}_j=\texttt e_{2j-1}$ and $w_j=0.$ Moreover, (2)--(4) of Assumption \ref{assum_measure} are satisfied due to Assumption \ref{assum_summary}.  
\end{remark}

\subsection{The spiked model}
We are now ready to state our model by adding $r$ spikes to $\Sigma_0,$ where $r \geq 0$ is some fixed integer. Let $\Sigma$ be a spiked sample covariance matrix based on $\Sigma_0$ so that it admits the following spectral decomposition
\begin{equation*}
\Sigma=\sum_{i=1}^M \widetilde{\sigma}_i \mathbf{v}_i \mathbf{v}_i^*,
\end{equation*}
where $\widetilde{\sigma}_i=(1+d_i) \sigma_i$ such that $d_i>0, i \leq r$  and $d_i=0, i>r.$ To ease our discussion, we assume the spikes are supercritical as summarized below following \cite{DRMTA}. 
% We can analogously define $\widetilde{G}_1, \widetilde{G}_2$ and $\widetilde{G}.$ 
\begin{assum}\label{assum_spikes}
For $i \leq r,$ we assume that there exists some constant $\varpi$ such that 
\begin{equation}\label{eq_outlierassumption}
\widetilde{\sigma}_i>-\texttt t_1^{-1}+\varpi. 
\end{equation}
We also assume that $\sigma_i, 1 \leq i \leq r$ are distinct and bounded.

%Moreover, let $\mathsf{s}$ be the number of distinct values of $\{\sigma_1,\cdots, \sigma_r\}$ and denote them as $\theta_1, \cdots, \theta_{\mathsf{s}}.$ We assume that $\min_{ 1 \leq i \neq j \leq \mathsf{s}}| \theta_i-\theta_j|>\varpi.$
\end{assum}

Then the spiked sample covariance matrix and its companion 
are defined, respectively, as follows
\begin{equation}\label{eq_spikedmodel}
\widetilde{\mathcal{Q}}_1:=\Sigma^{1/2}XX^* \Sigma^{1/2}, \ \widetilde{\mathcal{Q}}_2:=X^* \Sigma X. 
\end{equation}
The above model is a generalization of Johnstone's spiked sample covariance matrix model \cite{Johnstone2001}. Let $\{\lambda_i(\widetilde{\mathcal Q}_1)\}$ be the eigenvalues $\widetilde{\mathcal Q}_1$ in the decreasing order and $\{\widetilde{\mathbf{u}}_i\}$ be the associated eigenvector. 

 Under Assumption \ref{assum_spikes}, we have the following result \cite[Theorem 3.6]{DRMTA}. Recall $f(x)$ in (\ref{eq_defnstitlesjtransform}). 
\begin{lemma}\label{lem_edge} Suppose Assumptions \ref{assum_summary} and \ref{assum_spikes} hold. Then we have that for all $1 \leq i \leq r,$ 
\begin{equation*}
\left| \lambda_i(\widetilde{\mathcal{Q}}_1)-f(-\widetilde{\sigma}_i^{-1}) \right|=\OO_{\mathbb{P}}(N^{-1/2}),
\end{equation*}
and 
\begin{equation*}
\left| \langle \widetilde{\mathbf{u}}_i , \mathbf{v}_i \rangle^2-\frac{1}{\widetilde{\sigma}_i} \frac{f'(-\widetilde{\sigma}_i^{-1})}{f(-\widetilde{\sigma}_i^{-1})} \right|=\OO_{\mathbb{P}}(N^{-1/2}).
\end{equation*}
\end{lemma}

\subsection{VESDs and their limits}\label{sec_subsectionvesd}
In this subsection, we introduce the VESDs and their deterministic limits. To be consistent with the notation of Section \ref{sec_RHPframework}, we denote the VESDs  
as $\nu = \nu_N, \widetilde{\nu} = \widetilde{\nu}_N$ and their deterministic limits as $\mu = \mu_N, \widetilde{\mu} = \widetilde{\mu}_N$ for the non-spiked model in (\ref{eq_definitioncovariance}) and spiked model in (\ref{eq_spikedmodel}), respectively. 

For any projection, $\mathbf{b},$ we denote the VESD of $\mathcal{Q}_1$ as 
\begin{equation}\label{eq:VESD}
\nu=\sum_{i=1}^N |\langle \mathbf{u}_i, \mathbf{b} \rangle|^2 \delta_{\lambda_i(\mathcal{Q}_1)},    
\end{equation}    
where $\{\mathbf{u}_i\}$ are the eigenvectors of $\mathcal{Q}_1$ and $\{\lambda_i(\mathcal{Q}_1)\}$ are its eigenvalues. Similarly, we denote the VESD of $\widetilde{\mathcal{Q}}_1$ as 
\begin{equation*}
\widetilde{\nu}=\sum_{i=1}^N |\langle \widetilde{\mathbf{u}}_i, \mathbf{b} \rangle|^2 \delta_{\lambda_i(\widetilde{\mathcal{Q}}_1)}.   
\end{equation*}
The limits of $\nu$ and $\widetilde{\nu}$ can be characterized by the so-called anisotropic local law (c.f.~Lemmas \ref{lem_anisotropiclocallaw} and \ref{lem_spikedcase}). Especially, the Stieltjes transforms of  $\mu$ and $\widetilde{\mu}$ can be characterized, respectively, as \cite{DT1, Knowles2017}
\begin{equation}\label{eq_measureform}
m_{\mu}(z)=-\frac{1}{z} \mathbf{b}^* (1+m(z) \Sigma_0)^{-1} \mathbf{b}, \ m_{\widetilde{\mu}}(z)=\sum_{i=1}^N \frac{\omega_i^2}{1+d_i} \left(-\frac{1}{z}(1+m(z) \sigma_i)^{-1}-\mathcal{L}_i \right),
\end{equation} 
where we denote
\begin{equation}\label{eq_widefinition}
\omega_i=\mathbf{v}_i^* \mathbf{b}, \  \ \mathcal{L}_i=\mathbf{1}(i \leq r)z^{-1}(1+m(z)\sigma_i)^{-2} (d_i^{-1}+1-(1+m(z) \sigma_i)^{-1})^{-1},
\end{equation} 
and recall that $m(z)$ is the Stieltjes transform of the deformed MP law.

Before concluding this subsection, we explain how the measures $\mu$ and $\widetilde{\mu}$ satisfy Assumption \ref{assum_measure}. First, using the inversion formula that $\mu\{[a,b]\}=\pi^{-1} \int_a^b \Im m_{\mu}(x+\ri 0^{+}) \dd x,$ it is easy to see from (\ref{eq_measureform}) that the density of $\mu$, denoted as $\varrho_{\bm{b}}$ satisfies (see (3.4) of \cite{DT1})
\begin{equation}\label{eq_varhob}
\varrho_{\bm{b}}(x)=\frac{\varrho(x)}{x} \mathbf{b}^* \Sigma_0\left[ I+2\Re m(x+\ri 0^{+}) \Sigma_0+| m(x+\ri 0^{+})|^2 \Sigma_0^2 \right]^{-1} \mathbf{b}.
\end{equation}
Under Assumption \ref{assum_summary}, it is easy to see that $\varrho_{\bm{b}}(x) \sim \varrho(x)$ so that $\mu$ satisfies Assumption \ref{assum_measure} as discussed in Remark \ref{eq_esdsquare}. 

For the spiked model, it depends crucially on $\mathbf{b}.$ We will need the following assumption to match the condition (5) of Assumption \ref{assum_measure}.
\begin{assum}\label{assum_weights}
For $\omega_i$ defined in (\ref{eq_widefinition}) and all $1 \leq i \leq r,$ we assume that either of the following holds 
\begin{equation*}
\omega_i=0, \ \text{or} \ 1/D \leq |\omega_i| \leq D. 
\end{equation*}
\end{assum}
Under Assumption \ref{assum_weights}, on one hand, $\omega_i=0$ for all $1 \leq i \leq r,$ it is easy to see that $\mu$ and $\widetilde{\mu}$ coincide so that Assumption \ref{assum_measure} holds. On the other hand, if some of $\omega_i$ are nonzero satisfying Assumption \ref{assum_weights}, without loss of generality, say only $\omega_1 \asymp 1.$ Using the relation that $d_i^{-1}+1-(1+m(f(-\widetilde{\sigma}_i^{-1}))\sigma_i)^{-1}=0,$ according to (\ref{eq_measureform}), Lemma \ref{lem_edge} and Assumption \ref{assum_weights}, we find that $\widetilde{\mu}$ satisfies (\ref{eq:mu}) by setting $\mathtt{a}_j=\texttt e_{2j}, \mathtt{b}_j=\texttt e_{2j-1}$ and $ \mathtt{c}_1=f(-\widetilde{\sigma}_1^{-1}), w_1=\frac{1}{\widetilde{\sigma}_1} \frac{f'(-\widetilde{\sigma}_1^{-1})}{f(-\widetilde{\sigma}_1^{-1})}, p=1.$ The general setting can be analyzed similarly.

% with the deformed MP law as in Lemma \ref{lem_property} and $\widetilde{\mu}$ can be characterized using its Stiletjes transform, denoted as 

%{\color{red}[Need formula for the limiting spikes and some notation]}

\subsection{A general CLT}\label{sec_generalclt}

As we can see from Section \ref{sec_subasymptotics}, it suffices to establish the CLT of the following form, 
\begin{equation}\label{eq_yytildedefn}
\mathcal{Y}:=\sqrt{M \eta} \oint_{\Gamma} g(z) c_0(z; \mu-\nu) \dd z, \ \text{or} \  \widetilde{\mathcal{Y}}:=\sqrt{M \eta} \oint_{\Gamma} g(z) c_0(z; \widetilde{\mu}-\widetilde{\nu}) \dd z,
\end{equation}
where $g(z)$ is an analytic in a neighborhood of $\Gamma$ and $\eta = \eta(n)$ depending on some other parameter $n$ is defined in (\ref{eq_etachoice}). Here we recall again that $n$ can be the order of orthogonal polynomials or the number of iterations in the numerical algorithms.

 According to our applications,  by Lemma \ref{lem_edge} and the local law (c.f. Lemma \ref{lem_anisotropiclocallaw}),  $g(z)$ can be purely deterministic  and given by the entries of $M_n(z;\mu)/z^k,  k=0,1$ after some proper normalization so that $\oint_{\Gamma} |g(z)| |\dd z| \asymp 1$ and $\mathcal{Y}$ is a real-valued random variable as required. The main results are reported in Theorem \ref{thm_mainclt}. We first introduce the following definition.
\begin{definition}
For two sequences of random vectors $\mathbf{x}_N, \mathbf{y}_N \in \mathbb{R}^k, N \geq 1,$ we say they are asymptotically equal in distribution, denoted as $\mathbf{x}_N \simeq \mathbf{y}_N,$ if they are tight and satisfy 
\begin{equation*}
\lim_{N \rightarrow \infty} \left( \mathbb{E}l(\mathbf{x}_N)-\mathbb{E} l(\mathbf{y}_N)  \right)=0,
\end{equation*}    
for any bounded continuous function $l: \mathbb{R}^k \rightarrow \mathbb{R}.$
\end{definition}
Then we provide some notation. Denote
\begin{equation}\label{eq_defnpi1z}
\Pi_1(z):=-\frac{1}{z} (1+m(z)\Sigma_0)^{-1},
\end{equation}
and for any deterministic vectors $\vec{h}_1, \vec{h}_2 \in \mathbb{R}^N,$ we define 
\begin{align}\label{eq_v1}
\mathsf{V}_1(\vec{h}_1, \vec{h}_2):=\frac{\eta}{2\pi^2 } \oint_{\Gamma} \oint_{\Gamma} \sqrt{z_1 z_2} g(z_1) g(z_2) \left[ \vec{h}_1^*  (1+m(z_1) \Sigma_0)^{-1} \Sigma_0 \Pi_1(z_2) \vec{h}_2 \right] \\
\times  \left[ \frac{\vec{h}_1^*(\Pi_1(z_1)-\Pi_1(z_2))\vec{h}_2}{z_1-z_2}\right]   \dd z_1 \dd z_2,
\end{align}
where we used the convention that
\begin{equation*}
\lim_{z_1 \rightarrow z_2} \frac{\vec{h}_1^*(\Pi_1(z_1)-\Pi_1(z_2))\vec{h}_2}{z_1-z_2}=\vec{h}_1^* \Pi_1'(z_1) \vec{h}_2,
\end{equation*}
and 
\begin{equation}\label{eq_mathsfv2definition}
\mathsf{V}_2(\vec{h}_1, \vec{h}_2):=-\frac{\eta}{4 \pi^2} \Big( \oint_{\Gamma} \oint_{\Gamma} g(z_1) g(z_2) z_1 z_2 m(z_1) m(z_2) \mathcal{K}(z_1, z_2)\dd z_1 \dd z_2 \Big), 
\end{equation}
where $\mathcal{K}(z_1, z_2)$ is defined by
\begin{align*}
\mathcal{K}(z_1, z_2):=\sqrt{z_1} \sum_{i} (\Sigma_0^{1/2} \Pi_1(z_1) \vec{h}_1 \vec{h}_2^* \Pi_1(z_1) \Sigma_0^{1/2} )_{ii} (\Sigma_0^{1/2} \Pi_1(z_2) \vec{h}_1 \vec{h}_2^* \Pi_1(z_2))_{ii}. 
\end{align*}
 
Let $\kappa_4$ be the cumulant of the random variable $X_{ij}$ as defined in (\ref{eq_defncumulant}). 
\begin{theorem}\label{thm_mainclt}
Suppose $\mathcal{Y}$ and $\widetilde{\mathcal{Y}}$ are real valued.  Suppose that Assumption \ref{assum_summary} holds, then we have that 
\begin{equation*}
\mathcal{Y} \simeq \mathcal{N}(0, \mathsf{V}_1(\vec{b}, \vec{b})+\kappa_4 \mathsf{V}_2(\vec{b}, \vec{b})).
\end{equation*}
Moreover, if Assumptions \ref{assum_spikes} and \ref{assum_weights} hold, 
\begin{equation*}
\widetilde{\mathcal{Y}} \simeq \mathcal{N}(0, \widetilde{\mathsf{V}}_1+\kappa_4 \widetilde{\mathsf{V}}_2),
\end{equation*}
where we used the notation that for $k=1,2,$ 
\begin{equation*}
\widetilde{\mathsf{V}}_k:=\sum_{i=1}^N \frac{\omega^2_i}{1+d_i} \left(\mathsf{V}_k(\vec{v}_i, \vec{v}_i)-\mathsf{V}_k(\vec{l}_i,\vec{v}_i)-\mathsf{V}_k(\vec{v}_i,\vec{l}_i)- \mathsf{V}_k(z^{-1}\vec{l}_i, \vec{l}_i) \right).
\end{equation*}
where $\omega_i$ is defined in (\ref{eq_widefinition}) and $\vec{l}_i$ is defined in (\ref{eq_notationssummary}) after some additional necessary notation is introduced.
% and 
%\begin{equation*}
%\widetilde{\mathsf{V}}_2:=
%\end{equation*}
\end{theorem}
\begin{proof}
See Appendix \ref{sec_cltproof}. 
\end{proof}

\begin{remark}\label{rem:mom_decay}
In our applications, when the deterministic function $g(z)$ is properly normalized, it is easy to check that
\begin{equation*}
\mathsf{V}_1(\vec{b}, \vec{b}) \asymp 1, \ \mathsf{V}_2(\vec{b}, \vec{b}) \asymp \eta.
\end{equation*}
We recall from (\ref{eq_etachoice}) that $\mathsf{V}_2=n^{-2}.$ Consequently, when $n$ diverges in any polynomial order, $\mathsf{V}_2(\vec{b}, \vec{b})$ can be negligible asymptotically and hence the fluctuations only depend on the first two moments and is therefore more universal. Similar phenomenon has been observed in the mesoscopic CLT of random matrix theory, see, for example, \cite{10.1093/imrn/rnaa210, MR3459158,MR3678478,MR4255183, YF}. 
%We point out that $\mathsf{V}_1$ is always of constant order whereas $\mathsf{V}_2 \asymp \eta=n^{-2}$ {\color{red} add more discussion here}
\end{remark}

\begin{remark}
We provide a few examples to illustrate the results of the spiked model, i.e., the CLT of $\widetilde{\mathcal{Y}}.$ As can be seen in (\ref{eq_notationssummary}), 
\begin{equation*}
\vec{l}_i=\vec{0}, \ i>r.
\end{equation*}
Consequently, if $\vec{b} \in \operatorname{Span}(\{\vec{v}_i\}_{i>r}),$  we find that 
\begin{equation*}
\widetilde{\mathsf{V}}_k=\mathsf{V}_k(\vec{b}, \vec{b}), \ k=1,2. 
\end{equation*} 
That is to say, when $\vec{b}$ lies in the orthogonal complement of the spiked eigenvectors, the distribution is the same with the non-spiked model.  Moreover, when $\vec{b}=\vec{v}_{i_*},$ $1 \leq i_* \leq r,$ we have that 
\begin{equation*}
\vec{l}_{i_*}=\frac{-1-m(z) \sigma_{i_*}}{d_{i_*}+1-(1+m(z) \sigma_{i_*})^{-1}} \vec{b}.
\end{equation*}  
Consequently, we can simplify $\widetilde{\mathsf{V}}_k$ to 
\begin{equation*}
\widetilde{\mathsf{V}}_k:=(1+d_{i_*})^{-1} \left(\mathsf{V}_k(\vec{b}_{i_*}, \vec{b}_{i_*})-\mathsf{V}_k(\vec{l}_{i_*},\vec{b}_{i_*})-\mathsf{V}_k(\vec{b}_{i_*},\vec{l}_{i_*})-\mathsf{V}_k(z^{-1} \vec{l}_{i_*}, \vec{l}_{i_*}) \right).
\end{equation*}
%1. when $\vec{b} \in \{\vec{v}_i\}_{i>r}$, $\mathcal{Y}$ and $\widetilde{\mathcal{Y}}$ have the same distribution. 
%
%
% 2. if $b=\vec{v}_{i_0}, i_0 \leq r,$ the variance can be reduced to....
\end{remark}

As a consequence of Theorem \ref{thm_mainclt}, we can establish the asymptotic fluctuations of the associated orthogonal polynomials. 

\begin{corollary}\label{coro_explicitdistribution}
Suppose the assumptions of Theorem \ref{thm_mainclt} hold. Let the parameters $\mathfrak{c}, \mathfrak{g}, \mathsf{G}$ and $\{\texttt{c}_j\}$ in (\ref{eq:explicit_pin}), $\Theta$ in (\ref{eq_vectortheta}) and $\gamma(z)$ in (\ref{eq_gammaz}) defined by the limiting VESD as in (\ref{eq_varhob}). Denote 
\begin{equation*}
\mathsf{L}:=\mathfrak c^{(p-n)} \e^{(n-p) \mathfrak g(z) + \sG(z) - \sG(\infty)} \times \left[ \prod_{j=1}^p (z - \mathtt c_j) \right],
\end{equation*}
\begin{equation*}
\mathsf{E}_{1}:=\frac{\left(\frac{\gamma(z) + \gamma(z)^{-1}}{2}\right) \Theta_1(z;\vec d_2;(n-p)\bm{\Delta} + \bm{\zeta})}{\Theta_1(\infty;\vec d_2;(n-p)\bm{\Delta}+\bm{\zeta})},\  \mathsf{E}_2:=\e^{2 \mathsf{G}(\infty)}\frac{\left(\frac{\gamma(z)^{-1} - \gamma(z)}{2\I}\right) \Theta_1(z;\vec d_1;(n-p)\bm{\Delta} + \bm{\zeta})}{\Theta_2(\infty;\vec d_1;(n-p)\bm{\Delta}+\bm{\zeta})}.
\end{equation*}
For the non-spiked model, when $ C \log N \leq n \leq N^{1/6-\epsilon}$  for some large constant $C>0$ and small constant $\epsilon>0,$ for $z \in \mathbb{R} \setminus \mathrm{supp}(\mu)$, we have
\begin{equation*}
\frac{\sqrt{M}}{n \mathsf{C}_g}\left( \pi_n(z;\nu) -\mathsf{L}\mathsf{E}_{1} \right) \simeq \mathcal{N}\left(0,\frac{\mathsf{L}}{\mathsf{C}^2_g}(\mathsf{V}_1(\vec{b}, \vec{b})+\kappa_4 \mathsf{V}_2(\vec{b},\vec{b}))\right),
\end{equation*}
where $\mathsf{V}_1$ and $\mathsf{V}_2$ are defined as in (\ref{eq_v1}) and (\ref{eq_mathsfv2definition}) by letting $\eta=n^{-2}$ and $$g(z')= \frac{1}{2 \pi \I} \frac{1}{z'-z}\left( \mathsf{E}_1 M_n(z')_{11}+ \mathsf{E}_2 M_n(z')_{12} \right), \ \mathsf{C}_g=\oint_{\Gamma} |g(z')| |\dd z'|, $$
where $M_n$ is defined in (\ref{eq_MN(zmu)}). Similar results hold for the spiked model.

\end{corollary}
\begin{proof}
The proof follows directly from Theorem \ref{thm_mainclt} and (\ref{eq:explicit_pin}).
\end{proof}

% \begin{remark}
%   , with $\mathsf{L}\mathsf{E}_{1}$ being the leading-order term when the adjustments discussed in Remark~\ref{rmk_divergent} are incorporated.  Specifically, since $\mathtt c_j$ are random, $\sG$ has some randomness and the fluctuations of the spikes $\mathtt c_j$ will enter into the final result.
% \end{remark}

\begin{remark}\label{rmk_final}
The normalization is used to ensure that, $\oint_{\Gamma} |g(z')|/\mathsf{C}_g |\dd z'|$ is bounded from below and above so that Theorem \ref{thm_mainclt} applies.  Using an analogous discussion, we can derive the CLT for the Cauchy transforms as in (\ref{eq:explicit_pin}). Since the concerned quantities of the numerical algorithms depend on the orthogonal polynomials and their Cauchy transforms, we can also obtain the asymptotic fluctuation of these algorithms using Theorems \ref{thm_pertubed} and \ref{thm_mainclt}. We omit further details here.  
\end{remark}

%{\color{red} add some insights back to OPs and CGA. }

\section{Numerical simulations and some discussions}\label{sec_simu}

We now provide numerical simulations of our estimates and perturbation theory to demonstrate the asymptotic behavior of both matrix factorizations and iterative algorithms (c.f.~Section \ref{sec:matfac}) applied to the spiked sample covariance model.

\subsection{Calculations of key parameters}\label{sec_calculationofkeyparameters}
As we have seen in the results of Sections \ref{sec_RHPframework} and  \ref{sec_algorithmicapp} that many essential parameters need to be estimated before the application of Theorem~\ref{thm_pertubed}. The first quantity is the density of the limiting VESD, its support and strength of the spikes. In Appendix \ref{app:densityapproximation}, we provide a numerical method to approximate this. The outline of the procedure is:
\begin{itemize}
\item First, to compute the asymptotic support of the measure we use a rootfinder guided by Lemma~\ref{lem_property}.
\item Second, to compute the asymptotic location of the spikes we use Lemma~\ref{lem_edge}.
\item Then we fit the coefficients in a mapped Chebyshev approximation of the density $h_j$ on $[\mathtt a_j, \mathtt b_j]$ by solving a constrained optimization problem.
\end{itemize}
The method works with the empirical resolvent $\langle \vec b, (W - z)^{-1} \vec b \rangle$ or with the limiting Stieltjes transforms $m_\mu$ or $m_{\tilde \mu}$.  In the former case, one should average over a number of trials.   With the density function approximated in a useable form, we can calculate the other parameters.   Appendix~\ref{app:densityapproximation} outlines how to then approximate, with good accuracy, the limiting Jacobi matrix $\mathcal J(\mu)$ from which many other quantities of interest are easily computable.

Below, we also compute $\mathfrak g(0)$.  Since $g$ in Assumption~\ref{assum_measure} is small in our computations one can directly implement the procedure outlined in Section~\ref{subsubsec_differential} using the methodology of \cite[Section 11.6.1]{TrogdonSOBook} to compute the integrals that arise.  

\subsection{Performance of CGA with random inputs}

In this subsection, we work on CGA when $W$ is a spiked sample covariance matrix. We first work on an example where the support of the limiting VESD consists of two disjoint intervals (i.e., a single gap) with spikes. Then we study a three intervals (i.e., two gaps) case. Finally, we study the halting time of CGA, i.e., the number of iterations needed before CGA terminates according to some stopping rule. 

In the computations that follow, it is interesting to compare what results to the classical Chebyshev upper bound for the convergence of CGA \cite{Hestenes1952}:
\begin{align}\label{eq:classic_bound}
    \frac{\|\vec r_n\|_2}{\|\vec r_{n-1} \|_2} \leq \delta_{\mathrm{Cheb}}^{-1}, \quad \delta_{\mathrm{Cheb}} :=  \frac{ \sqrt{\kappa(W)}+1}{\sqrt{\kappa(W)} - 1}, \quad \kappa(W) = \lambda_{\max}(W)/\lambda_{\min}(W).
\end{align}

\subsubsection{CGA: Single gap with spikes}

Consider the spiked sample covariance matrix $W = \Sigma^{1/2} X X^* \Sigma^{1/2}$ where $X$ is $N \times M$, has iid entries,
\begin{align}\label{eq:single_gap_spiked}\begin{split}
    \Sigma_0 &= \mathrm{diag}( 8I, I), \quad X_{ij} \lawequals \mathcal N(0, M^{-1}),\\
    M &= \lfloor N /0.3 \rfloor, \quad d_1 = 1,  \quad d_2 = 0.5,  \quad d_i = 0, \ i \geq 3,
    \end{split}
\end{align}
and $I$ is the $N/2 \times N/2$ identity matrix.  We choose $X_{ij} \lawequals \mathcal N(0, M^{-1})$ for convenience because, as we have shown, the same limiting behavior will happen for any other admissible entry distribution. In Figure~\ref{fig:cg_two_spiked} we apply the CGA to $W \vec x = \vec b$ with $2 \vec b = \vec f_1  + \vec f_2  + \vec f_3 + \vec f_N$.  The residuals encountered at iteration $k$ concentrate on the black dashed curve that is computed utilizing the results of Section \ref{sec_algorithmicapp} with parameters calculated using methodology outlined in Section \ref{sec_calculationofkeyparameters}. In particular,  the density function is recorded Figure~\ref{fig:single_gap_density_spiked} and the choices of the parameters are
\begin{align}\label{eq:params}
  \begin{split}
  \mathtt a_1 \approx 0.279, & \quad \mathtt b_1 \approx 1.667, \quad  \mathtt a_2 \approx 3.192, \quad \mathtt b_2 \approx 15.562,\\
  \delta &:= \e^{ \mathfrak g(0)} \approx 1.322,\\
  \mathtt c_1 &\approx 20.319, \quad \mathtt c_2 \approx 33.755.
  \end{split}
\end{align}

In addition to the black curve, we also provide a red curve. The motivation is as follows. According to Corollary \ref{cor_cgadeterasymp} and Theorem \ref{thm_pertubed} (or Remark \ref{remark_formulaone}), we find that
\begin{equation*}
\frac{\|\vec r_n \|_2}{\| \vec r_{n-1} \|_2} \approx \e^{-\mathfrak{g}(0)}. 
\end{equation*}
Note that in this case $\delta_{\mathrm{Cheb}} \approx 1.309$ indicating that when one accounts for the gap in the spectrum, a faster convergence rate is predicted.

Then after being properly scaled, we can use $\e^{-n \mathfrak{g}(0)}$ for the prediction. We find that both our black and red curves are accurate even for small values of $N.$ Furthermore, a remarkable feature of \eqref{eq:res-formula} and \eqref{eq:error-formula} is that the random components are contained in the $P_{11}$ and $P_{12}$ terms which have a common exponential factor. This implies that the fluctuations are on the same exponential scale as the asymptotic mean. We demonstrate this by considering $\|\vec r_k\|_2 \delta^{k}$ in Figure~\ref{fig:scaled_cg_two_spiked}.  We emphasize that the parameters \eqref{eq:params}, along with the computation of $r(z) = \langle \vec b, (W - z)^{-1} \vec b \rangle$, are all that are needed as input to the algorithm in Appendix~\ref{app:densityapproximation} to produce the limiting curves in Figure~\ref{fig:single_gap_density_spiked}.

\begin{figure}[!h]  
\centering
\subfigure[]{\includegraphics[width=.49\linewidth]{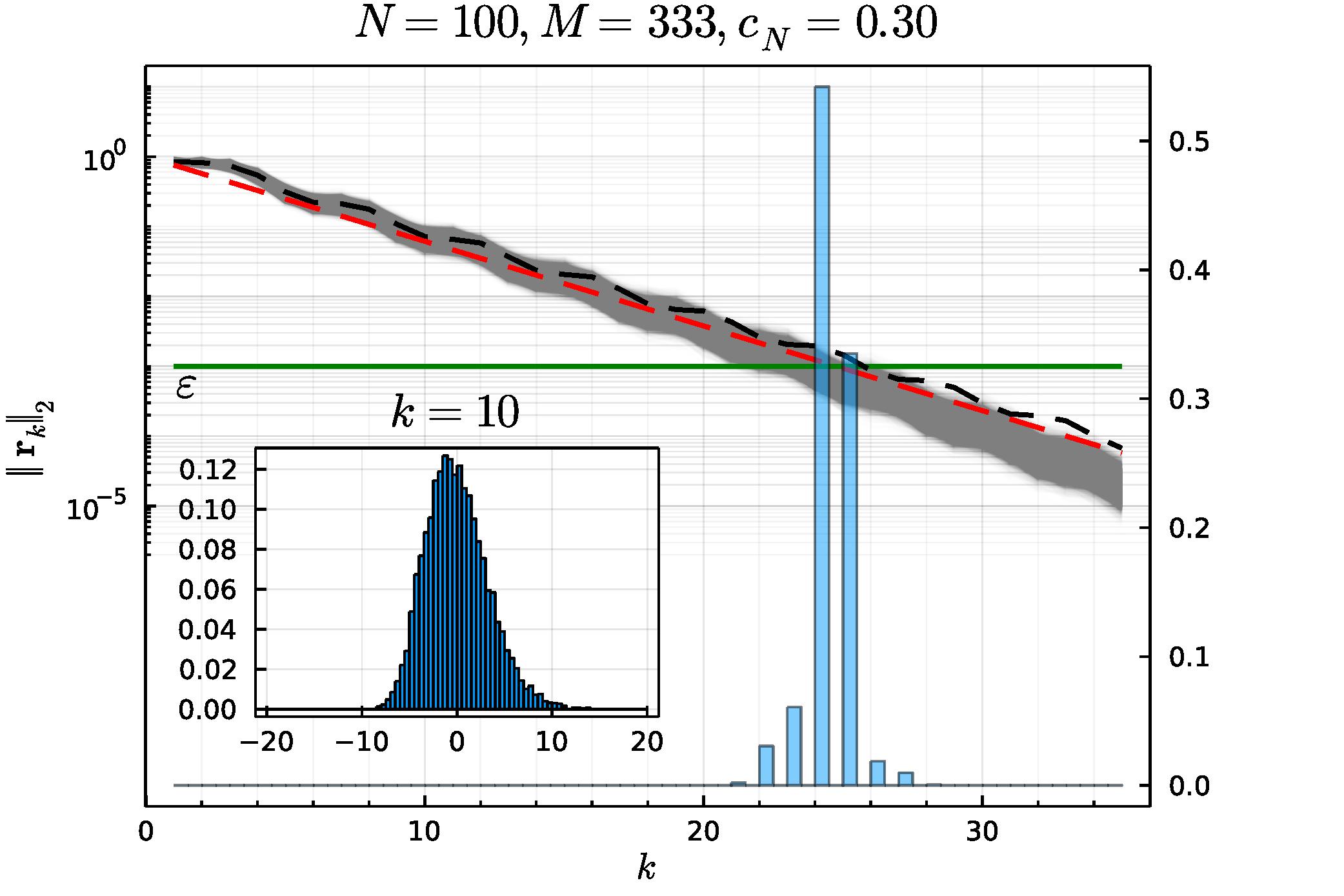}}
\subfigure[]{\includegraphics[width=.49\linewidth]{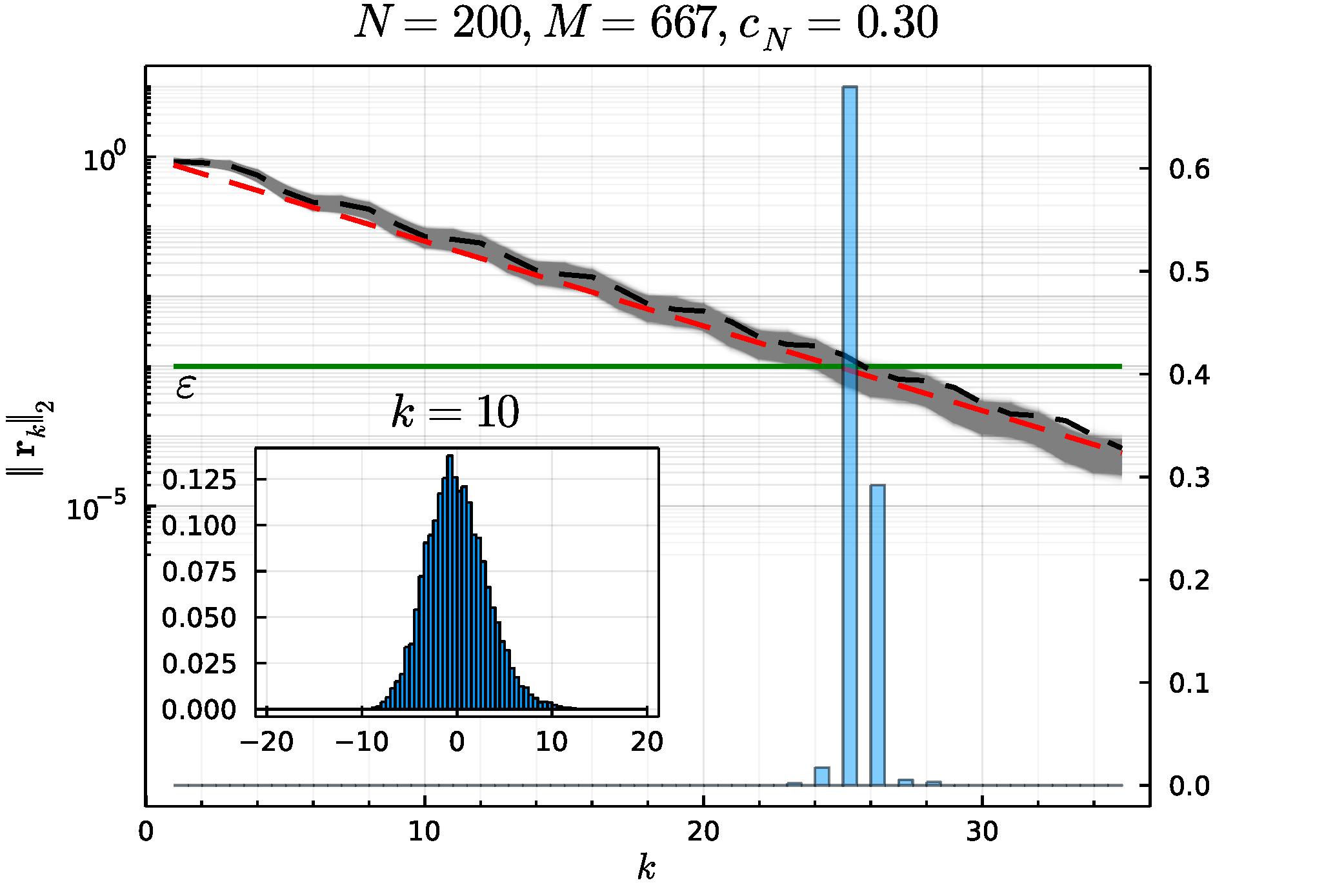}}
\subfigure[]{\includegraphics[width=.49\linewidth]{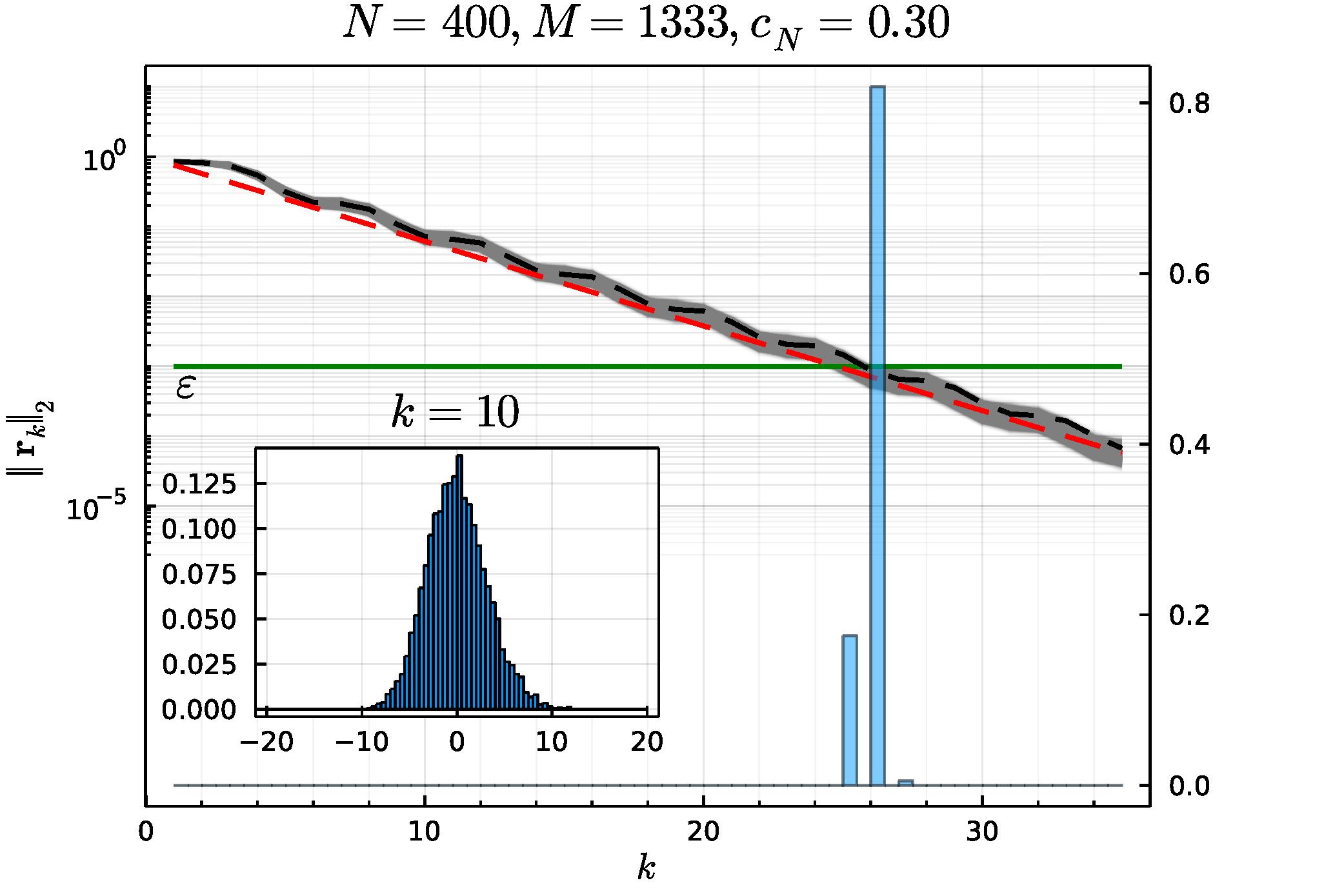}}
\subfigure[]{\includegraphics[width=.49\linewidth]{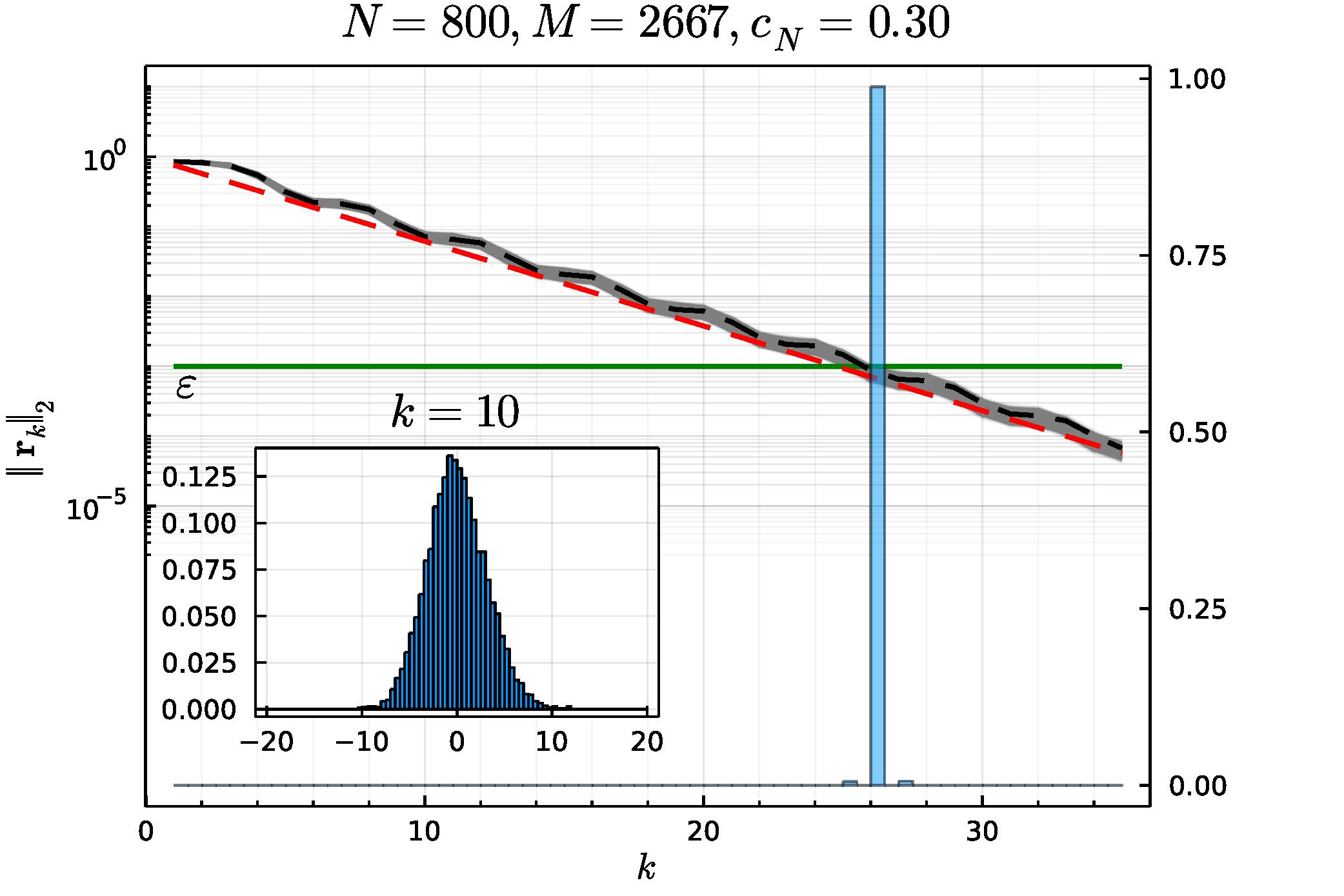}}
\caption{\label{fig:cg_two_spiked} The CGA runs on the single gap matrix in \eqref{eq:single_gap_spiked}.  The black oscillatory dashed curve indicates the large $N$ limit for the residual norms $\|\vec r_k\|_2$ at step $k$.  The shaded gray area is an ensemble of 10000 runs of the conjugate gradient algorithm, displaying the residuals that resulted.  The red dashed line is given by $\delta^{-k}$, $\delta = \e^{\mathfrak g(0)}$.  The overlaid histogram shows the rescaled fluctuations in the norm of the residual at $k = 10$.  As $N \to \infty$ this approaches a Gaussian density.  Lastly, the histogram in the main frame gives the halting distribution $\tau(W,\vec b,\vec \epsilon) = \min\{k : \|\vec r_k\|_2 < \epsilon\}$ for $\epsilon = 10^{-3}$ (green horizontal line), i.e., the statistics of the number of iterations required to achieve $\|\vec r_k\|_2 < \epsilon$.}
\end{figure}

%{\color{red}[Double check here. We also need to provide the details of other quantities. For example, $G(0),$ $\Theta$. ]}

\begin{figure}[!ht]  
\centering
\subfigure[]{\includegraphics[width=.49\linewidth]{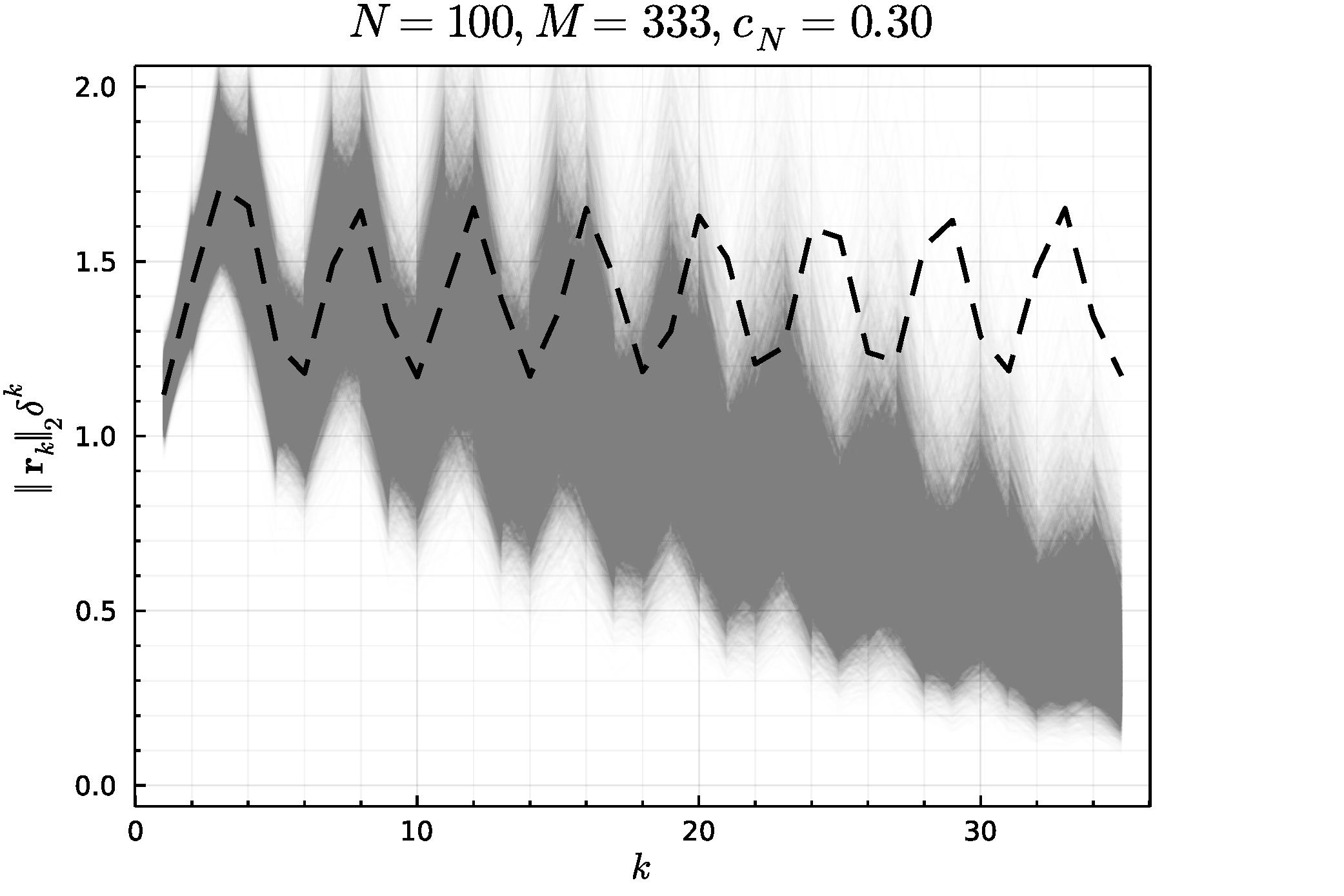}}
\subfigure[]{\includegraphics[width=.49\linewidth]{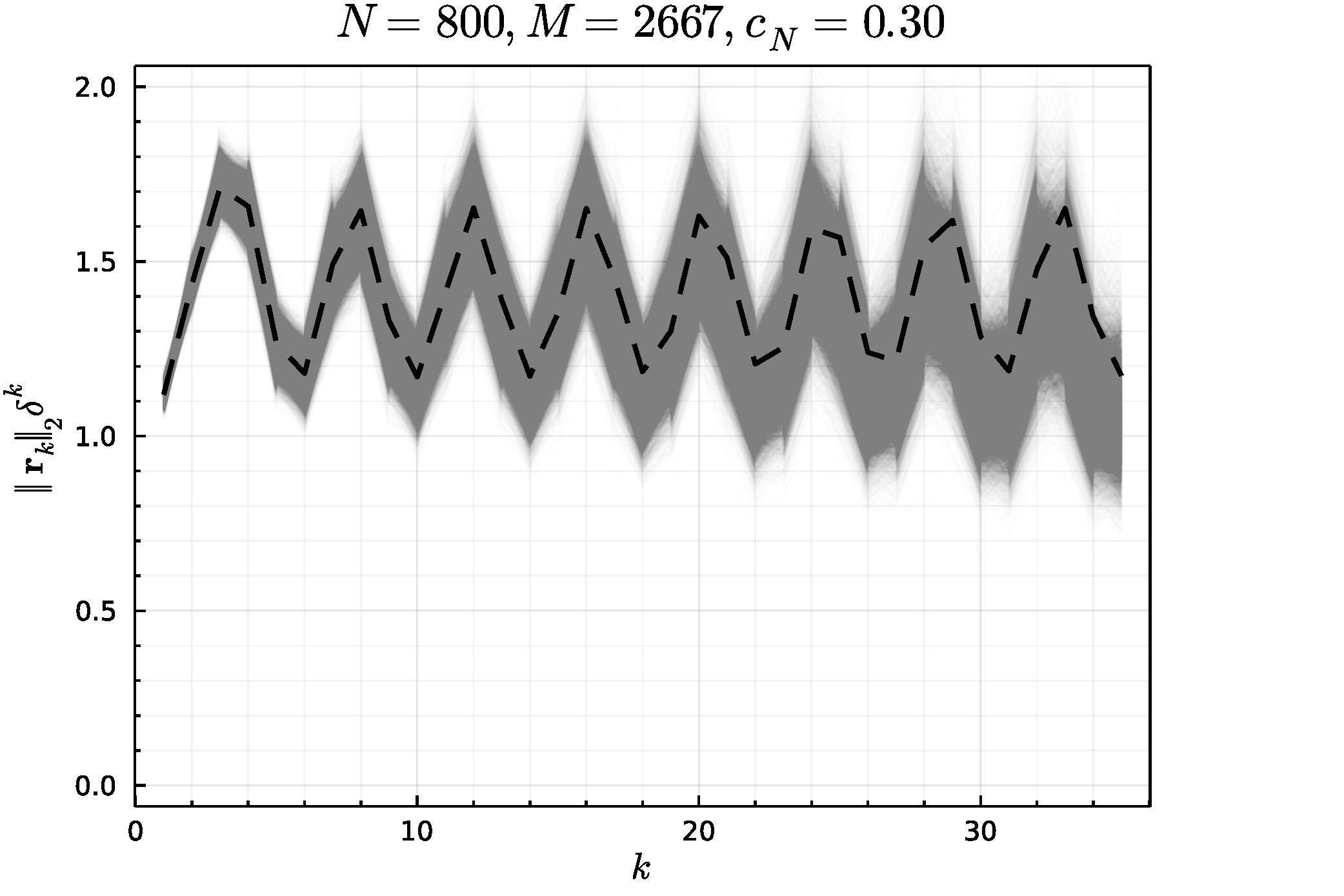}}
\caption{\label{fig:scaled_cg_two_spiked} The CGA runs on the single gap matrix in \eqref{eq:single_gap_spiked}.  The black oscillatory dashed curve indicates the large $N$ limit for the scaled residual norms $\|\mathbf r_k\|_2 \delta^{k}$ at step $k$.  The shaded gray area is an ensemble of 10000 runs of the conjugate gradient algorithm, displaying the scaled residuals that resulted.}
\end{figure}

\subsubsection{CGA: Two gaps}

Consider the non-spiked sample covariance matrix $W = \Sigma_0^{1/2} X X^* \Sigma_0^{1/2}$ where $X$ is $N \times M$, has iid entries,
\begin{align}\label{eq:two_gap}
  \Sigma_0 = \mathrm{diag}( 3.8I, 1.2I, 0.25I ), \quad X_{ij} \lawequals \mathcal N(0, M^{-1}), \quad M = \lfloor N /0.3 \rfloor,
\end{align}
and $I$ is the $N/3 \times N/3$ identity matrix\footnote{We choose $I$ here to be either $\lfloor N/3 \rfloor \times \lfloor N/3 \rfloor$ or $\lceil N/3 \rceil \times \lceil N/3 \rceil$.}. In Figure~\ref{fig:cg_two}  we apply the CGA to $W \vec x = \vec b$ with $\sqrt{3} \vec b = \vec f_1 + \vec f_{N/2} + \vec f_N$. We again report both the black and red curves and they are reasonably accurate. The approximate density for the limiting VESD is displayed in Figure~\ref{fig:two_gap_density} and the choices of the parameters we find that are
\begin{align*}
  \mathtt a_1 &\approx 0.080, \quad \mathtt b_1 \approx 0.349,\\
  \mathtt a_2 &\approx 0.496, \quad \mathtt b_2 \approx 1.828,\\
  \mathtt a_3 &\approx 2.029, \quad \mathtt b_3 \approx 6.767,\\
  \delta &:= \e^{ \mathfrak g(0)} \approx 1.248.
\end{align*}
Note that in this case $\delta_{\mathrm{Cheb}} \approx 1.244$.

\begin{figure}[!ht]  
\centering
\subfigure[]{\includegraphics[width=.49\linewidth]{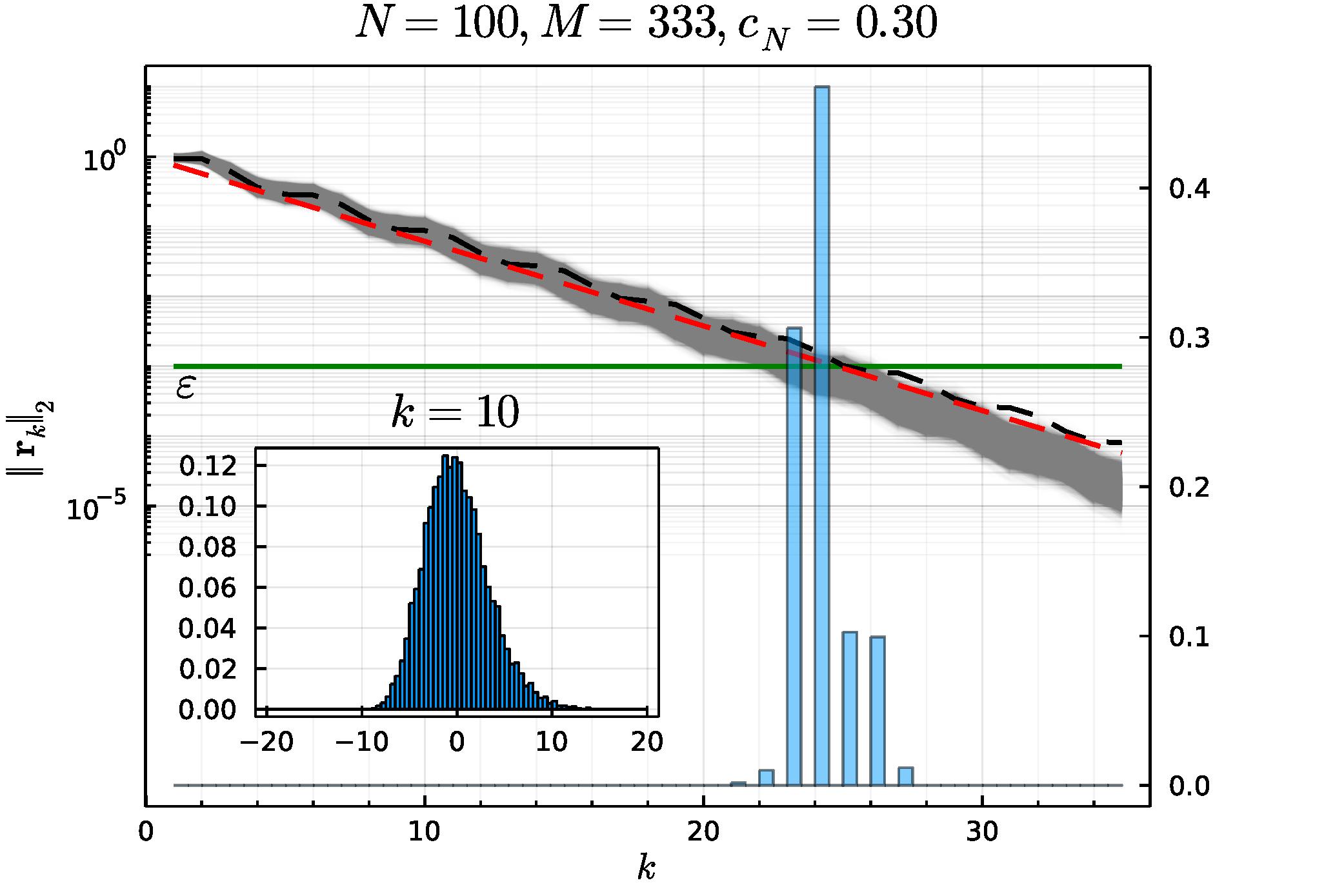}}
\subfigure[]{\includegraphics[width=.49\linewidth]{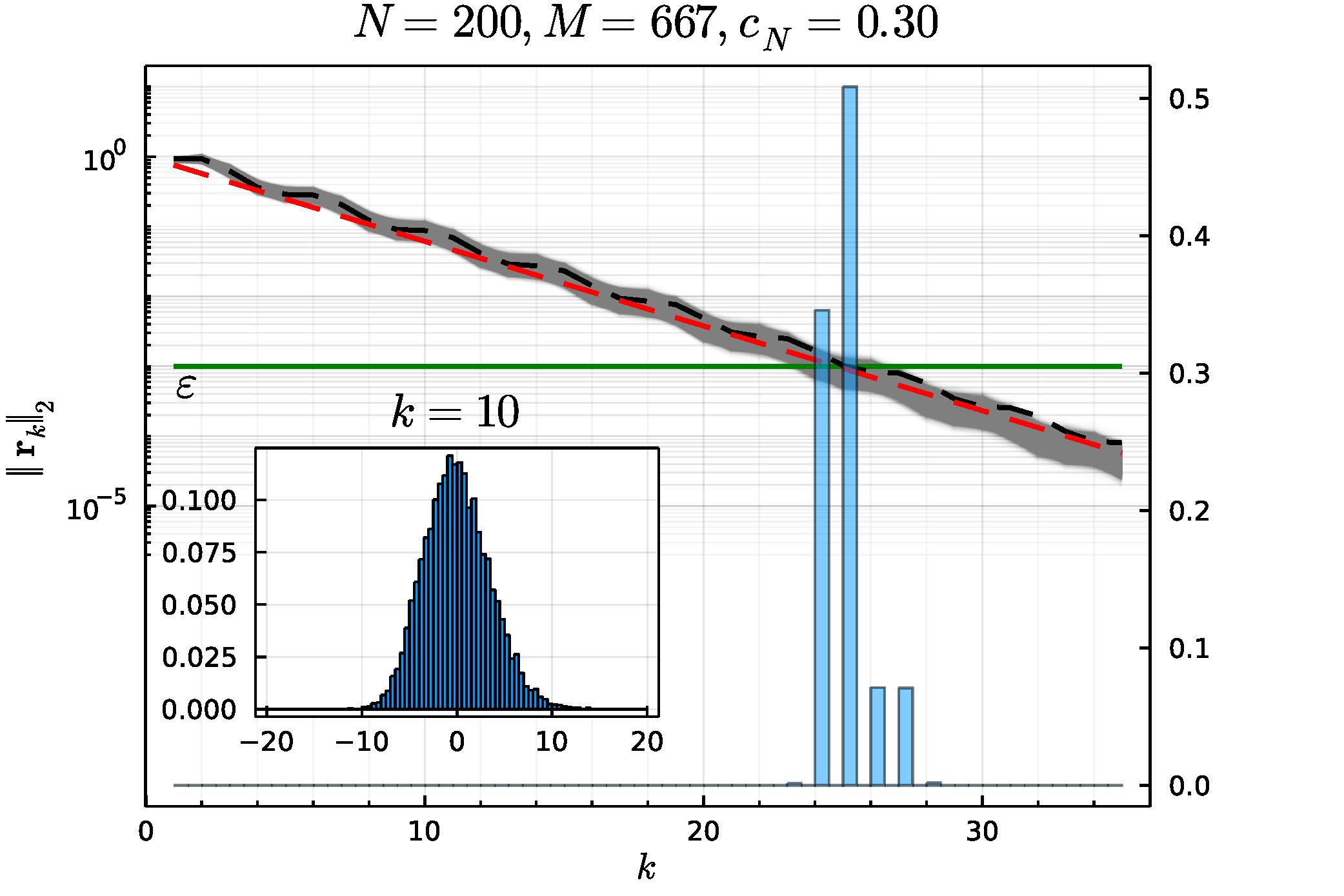}}
\subfigure[]{\includegraphics[width=.49\linewidth]{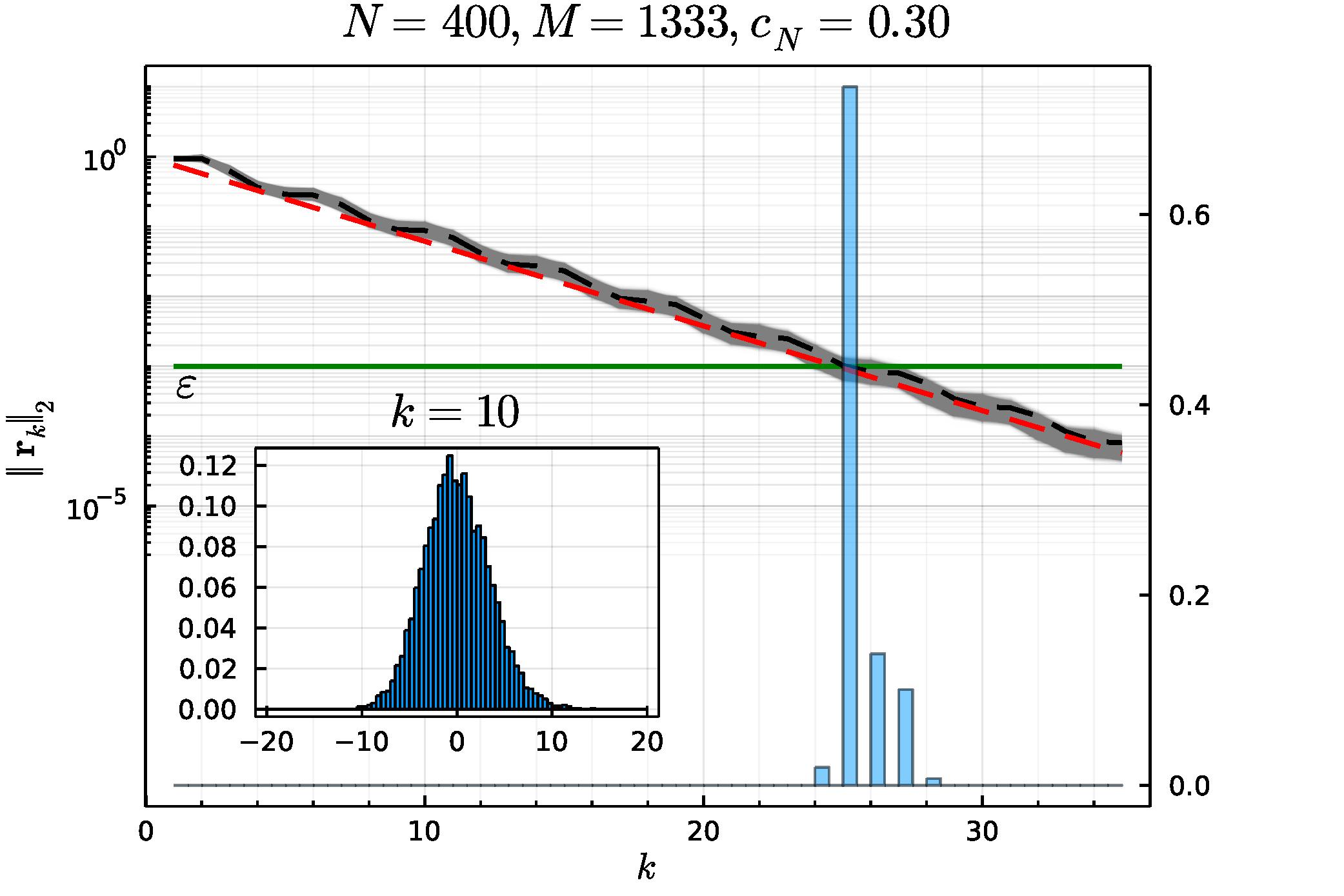}}
\subfigure[]{\includegraphics[width=.49\linewidth]{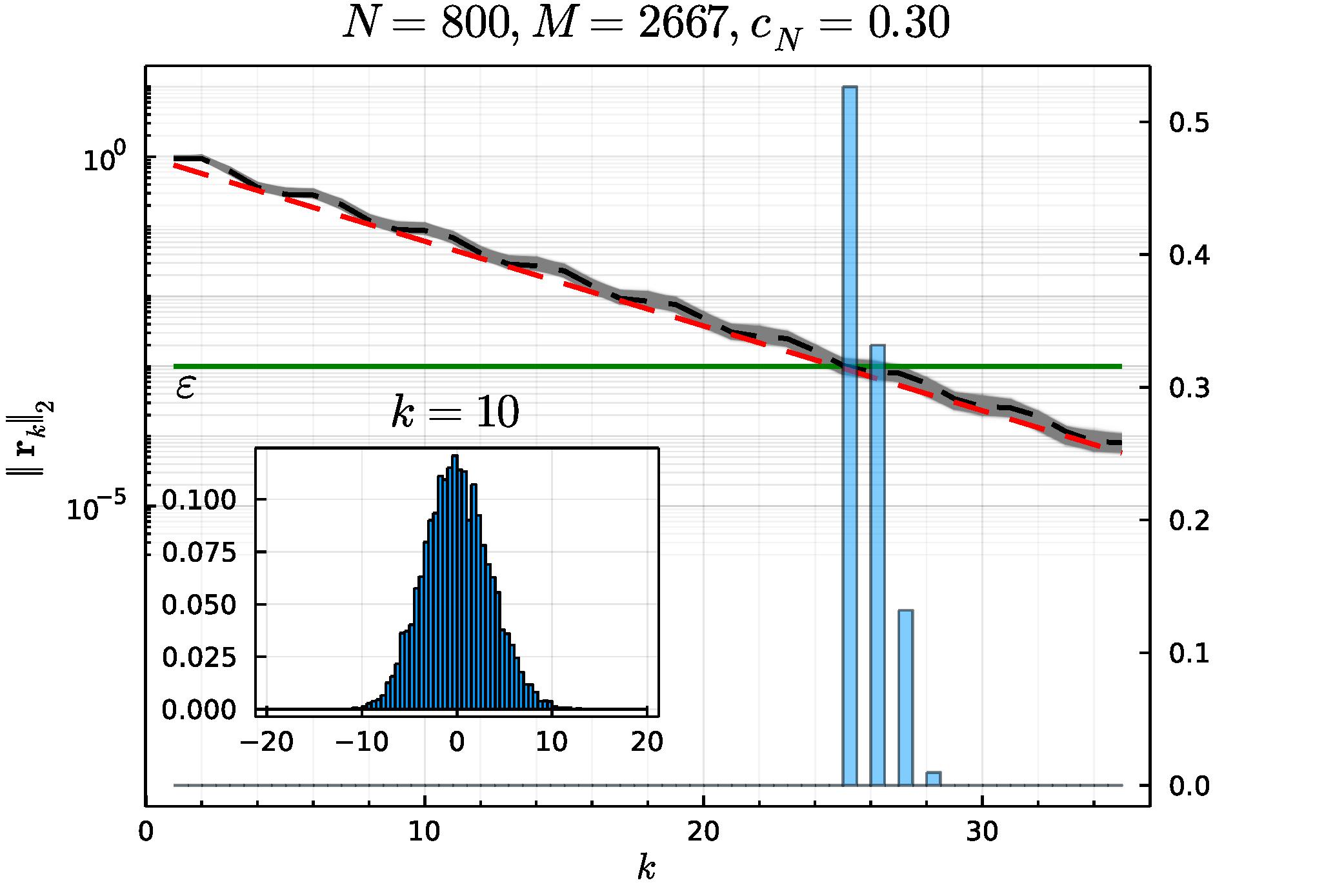}}
\caption{\label{fig:cg_two} The CGA runs on the single gap matrix in \eqref{eq:two_gap}. The details of the figures are similar to the captions of Figure \ref{fig:cg_two_spiked}.
% The black oscillatory dashed curve indicates the large $N$ limit for the residual norms $\|\vec r_k\|_2$ at step $k$.  The shaded gray area is an ensemble of 10000 runs of the conjugate gradient algorithm, displaying the residuals that resulted.  The red dashed line is given by $\delta^{-k}$, $\delta = \e^{k \mathfrak g(0)}$.  The overlaid histogram shows the rescaled fluctuations in the norm of the residual at $k = 10$.  As $N \to \infty$ this approaches a Gaussian density.  Lastly, the histogram in the main frame gives the halting distribution $\tau(W,\vec b,\vec \epsilon) = \min\{k : \|\vec r_k\|_2 < \epsilon\}$ for $\epsilon = 10^{-3}$ (green horizontal line), i.e., the statistics of the number of iterations required to acheive $\|\vec r_k\|_2 < \epsilon$.
}
\end{figure}

\subsubsection{Small $\epsilon$ runtime of CGA}

Let $\Sigma = I$ and consider the statistics of
\begin{align*}
  \tau(W,\vec b,\epsilon) := \min\{k : \|\vec r_k\|_2 < \epsilon\},
\end{align*}
where $\epsilon = \epsilon_0 p_N$ and $p_N$ tends to zero as $N \to \infty$.  For example, choosing $p_N = N^{-1/2}$ will ensure $\|\vec r_k\|_1 < \epsilon_0$ if $\|\vec r_k\|_2 < \epsilon$.  More generally, if $p_N$ tends to zero a polynomial rate, our results, combined with the formulae for the variance in \cite[Remark 4, Equation (7)]{Paquette2020} demonstrate that $\tau(W,\vec b, \epsilon)$ concentrates on
\begin{align*}
  \left\lceil  \frac{2\log \epsilon}{\log c_N} \right\rceil,
\end{align*}
as $N \to \infty$.  And while we do not consider the size of the limiting variance of $\|\vec r_k\|_2$ for the general spiked sample covariance model, it is expected to grow at the same rate as in \cite[Remark 4, Equation (7)]{Paquette2020} implying a similar statement using the exponential prefactors in \eqref{eq:res-formula}.

But, on the other hand, if $p_N$ tends to zero at an exponential rate, while concentration may still occur, the value about which $\tau(W,\vec b, \epsilon)$ concentrates may be different.    Concentration is demonstrated in Figures~\ref{fig:polyscale} and \ref{fig:expscale} and can be explained at a heuristic level by the transition from linear to superlinear convergence as discussed in \cite{Beckermann2001}.  Very large values of $N$ are considered using the explicit distributional formulae in \cite[Theorem 1.2]{Paquette2020}.

\begin{figure}[htbp]  
\centering
\includegraphics[width=.49\linewidth]{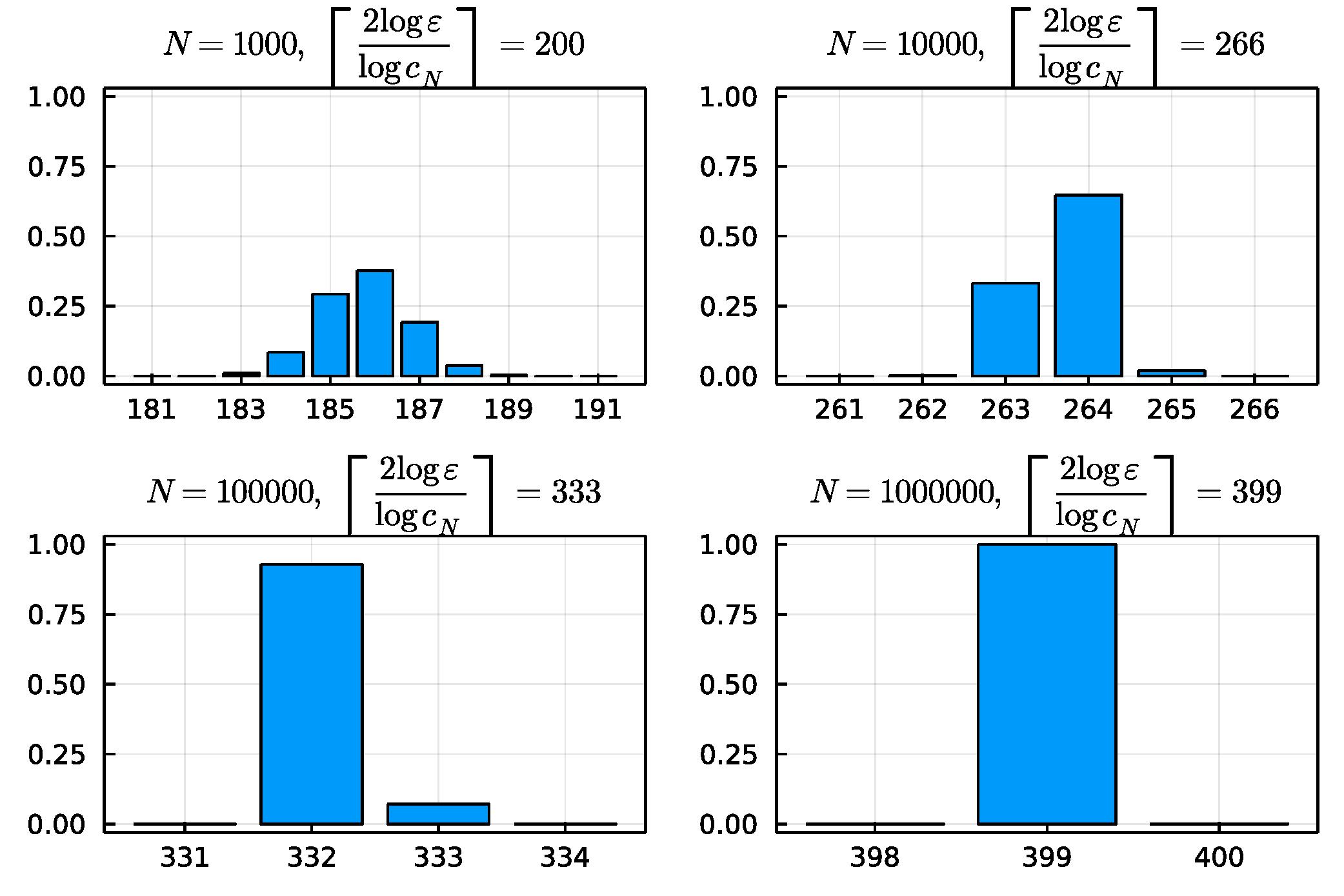}
\caption{\label{fig:polyscale} Statistics of $\tau(W,\vec b, \epsilon)$ when $\Sigma = \Sigma_0 = I$ and $p_N = N^{-10}$.  The results presented in this paper imply that the histogram for $\tau(W,\vec b, \epsilon)$ will concentrate on  $\left\lceil  \frac{2\log \epsilon}{\log c_N} \right\rceil$. }
\end{figure}

\begin{figure}[htbp]  
\centering
\includegraphics[width=.49\linewidth]{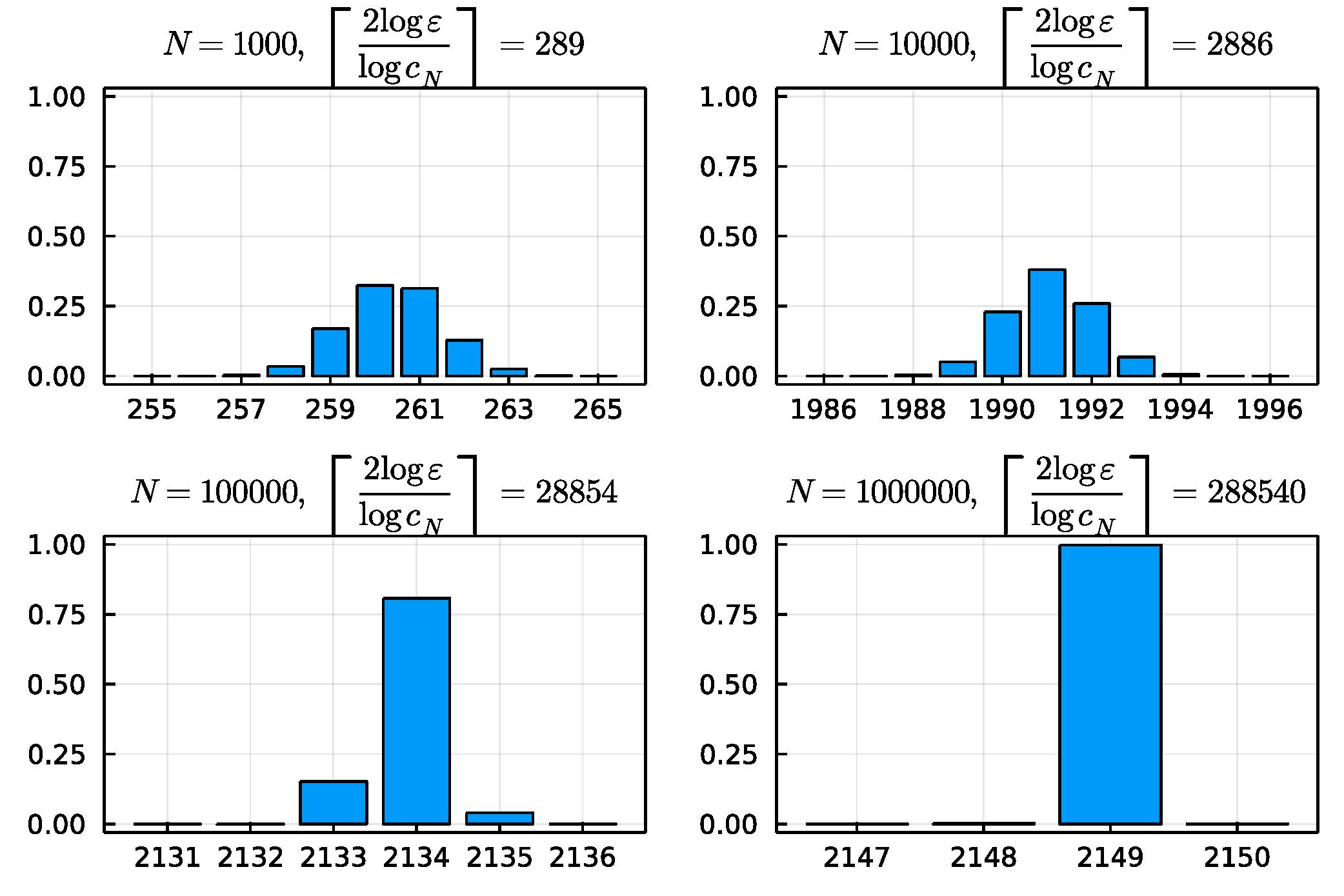}
\caption{\label{fig:expscale} Statistics of $\tau(W,\vec b, \epsilon)$ when $\Sigma = \Sigma_0 = I$ and $p_N = \e^{-N/10}$.  While it does appear that, remarkably, the histogram for $\tau(W,\vec b, \epsilon)$ concentrates the results of the current paper neither predict this nor determine this value.  Note that this value is orders of magnitude smaller than $\left\lceil  \frac{2\log \epsilon}{\log c_N} \right\rceil$ indicating that superlinear convergence is required to achieve this.}
\end{figure}

\subsection{The Jacobi and Cholesky matrices}
In this subsection, we analyze the entries of the Jacobi matrix in (\ref{eq:jacobi_def}) and its associated Cholesky decomposition in (\ref{eq:L_def}). We first pause to review some classical results.  The Householder tridiagonalization of a real symmetric or complex Hermitian matrix $W$ is a fundamental numerical process.  The process is succinctly described by the selection of a sequence of Householder reflectors, $U_1,\ldots, U_N$, so that
\begin{align*}
  U_NU_{N-1} \cdots U_1 W U_1^* \cdots U_{N-1}^* U_N^* = J,
\end{align*}
is a symmetric tridiagonal matrix.  The typical convention is to select each $U_j$ so that the only non-zero entry in the first row and first column is a one in the $(1,1)$-entry.  The off-diagonal entries of $J$ can be chosen to be non-negative. 

 When $W \lawequals XX^*$, $X_{ij} \lawequals \mathcal N(0, M^{-1})$, the case of a Wishart matrix, the distribution of $J$ can be calculated explicitly \cite{Silverstein1985,Dumitriu2002} and is given by
\begin{align}\label{eq:chol_wishart}
  J \lawequals L L^T, \quad L = \frac{1}{\sqrt{M}} \begin{bmatrix} \chi_{\beta M} \\
    \chi_{\beta(N-1)} & \chi_{\beta (M-1)} \\
    & \chi_{\beta(N-2)} & \chi_{\beta (M-2)} \\
    && \ddots & \ddots \\
    && & \chi_{\beta} & \chi_{\beta (M-N+1)}  \end{bmatrix},
\end{align}
where $\chi_\gamma$ is a $\chi$-distributed random variable with $\gamma$ degrees of freedom and all the entries of $L$ are independent. Here $\beta = 1$ if the matrix $W$ has real entries.  In another way of speaking, the matrix $L$ gives the distribution of the Cholesky factorization of the tridiagonalization of a Wishart matrix.  One can generalize this tridiagonalization by asking that the first column of $U_1^* \cdots U_{N-1}^* U_N^*$ be a prescribed vector $\vec b$ so that
\begin{align}\label{eq:T_L}
  T = T(W,\vec b), \quad L = L(W,\vec b).
\end{align}
This can be accomplished by simply constructing a matrix $U_0$, $U_0^*U_0 = I$ whose first column is $\vec b$ and apply the Householder tridiagonalization procedure to $U_0^* W U_0$. In this case, the tridiagonal matrix that results coincides with the output of the Lanczos algorithm\footnote{In numerical linear algebra these two methods are treated as distinct, in part, because they have vastly different behavior in finite-precision arithmetic.}.

More is true.  Consider the discrete measure
\begin{align*}
  \nu = \nu_{W,\vec b} = \sum_{i=1}^N |\langle \vec u_i, \vec b \rangle |^2 \delta_{\lambda_i(W)},
\end{align*}
for a general positive definite matrix $W$.  Then,
\begin{align*}
  T(W,\vec b) = J_N(\nu),\\ L(W,\vec b) = \mathcal L_N(\nu),
\end{align*}
which directly connects the output of the algorithms to the VESD. In the context of Wishart matrix, supposing $c_N = N/M \to \mathtt c \in (0,1]$, one can immediately see that the $(k,k)$ and $(k,k-1)$ entries of $L$ in \eqref{eq:chol_wishart} tend to $1$ and $\sqrt{\mathtt c}$, respectively, provided $k \ll N$.  Furthermore, the fluctuations will be Gaussian, by the central limit theorem.  It is of intrinsic interest to ask if this phenomenon persists for the spiked sample covariance model we analyze here. Our results establish this for $k \ll N^{1/6}$ and we conjecture it holds for $k \ll N$.

We now explain simulations based on the matrix model defined by \eqref{eq:single_gap_spiked} to demonstrate both our results and add evidence that that $k \ll N$ is necessary. Let $\nu$ be given by \eqref{eq:VESD} with limiting measure $\mu$ using the setting (\ref{eq:single_gap_spiked}).  As stated, the tridiagonalization of the spiked sample covariance model and its Cholesky factorization are given by $J_N(\nu)$ and $\mathcal L_N(\nu)$ using the notation of \eqref{eq:jacobi_def} and \eqref{eq:L_def}.  In this section, we examine $a_k(\nu)$ and $\alpha_k(\nu)$ for $k \leq 8 N^{1/6} + 10$ and $k \leq  N/3$ using the results of Section \ref{sec_algorithmicapp}.  Figure~\ref{fig:sublinear} demonstrates a consequence of our results\footnote{While our results technically only hold for $k \ll N^{1/6}$, allowing $k \leq 8 N^{1/6} + 10$ demonstrates that we expect our results up to hold up to this threshold.} that the entries of $J_k(\nu)$ concentrate on those of $J_k(\mu)$ for $k  \ll N^{1/6}$.  If we allow $k$ to be proportional to $N$ we do not expect this to occur as Figure~\ref{fig:linear} demonstrates.

\begin{figure}[!ht]  
\centering
\subfigure[]{\includegraphics[width=.49\linewidth]{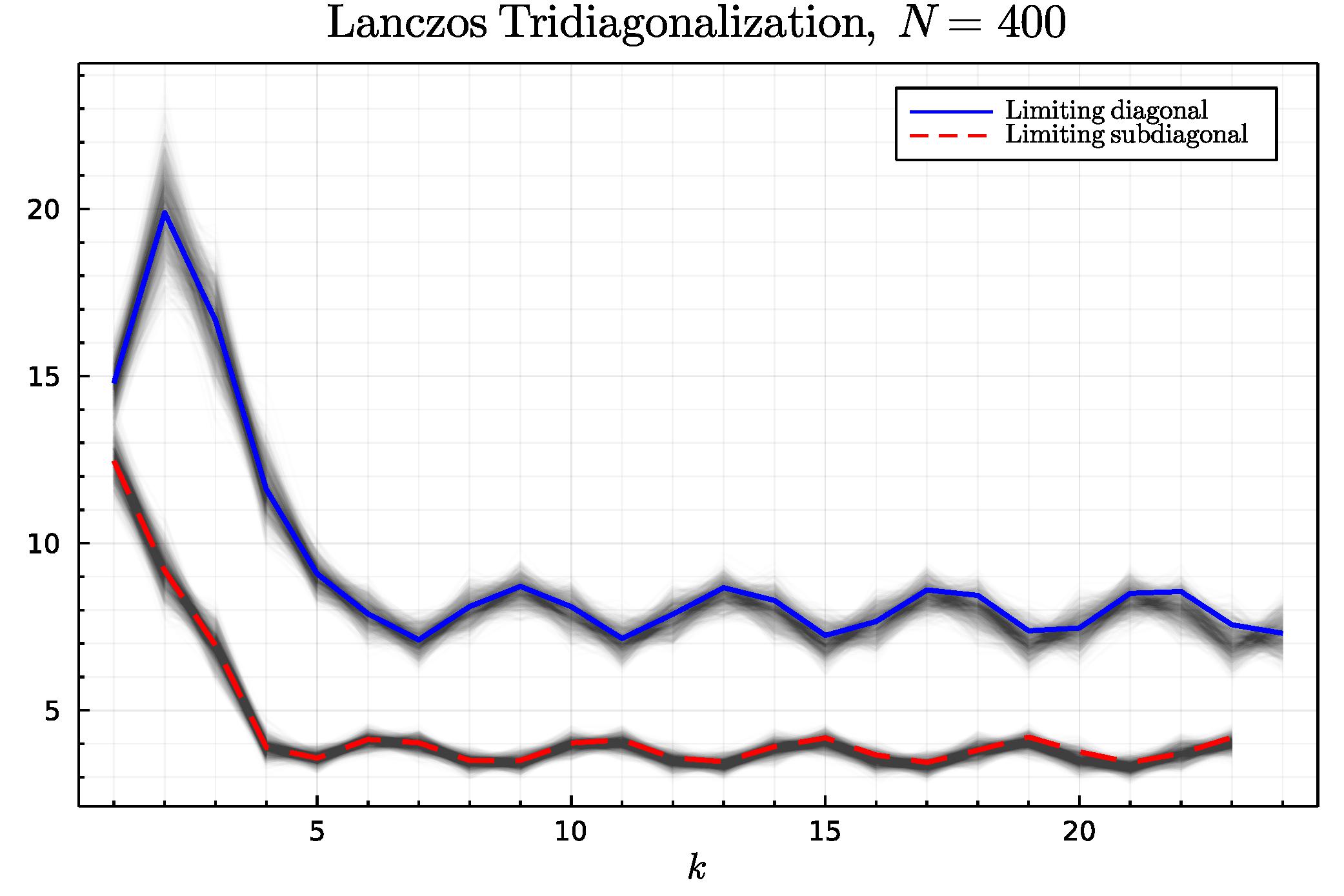}}
\subfigure[]{\includegraphics[width=.49\linewidth]{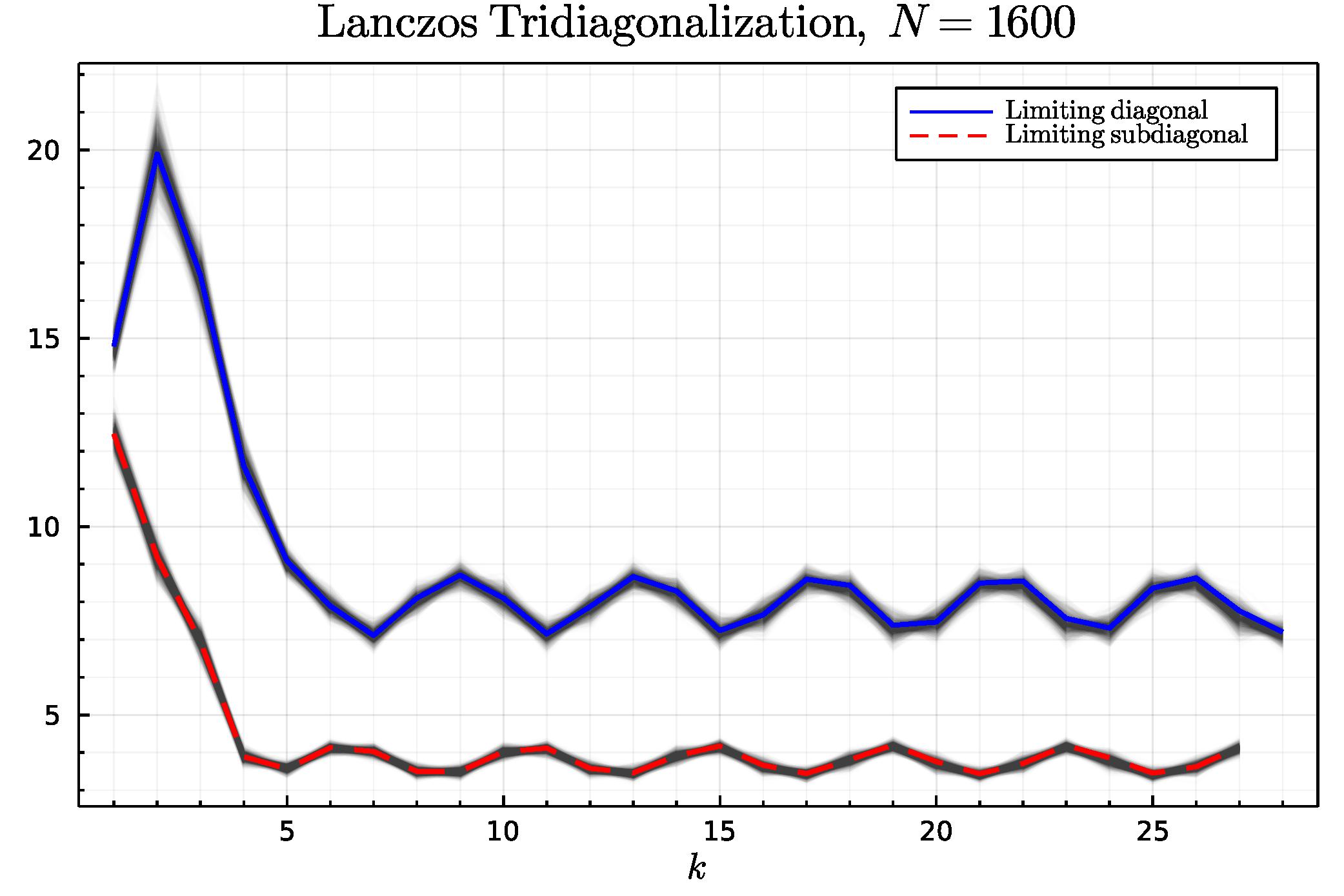}}
\subfigure[]{\includegraphics[width=.49\linewidth]{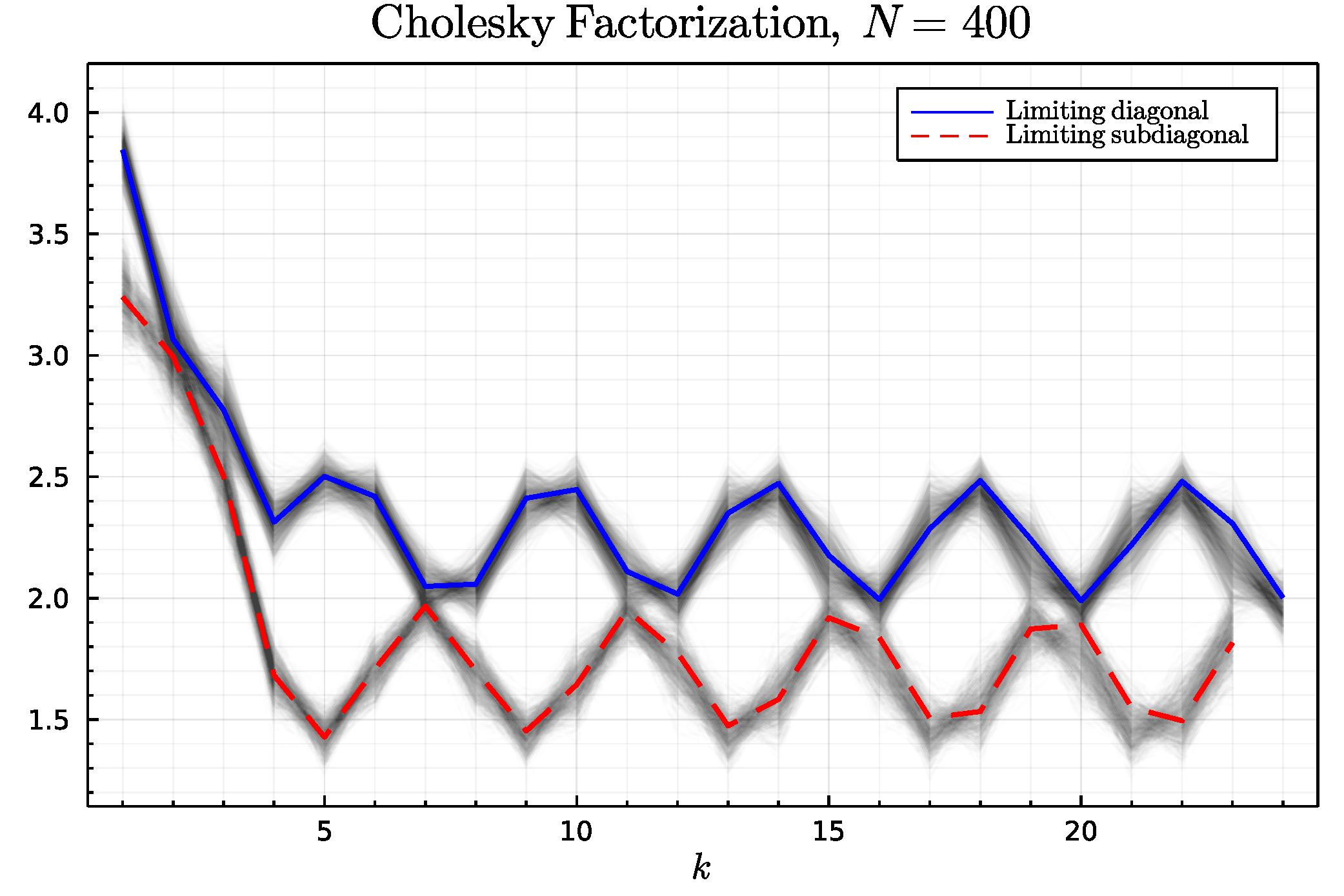}}
\subfigure[]{\includegraphics[width=.49\linewidth]{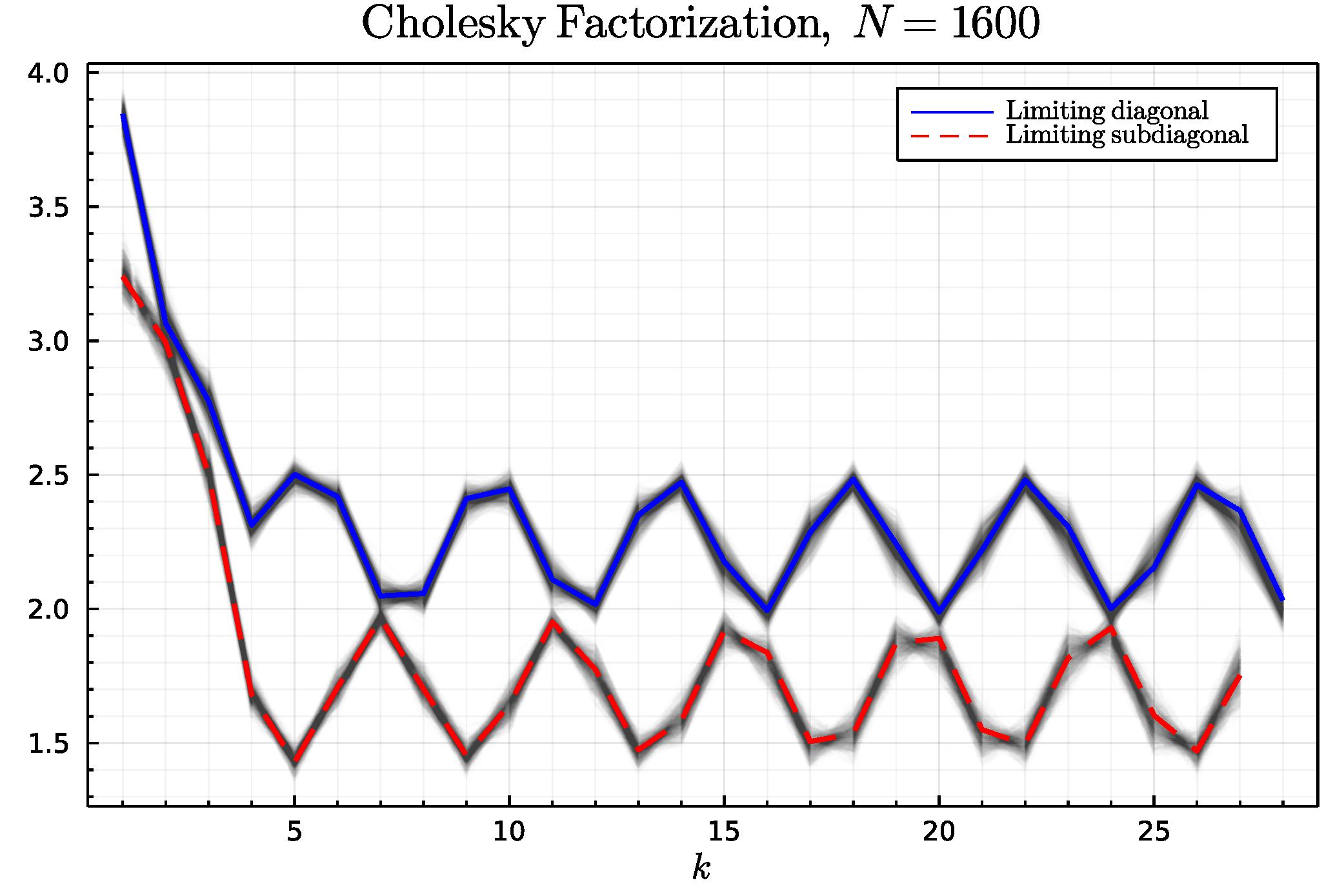}}
\caption{\label{fig:sublinear} The first $k$ entries of the matrices $J_N(\nu)$ and $\mathcal L_N(\nu)$ for $k \leq 5 N^{1/6} + 10$ in the case of \eqref{eq:single_gap_spiked}.  The solid blue and dashed red curves give the large $N$ limit of the diagonal and subdiagonal, respectively, computed using the results of Theorem \ref{thm_pertubed} with the parameters calculated using the methods outlined in Section \ref{sec_calculationofkeyparameters}.  The shaded region is produced using 1000 samples for the displayed value of $N$.     }
\end{figure}

\begin{figure}[!ht]  
\centering
\subfigure[]{\includegraphics[width=.49\linewidth]{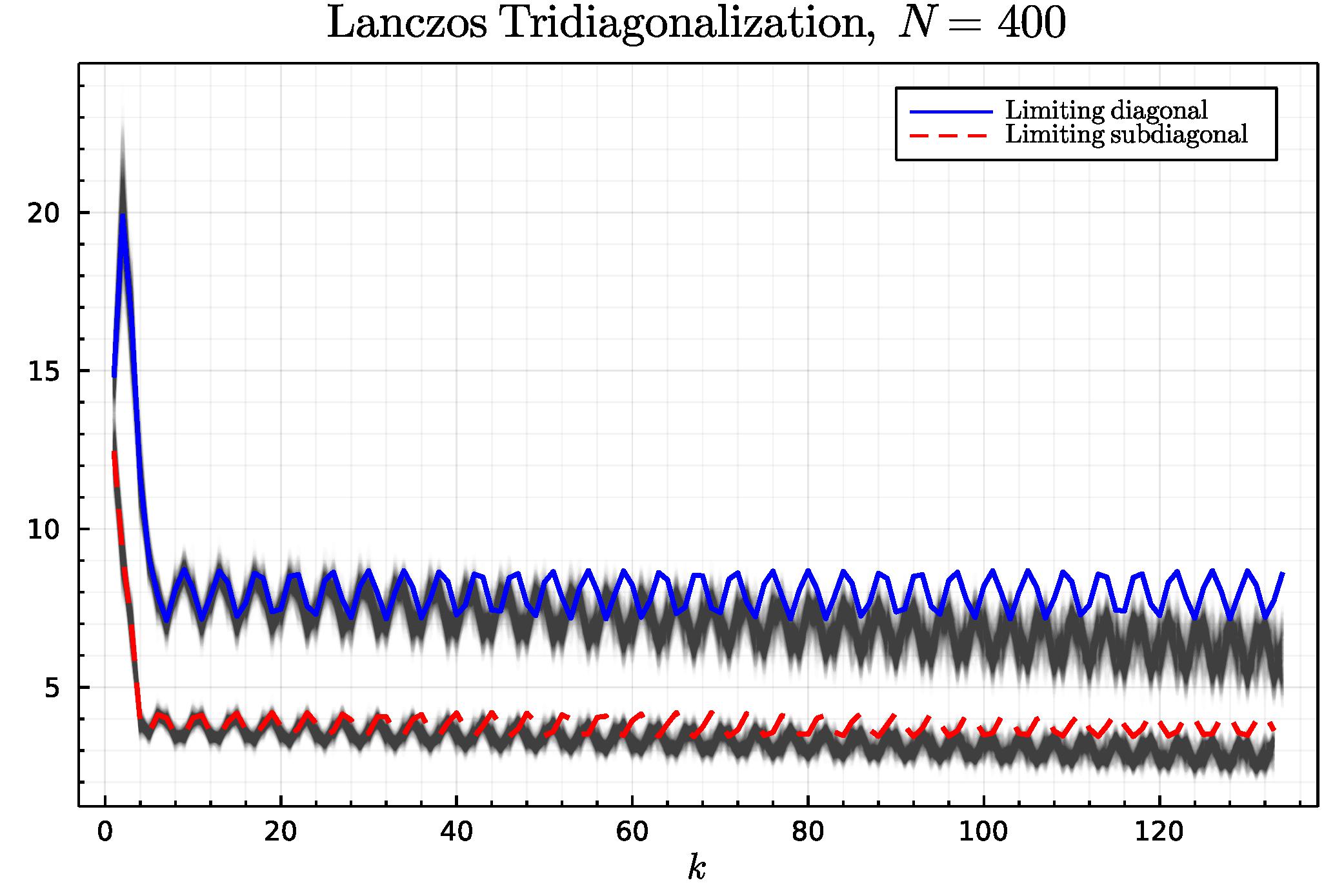}}
\subfigure[]{\includegraphics[width=.49\linewidth]{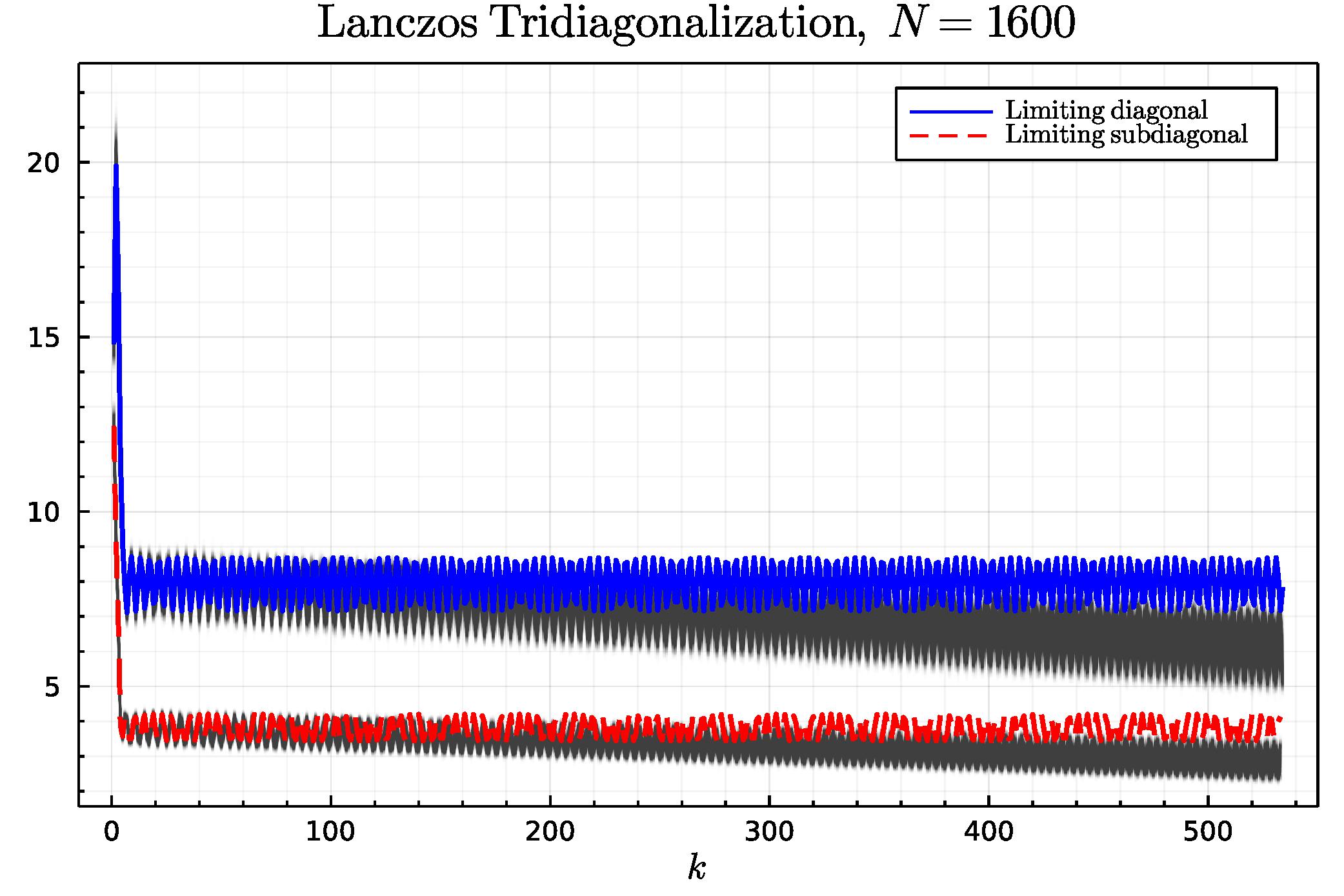}}
\subfigure[]{\includegraphics[width=.49\linewidth]{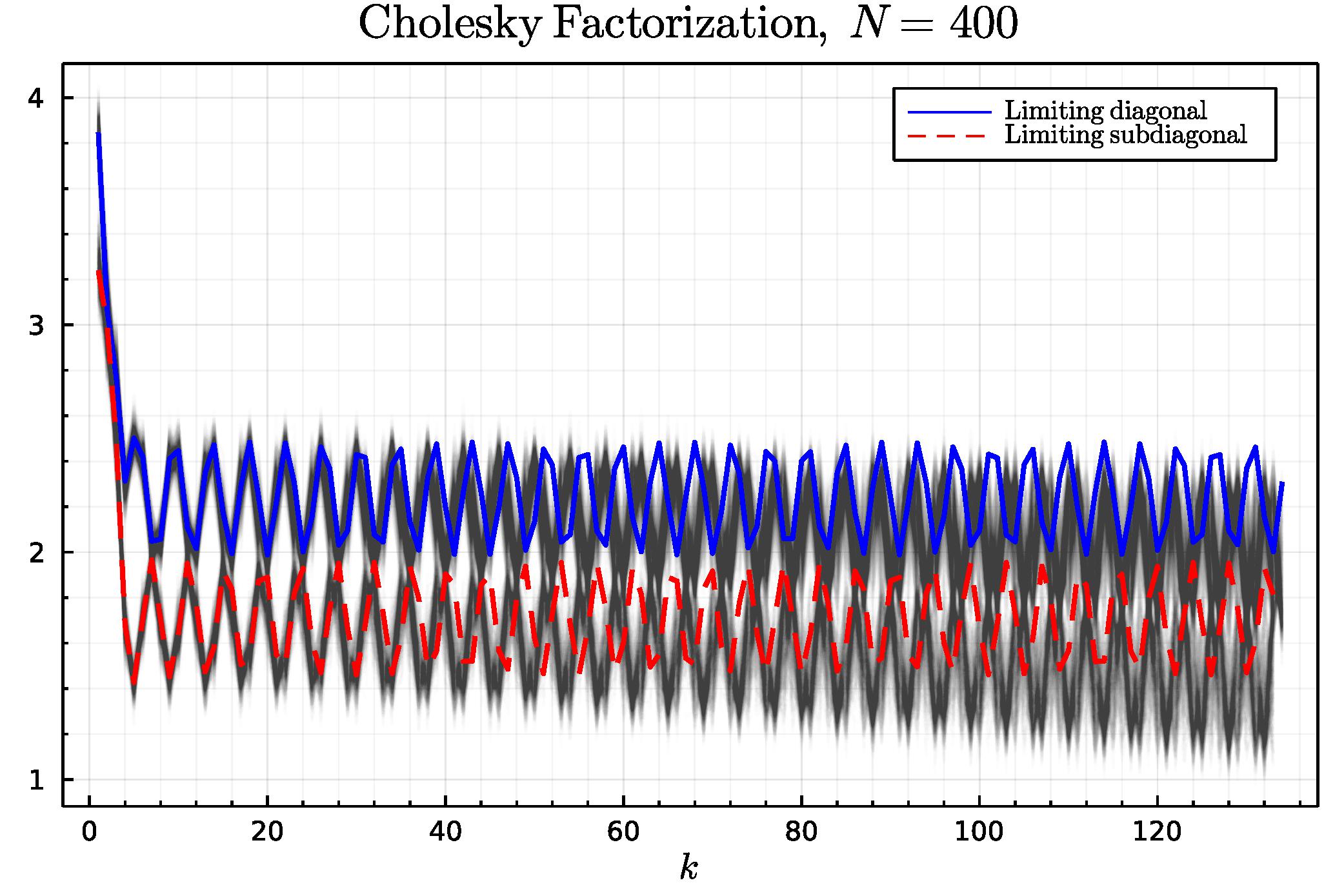}}
\subfigure[]{\includegraphics[width=.49\linewidth]{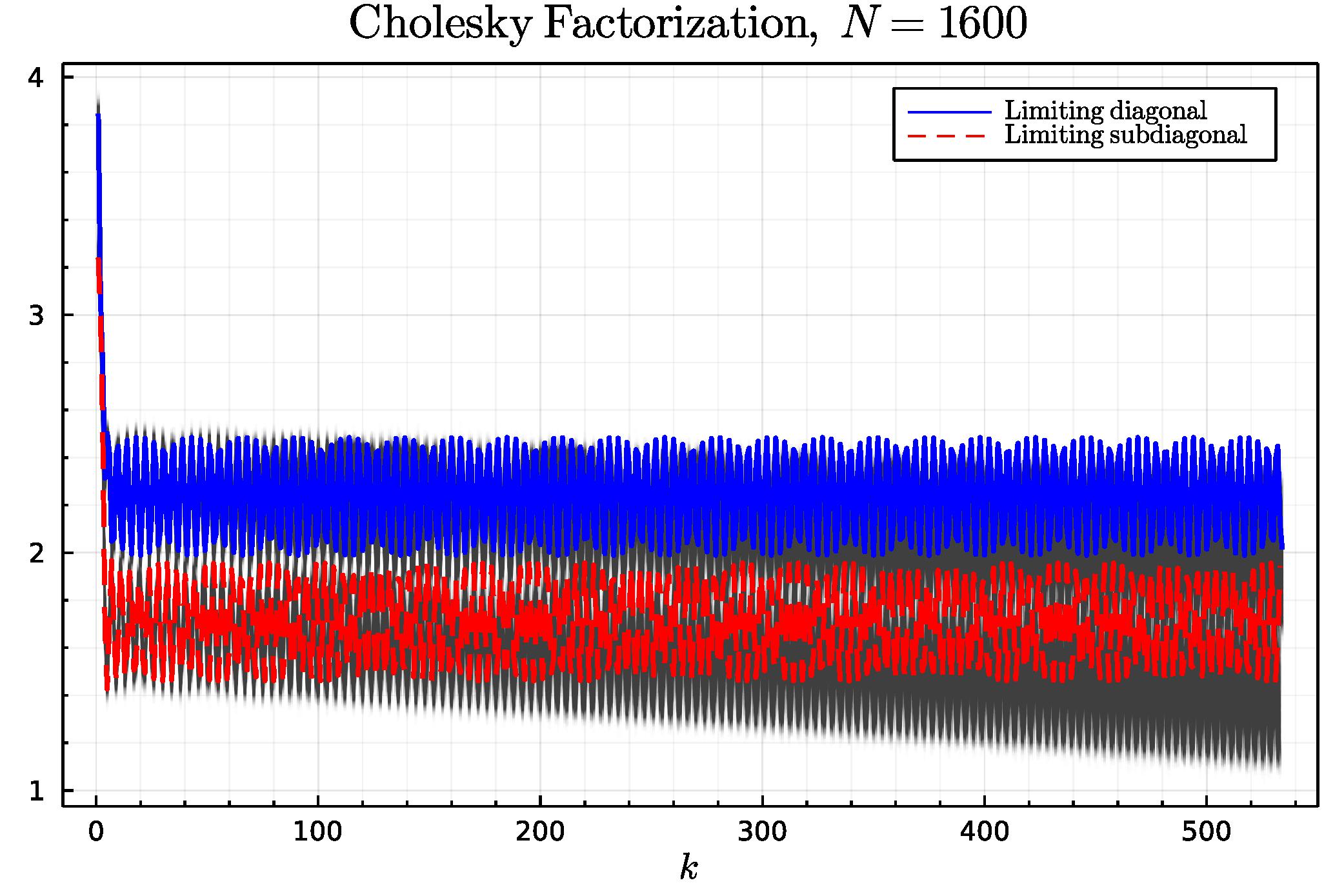}}
\caption{\label{fig:linear} The first $k$ entries of the matrices $J_N(\nu)$ and $\mathcal L_N(\nu)$ for $k \leq  N/3$ in the case of \eqref{eq:single_gap_spiked}. }
% The solid blue and dashed red curves give the diagonal and subdiagonal, respectively, of $\mathcal J(\mu)$ and $\mathcal L(\mu)$ computed using the methodology outlined in Appendix~\ref{app:densityapproximation}.  The shaded region is produced using 1000 samples for the displayed value of $N$. }
\end{figure}

Lastly, we consider the fluctuations of the diagonal elements of $\mathcal J(\nu)$ where $\nu$ is the VESD in \eqref{eq:VESD}.  We have shown that as $N \to \infty$, for $k$ fixed, the fluctuations of $a_k$ are Gaussian.  Furthermore, by Theorem~\ref{thm_mainclt}, Corollary \ref{coro_explicitdistribution} and Remark \ref{rmk_final},  the variance depends on the fourth moment of the matrix entries.  We confirm this clearly in the top two panels of Figure~\ref{fig:clt}.  But as Remark~\ref{rem:mom_decay} points out, as $k$ increases the dependence on the fourth moment should become negligible.  Figure~\ref{fig:clt} demonstrates that this happens quickly.

\begin{figure}[htbp]  
  \centering
  \includegraphics[width=.6\linewidth]{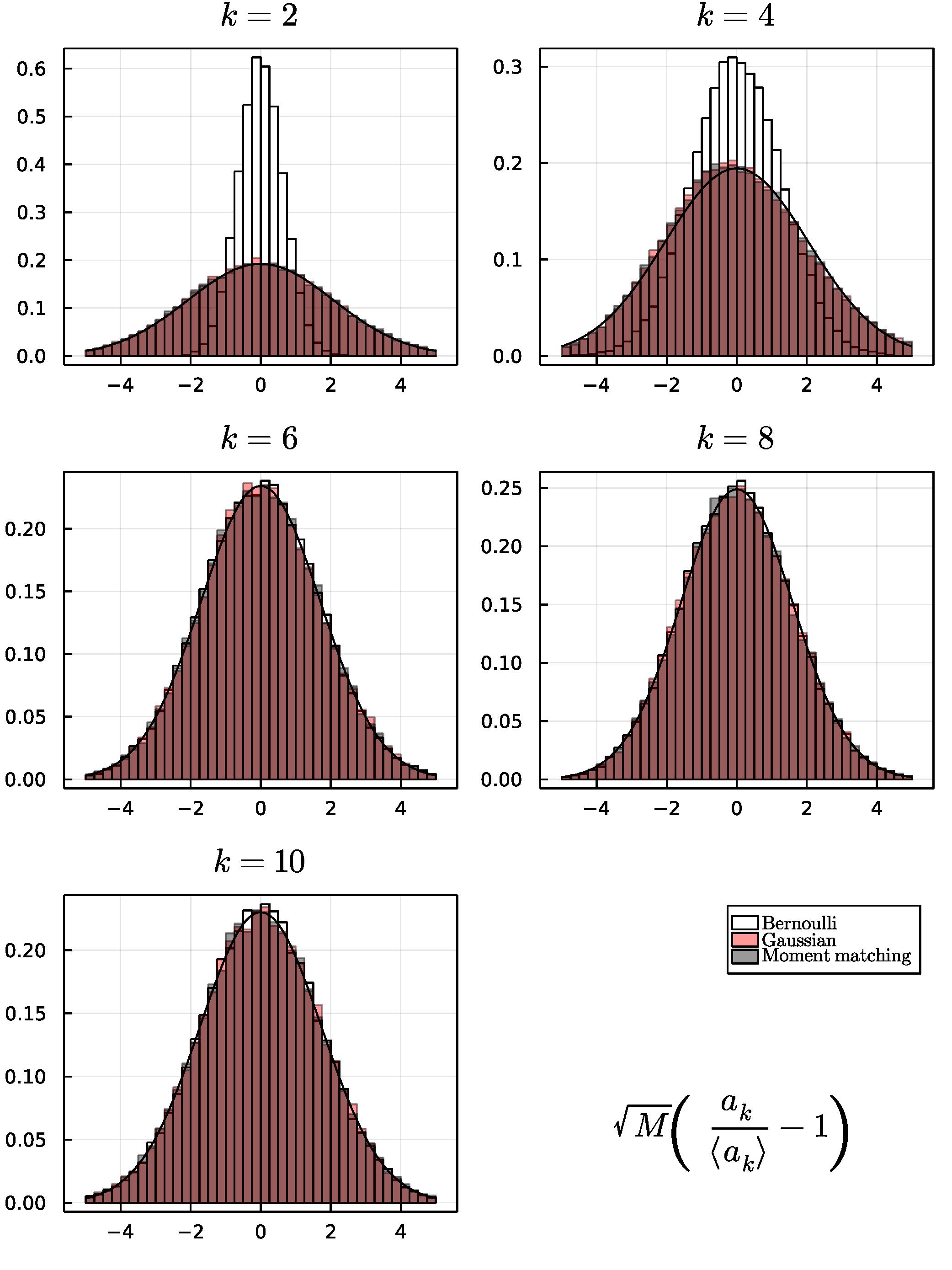}
  \caption{\label{fig:clt} Statistics of $a_k$ for the model in \eqref{eq:single_gap_spiked} for different choices of distributions on the entries $X_{ij}$ when $N = 1000$.  For each choice of distribution we plot a histogram for $\sqrt{M} ( a_k/\langle a_k \rangle  -1 )$  using $50,000$ samples where $\langle \cdot \rangle$ gives the sample average over these $50,000$ samples.  The thin black curve is the density for a normal distribution with mean zero and variance determined by the sample variance of $\sqrt{M} ( a_k/\langle a_k \rangle  -1 )$ when $X_{ij} \lawequals \mathcal N(0, M^{-1})$.      The shaded red area gives the histogram for $\mathcal N(0,M^{-1})$ entries, the shaded gray area gives the histogram  for the discrete distribution on $\{-1/\sqrt{M}, 0, 1/\sqrt{M} \}$ that matches its first four moments with $\mathcal N(0, M^{-1})$, and the white histogram is produced by $X_{ij} = \pm 1/\sqrt{M}$ with equal probability (Bernoulli). For smaller values of $k$ the variance clearly is different between the moment matching distribution and the Bernoulli distribution.  As $k$ increases, this difference dramatically diminishes, as predicted.}
\end{figure}

% {\color{red}[do we need to discuss the non-universal issue? In our current setting, the CLT is always Gaussian and universal. ]}

\appendix

\section{Orthogonal polynomials and their asymptotics: Proof of Theorem \ref{lem_deterministicexpansion}}\label{app:OPs}

\subsection{A Riemann surface}
In order to describe the asymptotics of polynomials orthogonal with respect to a measure $\mu$ from \eqref{eq:mu} satisfying the assumptions (1)-(5) we need to describe a Riemann surface. General references for what follows are \cite{Algebro,DeiftWeights1,Baik2007}. Associated with the intervals $[\mathtt a_j,\mathtt b_j], 1 \leq j \leq g+1,$ is a Riemann surface, described by the solution set of
\begin{align*}
  w^2 = \prod_{j=1}^{g+1} (z - \mathtt a_j)(z - \mathtt b_j) =: P_{2g+2}(z),
\end{align*}
in $\mathbb C^2$.  Consider a cut version of the complex plane:
\begin{align*}
  \hat {\mathbb C} &= \mathbb C \setminus \bigcup_{j=1}^{g+1} [\mathtt a_j,\mathtt b_j].
\end{align*}
Then define a sectionally analytic function
\begin{align*}
  &R: \hat {\mathbb C} \to \mathbb C, \quad R(z)^2 = P_{2g+2}(z), \quad R(z) \to 1, \quad \text{as} \quad z \to \infty.
\end{align*}
A Riemann surface $\Gamma$ can be constructed by adjoining copies of $\hat {\mathbb C}$; see Figure \ref{fig_riemman} for an illustration and a description of the $\mathfrak a$-cycles and $\mathfrak b$-cycles.  We have a natural projection $\pi: \Gamma \to \mathbb C$ defined by $\pi((z,w)) = z$ and its right-inverses $\pi_j^{-1}(z) = (z, (-1)^{j+1} R(z))$, $j = 1,2$.

\begin{figure}[h]
\begin{tikzpicture}[scale=1]
\coordinate (a1) at (0,0);
\coordinate (b1) at (1,0);
\coordinate (a2) at (2.5,0);
\coordinate (b2) at (3.5,0);
%\coordinate (zo) at (\zox,0);
%\coordinate (bpp) at (1,1);
%\coordinate (bmp) at (-1,1);
%\coordinate (bmm) at (-1,-1);
%\coordinate (bpm) at (1,-1);

%\draw () -- (1,0);

 \draw[smooth, line cap = round, line width=1pt] plot[tension=0.65] coordinates{(6.5,-0.85) (6,-0.75) (4.5,-1.5) (3,-0.75) (1.5,-1.5) (0,0) (1.5,1.5) (3, 0.75) (4.5, 1.5) (6,0.75) (6.5,0.85)};
 
 \draw[smooth, line cap = round, line width=1pt] plot[tension=0.65] coordinates{(7.5,-0.85) (8,-0.75) (9.5,-1.5) (11,0) (9.5, 1.5) (8, 0.75) (7.5,0.85)};
 
 \node at (7,0) {$\cdots$};
 
\draw[smooth, line cap = round, line width=1pt] plot[tension=0.65] coordinates{(1,0) (1.5,0.25) (2,0)};
\draw[smooth, line cap = round, line width=1pt] plot[tension=0.65] coordinates{(0.875, 0.25) (1,0) (1.5, -0.25) (2,0) (2.125,0.25)};

\draw[smooth, line cap = round, line width=1pt] plot[tension=0.65] coordinates{(4,0) (4.5,0.25) (5,0)};
\draw[smooth, line cap = round, line width=1pt] plot[tension=0.65] coordinates{(3.875, 0.25) (4,0) (4.5, -0.25) (5,0) (5.125,0.25)};

\draw[smooth, line cap = round, line width=1pt] plot[tension=0.65] coordinates{(9,0) (9.5,0.25) (10,0)};
\draw[smooth, line cap = round, line width=1pt] plot[tension=0.65] coordinates{(8.875, 0.25) (9,0) (9.5, -0.25) (10,0) (10.125,0.25)};

\node[left] at (0,0) {\small $\mathtt a_1$};
\node at (0,0) {\color{NicePurple}\textbullet};
\node[left] at (1,0) {\small $\mathtt b_1$};
\node at (1,0) {\color{NicePurple}\textbullet};
\node[right] at (2,0) {\small $\mathtt a_2$};
\node at (2,0) {\color{NicePurple}\textbullet};

\node[left] at (4,0) {\small $\mathtt b_2$};
\node at (4,0) {\color{NicePurple}\textbullet};
\node[right] at (5,0) {\small $\mathtt a_3$};
\node at (5,0) {\color{NicePurple}\textbullet};

\node[left] at (9,0) {\small $\mathtt b_g$};
\node at (9,0) {\color{NicePurple}\textbullet};
\node[right] at (9.96,0) {\small $\mathtt a_{g+1}$};
\node at (10,0) {\color{NicePurple}\textbullet};

\node[right] at (11,0) {\small $\mathtt b_{g+1}$};
\node at (11,0) {\color{NicePurple}\textbullet};

%a-cycles

\begin{scope}[decoration={
  	 markings,
	 mark=at position 0.56 with {\arrow[line width =1.2pt]{>}}
	 }
]
\draw[smooth cycle, line cap = round, line width=1pt, postaction=decorate, MidnightBlue] plot[tension=0.65] coordinates{(0.2,0) (1.5,-0.75) (2.8,0) (1.5, 0.75)};
\node[above, yshift=-2pt] at (1.5,0.75) {\color{MidnightBlue}\small ${\mathfrak{a}}_1$};

\draw[smooth cycle, line cap = round, line width=1pt, postaction=decorate, MidnightBlue] plot[tension=0.65] coordinates{(3+0.2,0) (3+1.5,-0.75) (3+2.8,0) (3+1.5, 0.75)};
\node[above, yshift=-2pt] at (3+1.5,0.75) {\color{MidnightBlue}\small ${\mathfrak{a}}_2$};

\draw[smooth cycle, line cap = round, line width=1pt, postaction=decorate, MidnightBlue] plot[tension=0.65] coordinates{(8+0.2,0) (8+1.5,-0.75) (8+2.8,0) (8+1.5, 0.75)};
\node[above, yshift=-2pt] at (8+1.5,0.75) {\color{MidnightBlue}\small ${\mathfrak{a}}_g$};

\end{scope}

%b-cycles

\begin{scope}[decoration={
  	 markings,
	 mark=at position 0.56 with {\arrow[line width =1.2pt]{<}}
	 }
]
\draw[smooth, line cap = round, line width=1pt, postaction=decorate] plot[tension=0.65] coordinates{(1.5,-0.25) (1.25,-0.95) (1.5, -1.5)};
\draw[smooth, line cap = round, line width=1pt, postaction=decorate, dashed] plot[tension=0.65] coordinates{(1.5, -1.5) (1.65,-0.95) (1.5,-0.25)};
\node[xshift=-5pt, yshift=-2pt] at (1.25,-0.95) {\small ${\mathfrak{b}}_1$};

\draw[smooth, line cap = round, line width=1pt, postaction=decorate] plot[tension=0.65] coordinates{(3+1.5,-0.25) (3+1.25,-0.95) (3+1.5, -1.5)};
\draw[smooth, line cap = round, line width=1pt, postaction=decorate, dashed] plot[tension=0.65] coordinates{(3+1.5, -1.5) (3+1.65,-0.95) (3+1.5,-0.25)};
\node[xshift=-5pt, yshift=-2pt] at (3+1.25,-0.95) {\small ${\mathfrak{b}}_2$};

\draw[smooth, line cap = round, line width=1pt, postaction=decorate] plot[tension=0.65] coordinates{(8+1.5,-0.25) (8+1.25,-0.95) (8+1.5, -1.5)};
\draw[smooth, line cap = round, line width=1pt, postaction=decorate, dashed] plot[tension=0.65] coordinates{(8+1.5, -1.5) (8+1.65,-0.95) (8+1.5,-0.25)};
\node[xshift=-5pt, yshift=-2pt] at (8+1.25,-0.95) {\small ${\mathfrak{b}}_g$};

 \end{scope}

\end{tikzpicture}
\caption{An illustration of the Riemann surface $\Gamma$. }\label{fig_riemman}
\end{figure}
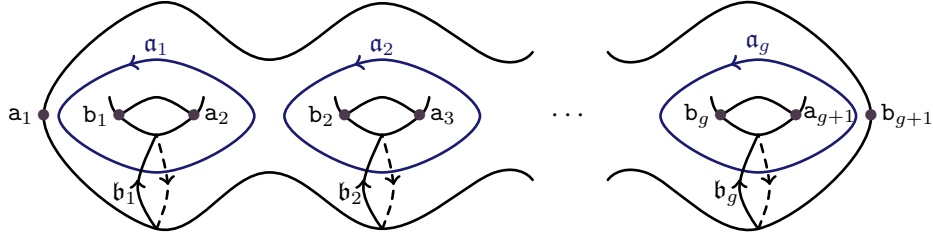

%%%%%

%%%%%

%%%%%

%%%%%

%%%%%

As is well-known (see \cite{Algebro}, for example) a basis for holomorphic differentials on $\Gamma$ is given by
\begin{align*}
  \sd \nu_j = \frac{z^{j-1}}{R(z)} \sd z, \quad j = 1,2,\ldots,g.
\end{align*}
Define the $g\times g$ period matrix $A$ by
\begin{align*}
  A_{ij} = \oint_{{\mathfrak a}_i} \sd \nu_j.
\end{align*}
Note that if $c = \begin{bmatrix} c_1 & c_2 & \cdots & c_g \end{bmatrix}^T  = {A}^{-1} e_j$ for the standard basis vector $e_j$, then
\begin{align*}
  \oint_{{\mathfrak a}_i} \sum_{k=1}^g c_k \sd \nu_k =  \sum_{k=1}^g c_k A_{ik} = e_i^T A c = e_i^T e_j = \delta_{ij}.
\end{align*}
So, we define a basis of normalized differentials
\begin{align*}
  \begin{bmatrix} \sd \omega_1 \\ \sd \omega_2 \\ \vdots \\ \sd \omega_g \end{bmatrix} = 2 \pi \I{ A}^{-1}  \begin{bmatrix} \sd \nu_1 \\ \sd \nu_2 \\ \vdots \\ \sd \nu_g \end{bmatrix},
\end{align*}
which satisfies
\begin{align*}
  \oint_{{\mathfrak a}_i} \sd \omega_j = 2 \pi \I \delta_{ij}.
\end{align*}
The invertiblity of the matrix $A$ follows from abstract theory as in \cite{Bobenko}.

Now fix the base point $a = \mathtt a_1$ and define
\begin{align*}
  u(z) = \left( \int_a^z \sd \omega_j \right)_{j=1}^g, \quad z \not \in \mathbb R, 
\end{align*}
where the path of integration is taken to be a straight line connecting $a$ to $z$.  Note that this extends to a vector-valued holomorphic function\footnote{We abuse notation here and treat $u$ as both a function of $z \in \mathbb C \setminus \mathbb R$ and a function of $P \in \Gamma$.}  $u(P)$ on the Riemann surface $ \Gamma$ provided $\Gamma$ is cut along the cycles $\{{\mathfrak a}_1, \ldots, {\mathfrak a}_g,  {\mathfrak b}_1, \ldots, {\mathfrak b}_g\}$, making it simply connected.  Another important feature is that for $z \in \hat{\mathbb C}$, $u(\pi_1^{-1}(z)) = -u(\pi_2^{-1}(z))$.  

Define the associated Riemann matrix of ${\mathfrak b}$ periods,
$$
\tau=\left( \tau_{i j}\right)=\left(\int_{{\mathfrak b}_{j}} \sd \omega_{i}\right)_{1 \leq i, j \leq g}. 
$$
Note that $ \tau$ is symmetric and pure imaginary and $-\I  \tau$ is positive definite.  Next, define the vector $\vec k$ of Riemann constants component wise via
\begin{align*}
  \vec k_j = \frac{2 \pi \I +  \tau_{jj}}{2} - \frac{1}{2 \pi \I} \sum_{\ell \neq j} \oint_{ {\mathfrak a}_\ell} u_j \sd \omega_\ell, \quad j =1,2,\ldots,g.
\end{align*}
The associated theta function is given by
\begin{equation}\label{eq_thetafunction}
\theta(z;\tau)=\sum_{m \in \mathbb{Z}^{g}} \exp \left( \frac 1 2 (m,\tau m) + (m,z) \right), \quad z \in \mathbb{C}^{g},
\end{equation}
where $(\cdot, \cdot)$ is the real scalar product.  This series is convergent because $\tau$ has a negative-definite real part. The following hold:
\begin{align*}
  \theta(z + 2\pi \I e_j;\tau) &= \theta(z;\tau),\\
  \theta(z + \tau e_j; \tau) &= \exp \left( - \frac 1 2 \tau_{jj} - z_k\right) \theta(z;\tau).
\end{align*}

A divisor $D = \sum_j n_j P_j$ is a formal sum of points $\{P_j\}$ on the Riemann surface $\Gamma$.  The Abel map of a divisor is defined via
\begin{align*}
  \mathcal A (D) = \sum_j n_j u(P_j).
\end{align*}

We now determine the jumps satisfied by the vector-valued function,
 \begin{align}\label{eq_vectortheta}
   \Theta(z;d;v) = \Theta(z) :=  \begin{bmatrix} \displaystyle \frac{\theta \left( u(z) + v - d; \tau\right)}{\theta \left( u(z) - d; \tau\right)} &  \displaystyle\frac{\theta \left( -u(z) + v - d; \tau\right)}{\theta \left( -u(z) - d; \tau\right)} \end{bmatrix}, \quad z \not\in \mathbb R.
 \end{align}
Note that the first component function is nothing more than $\frac{\theta \left( u(P) + v - d; \tau\right)}{\theta \left( u(P) - d; \tau\right)}$ restricted to the first sheet.  The same is true for the second component function on the second sheet.  The vector $v$ is left arbitrary for now, and it will be chosen in a crucial way in what follows.

Then note that
\begin{align*}
  u^+(z) + u^-(z) = \left( 2 \sum_{k = 1}^{j-1} \int_{\mathtt b_k}^{\mathtt a_{k+1}} \sd \omega_\ell\right)_{\ell=1}^g = \left( \sum_{k = 1}^{j-1} \oint_{\mathfrak a_k} \sd \omega_\ell\right)_{\ell=1}^g = 2 \pi \I \mathsf N, \quad z \in [\mathtt a_j,\mathtt b_j],
\end{align*}
for a vector $\mathsf N$ of zeros and ones.  Then we compute   
\begin{align*}
  u^+(z) - u^-(z) = \left( 2 \sum_{k = 1}^{j} \int_{\mathtt a_k}^{\mathtt b_{k}} \sd \omega_\ell\right)_{\ell=1}^g = \left(  \oint_{\mathfrak b_{j}} \sd \omega_\ell\right)_{\ell=1}^g = \tau e_{j}, \quad z \in [\mathtt b_j,\mathtt a_{j+1}].
\end{align*}
Then check
\begin{align*}
  \frac{\theta \left( \pm u(z) + \tau e_{j}+ v - d; \tau\right)}{\theta \left( \pm u(z) + \tau e_{j} - d; \tau\right)} = \e^{\pm v_k} \frac{\theta \left( \pm u(z)+ v - d; \tau\right)}{\theta \left(  \pm u(z) - d; \tau\right)}.
\end{align*}
Then on $(-\infty,a_1)$ we have $u^+(z) = u^-(z)$.  And on $(\mathtt b_{g+1},\infty)$ we have
\begin{align*}
   u^+(z) - u^-(z) = \left(\oint_\mathcal{C}  \sd \omega_j\right)_{j=1}^g,
\end{align*}
where $\mathcal{C}$ is a clockwise-oriented simple contour that encircles $[\mathtt a_1,\mathtt b_{g+1}]$. Then because all the differentials $\sd \omega_j$ are of the form $P(z)/R(z)$ where $P$ is a degree $g-1$ polynomial and $R(z) = \OO(z^{g+1})$ as $z \to \infty$, we see that $\oint_\mathcal{C} \sd \omega_j = 0$.
Thus, ignoring any poles $\Theta$ may have, we find that $\Theta$ satisfies the following jump conditions:
\begin{align*}
  \Theta^+(z) = \begin{cases} \Theta^-(z) \begin{bmatrix} 0 & 1 \\ 1 & 0 \end{bmatrix} & z \in (\mathtt a_j,\mathtt b_j),\\ \\
    \Theta^-(z) \begin{bmatrix} \e^{-v_j} & 0 \\ 0 & \e^{v_j} \end{bmatrix} & z \in (\mathtt b_j,\mathtt a_{j+1}),\\
     \Theta^-(z) & z \in (-\infty,\mathtt a_1) \cup (\mathtt b_{g+1},\infty). \end{cases}
\end{align*}
Also, note that since $u(\infty)$ is well-defined, $\Theta$ has a limit as $z \to \infty$ and is analytic at infinity.

Of particular importance are the poles of $\Theta$. It is known that (see \cite{Algebro}, for example) if $\theta( u(P) - \mathcal A(D) - \vec k)$, $D = P_1 + \cdots + P_g$, is not identically zero\footnote{This holds if $D$ is nonspecial.}, then, counting multiplicities, $\theta( u(P) - \mathcal A(D) - \vec k)$,  has $g$ zeros on $\Sigma$.  These zeros are characterized by
\begin{align*}
\theta( u(P) - \mathcal A(D) - \vec k) = 0 \quad \Leftrightarrow \quad P = P_j,
\end{align*}
for some $j$.   Next, define
\begin{equation}\label{eq_gammaz}
\gamma(z)=\left[\prod_{j=1}^{g+1}\left(\frac{z-\mathtt b_{j}}{z-\mathtt a_{j}}\right)\right]^{1 / 4},
\end{equation}
analytic on $\mathbb{C} \setminus  \cup_j [\mathtt a_j,\mathtt b_j]$, with $\gamma(z) \sim 1, z \rightarrow \infty$.  It follows that $\gamma - \gamma^{-1}$ has a single root $z_j$ in $(\mathtt b_j,\mathtt a_{j+1})$ for $j =1,2,\ldots,g$, while $\gamma + \gamma^{-1}$ does not vanish on $\mathbb C \setminus \cup_j [\mathtt a_j,\mathtt b_j]$.  So, define two divisors
\begin{align*}
  D_1 = \sum_{j=1}^g \pi_1^{-1}(z_j), \quad  D_2 = \sum_{j=1}^g \pi_2^{-1}(z_j).
\end{align*}
It follows from \cite{DubrovinTheta} (see also \cite[Lemma 11.10]{TrogdonSOBook}) that these divisors are nonspecial and therefore the $\theta$ functions we will consider do not vanish identically.

Note that for $\vec d_1 := \mathcal A(D_1) + \vec k, $ the function $z \mapsto \theta(u(z) - \vec d_1;\tau)$ has zeros at $z_j$, while the function $z \mapsto \theta(-u(z) - \vec d_1;\tau)$ is non-vanishing.  Similarly, for 
\begin{align}\label{eq_d2}
  \vec d_2 := \mathcal A(D_2) + \vec k,
\end{align}
the function $z \mapsto \theta(-u(z) - \vec d_2;\tau)$ has zeros at $z_j$, while the function $z \mapsto \theta(u(z) - \vec d_2;\tau)$ is non-vanishing.

This leads us to consider
\begin{align}\label{eq:L}
  L_n(z) = \begin{bmatrix}  \left(\frac{\gamma(z) + \gamma(z)^{-1}}{2}\right) \Theta_1(z;\vec d_2;v) & \left(\frac{\gamma(z) - \gamma(z)^{-1}}{2\I}\right) \Theta_2(z;\vec d_2;v) \\ \\ 
  \left(\frac{\gamma(z)^{-1} - \gamma(z)}{2\I}\right) \Theta_1(z;\vec d_1;v) & \left(\frac{\gamma(z) + \gamma(z)^{-1}}{2}\right) \Theta_2(z;\vec d_1;v) \end{bmatrix},
\end{align}
which is analytic in $\mathbb C \setminus \cup_j [\mathtt a_j,\mathtt b_j]$, with a limit as $z \to \infty$ and satisfies the jumps:
\begin{align*}
  L_n^+(z) = \begin{cases} L_n^-(z) \begin{bmatrix} 0 & 1 \\ -1 & 0 \end{bmatrix} & z \in (\mathtt a_j,\mathtt b_j),\\ \\
    L_n^-(z) \begin{bmatrix} e^{-v_j} & 0 \\ 0 & e^{v_j} \end{bmatrix} & z \in (\mathtt b_j,\mathtt a_{j+1}),\\
     L_n^-(z) & z \in (-\infty,\mathtt a_1) \cup (\mathtt b_{g+1},\infty). \end{cases}
\end{align*}
This follows because $\gamma^+(z) = \I \gamma^-(z)$ for $z \in (\mathtt a_j,\mathtt b_j)$ and therefore
\begin{align*}
  \gamma^+(z) + (\gamma(z)^{-1})^+ &= \I \left( \gamma^-(z) - (\gamma(z)^{-1})^-\right),\\
   \gamma^+(z) - (\gamma(z)^{-1})^+ &= \I \left( \gamma^-(z) + (\gamma(z)^{-1})^-\right).
\end{align*}

\subsection{Asymptotics of orthogonal polynomials: Proof of Theorem \ref{lem_deterministicexpansion}}

The derivation of the asymptotic formulae proceeds in six steps, each of which transforms $Y_n(z;\mu)$ by explicit algebraic transformations:  
\begin{itemize}
\item \underline{Step 1}: Turn residue conditions into rational jump conditions.
\item \underline{Step 2}: The determination of a differential, also called the exterior Green's function with pole at infinity, that is used to remove the singularites of $Y_n$ at infinity.
\item \underline{Step 3}: Lens the Riemann--Hilbert problem, invoking analyticity of functions in the jump matrix, to judiciously factor and move jumps into regions where exponential decay can be induced.
\item \underline{Step 4}: Use the differential to remove the singularities at infinity and induce exponential decay (decay to the identity matrix) on contours moved away the support of $\mu$.
\item \underline{Step 5}: Determine the Szeg\H{o} function that removes the details of the remaining jumps and converts them to piecewise constant jumps.
\item \underline{Step 6}:  Now that the original unknown $Y_n$ has been transformed to something that has jump matrices that are exponentially close to being piecewise constant, the limiting ``model'' Riemann--Hilbert problem is solved explicitly using theta functions.
\end{itemize}
The result, after unwinding all the transformations, is an explicit asymptotic expression for $Y_n$ with exponentially small error terms.  This procedure is far from new as it is applied in this form to measures supported on a single interval in \cite{MR2087231,MR2022855} and in greater generality in \cite{YATTSELEV201573}.  We rederive the results of \cite{YATTSELEV201573} in our special case to make them more explicit.  

\subsubsection{Step 1: Residue conditions to rational jumps}

Consider the function $Y_n(z;\mu)$ as defined in \eqref{eq:def_Y}. Now, consider a new unknown,
   \begin{align*}
     Z_n(z;\mu) = Y_n(z;\mu) \begin{bmatrix} \prod_{j=1}^p (z - \mathtt c_j)^{-1} & 0 \\
       0 &  \prod_{j=1}^p (z - \mathtt c_j) \end{bmatrix}.
   \end{align*}
   This eliminates poles in the second column and adds them to the first.  The residue condition implies that near $c_j$
   \begin{align*}
     Y_n(z;\mu) = \begin{bmatrix} Y_{n}(\mathtt c_j;\mu)_{11} + \OO(z-\mathtt c_j) & \frac{w_j}{2 \pi \I} \frac{Y_{n}(\mathtt c_j;\mu)_{11}}{z - \mathtt c_j} + \OO(1) \\
       Y_{n}(\mathtt c_j;\mu)_{21} + \OO(z-\mathtt c_j) & \frac{w_j}{2 \pi \I} \frac{Y_{n}(\mathtt c_j;n)_{21}}{z - \mathtt c_j} + \OO(1)
       \end{bmatrix}
   \end{align*}
   Then for $Z_n$ we have
   \begin{align*}
     Z_n(z;\mu) = \begin{bmatrix} \frac{Y_{n}(\mathtt c_j; \mu)_{11}}{z - \mathtt c_j} \prod_{k \neq j} (\mathtt c_j - \mathtt c_k)^{-1} + \OO(1) & \frac{w_j}{2 \pi \I} Y_{n}(\mathtt c_j;\mu)_{11} \prod_{k \neq j} (\mathtt c_j - \mathtt c_k) + \OO(z - \mathtt c_j) \\
       \frac{Y_{n}(\mathtt c_j;\mu)_{21}}{z - \mathtt c_j} \prod_{k \neq j} (\mathtt c_j - \mathtt c_k)^{-1} + \OO(1) & \frac{w_j}{2 \pi \I} Y_{21}(\mathtt c_j;n) \prod_{k \neq j} (\mathtt c_j - \mathtt c_k) + \OO(z - \mathtt c_j)
     \end{bmatrix}
   \end{align*}
   From this it follows that
   \begin{align*}
     \mathrm{Res}_{z = \mathtt c_j} Z_n(z;\mu) = \begin{bmatrix}  Y_{n}(\mathtt c_j;\mu)_{11}\prod_{k \neq j} (\mathtt c_j - \mathtt c_k)^{-1} & 0 \\
       Y_{n}(\mathtt c_j;\mu)_{21} \prod_{k \neq j} (\mathtt c_j - \mathtt c_k)^{-1} & 0
     \end{bmatrix} = \lim_{z \to \mathtt c_j} Z_n(z;\mu) \begin{bmatrix} 0 & 0 \\
       \frac{2 \pi \I}{w_j} \prod_{k \neq j} (\mathtt c_j - \mathtt c_k)^{-2} & 0
     \end{bmatrix}.
   \end{align*}
   The other important properties of $Z_n(z;\mu)$ are given by  
   \begin{align*}
    \lim_{\epsilon \to 0^+} Z_n(z + \I \epsilon;\mu) &=  \lim_{\epsilon \to 0^+} Z_n(z - \I \epsilon;\mu) \begin{bmatrix} 1 & \rho(z) \prod_{j=1}^p (z - \mathtt c_j)^2\\ 0 & 1 \end{bmatrix}, \quad z \in (-1,1),\\                                                                                                                                   Z_n(z;\mu) \begin{bmatrix} z^{-(n-p)} & 0 \\ 0 & z^{n-p} \end{bmatrix} &= I + \OO(1/z), \quad z \to \infty.
  \end{align*}
  
  Now, let $\Sigma_j$ be a small circle centered at $\mathtt c_j$ with radius sufficiently small so that it does not intersect any other $\Sigma_k$ for $k\neq j$ and so that it does not intersect any $\Sigma_j$ for all $j$.  Denote by $\mathring \Sigma_j$ the region enclosed by $\Sigma_j$. Define
  \begin{align*}
    \check Z_n(z;\mu) = \begin{cases} Z_n(z;\mu) & z \in \mathbb C \setminus \left(\bigcup_{j=1}^{g+1}[\mathtt a_j,\mathtt b_j] \cup \bigcup_{j=1}^p (\Sigma_j \cup \mathring \Sigma_j) \right),\\
      Z(z;n) \begin{bmatrix} 1 & 0 \\
        -\frac{\tilde w_j}{z - \mathtt c_j} & 1\end{bmatrix} & z \in \mathring \Sigma_j \setminus \{\mathtt c_j\},
    \end{cases}                                                
  \end{align*}
  where $\tilde w_j $ is defined as 
\begin{equation}\label{eq_tildewjdefn}
\tilde w_j :=\frac{2 \pi \I}{w_j} \prod_{k \neq j} (\mathtt c_j - \mathtt c_k)^{-2}.
\end{equation}  
  
Then it follows that $\check Z_n(z;\mu)$ has a removable singularity at $z = \mathtt c_j$ for each $j$.  We give $\Sigma_j$ counter-clockwise orientation and denote by $\check Z_n^\pm$ the limit to $\Sigma_j$ from the interior $(+)$ or exterior $(-)$. We have
  \begin{align*}
    \check Z_n^+(z;\mu) = \check Z_n^-(z;\mu) \begin{bmatrix} 1 & 0 \\
        \frac{\tilde w_j}{z - \mathtt c_j} & 1\end{bmatrix}, \quad z \in \Sigma_j.
  \end{align*}

\subsubsection{Step 2: Determine the correct differential}\label{subsubsec_differential}

  Our next task is to remove the growth/decay at infinity.  %A simple construction like $\varphi(z)$ is no longer useful.  We need to abstract away the properties that made it work:
  %\begin{enumerate}
  %\item $\varphi(z) = c z + O(1/z)$ as $z \to \infty$
  %\item $\varphi_+(z) \varphi_-(z) = 1$ on $[-1,1]$.  
  %\end{enumerate}
  %Motivated by this, we seek a function $\psi(z)$ that satisfies
  %\begin{enumerate}
  %  \item $\psi(z) = c z + O(1/z)$ as $z \to \infty$
  %\item $|\psi_+(z) \psi_-(z)| = 1$ on $[a_j,b_j]$.
  %\end{enumerate}
  We look for a function $\mathfrak g$ that satisfies:
 % \begin{itemize}
 % \item $\mathfrak g(z) = \log \mathfrak c z + O(1/z)$ as $z \to \infty$.
 % \item $\mathfrak g_+(z) + \mathfrak g_-(z) \in \I \mathbb R$ on $[a_j,b_j]$.
 % \end{itemize}
 % This implies that:
 % \begin{itemize}
 % \item $\mathfrak g'(z) = 1/z + O(1/z^2)$ as $z \to \infty$.
 % \item $\mathfrak g_+'(z) + \mathfrak g_-'(z) \in \I \mathbb R$ on $[a_j,b_j]$.
 % \item $\int_{b_j}^{a_{j+1}} \mathfrak g'(z) \sd z = 0$, $j = 1,2,\ldots,g$
 % \end{itemize}
 % It is sufficent to require:
\begin{enumerate}[(a)]
  \item $\mathfrak g'(z) = 1/z + \OO(1/z^2)$ as $z \to \infty$.
  \item $\mathfrak g_+'(z), \mathfrak g_-'(z) \in \I \mathbb R$ on $[\mathtt a_j,\mathtt b_j]$.
  \item $\int_{\mathtt b_j}^{\mathtt a_{j+1}} \mathfrak g'(z) \sd z = 0$, $j = 1,2,\ldots,g$
  \end{enumerate}

  Based on this, define
  \begin{align}\label{eq_gprimedefinition}
    \mathfrak g'(z) = \frac{Q_{g}(z)}{R(z)}, \quad \text{where} \quad R(z)^2 = \prod_{j=1}^{g+1} (z-\mathtt a_j)(z-\mathtt b_j),
  \end{align}
  where $Q_{g}$ is a monic polynomial of degree $g$, providing $g$ degrees of freedom to satisfy the requisite conditions.  We then see that $R_+(z)$ is purely imaginary in each interval $(\mathtt a_j,\mathtt b_j)$ and real-valued on $(\mathtt b_{j}, \mathtt a_{j+1})$.  The linear system that defines $Q_{g}(z) = \sum_k h_k z^k$ is given by:
  \begin{align*}
    \int_{\mathtt b_j}^{\mathtt a_{j+1}} \sum_{k=0}^{g-1} h_k \frac{z^k}{R(z)} dz = - \int_{\mathtt b_j}^{\mathtt a_{j+1}} \frac{z^{g}}{R(z)} dz, \quad j = 1,2,\ldots,g.
  \end{align*}
  Therefore $h_k$ are real-valued coefficients.  This implies (b).  The unique solvability of this system for these coefficients follows from the fact that $\frac{z^k}{R(z)} dz$, $k = 0,1,2,\ldots,g-1$ forms a basis for holomorphic differentials on hyperelliptic Riemann surface defined by $w^2 = R(z)^2$.  Then because $R(z)$ is sign definite in each gap $(\mathtt b_j,\mathtt a_{j+1})$, for (c) to hold, $\mathfrak g'(z)$ must vanish in this interval.  This implies that $Q_g(z)$ has one root $\mathtt d_j$ in each gap $(\mathtt b_j,\mathtt a_{j+1})$ and this accounts for all the roots of $Q_g(z)$.  This implies that $\mathfrak g'(z) < 0$ for $z < \mathtt a_1$ and $\mathfrak g'(z) > 0$ for $z > \mathtt b_{g+1}$.  With the notation $\mathtt b_0 = -\infty$ and $\mathtt a_{2g+1} = + \infty$, it follows that
  \begin{align*}
    R(z) R(z') < 0, \quad z \in (\mathtt b_j,\mathtt a_{j+1}), \quad z' \in (\mathtt b_{j+1}, \mathtt a_{j+2}),
  \end{align*}
  for $j = 0,1,2,\ldots,g-1$.  Since  $\mathfrak g'(z) < 0$ for $z < \mathtt a_1$, we see that $\mathfrak g'(z) > 0$ for $z \in (\mathtt b_1,\mathtt d_1)$ and $g'(z) < 0$ for $z \in (\mathtt d_1,\mathtt a_2)$. This is true, in general, with $\mathfrak g'(z)$ being positive on $(\mathtt b_j,\mathtt d_j)$ and negative on $(\mathtt d_j,\mathtt a_{j+1})$.

  Then $\mathfrak g(z)$ is defined by integration of $\mathfrak g'(z)$ from  $\mathtt a_1$ to $z$ by a straight line.  We can compute
  \begin{align*}
    \mathfrak g^+(z) + \mathfrak g^-(z) = 0, \quad z \in (\mathtt a_j,\mathtt b_j),
  \end{align*}
  where we use the fact that $R^+(z) = - R^-(z)$ for $z \in (\mathtt a_j,\mathtt b_j)$ along with $\int_{\mathtt b_j}^{\mathtt a_{j+1}} \mathfrak g'(z) dz = 0$ for each $j$.  And for $z \in (\mathtt b_{j},\mathtt a_{j+1})$ we find
  \begin{align}\label{eq:G-G}
    \mathfrak g^+(z) - \mathfrak g^-(z) = 2 \sum_{k=1}^j \int_{\mathtt a_k}^{\mathtt b_k} (\mathfrak g')^+(z) dz =: \Delta_j
  \end{align}
  So this is constant in each gap $(\mathtt b_j,\mathtt a_{j+1})$ and is purely imaginary.  Define the vector
\begin{equation}\label{eq_defndelta}  
   \bm{\Delta} = (\Delta_j)_{j=1}^g.  
 \end{equation} 
 All of this then implies that the real part of $\mathfrak g(z)$ is strictly positive on any closed subset of $\mathbb R \setminus \cup_j [\mathtt a_j,\mathtt b_j]$ and by the maximum modulus principle applied to $\e^{-\mathfrak g(z)}$ this statement extends to $\mathbb C \setminus \cup_j [\mathtt a_j,\mathtt b_j]$.  

Define
  \begin{align}\label{eq_capcity}
    \mathfrak c = \lim_{z\to \infty} \frac{\e^{\mathfrak g(z)}}{z}.
  \end{align}
  We remark that $| \mathfrak c|$ is classically known as the capacity of $\cup_j[\mathtt a_j,\mathtt b_j]$ \cite{Peherstorfer}.

  \subsubsection{Step 3: Lens the problem}\label{subsubsec_lenproblem}

  Define $\check \rho_j$ to be the analytic continuation of $\rho(z) \prod_{j=1}^p (z - \mathtt c_j)^2$ off $[\mathtt a_j,\mathtt b_j]$ to $\Omega_j$.  Then let $C_j$ be a curve the encircles $[\mathtt a_j,\mathtt b_j]$ lying in $\Omega_j$.  Denote the interior of this curve by $D_j$.  Then define
\begin{align*}
    S_n(z;\mu) = \begin{cases} \check Z_n(z;\mu) \begin{bmatrix} 1 & 0 \\ -1/\check \rho_j(z) & 1 \end{bmatrix} & z \in D_j \cap \mathbb C^+,\\ \\
      \check Z_n(z;\mu) \begin{bmatrix} 1 & 0 \\ 1/\check \rho_j(z) & 1 \end{bmatrix} & z \in D_j \cap \mathbb C^-, \\ \\
      \check Z_n(z;\mu) & \text{otherwise}.
      \end{cases}
  \end{align*}
  We find
   \begin{align*}
    S^+_n(z;\mu) = \begin{cases}
      S^-_n(z;\mu) \begin{bmatrix} 1 & 0 \\ 1/\check \rho_j(z) & 1 \end{bmatrix} & z \in C_j \setminus \mathbb R, \\ \\
      S^-_n(z;\mu) \begin{bmatrix} 0 & \check \rho_j(z) \\ -1/\check \rho_j(z) & 0 \end{bmatrix} & z \in (\mathtt a_j,\mathtt b_j),\\ \\
      S^-_n(z;\mu) \begin{bmatrix} 1 & 0 \\
        \frac{\tilde w_j}{z - \mathtt c_j} & 1\end{bmatrix} & z \in \Sigma_j
    \end{cases}
   \end{align*}
   Note that $S_n$ still has the same normalization at infinity as $\check Z_n$. And recalling that $\check Z_n(z;\mu)$ is bounded on $D_j$, we see that we have now introduced unbounded behavior in $S_n$, in an entrywise sense,
   \begin{align*}
       S_n(z;\mu) = \begin{bmatrix} O(|z-\mathtt a_j|^{-1/2}) & O(1)\\
       O(|z-\mathtt a_j|^{-1/2}) & O(1) \end{bmatrix}, \quad S_n(z;\mu) = \begin{bmatrix} O(|z-\mathtt b_j|^{-1/2}) & O(1)\\
       O(|z-\mathtt b_j|^{-1/2}) & O(1) \end{bmatrix},
   \end{align*}
   as $z\to \mathtt a_j, \mathtt b_j$, respectively.

   \subsubsection{Step 4: Normalize at infinity}

   Define
   \begin{align*}
    \check S_n(z;\mu) = \mathfrak c^{(n-p) \sigma_3} S_n(z;\mu) \e^{-(n-p)\mathfrak g(z)\sigma_3}
   \end{align*}
   Then it follows that $\check S_n(z;\mu) = I + \OO(z^{-1})$ as $z \to \infty$ and it satisfies the jumps
  \begin{align*}
    \check S^+_n(z;\mu) = \begin{cases}
     S^-_n(z;\mu) \begin{bmatrix} 1 & 0 \\ \e^{-2(n-p)\mathfrak g(z)}/\check \rho_j(z) & 1 \end{bmatrix} & z \in C_j \setminus \mathbb R, \\ \\
      S^-_n(z;\mu) \begin{bmatrix} 0 & \check \rho_j(z) \\ -1/\check \rho_j(z) & 0 \end{bmatrix} & z \in (\mathtt a_j,\mathtt b_j),  \\ \\
      S^-_n(z;\mu) \begin{bmatrix} \e^{-(n-p)\Delta_j} & 0 \\ 0 & \e^{(n-p)\Delta_j} \end{bmatrix} & z \in (\mathtt b_{j},\mathtt a_{j+1}),\\ \\
      S^-_n(z;\mu)   \begin{bmatrix} 1 & 0 \\
        \e^{-2(n-p)\mathfrak g(z)}\frac{\tilde w_j}{z - \mathtt c_j} & 1\end{bmatrix} & z \in \Sigma_j.
    \end{cases}
  \end{align*}

  %\subsubsection{Step 4: Remove diagonal jumps}

  %Define
  %\begin{align*}
  %  \Delta(z) = \frac{1}{2 \pi \I} \sum_{j=1}^{g-1} \int_{b_j}^{a_{j+1}} \frac{ \log \Delta_j}{z' - z} \sd z'.
  %\end{align*}
  %and consider
  % \begin{align*}
  %  T(z;n) = S(z;n) \e^{\sigma_3\Delta(z)}.
  % \end{align*}
  % We then have
  % \begin{align*}
  %  T_+(z;n) = \begin{cases}
  %   T_-(z;n) \begin{bmatrix} 1 & 0 \\ \e^{-2n(\mathfrak g(z)-\Delta(z))\sigma_3}/\check \rho_j(z) & 1 \end{bmatrix} & z \in C_j \setminus \mathbb R, \\ \\
  %    T(z;n) \begin{bmatrix} 0 &\e^{-2n \Delta(z)} \check \rho_j(z) \\ -\e^{2n \Delta(z)}/\check \rho_j(z) & 0 \end{bmatrix} & z \in (a_j,b_j).
 %   \end{cases}
 %  \end{align*}
 %  At this point the real part of $\mathfrak g(z) - \Delta(z)$ is incredibly important and needs to be investigated...

   \subsubsection{Step 5: Determine the Szeg\H{o} function}

   The point of the Szeg\H{o} function is to replace the jumps on $(\mathtt a_j,\mathtt b_j)$ with something simpler at the cost of adding to the jumps on $(\mathtt b_j,\mathtt a_{j+1})$.  Define
   \begin{align}\label{eq_definG}
     \sG(z) = -\frac{R(z)}{2 \pi \I} \left[\sum_{j=1}^{g+1} \int_{\mathtt a_j}^{\mathtt b_j} \frac{\log \check \rho_j(\lambda) }{\lambda -z} \frac{\sd \lambda}{R_+(\lambda)} +  \sum_{j=1}^{g} \int_{\mathtt b_j}^{\mathtt a_{j-1}} \frac{{\zeta}_j}{\lambda -z} \frac{\sd \lambda}{R(\lambda)}\right],
   \end{align}
   where the constants $\zeta_j$ are yet to be determined.

   Before we determine these constants, note that
   \begin{align*}
     \sG^+(z) + \sG^-(z) &= -\log \check \rho_j(z), \quad z \in (\mathtt a_j,\mathtt b_j),\\
     \sG^+(z) - \sG^-(z) &= -\zeta_j, \quad z \in (\mathtt b_j,\mathtt a_{j+1}).
   \end{align*}
   Since $R(z) = \OO(z^g)$, we see that $\sG(z) = \OO(z^{g-1})$.  To avoid unbounded behavior of $\sG$ at infinity, we choose $\bm{\zeta} = (\zeta_j)_{j=1}^{g}$ so that as $z \to \infty$
   \begin{equation}\label{eq_G(z)}   
   \sG(z) = \OO(1). 
   \end{equation}
    Indeed, we find a linear system of equations
   \begin{align}\label{eq_zetaequation}
     m_\ell = -\sum_{j=1}^g \int_{\mathtt a_j}^{\mathtt b_j} & \log \check \rho_j(\lambda) \lambda^{\ell-1} \frac{\sd \lambda}{R_+(\lambda)} \\
   &   -  \sum_{j=1}^{g-1} \int_{b_j}^{a_{j-1}} \zeta_j \lambda^{\ell-1} \frac{\sd \lambda}{R_+(\lambda)} = 0, \quad \ell = 1,2,\ldots,g-1. \notag
   \end{align}

   We pause briefly to discuss the singularity behavior of $\sG$ and note that we have to take some care because in Assumption~\ref{assum_measure} we allow $\mu$ to depend on $N$.
   \begin{lemma}\label{l:Gsing}
     Given Assumption~\ref{assum_measure}, for some $\epsilon > 0$, and for every $j =1,2,\ldots,g+1$ we have
     \begin{align*}
       \sG(z) = -\frac 1 4 \log[(z - \mathtt b_j)(\mathtt a_j - z)] + \mathtt R_{j}(z), \quad \mathrm{dist}(z,[\mathtt a_j,\mathtt b_j]) \leq \epsilon,
     \end{align*}
     where $\mathtt R_j(z)$ is a uniformly bounded function for $\mathrm{dist}(z,[\mathtt a_j,\mathtt b_j]) \leq \epsilon$.
   \end{lemma}
      \begin{proof}
     We first observe that if $h$ is a uniformly bounded analytic function in the $O_\epsilon = \{ z: \mathrm{dist}(z,[\mathtt a_j,\mathtt b_j]) \leq \epsilon\}$ then
     \begin{align}\label{eq:E}
       E(\lambda) = \frac{R(z)}{2 \pi \I}  \int_{\mathtt a_j}^{\mathtt b_j} \frac{h(\lambda)}{\lambda -z} \frac{\sd \lambda}{R_+(\lambda)}
     \end{align}
     is bounded for $z$ in any fixed bounded set.  Indeed, for $z \in O_{\epsilon/2}$
     \begin{align*}
       \frac{R(z)}{2 \pi \I}  \int_{\mathtt a_j}^{\mathtt b_j} \frac{h(\lambda)}{\lambda -z} \frac{\sd \lambda}{R_+(\lambda)} =  \frac{h(z)}{2} - \frac{R(z)}{4 \pi \I}  \int_{\Sigma} \frac{h(z')}{z' -z} \frac{\sd z'}{R(z')},
     \end{align*}
     where $\Sigma = \partial O_{2\epsilon/3}$.  This is uniformly bounded. This function is then evidently bounded uniformly on $\{|z| \leq R \} \setminus O_{\epsilon/2}$ for any $R> 0$. So, consider
     \begin{align*}
       H(z) %&= \frac{R(z)}{2 \pi \I}  \int_{\mathtt a_j}^{\mathtt b_j} \frac{\frac 1 2 \log (\lambda-\mathtt a_j)_+}{\lambda -z} \frac{\sd \lambda}{R_+(\lambda)}\\
       %& = \frac{R(z)}{4 \pi \I}  \int_{\mathtt a_j}^{\mathtt b_j} \frac{\frac 1 2 \log (\lambda-\mathtt a_j)_+}{\lambda -z} \frac{\sd \lambda}{R_+(\lambda)} + \frac{R(z)}{4 \pi \I}  \int_{\mathtt b_j}^{\mathtt a_j} \frac{\frac 1 2 \log (\lambda-\mathtt a_j)_+}{\lambda -z} \frac{\sd \lambda}{R_-(\lambda)}\\
        = \frac{R(z)}{2 \pi \I}  \int_{\mathtt a_j}^{\mathtt b_j} \frac{\frac 1 2 \log (\lambda-\mathtt a_j)_+}{\lambda -z} \frac{\sd \lambda}{R_+(\lambda)},
     \end{align*}
     in a neighborhood of $\mathtt a_j$ where the branch cut of $\log z$ here is chosen to be $[0,\infty)$. By choosing $\epsilon$ sufficiently small we find
     \begin{align*}
       H(z) = \frac 1 4 \log(z-\mathtt a_j) - \frac{R(z)}{4 \pi \I}  \int_{\partial \tilde O_{2 \epsilon}} \frac{\frac 1 2 \log(z'-\mathtt a_j)}{z' -z} \frac{\sd z'}{R(z)} +E(z),
     \end{align*}
     where $\tilde O_{2 \epsilon} = O_{2 \epsilon} \cap \{ \Re z \leq \mathtt b_j\}$ and $h = - \pi \I$ in \eqref{eq:E} is a constant.  This holds for $z \in O_{\epsilon} \cap \{ \Re z \leq \mathtt b_j - \epsilon\}$.  The second two terms are uniformly bounded for these values of $z$. We exchange $\log z$ for the principal branch in the initial integral for $H(z)$ and find
     \begin{align*}
       H(z) = \frac 1 4 \log(z-\mathtt a_j) - \frac{R(z)}{4 \pi \I}  \int_{\partial \check O_{2 \epsilon}} \frac{\frac 1 2 \log(z'-\mathtt a_j)}{z' -z} \frac{\sd z'}{R(z)} - E(z),
     \end{align*}
       where $\check O_{2 \epsilon} = O_{2 \epsilon} \cap \{ \Re z \geq \mathtt a_j\}$.  The second two terms are uniformly bounded for $z \in O_{\epsilon} \cap \{ \Re z \geq \mathtt a_j + \epsilon\}$.  Similar arguments hold after exchanging $\log (\lambda- \mathtt a_j)$ for $\log(\lambda - \mathtt b_j)$ and the lemma follows.
   \end{proof}

   This system of equations is uniquely solvable for $\bm{\zeta}$ using the fact that the normalized differentials exist, and involves the same coefficient matrix that is used to determine the polynomials $Q_{g}$ above.

   Then consider
   \begin{align*}
     T_n(z;\mu) = \e^{\sigma_3 \sG(\infty)} \check S_n(z;\mu) \e^{-\sigma_3 \sG(z)}.
   \end{align*}
   We check the jumps of $T_n$:
\begin{align*}
    T^+_n(z;\mu) = \begin{cases}
     T^-_n(z;\mu) \begin{bmatrix} 1 & 0 \\ \e^{-2((n-p)\mathfrak g(z)-\sG(z))}/\check \rho_j(z) & 1 \end{bmatrix} & z \in C_j \setminus \mathbb R, \\ \\
      T^-_n(z;\mu) \begin{bmatrix} 0 & 1 \\ -1 & 0 \end{bmatrix} & z \in (\mathtt a_j,\mathtt b_j),  \\ \\
      T^-_n(z;\mu)) \begin{bmatrix} \e^{-(n-p)\Delta_j-\zeta_j} & 0 \\ 0 & \e^{(n-p)\Delta_j+\zeta_j} \end{bmatrix} & z \in (\mathtt b_{j},\mathtt a_{j+1}),\\ \\
      T^-_n(z;\mu) \begin{bmatrix} 1 & 0 \\
        \e^{-2((n-p)\mathfrak g(z)-\sG(z))}\frac{\tilde w_j}{z - \mathtt c_j} & 1\end{bmatrix} & z \in \Sigma_j.
    \end{cases}
\end{align*}
Since the first and last jumps tend to the identity matrix exponentially fast, uniformly at a rate $\OO(\e^{-cn})$ for some $c > 0$ and we look to solve the model problem
  \begin{align*}
    \check T^+_n(z;\mu) = \begin{cases}
       \check T^-_n(z;\mu) \begin{bmatrix} 0 & 1 \\ -1 & 0 \end{bmatrix} & z \in (\mathtt a_j,\mathtt b_j),  \\ \\
      \check T^-_n(z;\mu) \begin{bmatrix} \e^{-(n-p)\Delta_j-\zeta_j} & 0 \\ 0 & \e^{(n-p)\Delta_j+\zeta_j} \end{bmatrix} & z \in (\mathtt b_{j},\mathtt a_{j+1}),\\  \\
    \end{cases}
  \end{align*}
  with the condition that $\check T_n(\infty;\mu) = I$.

  \subsubsection{Step 6: Solution of the model problem}

  From \eqref{eq:L} we find that $\check T_n(z;\mu) = L_n(\infty)^{-1} L_n(z),$
  with $v_j = (n-p)\Delta_j + \zeta_j$, $j =1, 2,\ldots, g$, i.e., $v = (n-p)\bm{\Delta} + \bm{\zeta}$. It then follows that
  \begin{align*}
    R_n(z;\mu) := T_n(z;\mu) \check T_n(z;\mu)^{-1},
  \end{align*}
  using the fact that $\check T_n(z;\mu)$ and its inverse are uniformly bounded (see \cite{MR1702716}, for example) on sets bounded away from the support of $\mu$, it follows that
  \begin{align*}
    R_n(z;\mu)  = I + \OO\left( \frac{\e^{- c n}}{1 + |z|} \right).
  \end{align*}

  \section{Algorithmic asymptotic expansions: Proof of Theorem \ref{thm_main_asympt} and its corollaries}\label{a:forms}
  
\subsection{Detailed expressions of (\ref{eq:Yn})}\label{sec_detailedexpression} We provide some explicit entry-wise formulae for (\ref{eq:Yn}). Denote  
\begin{equation}\label{eq_defn}
  E(z;n) =  \left ( I + \OO \left(\frac{\e^{-cn}}{1 + |z|}\right) \right)L_n(\infty)^{-1} L_n(z) .
\end{equation} 
According to (\ref{eq:Yn}), we readily obtain that for $z$ outside any region of deformation
\begin{align}
  Y_n(z;\mu)_{11} &= \mathfrak c^{(p-n)} \e^{-\sG(\infty)} \e^{\sG(z)} \e^{(n-p) \mathfrak g(z)} \left[\prod_{j=1}^p(z-\mathtt c_j) \right] E_{11}(z;n), \label{eq_Yn11}\\
  Y_n(z;\mu)_{12} &=  \mathfrak c^{(p-n)} \e^{-\sG(\infty)} \e^{-\sG(z)} \e^{-(n-p) \mathfrak g(z)} \left[\prod_{j=1}^p(z-\mathtt c_j)^{-1} \right] E_{12}(z;n), \label{eq_Yn12}\\
  Y_n(z;\mu)_{21} &=  \mathfrak c^{-(p-n)} \e^{\sG(\infty)} \e^{\sG(z)} \e^{(n-p) \mathfrak g(z)} \left[\prod_{j=1}^p(z-\mathtt c_j) \right] E_{21}(z;n), \label{eq_Yn21}\\
  Y_n(z;\mu)_{22} & = \mathfrak c^{-(p-n)} \e^{\sG(\infty)} \e^{-\sG(z)} \e^{-(n-p) \mathfrak g(z)} \left[\prod_{j=1}^p(z-\mathtt c_j)^{-1} \right] E_{22}(z;n). \label{eq_Yn22}
  \end{align} 

Recall (\ref{eq_gammaz}). As $\gamma(z) \rightarrow 1$ when $z \rightarrow \infty,$ using the definition of $L_n(z)$ in (\ref{eq:L}), we see that 
  \begin{align}
  L_n(\infty)^{-1} L_n(z) &= \begin{bmatrix} \Theta_1(\infty;\vec d_2;(n-p)\bm{\Delta} + \bm{\zeta})^{-1} & 0 \\ 0 &  \Theta_2(\infty;\vec d_1,;(n-p)\bm{\Delta} + \bm{\zeta})^{-1} \end{bmatrix}\notag \\
  &\times \begin{bmatrix}  \left(\frac{\gamma(z) + \gamma(z)^{-1}}{2}\right) \Theta_1(z;\vec d_2;(n-p)\bm{\Delta} + \bm{\zeta}) & \left(\frac{\gamma(z) - \gamma(z)^{-1}}{2\I}\right) \Theta_2(z;\vec d_2;(n-p)\bm{\Delta} + \bm{\zeta}) \\ \\ 
  \left(\frac{\gamma(z)^{-1} - \gamma(z)}{2\I}\right) \Theta_1(z;\vec d_1;(n-p)\bm{\Delta} + \bm{\zeta}) & \left(\frac{\gamma(z) + \gamma(z)^{-1}}{2}\right) \Theta_2(z;\vec d_1;(n-p)\bm{\Delta} + \bm{\zeta}) \end{bmatrix}. \label{eq_LninLn}
\end{align}

Finally, we provide more explicit formulae for $E$ in (\ref{eq_defn}). Note that 
\begin{align}\label{eq_e11}
  E_{11}(0;n) = \frac{1}{2} \left( \prod_{j=1}^{g+1} \left( \frac{ \mathtt b_j}{\mathtt a_j} \right)^{1/4} + \prod_{j=1}^{g+1} \left( \frac{ \mathtt a_j}{\mathtt b_j} \right)^{1/4}\right) \frac{\Theta_1(0;\vec d_2;(n-p)\bm{\Delta} + \bm{\zeta})}{\Theta_1(\infty;\vec d_2;(n-p)\bm{\Delta} + \bm{\zeta})} + \OO(\e^{-cn}),
\end{align}
and similar expressions are easily derivable for the other entries of $E(0;n)$.

Define $\Theta_1^{(1)}$ by,
\begin{align}\label{eq_thetaone}
  \Theta_1(z;\vec d_2; (n-p)\bm{\Delta} + \bm{\zeta}) = \Theta_1(\infty;\vec d_2; (n-p)\bm{\Delta} + \bm{\zeta}) + \frac{1}{z} \Theta_1^{(1)}(\vec d_2; (n-p)\bm{\Delta} + \bm{\zeta}) + \OO(z^{-2}),
\end{align}  
so that $\Theta_1^{(1)}$ denotes the residue of $\Theta_1$ at infinity.
%{\color{blue}[What is the formal mathematical definition of $\Theta_1^{(1)}$?]}
Together with (\ref{eq_LninLn}), as $z \rightarrow \infty,$ it leads to
\begin{align}\label{eq_e11expansion}
  E_{11}(z;n) = 1 + \frac{1}{z} \left[\frac{\Theta_1^{(1)}(\vec d_2; (n-p)\bm{\Delta} + \bm{\zeta})}{ \Theta_1(\infty;\vec d_2;(n-p)\bm{\Delta} + \bm{\zeta})}\right] + \OO(\e^{-cn}+z^{-2}) .  
\end{align}
Moreover, using (\ref{eq_zetaequation}), we see that as $z \rightarrow \infty$
\begin{align*}
  \e^{-\sG(\infty)} \e^{\sG(z)} = 1  + \frac{1}{z} \left[ \frac{m_{g+1}}{2 \pi \I} -  \frac{m_g}{2 \pi \I} \sum_{j=1}^{g+1} (\mathtt a_j + \mathtt b_j) \right] + \OO(z^{-2}),\\
  \mathfrak c^{(p-n)} \frac{\e^{(n-p) \mathfrak g(z)}}{z^{n}}\left[\prod_{j=1}^p(z-\mathtt c_j) \right] = 1  +  \frac{1}{z} \left[ \mathfrak g_1 - \sum_{j=1}^p \mathtt c_j \right]+\OO(z^{-2}), 
\end{align*}
where $\mathfrak{g}_1$ is defined so that  $\mathfrak g(z) = \log z + \log \mathfrak c + \mathfrak g_1/z + \OO(z^2)$ as $z \to \infty$.  Combining (\ref{eq_anexpansion}),(\ref{eq_thetaone}), (\ref{eq_Yn11}), (\ref{eq_LninLn}), (\ref{eq_capcity}) and
 \begin{align*}
  z \left[ z^{-n}Y_n(z;\mu)_{11} - 1 \right] & = z \left[ \mathfrak c^{(p-n)} \e^{-\sG(\infty)} \e^{\sG(z)} \frac{\e^{(n-p) \mathfrak g(z)}}{z^{n}} \left[\prod_{j=1}^p(z-\texttt c_j) \right] E_{11}(z;n) - 1 \right],
 \end{align*}
 one finds
 \begin{align*}
   \lim_{z \to \infty} z \left(z^{-n} Y_n(z;\mu)_{11} - 1 \right) &= \frac{m_{g+1}}{2 \pi \I} -  \frac{m_g}{2 \pi \I} \sum_{j=1}^{g+1} (\mathtt a_j + \mathtt b_j)  +  (n-p)\mathfrak g_1 - \sum_{j=1}^p \mathtt c_j \\
                                                                  & + \frac{\Theta_1^{(1)}(\vec d_2; (n-p)\bm{\Delta} + \bm{\zeta})}{ \Theta_1(\infty;\vec d_2;(n-p)\bm{\Delta} + \bm{\zeta})} + \OO(\e^{-cn}).
 \end{align*}

 Also, according to (\ref{eq_defn}) and (\ref{eq_LninLn}), we see that 
\begin{equation}\label{eq_largezlimitexample}
\lim_{z \rightarrow \infty}zE_{12}(z;n)=\frac{\I}{4} \sum_{j=1}^{g+1} (\texttt b_j - \texttt a_j) \frac{\Theta_2(\infty;\vec d_2;(n-p)\bm{\Delta} + \bm{\zeta})}{\Theta_1(\infty;\vec d_2;(n-p)\bm{\Delta} + \bm{\zeta})} + \OO(\e^{-cn}) ,
\end{equation}
where we used the definition (\ref{eq_gammaz}). Consequently, using (\ref{eq_Yn12}), we readily obtain that 
\begin{align*}
- 2 \pi \I \lim_{z \rightarrow \infty} z^{n+1} Y_n(z;\mu)_{12}& =\left[\frac{\pi}{2} \sum_{j=1}^{g+1}(\texttt{b}_j-\texttt{a}_j) \frac{\Theta_2(\infty;\vec d_2;(n-p)\bm{\Delta} + \bm{\zeta})}{\Theta_1(\infty;\vec d_2;(n-p)\bm{\Delta} + \bm{\zeta})}+\OO(\e^{-cn}) \right] \\
& \times \lim_{z \rightarrow \infty} e^{-\sG(\infty)-\sG(z)} \mathfrak{c}^{p-n} \frac{z^{n-p}}{e^{(n-p)\mathfrak{g}(z)}} \frac{z^p}{\prod_{j=1}^p (z-\texttt{c}_j)^{-1}} \\
& =e^{-2\sG(\infty)} \mathfrak{c}^{2(p-n)} \frac{\pi}{2} \sum_{j=1}^{g+1}(\texttt{b}_j-\texttt{a}_j) \frac{\Theta_2(\infty;\vec d_2;(n-p)\bm{\Delta} + \bm{\zeta})}{\Theta_1(\infty;\vec d_2;(n-p)\bm{\Delta} + \bm{\zeta})}+\OO(\e^{-cn}), 
\end{align*}
where in the last step we used the definition (\ref{eq_capcity}).
                                                                   
\subsection{Asymptotic formulae of Section \ref{sec_asymptotics}}\label{sec_limitformula}

\begin{proof}[\bf Proof of Corollary \ref{cor_deterthreeterm}] 
  Recall (\ref{eq_ellndefn}).
  By equating coefficients in \eqref{eq:three-term} and the definition of $p_n(z;\mu),$ we find that
\begin{align} \label{eq:anbn}
    \ell_n(\mu) &= b_n(\mu) \ell_{n+1}(\mu),\\
    s_n(\mu) &= a_n(\mu) \ell_n(\mu) + b_n(\mu) s_{n+1}(\mu). \notag
\end{align}
There is then, of course, the relation $\gamma_n(\mu) = -2 \pi \I \ell_n^2(\mu)$.  A direct calculation, using orthogonality and the definitions (\ref{eq_gammangammmu}) and (\ref{eq_defncauchytransform}), leads to
\begin{align*}
  \lim_{z \to \infty} z^{n+1} c_n(z;\mu) & = - \frac{1}{2 \pi \I} \lim_{z \to \infty} z^n \int \frac{\pi_n(x;\mu)}{ 1 - (x/z)} \sd x = -\frac{1}{2 \pi \I} \int x^n \pi_n(x;\mu) \sd x  \\
&  = -\frac{1}{2 \pi \I} \ell_n^{-2}(\mu).
\end{align*}
This gives
\begin{align}
  b_n(\mu)^2 &= \frac{\gamma_n(\mu)}{\gamma_{n+1}(\mu)} = \frac{\lim_{z \to \infty} z^{n+2} c_{n+1}(z;\mu)}{\lim_{z \to \infty} z^{n+1} c_n(z;\mu)} = \lim_{z \to \infty} \frac{zY_{n+1}(z;\mu)_{12}}{Y_n(z;\mu)_{12}} \label{eq_bnmuexpansion}.
\end{align}
 The expression
\begin{align*}
  a_n(\mu) &= \lim_{z \to \infty}  z^{-n+1}(\pi_n(z;\mu) - z^n) - \lim_{z \to \infty} z^{-n}(\pi_{n+1}(z;\mu) - z^{n+1}),
\end{align*}
directly follows from \eqref{eq:anbn} and the definition of $s_n,\ell_n$.  From the definition (\ref{eq:Yn}) one has
\begin{align}
  a_n(\mu) &= - \lim_{z \to \infty} z^{-n} (Y_{n+1}(z;\mu)_{11} - z Y_n(z;\mu)_{11}). \label{eq_anexpansion}
\end{align}
Then the proof follows immediately from Theorem~\ref{thm_main_asympt}.

%
%Then it is easy to check that the proof follows from , . This completes our proof.  
\end{proof}

\begin{proof}[\bf Proof of Corollary \ref{cor_cgadeterasymp}]  By \cite[Proposition 4.1]{Paquette2020} and the definition (\ref{eq:Yn}), we have
\begin{align}\label{eq_enw}
 \|\vec e_n\|_W^2 =  \frac{c_n(0;\mu)}{\pi_n(0;\mu)} = \frac{Y_n(0;\mu)_{12}}{Y_n(0;\mu)_{11}}, 
 %= \e^{- 2 G(0)} \e^{-2 (n-p) \mathfrak g(0)} \left[\prod_{j=1}^p c_j^{-2}\right] \frac{E_{12}(0;n)}{E_{11}(0;n)},
\end{align}
and 
%\footnote{Here $\|\vec x\|_W^2 := \vec x^* W \vec x$.}
\begin{align}\label{eq_rn}
  \| \vec r_n \|_2^2 = \frac{\prod_{j=0}^{n-1} b_j(\mu)^2}{\pi_n(0;\mu)^2} &= \frac{ -2 \pi \I \lim_{z \to \infty} z^{n+1} Y_n(z;\mu)_{12}}{Y_n(0;\mu)_{11}^2},
  % \\
 % & = \frac{\displaystyle \frac{\pi}{2} \sum_{j=1}^{g+1} (\mathtt b_j - \mathtt a_j) \frac{\Theta_2(\infty;\vec d_2;(n-p)\bm{\Delta} + \zeta)}{\Theta_1(\infty;\vec d_2;(n-p)\bm{\Delta} + \zeta)} + O(\e^{-cn})}{\displaystyle  \e^{2 (n-p) \mathfrak g(0)} \left[\prod_{j=1}^pc_j^2\right] E_{11}(0;n)^2}
\end{align}
where we used (\ref{eq_bnmuexpansion}):
\begin{align*}
  \prod_{j=0}^{n-1} b_j(\mu)^2= \lim_{z \to \infty} \prod_{j=0}^{n-1} \frac{zY_{j+1}(z;\mu)_{12}}{Y_j(z;\mu)_{12}} = \lim_{z \to \infty} z^n \frac{Y_n(z;\mu)_{12}}{Y_0(z;\mu)_{12}},
\end{align*}
and that $Y_0(z;\mu)_{12} = - \frac{1}{2 \pi \I z} ( 1  + \OO(z^{-1}))$.

The proof of the first equation follows directly from the above formula and (\ref{eq_Yn11}) and (\ref{eq_Yn12}). For the second equation,  Combining with Theorem~\ref{thm_main_asympt}, we can complete the proof. 
\end{proof}

\begin{proof}[\bf Proof of Corollary \ref{cor_choleskeylimit}]
Using the facts
\begin{align*}
  \det \mathcal J_n(\mu) = \prod_{j=0}^{n-1} \alpha_j(\mu)^2, \quad \pi_n(z;\mu) = \det ( z I - \mathcal J_n(\mu)),
\end{align*}
we obtain that 
\begin{align*}
  (-1)^n \pi_n(0;\mu) = \prod_{j=0}^{n-1} \alpha_j(\mu)^2.
\end{align*}
Combining with $(\mathcal{J})_{j,j+1}=\alpha_j \beta_j,$ we immediately see that 
\begin{align*}
  \prod_{j=0}^{n-1} \frac{\beta_j(\mu)^2}{\alpha_j(\mu)^2} = \frac{\prod_{j=0}^{n-1} b_j(\mu)^2}{\pi_n(0;\mu)^2}.
\end{align*}
This gives the expressions
\begin{align}\label{eq_alphabetarepresentation}
  \alpha_n(\mu)^2 = -\frac{ \pi_{n+1}(0;\mu)}{ \pi_n(0;\mu)}, %\quad \frac{\beta_n^2(\mu)}{\alpha_n^2(\mu)}  = \frac{\prod_{j=0}^{n} b_j(\mu)^2}{\pi_{n+1}(0;\mu)^2}\frac{\pi_n(0;\mu)^2}{\prod_{j=0}^{n-1} b_j(\mu)^2},\\
  \quad \beta_n(\mu)^2 = -b_n^{2}(\mu) \frac{\pi_{n}(0;\mu)}{\pi_{n+1}(0;\mu)}.
\end{align}
The proof then follows from Theorem~\ref{thm_main_asympt} and Corollary~\ref{cor_deterthreeterm}. 
\end{proof}

 \section{CLT for spiked sample covariance matrix model: Proof of Theorem \ref{thm_mainclt}}\label{sec_cltproof}
In this section, we prove the CLT as in Section \ref{sec_generalclt}. Throughout this section, we will consistently use the notion of \emph{stochastic domination}, which provides a precise statement of the form ``$\xi_N$ is bounded by $\zeta_N$ up to a small power of $N$ with high probability".  
\begin{definition}
(i) Let
\[\xi=\left(\xi^{(N)}(u):N \in \mathbb{N}, u\in U^{(N)}\right),\hskip 10pt \zeta=\left(\zeta^{(N)}(u):N \in\mathbb{N}, u\in U^{(N)}\right)\]
be two families of nonnegative random variables, where $U^{(N)}$ is a possibly $n$-dependent parameter set. We say $\xi$ is stochastically dominated by $\zeta$, uniformly in $u$, if for any fixed (small) $\epsilon>0$ and (large) $D>0$, 
\[\sup_{u\in U^{(N)}}\mathbb{P}\left(\xi^{(N)}(u)>N^\epsilon\zeta^{(N)}(u)\right)\le N^{-D},\]
for large enough $N \ge N_0(\epsilon, D)$, and we shall use the notation $\xi\prec\zeta$. Throughout this paper, the stochastic domination will always be uniform in all parameters that are not explicitly fixed (such as matrix indices, and $z$ that takes values in some compact set). Note that $N_0(\epsilon, D)$ may depend on quantities that are explicitly constant, such as $\tau$ in Assumption \ref{assum_summary}. If for some complex family $\xi$ we have $|\xi|\prec\zeta$, then we will also write $\xi \prec \zeta$ or $\xi=\OO_\prec(\zeta)$.
%\item[(ii)] 

%(ii) We extend the definition of $\OO_\prec(\cdot)$ to matrices in the weak operator norm sense as follows. Let $A$ be a family of random matrices and $\zeta$ be a family of nonnegative random variables. Then $A=\OO_\prec(\zeta)$ means that $\left|\left\langle\mathbf v, A\mathbf w\right\rangle\right|\prec\zeta \| \mathbf v\|_2 \|\mathbf w\|_2 $ uniformly in any deterministic vectors $\mathbf v$ and $\mathbf w$. Here and throughout the following, whenever we say ``uniformly in any deterministic vectors", we mean that ``uniformly in any deterministic vectors belonging to a set of cardinality $n^{\OO(1)}$".
%\item[(iv)] 

(ii) We say an event $\Xi$ holds with high probability if for any constant $D>0$, $\mathbb P(\Xi)\ge 1- N^{-D}$ for large enough $N$.
\end{definition}

\subsection{Technical Tools} In this subsection, we collect some preliminary results which will be used in our proof. Recall (\ref{eq_definitioncovariance}) and (\ref{eq_spikedmodel}). Denote their resolvents as 
\begin{equation}
G_k=(\mathcal{Q}_k-z)^{-1}, \ \widetilde{G}_k=(\widetilde{\mathcal{Q}}_k-z)^{-1}, \ k=1,2. 
\end{equation}
We will use the following linearization. For simplicity, let $Y=\Sigma_0^{1/2} X.$ 
Denote the $(N+M) \times (N+M)$ linearized matrix $H$ by 
\begin{equation}\label{eq_defnh}
H = H(z,X):=\sqrt{z}
\begin{pmatrix}
0 &   Y\\
 Y^* & 0
\end{pmatrix}.
\end{equation}
Similarly, we can define $\widetilde{H}$ by replacing $\Sigma_0$ with $\Sigma.$ By Schur's complement, we have that 
\begin{equation}\label{eq_defnG}
G(z) = G(z, X):=(H-z)^{-1}= 
\begin{pmatrix}
G_1(z) & \frac{1}{\sqrt{z}} G_1(z) Y  \\
\frac{1}{\sqrt{z}}  Y^* G_1(z) & G_2(z)
\end{pmatrix}.
\end{equation}
The resolvents and related quantities are very convenient for us to analyze the VESD and ESD. Recall the notation in Section \ref{sec_subsectionvesd} and the ESD of $\mathcal{Q}_2$ is
\begin{equation*}
\zeta = \zeta_N:=\frac{1}{M} \sum_{i=1}^M \delta_{\lambda_i(\mathcal{Q}_2)}. 
\end{equation*}
%Let the eigenvectors of $Z$ be $\{\vec{u}_i\}.$ For any projection $\vec{b},$ we denote the eigenvector empirical spectral distribution (VESD) as \cite{BMP}
%\begin{equation*}
%\mu_{Z,\vec{b}}=\sum_{i=1}^M |\langle \vec{u}_i, \vec{b} \rangle|^2 \delta_{\lambda_i(Z)}.    
%\end{equation*}
Denote $m_N$ and $m_{N,\vec{b}}$ as the Stieltjes transforms of $\zeta$ and $\nu,$ respectively. Then we have that
\begin{equation}\label{eq_empiricalresolventform}
m_N=\frac{1}{M} \operatorname{Tr} G_2(z), \ m_{N, \vec{b}}=\vec{b}^* G_1(z) \vec{b}.     
\end{equation}

%Denote the deterministic matrix 
%\begin{equation}\label{eq_defnpi}
%\Pi(z)= 
%\begin{pmatrix}
%\Pi_1(z) & 0\\
%0 & \Pi_2(z) 
%\end{pmatrix}
%:=
%\begin{pmatrix}
%-\frac{1}{z} (1+m(z)\Sigma_0)^{-1} & 0 \\
%0 & m(z)
%\end{pmatrix}.
%\end{equation}

First, we state  the anisotropic laws in the following lemma. Recall (\ref{eq_defnpi1z}). Define the deterministic matrix 
\begin{equation}\label{eq_defnpi}
\Pi(z):= 
\begin{pmatrix}
\Pi_1(z) & 0\\
0 & \Pi_2(z) 
\end{pmatrix}
:=
\begin{pmatrix}
-\frac{1}{z} (1+m(z)\Sigma_0)^{-1} & 0 \\
0 & m(z)
\end{pmatrix}.
\end{equation} 
Fix some small constant $\tau>0,$ denote the set of spectral parameters as 
\begin{equation}\label{eq_setmathcald}
\mathcal{D} = \mathcal{D}(z,\tau)=\left\{z=E+\ri \eta: |z| \geq \tau, \ M^{-1+\tau} \leq \eta \leq \tau^{-1} \right\}.
\end{equation}
 Moreover, we denote a subset of $\mathcal{D}$ as 
\begin{equation}\label{eq_setmathcaldout}
\mathcal{D}_o =  \mathcal{D}_o(z,\tau)=\mathcal{D} \cap \left\{\operatorname{dist}(E, \operatorname{supp}(\varrho)) \geq M^{-2/3+\tau}\right\},
\end{equation}
and the control parameter as
\begin{equation*}
\Psi(z):= \sqrt{\frac{\Im m(z)}{M\eta}}+\mathbf{1}(z \in \mathcal{D} \backslash \mathcal{D}_o)\frac{1}{M \eta}.   
\end{equation*}
\begin{lemma}[Anisotropic local law]\label{lem_anisotropiclocallaw} 
For any deterministic unit vectors $\mathbf{u}, \mathbf{v} \in \mathbb{R}^{M+N},$ we have that for all $z \in \mathcal{D}(z,\tau)$
\begin{equation*}
\left| \mathbf{u}^* G(z)\mathbf{v} - \mathbf{u}^* \Pi(z) \mathbf{v} \right| \prec \Psi(z).
\end{equation*}
Moreover, we have for all $z \in \mathcal{D}(z,\tau)$
\begin{equation*}
|m_N(z)-m(z)|\prec \frac{1}{N \eta}. 
\end{equation*}
Furthermore, when $z \in \mathcal{D}_0(z,\tau),$ we have that
\begin{equation*}
|m_N(z)-m(z)|\prec \frac{1}{N (\kappa+\eta)}. 
\end{equation*}
\end{lemma}
\begin{proof}
See \cite{Knowles2017}. 
\end{proof}
We point out that $\Im m(z)$ can be controlled in the following way. Recall $\varrho$ is the measure associated with $m(z).$ We have that
\begin{equation*}
\Im m(z) \asymp 
\begin{cases}
\sqrt{\kappa+\eta}, & \text{if} \ E \in \operatorname{supp} \varrho \\
\frac{\eta}{\sqrt{\kappa+\eta}}, & \text{Otherwise}
\end{cases},
\end{equation*}
where $\kappa:=\operatorname{dist}(E, \operatorname{supp} \varrho).$ Moreover, according to \cite[(4.15) and (4.16)]{YF}, we have that for $z \in \mathcal{D}$
\begin{equation}\label{eq_stjbound}
|m(z)|=\OO(1), \  |m'(z)|=\OO\left( \frac{1}{\sqrt{\kappa+\eta}}\right).
\end{equation}

 Throughout this section, for simplicity of notation, we define the index sets $\mathcal{I}_1:=\{1,2,\cdots, N\}, \ \mathcal{I}_2:=\{N+1, \cdots, N+M\},$$ \ \mathcal{I}:=\mathcal{I}_1 \cup \mathcal{I}_2. $ 
We shall consistently use the Latin letters $i,j \in \mathcal{I}_1,$ Greek letters $\mu, \nu \in \mathcal{I}_2,$ and $\mathsf{a}, \mathsf{b} \in \mathcal{I}.$ Then we can label the indices of $X$ as $X=(X_{i \mu}: i \in \mathcal{I}_1, \mu \in \mathcal{I}_2 ).$ For simplicity, given a vector $\vb \in \mathcal{C}^{\mathcal{I}_{1,2}},$ we always identity it with its natural embedding in $\mathbb{C}^{\mathcal{I}}.$ For example, we shall identify $\xb \in \mathbb{C}^{\mathcal{I}_1}$ with $\begin{pmatrix}
 \xb \\ \mathbf{0}
\end{pmatrix},$ and $\yb \in \mathbb{C}^{\mathcal{I}_2}$ with $\begin{pmatrix}
  \mathbf{0} \\ \yb
\end{pmatrix}.$ We will also consistently use the notation $G_{\xb \yb}(z)=\xb^*G(z) \yb.$ Second, we will frequently use the following identities.
\begin{lemma}[Ward's identity]\label{lem_wardidentity}
Let $\{\ub_i\}_{i \in \mathcal{I}_1}$ and $\{\vb_{\mu}\}_{\mu \in \mathcal{I}_2}$ be orthonormal basis vectors in $\mathbb{R}^{\mathcal{I}_1}$ and $\mathbb{R}^{\mathcal{I}_2},$ respectively. For $\mathbf{x} \in \mathbb{C}^{\mathcal{I}_1}$  and $\mathbf{y} \in \mathbb{C}^{\mathcal{I}_2},$ we have  
\begin{equation*}
\sum_{i \in \mathcal{I}_1 }|G_{ \xb \ub_i }|^2=\sum_{i \in \mathcal{I}_1 }|G_{ \ub_i \xb}|^2=\frac{|z|^2}{\eta} \Im \left( \frac{G_{\xb \xb}}{z} \right),
\end{equation*}
\begin{equation*}
\sum_{\mu \in \mathcal{I}_2} | G_{\yb \vb_{\mu}}|^2 =\sum_{\mu \in \mathcal{I}_2} | G_{ \vb_{\mu} \yb}|^2=\frac{\Im G_{\yb \yb}(z)}{\eta},  
\end{equation*}
\begin{equation*}
\sum_{i \in \mathcal{I}_1} | G_{\yb \ub_i}|^2=\sum_{i \in \mathcal{I}_1} | G_{\ub_i \yb }|^2=G_{\yb \yb}+\frac{\bar{z}}{\eta} \Im G_{\yb \yb},  
\end{equation*}
\begin{equation*}
\sum_{\mu \in \mathcal{I}_2} | G_{\xb \vb_\mu}|^2=\sum_{\mu \in \mathcal{I}_2} | G_{ \vb_\mu \xb}|^2=\frac{G_{\xb \xb}}{z}+\frac{\bar{z}}{\eta} \Im \left( \frac{G_{\xb \xb}}{z} \right).  
\end{equation*}
\end{lemma}
\begin{proof}
The proofs follow from the spectral decomposition of $G$ as in (\ref{eq_defnG}) and the orthonormality of the basis. See Lemma 4.1 of \cite{YF} for details. 
\end{proof}

Third, we will also need the following estimate. 
%{\color{red} start from here.} 
\begin{lemma}\label{lem_derivativebound}
For any two vectors $\bb_1, \bb_2 \in \mathbb{R}^{\mathcal{I}},$ we have that 
\begin{equation*}
\sum_{\mu \in \mathcal{I}_2} G_{\bb_1 \mu}(z_1) G_{\bb_2 \mu}(z_2)=\frac{\Pi_{\bb_1 \bb_2}(z_1)-\Pi_{\bb_1 \bb_2}(z_2)}{z_1-z_2}+\OO_{\prec}(\eta^{-1}(N \eta)^{-1/2}), 
\end{equation*}
where we used the convention that 
\begin{equation*}
\lim_{z_2 \rightarrow z_1} \frac{\Pi_{\bb_1 \bb_2}(z_1)-\Pi_{\bb_1 \bb_2}(z_2)}{z_1-z_2}=\Pi'_{\bb_1 \bb_2}(z_1).
\end{equation*}
\end{lemma}
\begin{proof}
Note that by spectral decomposition, we have that
\begin{equation}
\sum_{\mu \in \mathcal{I}_2} G_{\bb_1 \mu}(z_1) G_{\bb_2 \mu}(z_2)=\frac{G_{\bb_1 \bb_2}(z_1)-G_{\bb_1 \bb_2}(z_2)}{z_1-z_2}. 
\end{equation}
The proof follows from local law and Cauchy's integral formula. See equation (5.21) of \cite{YF} for more details. 
\end{proof}

Finally, we introduce the device of cumulant expansion. 
%which generalizes the Stein's identity \cite{MR2732624}.    
Recall that for any random variable $h$ its $k$th cumulant is defined as 
\begin{equation}\label{eq_defncumulant}
\kappa_k(h)=\left(\partial_t^k \log \mathbb{E} e^{th} \right)|_{t=0}. 
\end{equation}
\begin{lemma}[Cumulant expansion]\label{lemma_cumulant} Fix any $\ell \in \mathbb{N}$ and let $f \in \mathsf{C}^{\ell+1}(\mathbb{R})$ be a complex-valued function. Suppose $h$ is a real valued random variable with finite moments up to order $\ell+2.$ Then we have that 
\begin{equation*}
\mathbb{E}(f(h)h)=\sum_{k=0}^{\ell} \frac{1}{k!} \kappa_{k+1}(h) \mathbb{E} f^{(k)}(h)+R_{\ell+1},
\end{equation*}
where $\kappa_k(h)$ is the $k$th cumulant of $h$ and $\mathbb{R}_{\ell+1}$ satisfies 
\begin{equation*}
R_{\ell+1} \precsim \mathbb{E}\left| h^{\ell+2} \mathbf{1}_{|h|>N^{-1/2+\epsilon}} \right| \cdot \| f^{(\ell+1)}\|_{\infty}+\mathbb{E}|h|^{\ell+2} \cdot \sup_{|x| \leq N^{-1/2+\epsilon}}|f^{(\ell+1)}(x)|,
\end{equation*}
for any constant $\epsilon>0.$
\end{lemma}
\begin{proof}
See \cite[Proposition 3.1]{MR2561434} or \cite[Section II]{MR1411619}.  
\end{proof}

\subsection{The non-spiked case: CLT for $\mathcal{Y}$} \label{sec_nonpiskedproofgeneral} 
In this subsection, we prove the CLT for $\mathcal{Y}$ defined in (\ref{eq_yytildedefn}). Note that we can write the integrand into a trace form. As before, we set the natural embedding of $\vec{b} \in \mathbb{R}^N$ as $\mathbf{b} \in \mathbb{R}^{N+M}$ such that 
\begin{equation}\label{bembedding}
\mathbf{b}=\begin{pmatrix}
\vec{b} \\
\vec{0}
\end{pmatrix}. 
\end{equation}
Additionally, we denote $B=\bb \bb^*.$ According to (\ref{eq_empiricalresolventform}) and (\ref{eq_measureform}), we can write
\begin{equation}\label{eq_yreducedform}
\mathcal{Y}=\sqrt{M \eta} \oint_{\Gamma} \mathsf g(z) \Tr([G(z)-\Pi(z)]B) \dd z, \ \mathsf g(z)=\frac{g(z)}{2 \pi \ri}.    
\end{equation}
Recall that if $x$ is a real Gaussian random variable, i.e., $x \sim \mathcal{N}(0, \sigma^2),$ denote $\texttt m_n=\mathbb{E} x^n,$ we have that 
\begin{equation}\label{eq_realrelation}
\texttt m_{n+2}=(n+1) \sigma^2 \texttt m_{n}. 
\end{equation}
Our goal is to prove an asymptotic version of (\ref{eq_realrelation}) for $\mathcal{Y}.$

For $\Pi(z)$ defined in (\ref{eq_defnpi}), we introduce the following auxiliary quantities for the ease of statements
\begin{equation}\label{eq_defnA1A2}
A_1=-z \Pi(z), \ A_2=I-A_1. 
\end{equation}
%Since $\Sigma_0$ is diagonal, we observe that for $1 \leq i \leq N$
%\begin{equation*}
%(A_1)_{ii}=\frac{1}{1+m(z) \sigma_i}, \ (A_2)_{ii}=\frac{\sigma_i m(z)}{1+\sigma_i m(z)}. 
%\end{equation*}
In view of (\ref{eq_defnG}),  we will frequently use the following identity 
\begin{equation}\label{eq_identityhg}
G=\frac{1}{z}(HG-I). 
\end{equation}
%Now we decompose $\mathcal{Z}$ as follows {\color{red}[start from here]} 
%\begin{equation*}
%\mathcal{Z}=\sqrt{M}
%\end{equation*}
The starting point is to decompose the following quantity
\begin{equation*}
\mathcal{Z}:=\sqrt{M\eta} \operatorname{Tr}([G(z)-\Pi(z)]B),
\end{equation*}
%and the starting point is the following decomposition 
so that 
\begin{align}\label{eq_decompositionz}
\mathcal{Z}&=\sqrt{M \eta} \left( \Tr(GBA_1)-\Tr (\Pi B)+\Tr(GBA_2)\right) \nonumber \\
& =\sqrt{M \eta} \left( \frac{1}{z} \Tr(HGBA_1)-\frac{1}{z} \Tr BA_1-\Tr (\Pi B)+\Tr(G BA_2) \right) \nonumber \\
& =\sqrt{M \eta} \left( \frac{1}{z} \Tr(HGBA_1)+\Tr(GBA_2) \right),
\end{align}
where in the second step we used (\ref{eq_identityhg}) and in the third step we used the definition of $A_1$ as in (\ref{eq_defnA1A2}). Together with (\ref{eq_yreducedform}), for any integer $k,$ we have that 
\begin{align}\label{eq_basicexpansiononeoneone}
\mathbb{E} \mathcal{Y}^k &= \sqrt{M \eta}\left[\mathbb{E} \oint_{\Gamma} \frac{\mathsf g(z)}{z} \Tr(HGBA_1) \dd z  \mathcal{Y}^{k-1} \right]  \\
& +\sqrt{M \eta} \left[\mathbb{E} \oint_{\Gamma} \mathsf g(z) \Tr(GBA_2) \dd z  \mathcal{Y}^{k-1} \right].  \notag
\end{align}
%We now pause to provide a simple but useful representation for $\Tr(HGBA_1)$ which follows from elementary calculations.
Denote $j'=j+N$ and  
 $\Lambda \in \mathbb{R}^{(M+N) \times (M+N)}$  as
 \begin{equation}\label{eq_defnlambda}
 \Lambda:=
 \begin{pmatrix}
 \Sigma_0^{1/2} & 0\\
 0 & I
 \end{pmatrix}.
 \end{equation}
 We have that 
  \begin{equation} \label{eq_traceexpression}
 \Tr HGBA_1=\sqrt{z}\sum_{i,j} X_{ij} (GBA_1 \Lambda)_{ji'}.
 \end{equation}
%\begin{proof}
%Note that  and the structure of $B$, we can write
%\begin{equation*}
%\Tr HGBA_1=\sqrt{z} \sum_{i,j} Y_{ij} (GBA_1)_{j' i}, Y=\Sigma_0^{1/2} X. 
%\end{equation*}
%For $H$ in (\ref{eq_defnh}), we have that 
%\begin{equation*}
%H=\Lambda \mathcal{X} \Lambda=H_1 \Lambda, \ \mathcal{X}=  
% \begin{pmatrix}
% 0 & \sqrt{z}X\\
% \sqrt{z} X^* & I
% \end{pmatrix}. 
%\end{equation*}
%Using the structure of $B$ and (\ref{eq_defnh}), we immediately see from the %definition of trace that 
%\begin{equation*}
%\Tr HGBA_1=\Tr H_1 GBA_1=\Tr \mathcal{X} GBA_1 \Lambda, 
%\end{equation*}
%where we used the block diagonal structure of  $A_1.$ This completes the %proof using the elementary identity $\Tr (S T)=\sum_{i,j} S_{ij} T_{ji}$ for %any compatible matrices $S$ and $T$ and the structure of $\mathcal{X}$. 
%\end{proof} 

Let $E_{ij'}$ be an $(M+N) \times (M+N)$ matrix whose only nonzero entry is the $(i,j')$th entry and equals to one. We next prepare some expressions for derivatives which follow from elementary calculations. Note that
\begin{equation}\label{eq_derivativeone}
\frac{\partial G}{\partial X_{ij}}=-G \frac{\partial H}{\partial X_{ij}} G=-\sqrt{z} G (\Lambda E_{ij'}+E_{j' i}\Lambda) G.
\end{equation}
Consequently, for any block diagonal matrix $D,$ we have that
\begin{align}\label{eq_partialbdij}
\left( \frac{\partial G}{\partial X_{ij}} BD\right)_{j' i}  =-\sqrt{z} \left[ (G \Lambda)_{j'i}(GBD)_{j'i}+G_{j'j'} (\Lambda GBD)_{ii} \right].
\end{align}
Additionally, we have that 
\begin{align}\label{eq_zxij}
\frac{\partial \mathcal{Z}}{\partial X_{ij}}& =-\sqrt{zM \eta} \Tr \left( G (\Lambda E_{ij'}+E_{j' i}\Lambda) G B\right) \nonumber \\
&=-\sqrt{zM \eta}\left[ (GBG\Lambda)_{ij'}+(\Lambda GBG)_{j'i} \right].
\end{align}

From now on, we will conduct calculations on (\ref{eq_basicexpansiononeoneone}). Our strategy is to focus on the first term of the right-hand side of (\ref{eq_basicexpansiononeoneone}) as we will see later that that second term will be canceled algebraically. 
Denote
\begin{equation}\label{eq_h1h2definition}
h_1= h_1(i,j):=(GB A_1 \Lambda)_{j'i}, \ h_2=\mathcal{Y}^{k-1}.
\end{equation}
Since $\mathsf g(z)$ is purely deterministic, using Lemma \ref{lemma_cumulant}, (\ref{eq_traceexpression}) and (\ref{eq_partialbdij}), we readily obtain 
\begin{align*}
 \sqrt{M \eta} \mathbb{E} \oint_{\Gamma} \frac{\mathsf g(z)}{z} \Tr(HGBA_1) \dd z  \mathcal{Y}^{k-1} = \mathbb{E}( P_1+P_2+P_3).
\end{align*}
Here $P_1$ is defined as 
\begin{align*}
P_1:=\oint_{\Gamma} \mathsf g(z)\left(-\frac{\sqrt{\eta}}{\sqrt{M}} \sum_{i,j} (G \Lambda)_{j'i}(GBA_1 \Lambda)_{j'i}-\frac{\sqrt{\eta}}{\sqrt{M}} \sum_{i,j} G_{j' j'} (\Lambda GBA_1 \Lambda)_{ii} \right) \dd z h_2,
\end{align*}
and $P_2$ is defined as 
\begin{align}\label{eq_defnp2}
P_2:=\frac{\sqrt{\eta}}{\sqrt{M}}  \oint_{\Gamma} \sqrt{z} \mathsf g(z) \sum_{i,j} \left(GBA_1 \Lambda \right)_{j'i} \dd z \frac{\partial h_2}{\partial X_{ij}} ,
\end{align}
and $P_3$ is defined as 
\begin{align}\label{eq_p3decomposition}
P_3:&=\sqrt{\eta} \sum_{l=2}^3 \frac{\kappa_{l+1}}{l!M^{l/2}}\sum_{i,j} \frac{\partial^l }{\partial X^l_{ij}}(\oint_{\Gamma} \mathsf g(z)  h_1 \dd z h_2) +R_1 \nonumber \\
&:=P_{31}+P_{32}+R_1.
\end{align}
In the last equation, $P_{31}$ collects the summation for $l=2$, $P_{32}$ collects that of the summation for $l=3$ and $R_1:=P_3-P_{31}-P_{32}$ is the residual. Here we used the notation that 
\begin{equation*}
\frac{\partial^l }{\partial X^l_{ij}}(h_1 h_2)=\sum_{l_1+l_2=l} {l \choose l_1, l_2} \frac{\partial^{l_1}h_1}{\partial X_{ij}^{l_1}} \frac{\partial^{l_2}h_2}{\partial X_{ij}^{l_2}}, \ {l \choose l_1, l_2}=\frac{l!}{l_1! l_2!}. 
\end{equation*}
We will see later that $l=2$ will contribute nothing, $l=3$ will give some extra terms which explains the fourth moment contributes, and $l=4$ is needed to show $R_1$ is small. 

 For $P_1,$ on one hand, using (\ref{eq_defnG}),  by definitions of $A_1$ and $\Lambda$,  we have that 
\begin{align*}
\frac{\sqrt{\eta}}{\sqrt{M}} \sum_{i,j} (G \Lambda)_{j'i}(GBA_1 \Lambda)_{j'i} & \asymp \frac{\sqrt{\eta}}{\sqrt{M}} \sum_{i,j} (Y^* G_1 \Sigma_0^{1/2})_{ji}(Y^* G_1 \vec{b} \vec{b}^* \Pi \Sigma_0^{1/2})_{ji} \\
& =\frac{\sqrt{\eta}}{\sqrt{M}} \Tr \Sigma_0^{1/2} G_1 YY^* G_1 \vec{b} \vec{b}^* \Pi_1 \Sigma_0^{1/2} \\
&=\frac{\sqrt{\eta}}{\sqrt{M}} \vec{b}_1^* G_1 Y Y^*G_1 \vec{b},  
%\asymp \frac{\sqrt{\eta}}{\sqrt{M}} \vec{b}^* G_1(I+z G_1) \vec{b} \\
%& \prec \frac{\sqrt{\eta}}{\sqrt{M}},
\end{align*}
where we denote $\vec{b}_1:=\Pi_1 \Sigma_0 \vec{b}.$  In view of (\ref{eq_defnG}), we have that for $z \in \Gamma$ 
\begin{align*}
\frac{\sqrt{\eta}}{\sqrt{M}} \sum_{i,j} (G \Lambda)_{j'i}(GBA_1 \Lambda)_{j'i}  \asymp \frac{\sqrt{\eta}}{\sqrt{M}} \sum_{\mu \in \mathcal{I}_2}\bb_1^*G \eb_{\mu}  \bb^* G \eb_{\mu} \prec \frac{1}{\sqrt{M}},
\end{align*}
where we used Lemma \ref{lem_anisotropiclocallaw}, the definition of $\Gamma$ and (\ref{eq_stjbound}) in the last step. On the other hand,  using Lemma \ref{lem_anisotropiclocallaw}, we have that
\begin{align*}
-\frac{\sqrt{\eta}}{\sqrt{M}} \sum_{i,j} G_{j' j'} (\Lambda GBA_1 \Lambda)_{ii}& =-\sqrt{M \eta} \Tr G_2 \Tr(GBA_1 \Lambda^2) \\
& =-m(z)\sqrt{M \eta} \Tr GBA_1 \Lambda^2+\OO_{\prec} \left( (M \eta)^{-1/2} \right) \\
& =-\sqrt{M \eta}\Tr GBA_2+\OO_{\prec} \left( (M \eta)^{-1/2} \right),
\end{align*} 
where in the last step we used the fact that $m(z)B A_1 \Lambda^2=BA_2$ which follows directly from (\ref{eq_defnA1A2}). 

Using the above calculations and inserting them back into (\ref{eq_basicexpansiononeoneone}), we find that 
\begin{equation}\label{eq_newstart}
\mathbb{E} \mathcal{Y}^k=\mathbb{E} P_2+\mathbb{E} P_3+\OO_{\prec}((M \eta)^{-1/2}).
\end{equation}
We summarize the properties of $P_2$ and $P_3$ in the following lemma and defer its proof to Sections \ref{sec_lemmaproofp2p3} and \ref{sec_lemmaproofp2p32nd}.
\begin{lemma}\label{lem_p2p3} We have that
\begin{equation}\label{eq_p2equation}
P_2=(k-1)\mathsf{V}_1(\vec{b}, \vec{b}) \mathcal{Y}^{k-2}+\OO_{\prec}((N \eta)^{-1/2}),
\end{equation}
and 
\begin{equation}\label{eq_p3equation}
P_3=(k-1) \kappa_4 \mathsf{V}_2(\vec{b}, \vec{b})  \mathcal{Y}^{k-2}+\OO_{\prec}((N \eta)^{-1/2}). 
\end{equation}
\end{lemma}

Recall (\ref{eq_realrelation}). It is easy to see that Theorem \ref{thm_mainclt} follows from Lemma \ref{lem_p2p3} and (\ref{eq_newstart}).

\subsection{Proof of Lemma \ref{lem_p2p3}: Verification of (\ref{eq_p2equation})}\label{sec_lemmaproofp2p3} We first provide some useful results. By Lemma \ref{lem_anisotropiclocallaw}, it is easy to see that 
\begin{equation}\label{y_bound}
\mathcal{Y}=\OO_{\prec}(1). 
\end{equation} 
Moreover, using the definition of $\mathcal{Y}$ and (\ref{eq_zxij}), we have 
\begin{align}\label{eq_firstorderh2}
\frac{\partial h_2}{\partial X_{ij}} & =(k-1) \mathcal{Y}^{k-2} \oint_{\Gamma} \mathsf g(z) \frac{\partial \mathcal{Z}}{\partial X_{ij}} \dd z \nonumber \\
&=(k-1) \mathcal{Y}^{k-2} \oint_{\Gamma} \mathsf g(z) \left(-\sqrt{zM\eta}\left[ (GBG\Lambda)_{ij'}+(\Lambda GBG)_{j'i} \right] \right) \dd z. 
\end{align}
 Consequently, in view of (\ref{eq_defnp2}), we can write 
 \begin{align}\label{eq_P2representation}
 P_2=-(k-1) \eta \mathbb{E} \mathcal{L} \mathcal{Y}^{k-2},
 \end{align}
 where $\mathcal{L}$ is defined as 
 \begin{align}\label{eq_mathcalLrep}
 \mathcal{L}:& = 2\sum_{i,j} \oint_{\Gamma} \oint_{\Gamma} \sqrt{z_1 z_2}\mathsf g(z_1) \mathsf g(z_2) (G(z_1) BA_1(z_1) \Lambda)_{j'i} (G(z_2) BG(z_2) \Lambda)_{ij'}  \dd z_1 \dd z_2 \nonumber \\
 &=2 \oint_{\Gamma} \oint_{\Gamma} \mathsf g(z_1) \mathsf g(z_2) \Tr (\Sigma_0^{1/2} G_1(z_2) \vec{b} \vec{b}^* G_1(z_2) YY^* G_1(z_1) \vec{b} \vec{b}^* (1+m(z_1) \Sigma_0)^{-1} \Sigma_0^{1/2} )  \dd z_1 \dd z_2 \nonumber \\
 &=2  \oint_{\Gamma} \oint_{\Gamma} \mathsf g(z_1) \mathsf g(z_2) \left[ \vec{b}^*  (1+m(z_1) \Sigma_0)^{-1} \Sigma_0  G_1(z_2) \vec{b} \right] \left[ \vec{b}^*G_1(z_2) YY^*G_1(z_1)  \vec{b} \right]    \dd z_1 \dd z_2. 
 \end{align}
Using the structure of (\ref{eq_defnG}) and (\ref{bembedding}), we have that 
\begin{align}
\vec{b}^*G_1(z_2) YY^*G_1(z_1)  \vec{b} &=\sqrt{z_1 z_2} \sum_{\mu \in \mathcal{I}_2}\bb^*G \eb_{\mu}  \bb^* G \eb_{\mu} \nonumber \\
&=\sqrt{z_1 z_2} \frac{\vec{b}^*(\Pi_1(z_1)-\Pi_1(z_2))\vec{b}}{z_1-z_2}+\OO_{\prec}(\eta^{-1}(N \eta)^{-1/2}), 
\end{align} 
 where in the last step we used Lemma \ref{lem_derivativebound}. The rest of the proof follows from Lemma \ref{lem_anisotropiclocallaw} and (\ref{y_bound}).
% 
% 
%  Then the proof follows from the local laws and Ward's identities. 

\subsection{Proof of Lemma \ref{lem_p2p3}: Verification of (\ref{eq_p3equation})}\label{sec_lemmaproofp2p32nd}

To control $P_3,$ we separate our discussion in the following three subsections according to the order of the expansion as in (\ref{eq_p3decomposition}).   

\subsubsection{$l=2$}\label{sec_l2discussion} This corresponds to the term $P_{31}$ in (\ref{eq_p3decomposition}). Formally, we can write
\begin{align*}
 P_{31}=\frac{\kappa_3}{2} \frac{\sqrt{\eta}}{M} \mathbb{E} \sum_{i,j} \left(\mathsf{P}_{31}(2,0)+\mathsf{P}_{31}(1,1)+\mathsf{P}_{31}(0,2)\right), 
 \end{align*}
 where we denote 
 \begin{equation}\label{eq_p3120}
 \mathsf{P}_{31}(2,0) =  \mathsf{P}_{31}(2,0; i,j)=: \oint_{\Gamma} \mathsf g(z) \left( \frac{\partial^2 G}{\partial X_{ij}^2} BA_1 \Lambda \right)_{j'i} \dd z h_2, 
 \end{equation}
  \begin{equation}\label{eq_p3111}
 \mathsf{P}_{31}(1,1) =  \mathsf{P}_{31}(1,1; i,j)=2 \oint_{\Gamma} \mathsf g(z) \left( \frac{\partial G}{\partial X_{ij}} BA_1 \Lambda \right)_{j'i} \dd z \frac{h_2}{\partial X_{ij}}, 
 \end{equation}
  \begin{equation}\label{eq_p3102}
 \mathsf{P}_{31}(0,2)=  \mathsf{P}_{31}(0,2; i,j)= \oint_{\Gamma} \mathsf g(z) \left( G BA_1 \Lambda \right)_{j'i} \dd z \frac{\partial^2 h_2}{\partial X^2_{ij}} .
 \end{equation}

We first prepare some useful identities, which can be obtained using some elementary calculation. Using (\ref{eq_zxij}) and (\ref{eq_derivativeone}), we have that 
\begin{align}\label{eq_form}
\frac{\partial^2 \mathcal{Z}}{\partial X_{ij}^2}= &z \sqrt{M \eta} \Big( (G\Lambda)_{ii} (GBG\Lambda)_{j'j'}+G_{ij'}(\Lambda GBG\Lambda)_{ij'}+(GBG\Lambda)_{ii} (G \Lambda)_{j'j'}+(GBG)_{ij'}(\Lambda G \Lambda)_{ij'} \nonumber \\
&+(\Lambda G \Lambda)_{j'i} (GBG)_{j'i}+(\Lambda G)_{j'j'} (\Lambda GBG)_{ii}+(\Lambda GBG\Lambda)_{j'i} G_{j'i}+(\Lambda GBG)_{j'j'} (\Lambda G)_{ii} \Big). 
\end{align}
Moreover, we have that
\begin{align}\label{eq_partial2hpartialxij}
\frac{\partial^2 h_2}{\partial X_{ij}^2}&=(k-1)(k-2) \mathcal{Y}^{k-3} \left( \oint_{\Gamma} \mathsf g(z) \frac{\partial \mathcal{Z}}{\partial X_{ij}} \dd z \right)^2 \nonumber \\
&+(k-1) \mathcal{Y}^{k-2} \oint_{\Gamma} \mathsf g(z) \frac{\partial^2 \mathcal{Z}}{\partial X^2_{ij}} \dd z.
\end{align}
Additionally, we have 
\begin{equation*}
\frac{\partial^2 G}{\partial X^2_{ij}}=2 z \left[G(\Lambda E_{ij'}+E_{j'i} \Lambda) \right]^2 G. 
\end{equation*}
For the ease of discussion, in what follows, we use the following shorthand notation
\begin{equation}\label{eq_shorhandnotation}
\mathcal{L}_{ij}:=\Lambda E_{ij'}+E_{j' i} \Lambda. 
\end{equation}
For any block-diagonal matrix $D=D_1 \oplus D_2$, we have that 
\begin{align}\label{eq_secondorderiative}
\left( \frac{\partial^2 G}{ \partial X_{ij}^2} BD \right)_{j'i}& =2z \left( G \mathcal{L}_{ij} G \mathcal{L}_{ij} GBD \right)_{j'i} \nonumber \\
& =2z\left[(G \mathcal{L}_{ij}G)_{j'j'}(GBD)_{ii}+ (G \mathcal{L}_{ij}G)_{j'i}(GBD)_{j'i}\right]. 
\end{align} 
Note that 
\begin{equation}\label{eq_extensionone}
(G\mathcal{L}_{ij} G)_{j'j'}=2(G\Lambda)_{j'i} G_{j'j'}, \ \ (GBD)_{ii}=(G_1 \vec{b} \vec{b}^* D_1)_{ii}, \ (G\mathcal{L}_{ij}G)_{ii}=(G\Lambda)_{ii}G_{j'i}+G_{ij'}(\Lambda G)_{ii}. 
\end{equation}
and 
\begin{equation}\label{eq_extensiontwo}
(G \mathcal{L}_{ij}G)_{j'i}=(G \Lambda)_{j' i} G_{ii}+G_{j'i} (\Lambda G)_{ii},
\end{equation}
\begin{equation}\label{eq_extensionthree}
(GBD)_{j'i}=(G_{21} \vec{b} \vec{b}^* D_1)_{ji}.
\end{equation}
Moreover, we have that  
\begin{equation}\label{eq_oneoneoneone}
(\Lambda GBG)_{ii}=(\Sigma^{1/2} G_1 \vec{b} \vec{b}^* G_1)_{ii}, \ (\Lambda GBG)_{ij'}=(\Sigma_0^{1/2} G_{1} \vec{b} \vec{b}^* G_{12})_{ij},
 \end{equation}
 and
 \begin{equation}\label{eq_formulajjprime}
 (\Lambda GBG)_{j'j'}=(G_{21} \vec{b} \vec{b}^* G_{12})_{j'j'},
 \end{equation}
 where we used the conventions that $G_{12}=z^{-1/2} G_1 Y$ and $G_{21}=G_{12}^*.$

We give a more explicitly form of $D_1.$ As in (\ref{eq_traceexpression}), $D=A_1 \Lambda.$ Consequently, we shall have that
\begin{equation}\label{eq_defnd1}
D_1=-z \Pi_1(z) \Sigma_0^{1/2}. 
\end{equation}
This leads to that
\begin{equation}\label{eq_decomposeaaaaaa}
\vec{b}^* D_1 \vec{f}_i= -z \vec{b}^* \Pi_1(z) \Sigma_0^{1/2} \vec{f}_i=\vec{f}_i^* \Sigma^{1/2} \Pi_1(z) \vec{b}. 
\end{equation}
Note that $D_1$ is symmetric since $\Pi_1$ and $\Sigma_0$ share the same eigenvectors. 

We summarize the main estimates in the following lemma.
\begin{lemma}
We have the following estimates
\begin{equation}\label{eq_p31}
\frac{\kappa_3}{2} \frac{\sqrt{\eta}}{M}  \sum_{i,j} \mathsf{P}_{31}(2,0)=\OO_{\prec}\left( \frac{1}{\sqrt{M \eta}} \right), 
\end{equation} 
\begin{equation}\label{eq_p32}
\frac{\kappa_3}{2} \frac{\sqrt{\eta}}{M} \sum_{i,j} \mathsf{P}_{31}(1,1)=\OO_{\prec}\left( \frac{1}{\sqrt{M \eta}} \right), 
\end{equation} 
\begin{equation}\label{eq_p33}
\frac{\kappa_3}{2} \frac{\sqrt{\eta}}{M}  \sum_{i,j} \mathsf{P}_{31}(0,2)=\OO_{\prec}\left( \frac{1}{\sqrt{M \eta}} \right). 
\end{equation} 
\end{lemma} 
\begin{proof}
\noindent {\bf (1). Justification of (\ref{eq_p31}).} In view of (\ref{eq_secondorderiative}), (\ref{eq_extensionone}) and the definition of $\mathsf{P}_{31},$ we focus our discussion on some typical terms. By Lemma \ref{lem_anisotropiclocallaw}, we see that  
\begin{align*}
\left|\frac{\sqrt{\eta}}{M} \sum_{i,j} G_{j' j'} (G\Lambda)_{j' i} (G_1 \vec{b} \vec{b}^* D_1)_{ii} \right| & \prec   \frac{1}{\sqrt{\eta}}\frac{\sqrt{\eta}}{M^{3/2}} \sum_{i,j} |\vec{f}_i^* G_1 \vec{b}| | \vec{f}_i^* D_1 \vec{b}| \\
& \prec \frac{1}{\sqrt{M}} \sum_i |\vec{f}_i^* G_1 \vec{b}| | \vec{f}_i^* D_1 \vec{b}|, 
%\prec \frac{1}{\sqrt{M \eta}},  
\end{align*} 
where in the first step we used $(\Pi \Lambda)_{j'i}=0$ 
and the symmetry of $D_1.$ Applying the Cauchy-Schwarz inequality, we have that  
\begin{align}\label{eq_controlcontrolcontrol}
\sum_i |\vec{f}_i^* G_1 \vec{b}| | \vec{f}_i^* D_1 \vec{b}| & \leq \left( \sum_i | \vec{f}_i^* D_1 \vec{b}|^2 \right)^{1/2} \left( \sum_i | \vec{f}_i^* G_1 \vec{b}|^2 \right)^{1/2} \nonumber \\
&=\left( \sum_i | \vec{f}_i^* D_1 \vec{b}|^2 \right)^{1/2} \left( \sum_i |G_{\eb_i \bb}|^2 \right)^{1/2}.
\end{align}
 Using (\ref{eq_decomposeaaaaaa}), it is easy to see that  
\begin{equation}\label{eq_squarebound}
\sum_i |\vec{f}_i^* D_1 \vec{b}|^2 \asymp  \sum_i \vec{f}_i^* \Sigma_0^{1/2} \Pi_1(z) \vec{b} \vec{b}^* \Pi_1 \Sigma_0^{1/2} \vec{f}_i=\vec{b}^* \Pi_1 \Sigma_0 \Pi_1 \vec{b} \asymp 1.  
\end{equation}
Inserting the above estimate back into (\ref{eq_controlcontrolcontrol}), together with Ward's identities in Lemma \ref{lem_wardidentity}, we find that 
\begin{align}\label{eq_keyestimationone}
\sum_i |\vec{f}_i^* G \vec{b}| | \vec{f}_i^* D \vec{b}| \prec \frac{1}{\sqrt{\eta}}, 
\end{align}
where we used Lemma \ref{lem_anisotropiclocallaw} and (\ref{eq_stjbound}) to obtain that $\Im (G_{\bb \bb}/z) \asymp 1.$ This yields that  
\begin{equation*}
\left|\frac{\sqrt{\eta}}{M} \sum_{i,j} G_{j' j'} (G\Lambda)_{j' i} (G_1 \vec{b} \vec{b}^* D_1)_{ii} \right|=\OO_{\prec}\left(\frac{1}{\sqrt{M \eta}} \right). 
\end{equation*}

Similarly,  we have that 
\begin{align*}
\left|\frac{\sqrt{\eta}}{M} \sum_{i,j} (G \Lambda)_{j'i} G_{ii} (Y^* G_1 \vec{b} \vec{b}^* D_1)_{ji} \right| & \prec \frac{1}{M^{3/2}} \sum_{i,j} |\vec{f}_j^* Y^* G_1 \vec{b}^*| |\vec{b}^* D_1 \vec{f}_i| \\
& \prec \frac{1}{M \sqrt{\eta}} \sum_i |\vec{b}^* D \vec{f}_i| \prec \frac{1}{\sqrt{M \eta}},
\end{align*}
where in the second step we used Lemma \ref{lem_anisotropiclocallaw} and in the last step we used the Cauchy-Schwarz inequality and (\ref{eq_squarebound}) to obtain that for some constant $C>0$ 
\begin{equation}\label{eq_sqrtnbound}
\sum_i|\vec{b}^* D_1 \vec{f}_i| \leq \sqrt{N} \left(\sum_{i} |\vec{b}^* D_1 \vec{f}_i|^2 \right) \leq C \sqrt{N}.
\end{equation}
Analogously, we can show that 
\begin{align*}
\left|\frac{\sqrt{\eta}}{M} \sum_{i,j} G_{j'i} (\Lambda G)_{ii} (Y^* G_1 \vec{b} \vec{b}^* D_1)_{ji} \right| \prec \frac{1}{\sqrt{M \eta}}.
\end{align*}

Using (\ref{eq_secondorderiative}) and the formulas below, in view of the definition (\ref{eq_p3120}), combing the above estimates and (\ref{y_bound}), we have concluded our proof. 

\noindent {\bf (2). Justification of (\ref{eq_p32}).} We again work with some typical terms. Set 
\begin{equation}\label{eq_defnvi}
\vec{v}_i=\Sigma_0^{1/2} \vec{f}_i.
\end{equation}
By Lemma \ref{lem_anisotropiclocallaw}, we have that  
\begin{align*}
& \left| \sqrt{M \eta} \frac{\sqrt{\eta}}{M} \sum_{i,j} G_{j'j'}(z_1) (\Lambda G(z_1)BD(z_1))_{ii} (G(z_2)BG(z_2)\Lambda)_{ij'} \right| \nonumber \\
 & \prec  \frac{\eta }{\sqrt{M}} \sum_{i,j} (\Sigma_0^{1/2} G_1(z_1) \vec{b} \vec{b}^* D_1(z_1))_{ii} (G_1(z_2) \vec{b} \vec{b}^* G_1(z_2)Y)_{ij}+ \frac{1}{M^{3/2}} \sum_{i,j} |\vec{b}^* G_1(z_1) \vec{v}_i \vec{b}^* D_1(z_1) \vec{f}_i | \nonumber \\ 
%&= \frac{\eta}{\sqrt{M}}  \sum_{i,j} \vec{v}_i^* G_1 \vec{u} \vec{u}^* D_1(z_1) \vec{f}_i \vec{f}^*_i G_1 \vec{u} \vec{u}^* G_1 \Sigma^{1/2}X \vec{f}_j \\
& \prec \frac{\eta}{\sqrt{M}}  \sum_{i,j}\vec{b}^* G_1(z_1) \vec{v}_i \vec{b}^* D_1(z_1) \vec{f}_i \vec{b}^* G_1(z_2) Y \vec{f}_j+\frac{1}{\sqrt{M}} \sum_{i} |\vec{b}^* G_1(z_1) \vec{v}_i \vec{b}^* D_1(z_1) \vec{f}_i |.
%+ \frac{1}{\sqrt{M \eta}}
%\prec \frac{1}{\sqrt{M \eta}}, 
\end{align*}
We first consider the second term of the right-hand side of the above equation. The discussion is similar to (\ref{eq_controlcontrolcontrol}) and (\ref{eq_keyestimationone}) except that $\{\vec{v}_i\}$ may not be an orthonormal basis so that Lemma \ref{lem_wardidentity} cannot be applied directly. Note that since $\| \Sigma_0 \|$ is bounded, by the Cauchy-Schwarz inequality,  we have that for some constant $C>0$
\begin{align}\label{eq_nononhandle}
\sum_i |\vec{b}^* G_1(z_1) \vec{v}_i|^2=\vec{b}^* G_1(z_1) \Sigma_0 \overline{G}_1(z_1) \vec{b} \leq C \sum_i |\vec{b}^* G_1(z_1) \vec{f}_i|^2.
\end{align}
As a result, together with (\ref{eq_keyestimationone}), we readily obtain that 
\begin{equation}\label{eq_newnewupdateupdate}
\frac{1}{\sqrt{M}} \sum_{i} |\vec{b}^* G_1(z_1) \vec{v}_i \vec{b}^* D_1(z_1) \vec{f}_i | \prec \frac{1}{\sqrt{M \eta}}. 
\end{equation}
The first term can be controlled similarly using that
\begin{align*}
\sum_j \vec{b}^* G_1(z_2) Y \vec{f}_j =\vec{b}^* G_1(z)Y \bm{1} \prec \frac{\sqrt{N}}{\sqrt{N\eta}}=\frac{1}{\sqrt{\eta}},
\end{align*}
where $\mathbf{1}$ is a vector with all unity and in the last step we used Lemma \ref{lem_anisotropiclocallaw}. This yields that 
\begin{equation*}
 \left| \sqrt{M \eta} \frac{\sqrt{\eta}}{M} \sum_{i,j} G_{j'j'}(z_1) (\Lambda G(z_1)BD(z_1))_{ii} (G(z_2)BG(z_2)\Lambda)_{ij'} \right|  \prec \frac{1}{\sqrt{M \eta}}. 
\end{equation*}

Similarly, we can show that  
\begin{align*}
\left| \sqrt{M \eta} \frac{\sqrt{\eta}}{M} \sum_{i,j} (G(z_1) \Lambda)_{j'i} (G(z_1)BD(z_2))_{j'i} (G(z_2)BG(z_2)\Lambda)_{ij'} \right|   
%\prec \frac{1}{M^{3/2}} \sum_{i,j} |(GBD)_{j'i}||\vec{f}_i^* G_1 \vec{u}| 
\prec \frac{1}{\sqrt{M \eta}}. 
\end{align*}
Using(\ref{eq_zxij}) and  (\ref{eq_partialbdij}), in view of the definition (\ref{eq_p3111}), by (\ref{y_bound}),  we have completed the proof. 
%\begin{equation*}
%\frac{\kappa_3}{2} \frac{\sqrt{\eta}}{M} \mathbb{E} \sum_{i,j} \mathsf{P}_{31}(1,1)=\OO_{\prec}\left( \frac{1}{\sqrt{M \eta}} \right), 
%\end{equation*} 
%where we used (\ref{eq_constantorderterm}). 

\noindent{\bf (3). Justification of (\ref{eq_p33}).} We work on some typical terms according to (\ref{eq_partial2hpartialxij}). By Lemma \ref{lem_anisotropiclocallaw}, we have that 
\begin{align*}
& \left|\frac{\sqrt{\eta}}{M} M \eta \sum_{i,j} (G(z_0)BA_1(z_0)\Lambda)_{j'i} (G(z_1)BG(z_1) \Lambda)_{ij'} (G(z_2)BG(z_2) \Lambda)_{ij'} \right|  \\
& \asymp \eta^{3/2}  \sum_{i,j} |\vec{b}^* G_1(z_1) \vec{f}_i| |\vec{b}^* G_1(z_2) \vec{f}_i| |\vec{f}_j^* Y^* G_1(z_1) \vec{b}| |\vec{f}_j^* Y^* G_1(z_2) \vec{b}| |\vec{f}_j^* Y^* G_1(z_0) \vec{b}| |\vec{b}^* \Pi_1(z_0) \vec{v}_i \vec{b}| \\
& \prec  \eta^{3/2} \frac{M}{(M \eta)^{3/2}} \sum_{i,j}  |\vec{b}^* G_1(z_1) \vec{f}_i| |\vec{b}^* G_1(z_2) \vec{f}_i| |\vec{b}^* \Pi_1(z_0) \vec{v}_i | \\
& \prec \frac{1}{\sqrt{M}} \sum_i  |\vec{b}^* G_1(z_2) \vec{f}_i| |\vec{b}^* \Pi_1(z_0) \vec{v}_i | \prec \frac{1}{\sqrt{M \eta}},
\end{align*} 
where in the last step we used (\ref{eq_newnewupdateupdate}). Similarly, we have that 
\begin{align*}
& \left| \frac{\sqrt{\eta}}{M} \sqrt{M \eta} \sum_{i,j} \left( G(z_1) B A_1(z_1) \Lambda \right)_{j'i} (G(z_2) \Lambda)_{ii}(G(z_2)BG(z_2) \Lambda)_{j'j'} \right| \\
& \prec \frac{\eta}{\sqrt{M}} \frac{M}{(M \eta)^{3/2}} \sum_i|\vec{b}^* \Pi_1(z_1) \vec{v}_i| \asymp \frac{1}{\sqrt{M \eta}},
\end{align*}
and 
\begin{align*}
& \left| \frac{\sqrt{\eta}}{M} \sqrt{M \eta} \sum_{i,j} (G(z_1) BA_1(z_1) \Lambda)_{j'i} (\Lambda G(z_2) \Lambda)_{j'i}(G(z_2) BG(z_2))_{j'i} \right| \\
& \prec \frac{\eta}{\sqrt{M}} \frac{M}{(M \eta)^{3/2}} \sum_i|\vec{b}^* \Pi_1(z_1) \vec{v}_i| \asymp \frac{1}{\sqrt{M \eta}}.
\end{align*}
The other terms can be analyzed in the same way. Using (\ref{eq_partial2hpartialxij}) and (\ref{eq_form}), in view of the definition (\ref{eq_p3102}), combing the above estimates and (\ref{y_bound}), we have concluded our proof. 

\end{proof}

\subsubsection{$l=3$}\label{sec_l3}
This corresponds to the term $P_{32}$ in (\ref{eq_p3decomposition}).  We decompose $P_{32}$ as follows 
\begin{equation*}
P_{31}=\frac{\kappa_4}{6} \frac{\sqrt{\eta}}{M^{3/2}} \mathbb{E} \sum_{i,j} \left(\mathsf{P}_{32}(1,2)+\mathsf{P}_{32}(2,1)+\mathsf{P}_{32}(0,3)+\mathsf{P}(3,0) \right),
\end{equation*} 
where we denote
\begin{equation*}
\mathsf{P}_{32}(1,2) = \mathsf{P}_{32}(1,2,i,j):=3 \oint_{\Gamma} \mathsf g(z) \left( \frac{\partial G}{\partial X_{ij}} BA_1 \Lambda \right)_{j'i} \dd z  \frac{\partial^2 h_2 }{\partial X_{ij}^2},
\end{equation*}
\begin{equation*}
\mathsf{P}_{32}(2,1):=3 \oint_{\Gamma} \mathsf g(z) \left( \frac{\partial^2 G}{\partial X^2_{ij}} BA_1 \Lambda \right)_{j'i} \dd z  \frac{\partial h_2 }{\partial X_{ij}},
\end{equation*}
\begin{equation*}
\mathsf{P}_{32}(3,0):=\oint_{\Gamma} \mathsf g(z) \left( \frac{\partial^3 G}{\partial X^3_{ij}} BA_1 \Lambda \right)_{j'i} \dd z h_2,
\end{equation*}
\begin{equation*}
\mathsf{P}_{32}(0,3):=\oint_{\Gamma} \mathsf g(z) \left( G BA_1 \Lambda \right)_{j'i} \dd z \frac{\partial^3 h_2 }{\partial X_{ij}^3}.
\end{equation*}
We first prepare some identities. Using (\ref{eq_shorhandnotation}), observe that
\begin{equation*}
\frac{\partial^3 G}{\partial X_{ij}^3}=-6 z^{3/2}\left[G \mathcal{L}_{ij} \right]^3 G,
\end{equation*}
which yields that   
% The trick is to use the matrix L_{ij}
\begin{equation}\label{eq_aaaa}
\left( \frac{\partial^3 G}{\partial X_{ij}^3} BD \right)_{j'i}=-6z^{3/2}\left[(G \mathcal{L}_{ij} G \mathcal{L}_{ij} G )_{j'j'}(GBD)_{ii}+(G \mathcal{L}_{ij} G \mathcal{L}_{ij} G )_{j'i} (GBD)_{j'i} \right].
\end{equation}
Using (\ref{eq_extensionone}) and (\ref{eq_extensiontwo}), we readily obtain that 
\begin{equation}\label{eq_3rdorder}
(G \mathcal{L}_{ij} G \mathcal{L}_{ij} G )_{j'j'}=2 (G\mathcal{L}_{ij} G)_{j'i} G_{j'j'}, \  (G \mathcal{L}_{ij} G \mathcal{L}_{ij} G )_{j'i}=(G \mathcal{L}_{ij}G \Lambda)_{j'i} G_{ii}+(G \mathcal{L}_{ij} \Lambda G)_{ii} G_{j'i}.    % The trick is to use the matrix L_{ij}
\end{equation}
%Moreover,
%\begin{align*}
%\frac{\partial^3 \mathcal{Z}}{\partial X_{ij}^3}=
%\end{align
%{\color{red} continue from here.}*}

 We summarize the results in the following lemma. Recall (\ref{eq_mathsfv2definition}). 
\begin{lemma} We have the following estimates,
\begin{align}\label{eq_p3212}
\frac{\kappa_4}{6} \frac{\sqrt{\eta}}{M^{3/2}} \sum_{i,j} \mathsf{P}_{32}(1,2)= \kappa_4 \mathsf{V}_2(\vec{b}, \vec{b})  \mathcal{Y}^{k-2}+\OO_{\prec} \left( \frac{1}{\sqrt{M \eta}} \right), 
\end{align}
\begin{align}\label{eq_p3221}
\frac{\kappa_4}{6} \frac{\sqrt{\eta}}{M^{3/2}} \sum_{i,j}  \mathsf{P}_{32}(2,1)= \OO_{\prec} \left( \frac{1}{M \sqrt{\eta}} \right), 
\end{align}
\begin{align}\label{eq_p3230}
\frac{\kappa_4}{6} \frac{\sqrt{\eta}}{M^{3/2}} \sum_{i,j}  \mathsf{P}_{32}(3,0)= \OO_{\prec} \left( \frac{1}{M \sqrt{\eta}} \right), 
\end{align}
\begin{align}\label{eq_p3203}
\frac{\kappa_4}{6} \frac{\sqrt{\eta}}{M^{3/2}} \sum_{i,j}  \mathsf{P}_{32}(0,3)= \OO_{\prec} \left( \frac{1}{M \sqrt{\eta}} \right). 
\end{align}
\end{lemma}
\begin{proof}
\noindent{\bf (1). Justification of (\ref{eq_p3212}). } As before, we first study discussion on some typical terms. Especially, we  focus on the following term
\begin{align*}
\frac{\kappa_4\sqrt{\eta}}{2 M^{3/2}}  \oint_{\Gamma} \mathsf g(z) \frac{\partial^2 \mathcal{Z}}{\partial X^2_{ij}} \dd z \oint_{\Gamma} \mathsf g(z) \left( \frac{\partial G}{\partial X_{ij}} BA_1 \Lambda \right)_{j'i} \dd z. 
\end{align*}
We analyze several terms according to (\ref{eq_form}) and (\ref{eq_partialbdij}). For notational convenience, we set 
$D=A_1 \Lambda.$ We claim that 
\begin{align}\label{eq_controlkeypartone}
 \frac{\kappa_4 \sqrt{\eta}}{2 M^{3/2}}  \oint_{\Gamma} \oint_{\Gamma} & \sum_{i,j} \mathsf g(z_1) \mathsf g(z_2)(-\sqrt{z}_1 G_{j'j'}(z_1) (\Lambda G(z_1)BD(z_1))_{ii}) z_2 \sqrt{M \eta} (\Lambda G(z_2))_{j'j'} (\Lambda G(z_2)BG(z_2))_{ii} \nonumber \\
&=-\frac{\kappa_4 \eta}{2}  \oint_{\Gamma} \oint_{\Gamma} \mathsf g(z_1) \mathsf g(z_2) z_1^{1/2} z_2  m(z_1) m(z_2) \mathcal{J}(z_1,z_2) \dd z_1 \dd z_2+\OO_{\prec}((M\eta)^{-1/2}),   
\end{align}
where $\mathcal{J}(z_1, z_2)$ is defined as 
\begin{equation*}
\mathcal{J}(z_1, z_2):=\sum_{i} (\Sigma_0^{1/2} \Pi_1(z_1) \vec{b} \vec{b}^* D_1(z_1))_{ii} (\Sigma_0^{1/2} \Pi_1(z_2) \vec{b} \vec{b}^* \Pi_1(z_2))_{ii}, 
\end{equation*}
where we recall the definition (\ref{eq_defnd1}). 

To see (\ref{eq_controlkeypartone}), by (\ref{eq_oneoneoneone}),  we notice that 

%This the case when $\kappa_4$ comes in to play a role.  We again start with the calculations of some typical terms. Applying local laws and (\ref{eq_keyestimationone}), we have that 
\begin{align*}
& \frac{\eta}{M} \sum_{i,j} G_{j' j'}(z_1)  (\Lambda G(z_1)BD(z_1))_{ii} (\Lambda G(z_2))_{j'j'} (\Lambda G(z_2)BG(z_2) )_{ii} \\
&=\frac{\eta}{M} \sum_{i,j} G_{j'j'}(z_1) G_{j'j'}(z_2) (\Sigma_0^{1/2} G_1(z_2) \vec{b} \vec{b}^* G_1(z_2))_{ii} (\Sigma_0^{1/2} G_1(z_1) \vec{b} \vec{b}^* D_1(z_1))_{ii} \\
&=  \left( \frac{1}{M} \sum_{j} G_{j' j'}(z_1) G_{j'j'}(z_2) \right) \left( \eta  \sum_i [\vec{v}_i^*  G_1(z_2) \vec{b}][\vec{b}^* G_1(z_2) \vec{f}_i][\vec{v}_i^* G_1(z_1) \vec{b}][ \vec{b}^* D_1(z_1) \vec{f}_i]   \right) \\
&:=\mathcal{L}_1 \mathcal{L}_2. 
%& =m(z) \overline{m(z)}\left( \eta \sum_i [\vec{f}_i^* \Sigma^{1/2} \bar{G}_1 \vec{u}][\vec{u}^* \bar{G}_1 \vec{f}_i][\vec{f}_i^* \Sigma^{1/2} G_1 \vec{u} \vec{u}^* D_1 \vec{f}_i]   \right)+\OO_{\prec}(M^{-1/2}) \\
%&=\eta m(z) \overline{m(z)}\left( \sum_i [\vec{f}_i^* \Sigma^{1/2} \bar{\Pi}_1 \vec{u}][\vec{u}^* \bar{\Pi}_1 \vec{f}_i][\vec{f}_i^* \Sigma^{1/2} \Pi_1 \vec{u} \vec{u}^* D_1 \vec{f}_i]   \right) +\OO_{\prec}(M^{-1/2}) \\
%& =\OO_{\prec} (\eta+M^{-1/2}).  
\end{align*}
where we recall (\ref{eq_defnvi}). On one hand, we have  from Lemma \ref{lem_anisotropiclocallaw} that 
\begin{equation*}
\mathcal{L}_1=m(z_1)m(z_2)+\OO_{\prec}\left( \frac{1}{\sqrt{M \eta}} \right).  
\end{equation*}
On the other hand, by a discussion similar to (\ref{eq_newnewupdateupdate}), together with Lemma \ref{lem_anisotropiclocallaw}, we obtain that 
\begin{align*}
\mathcal{L}_2-\eta  \sum_i [\vec{v}_i^*  \Pi_1(z_2) \vec{b}][\vec{b}^* \Pi_1(z_2) \vec{f}_i][\vec{v}_i^* \Pi_1(z_1) \vec{b}][ \vec{b}^* D_1(z_1) \vec{f}_i]=\OO_{\prec}(M^{-1/2}).  
\end{align*}
Consequently, we have that 
\begin{equation*}
\mathcal{L}_1 \mathcal{L}_2=\eta m(z_1) m(z_2)  \sum_i [\vec{v}_i^*  \Pi_1(z_2) \vec{b}][\vec{b}^* \Pi_1(z_2) \vec{f}_i][\vec{v}_i^* \Pi_1(z_1) \vec{b}][ \vec{b}^* D_1(z_1) \vec{f}_i]+\OO_{\prec}((M \eta)^{-1/2}).
\end{equation*}

Analogously, by Lemma \ref{lem_anisotropiclocallaw}, using (\ref{eq_formulajjprime}) and (\ref{eq_newnewupdateupdate}),  we can show that 
\begin{align}\label{eq_typicalform}
& \left| \frac{\eta}{M} \sum_{i,j} G_{j' j'}(z_1)  (\Lambda G(z_1)BD(z_1))_{ii} (\Lambda G(z_2))_{ii} (\Lambda G(z_2)BG(z_2) )_{j'j'} \right| \nonumber \\
& \prec \frac{\eta}{M} \frac{M}{M \eta} \sum_{i}  |\vec{v}_i^* G_1(z_1) \vec{b}| | \vec{b}^* D_1(z_1) \vec{f}_i| \prec \frac{1}{M \sqrt{\eta}},
\end{align}
and
\begin{align*}
& \left| \frac{\eta}{M}\sum_{i,j} (G\Lambda)_{j'i}(z_1) (G(z_1)BD(z_1))_{j'i} (\Lambda G(z_2))_{j'j'} (\Lambda G(z_2) BG(z_2))_{ii} \right|  \\
& \prec \frac{\eta}{M} \frac{M}{M \eta} \sum_{i}  |\vec{v}_i^* G_1(z_2) \vec{b}| | \vec{b}^* D_1(z_1) \vec{f}_i| | \vec{b}^* D_1(z_2) \vec{f}_i|\prec \frac{1}{M \sqrt{\eta}}. 
%& \leq C \frac{1}{M} \sum_i |(\Lambda GBG)_{ii}| \\
%&=\frac{C}{M} \sum_i |\vec{f}_i^* \Sigma^{1/2} G_1 \vec{u}| |\vec{u}^* G_1 \vec{f}_i| \leq \frac{C_1}{M}, 
\end{align*} 
The rest of the terms can be analyzed since they can all be reduced to the form (\ref{eq_typicalform}). This completes the proof using the above estimates and (\ref{y_bound}).

\noindent{\bf (2). Justification of (\ref{eq_p3221}).}  According to (\ref{eq_secondorderiative}) and (\ref{eq_firstorderh2}), we focus on the following term which is the leading term 
\begin{align*}
& \left| \frac{\sqrt{\eta}}{M^{3/2}} \sqrt{M \eta} \sum_{i,j} (G(z_1) \Lambda)_{j'i} G(z_1)_{j'j'}(G_1(z_1) \vec{b} \vec{b}^* D_1(z_1))_{ii} (G(z_2)BG(z_2)\Lambda)_{j'i} \right| \\
& \prec \frac{\eta}{M} \frac{1}{M \eta} \sum_{i,j} |\vec{f}_i^* G_1(z_1) \vec{b}| |\vec{b}^* D_1(z_1) \vec{f}_i| |\vec{v}_i^* G_1(z_2) \vec{b}| \prec \frac{1}{M \sqrt{\eta}}.
\end{align*}  
Here in the first step we used Lemma \ref{lem_anisotropiclocallaw} and in the second step we used a discussion similar to (\ref{eq_typicalform}).  The other terms can be analyzed similarly. This completes our proof. 

\noindent{\bf (3). Justification of (\ref{eq_p3230}).}  According to (\ref{eq_aaaa}), (\ref{eq_3rdorder}), (\ref{eq_extensiontwo}) and (\ref{eq_extensionthree}), we focus our discussion on the following terms which is the leading term 
\begin{align*}
&\left|\frac{\sqrt{\eta}}{M^{3/2}} \sum_{i,j} G_{j'j'}(G\Lambda)_{j'i}G_{ii}(G_1 \vec{b} \vec{b}^*D_1)_{ii} \right| \\
& \prec \frac{1}{M} \sum_i|\vec{f}_i^* G_1 \vec{b}| |\vec{f}_i^* D_1 \vec{b}| \prec  \frac{1}{M \sqrt{\eta}}, 
\end{align*}
where in the first step we used Lemma \ref{lem_anisotropiclocallaw} and in the second step we used (\ref{eq_keyestimationone}). The other terms can be studied similarly, and this completes the proof. 

\noindent{\bf (4). Justification of (\ref{eq_p3203}).}    According to (\ref{eq_form}), (\ref{eq_derivativeone}), (\ref{eq_extensionone}) and (\ref{eq_extensiontwo}), we have found that it suffices to focus on the following leading term 
\begin{align*}
& \left| \frac{\eta}{M} \sum_{i,j} (G(z_1)BA_1(z_1) \Lambda)_{j'i} G_{j'j'}(z_2) (\Lambda G(z_2))_{ii} (\Lambda G(z_2) BG(z_2)\Lambda)_{j'i} \right| \\
& \prec \frac{\eta}{M^2 \eta} \sum_{i,j} |\vec{f}^*_i \Pi_1(z_1) \vec{b}| |\vec{b}^* G_1(z_2) \vec{f}_i| \prec \frac{1}{M \sqrt{\eta}}.  
\end{align*} 
The other terms can be analyzed similarly.  This completes our proof. 
\end{proof}

\subsubsection{The error term $R_1$}
Finally, we control the error term $R_1$ in the cumulant expansion to complete the verification of (\ref{eq_p3equation}). Recall (\ref{eq_h1h2definition}).   According to Lemma  \ref{lemma_cumulant}, it suffices to control the following two terms 
\begin{equation*}
\mathcal{E}_1:=\sqrt{M \eta} \sum_{i,j} \mathbb{E}\left|X_{ij}^5 \mathbf{1}_{\{|X_{ij}|>N^{\epsilon-1/2}\}} \right| \cdot \left\| \frac{\partial^4 w}{\partial X^4_{ij}} \right \|_{\infty}, \  w=\oint_{\Gamma} \mathsf g(z) h_1 \dd z h_2, 
\end{equation*}
and 
\begin{equation*}
\mathcal{E}_2:=\sqrt{M \eta} \sum_{i,j} \mathbb{E}|X_{ij}^5| \cdot \sup_{|x| \leq N^{\epsilon-1/2}} \left|\frac{\partial^4 w(x)}{\partial X^4_{ij}} \right|.
\end{equation*}
By Lemma \ref{lemma_cumulant}, it is easy to see that $R_1 \prec M^{-1/2},$ which follows from the lemma below. Its proof is similar to the discussions in Sections \ref{sec_l2discussion} and \ref{sec_l3} and we only provide the key points.  
\begin{lemma} We have that
\begin{equation*}
\mathcal{E}_1, \mathcal{E}_2 \prec M^{-1/2}.
\end{equation*}
\end{lemma}
\begin{proof}
Using an argument similar to the previous subsections on the control of $\partial^k w/\partial X_{ij}^k,  1 \leq k \leq 3$, we can show that 
\begin{equation}\label{eq_fourthbound}
\left|\frac{\partial^4 w}{\partial X_{ij}^4} \right| \prec \frac{1}{\sqrt{M \eta}}. 
\end{equation} 
For $\mathcal{E}_1,$ using the assumption (\ref{eq_momentassumption}), we find that for any fixed large constant $D>0,$
\begin{equation*}
\mathbb{E}\left|X_{ij}^5 \mathbf{1}_{\{|X_{ij}|>N^{\epsilon-1/2}\}} \right| \leq N^{-D}. 
\end{equation*}
Similar arguments hold for $\mathcal{E}_2$ using (\ref{eq_momentassumption}) and  (\ref{eq_fourthbound}).    This completes our proof. 
\end{proof}
\subsection{The spiked case: CLT for $\widetilde{\mathcal{Y}}$} In this subsection, we briefly discuss how to handle the spiked model and establish the CLT for $\widetilde{\mathcal{Y}}$ as in (\ref{eq_yytildedefn}). Due to similarity, we focus on explaining the main differences from $\widetilde{Y}$. We will utilize the following identity. It reveals the message that the spiked model can be efficiently reduced to the non-spiked model so that the arguments of Sections \ref{sec_nonpiskedproofgeneral}--\ref{sec_lemmaproofp2p32nd} apply. 
\begin{lemma}\label{lem_spikedcase} 
Recall that $\Db=\operatorname{diag}\{d_1, d_2, \cdots, d_r\}$ and $\Vb_r$ be the collection of the first $r$ eigenvectors.  Then we have that 
\begin{equation*}
\widetilde{G}_1(z)=\Sigma^{-1/2} \Sigma_0^{1/2}  \left[ G_1(z)-zG_1(z) \Vb_r(\Db^{-1}+1+z\Vb_r^*G_1(z) \Vb_r)^{-1} \Vb_r^*G_1(z) \right] \Sigma_0^{1/2} \Sigma^{-1/2}. 
\end{equation*}
\end{lemma}
\begin{proof}
See Lemma C.1 of \cite{DT1}. 
\end{proof}
According to Lemma \ref{lem_spikedcase}, we have that
\begin{equation*}
\vec{b}^*\widetilde{G}_1(z)\vec{b}=\sum_{i=1}^N \frac{\omega_i^2}{1+d_i} \left( \vec{v}^*_i G_1(z) \vec{v}_i-z \vec{v}_i^* G_1(z) \mathbf{V}_r (\Db^{-1}+I+z \Vb_r^* G_1(z) \Vb_r)^{-1} \Vb_r^* G_1(z) \vec{v}_i \right),
\end{equation*}  
where we used the convention that  $d_i \equiv 0, i >r. $ Denote
\begin{equation}\label{eq_deltaz}
\Delta(z)=\Vb_r^*(G_1(z)-\Pi_1(z)) \Vb_r,
\end{equation}
and
\begin{equation*}
\Hb:=(\Db^{-1}+I+z \Vb_r^* G_1(z) \Vb_r)^{-1}, \ \Lb_1:=(\Db^{-1}+I+z \Vb_r^* \Pi_1(z) \Vb_r)^{-1}.
\end{equation*}
Then applying a resolvent expansion till the order of two leads to 
\begin{equation*}
\Hb=\Lb_1+\Lb_1 \Delta(z) \Lb_1+(\Lb_1 \Delta(z))^2 \Hb.
\end{equation*}
We now pause to provide the following control.  
\begin{lemma}\label{lemma_outlierbound}
We have that 
\begin{equation*}
\sup_{z \in \Gamma} \| \Lb_1(z) \| \geq \vartheta,
\end{equation*}
for some constant $\vartheta>0.$ 
\end{lemma}
\begin{proof}
Note that for $1 \leq i \leq r,$ we have that 
\begin{equation*}
d_i^{-1}+1+z \vec{v}_i^* \Pi_1(f(-\widetilde{\sigma}_i^{-1})) \vec{v}_i=0, 
\end{equation*}
where $f(\cdot)$ is defined in (\ref{eq_defnstitlesjtransform}). Consequently, according to Assumption \ref{assum_spikes}, we see that for some constant $C>0,$
\begin{equation*}
\sup_{z \in \Gamma}\left| d_i^{-1}+1+z \vec{v}_i^* \Pi_1(f(z)) \vec{v}_i \right|=  \sup_{z \in \Gamma} \left|\frac{1}{1-\widetilde{\sigma}_i^{-1} \sigma_i }-\frac{1}{1+z \sigma_i} \right| \geq C |\widetilde{\sigma}_i^{-1}-z| \geq \vartheta.   
\end{equation*}
This completes our proof. 
% By \cite[Lemma A.3 and equation (A.11)]{Knowles2017}, we have that
%\begin{equation*}
%m(-\widetilde{\sigma}^{-1}_i)-m(b_1)=C 
%\end{equation*} 

\end{proof}

By Lemmas \ref{lem_anisotropiclocallaw} and \ref{lemma_outlierbound}, we have that 
 \begin{align*}
 \vec{b}^*\widetilde{G}_1(z)\vec{b}=&\sum_{i=1}^N \frac{\omega_i^2}{1+d_i}\left(\vec{v}_i^* G_1(z) \vec{v}_i-z \vec{v}_i^* G_1(z) \Vb_r \Lb_1 \Vb_r^* G_1(z) \vec{v}_i-z\vec{v}_i^* \Pi_1(z) \Vb_r \Lb_1 \Delta(z) \Lb_1 \Vb_r^* \Pi_1(z) \vec{v}_i \right) \\
& +\OO_{\prec}\left(\frac{1}{M \eta}\right).
 \end{align*}
Denote 
\begin{equation*}
\Kb:=\sum_{i=1}^N \frac{\omega_i^2}{1+d_i}\left(\vec{v}_i^* \Pi_1(z) \vec{v}_i-z \vec{v}_i^* \Pi_1(z) \Vb_r \Lb_1 \Vb_r^* \Pi_1(z) \vec{v}_i \right).
\end{equation*} 
Applying Lemma \ref{lem_anisotropiclocallaw}, we have that 
\begin{align*}
 \vec{b}^*\widetilde{G}_1(z)\vec{b}-\Kb=\operatorname{Tr}\left((G_1(z)-\Pi_1(z)) \Ab\right)+\OO_{\prec} \left(\frac{1}{M \eta} \right),
\end{align*}
where $\Ab$ is defined as 
\begin{align*}
\Ab:=\sum_{i=1}^N  \frac{\omega_i}{1+d_i} \Big( &\vec{v}_i \vec{v}_i^*-z\Vb_r \Lb_1 \Vb_r^* \Pi_1(z) \vec{v}_i \vec{v}_i^*-z \vec{v}_i \vec{v}_i^* \Pi_1(z) \Vb_r \Lb_1 \Vb_r^* \\
&-z \Vb_r \Lb_1 \Vb^*_r \Pi_1(z) \vec{v}_i \vec{v}_i^*\Pi_1(z) \Vb_r \Lb_1 \Vb_r^* \Big), 
\end{align*}
where we used the definition (\ref{eq_deltaz}). 

To ease our discussion, we denote
\begin{equation}\label{eq_notationssummary}
\vec{l}_i:=z\Vb_r \Lb_1 \Vb_r^* \Pi_1(z) \vec{v}_i
\end{equation}
so that we can rewrite 
\begin{align*}
\Ab:=\sum_{i=1}^N \frac{\omega_i}{1+d_i} \left(\vec{v}_i \vec{v}_i^*-\vec{l}_i \vec{v}_i^*-\vec{v}_i \vec{l}_i^*-z^{-1} \vec{l}_i \vec{l}_i^* \right). 
\end{align*}

Similar to (\ref{eq_yreducedform}), by setting
\begin{equation*}
A:=\begin{pmatrix}
\Ab &0 \\
0& 0
\end{pmatrix},
\end{equation*}
we find that it suffices to study the distribution of 
\begin{equation}
\oint_{\Gamma} \mathsf g(z) \sqrt{M \eta} \Tr((G(z)-\Pi(z))A) \dd z. 
\end{equation}  
Compared to  (\ref{eq_yreducedform}), the only difference  is the deterministic part $A.$ The calculations of Sections \ref{sec_nonpiskedproofgeneral}--\ref{sec_lemmaproofp2p32nd} for $\widetilde{\mathcal{Y}}$ still hold here. In what follows, we only explain how to modify the steps. Denote $\widetilde{P}_2$ and $\widetilde{P}_3$ in (\ref{eq_defnp2}) and (\ref{eq_p3decomposition}) by simply replacing $B$ with $A.$ First, by a discussion similar to (\ref{eq_P2representation}) and (\ref{eq_mathcalLrep}),
we can obtain that 
\begin{equation*}
\widetilde{P}_2=-(k-1) \eta \widetilde{\mathcal{L}}\mathcal{Y}^{k-2}, 
\end{equation*}
where $\widetilde{\mathcal{L}}$ is defined similar to (\ref{eq_mathcalLrep}) as follows 
 \begin{align*}
 \widetilde{\mathcal{L}}:=2 \oint_{\Gamma} \oint_{\Gamma} \mathsf g(z_1) \mathsf g(z_2) \Tr ( \Sigma_0^{1/2} G_1(z_2) \Ab G_1(z_2) YY^* G_1(z_1) \Ab  (1+m(z_1) \Sigma_0)^{-1} \Sigma_0^{1/2} )  \dd z_1 \dd z_2.  
 \end{align*}
Note that $\widetilde{\mathcal{L}}$ can be controlled using Lemma \ref{lem_anisotropiclocallaw} as in (\ref{eq_p2equation}) so that we have
\begin{equation*}
\widetilde{P}_2=(k-1) \widetilde{\mathsf{V}}_1 \widetilde{\mathcal{Y}}^{k-2}+\OO_{\prec}((N \eta)^{-1/2}), 
\end{equation*}
Second, for the high order terms, using an analogous argument, we find that (\ref{eq_controlkeypartone})  holds true by replacing $\vec{b} \vec{b}^*$ with $\Ab$ and using the fact that $\sum_{i=1}^N \omega_i^2=1$  so that as in (\ref{eq_p3equation}) we have
\begin{equation*}
\widetilde{P}_3=(k-1) \kappa_4 \widetilde{\mathsf{V}}_2 \widetilde{\mathcal{Y}}^{k-2}+\OO_{\prec}((N \eta)^{-1/2}). 
\end{equation*} 
This completes our proof. 

% {\color{red} start from here.  By a discussion similar  to }

%We only point out the main differences here. {\color{red}[start from here]}  
%Note that we can write the above form as 
%\begin{equation*}
%\Tr(G_1(z) B_1)+\Tr G_1(z) B_2,
%\end{equation*}
%where $\| B_1\|=\OO_{\prec}(1)$ is defined as 
%\begin{equation*}
%B_1=\sum_{i=1}^N  \frac{w_i}{1+d_i} (\vec{v}_i \vec{v}_i^* -z\Vb_r \Lb_1 \Vb_r^* G_1(z) \vec{v}_i \vec{v}_i^*),
%\end{equation*}
%and $\|B_2 \|=\OO_{\prec}(M^{-1/2})$ is defined as 
%\begin{equation*}
%B_2=-z \sum_{i=1}^N \frac{w_i^2}{1+d_i} \Vb_r \Lb_1 \Delta(z) \Lb_1 \Vb_r^* G_1(z) \vec{v}_i \vec{v}_i^*. 
%\end{equation*}
%Therefore, we have shown that the spiked case can be efficiently reduced to the non-spiked case.  

%{\color{red}calculate use the real directly. }

\section{Density and Jacobi matrix approximation}\label{app:densityapproximation}

In this section we first discuss a method to compute an approximation of measures of the form \eqref{eq:mu} given a (possibly random) approximation $r(z)$ of
\begin{align*}
  \int_{\mathbb R} \frac{\mu(\sd \lambda)}{\lambda - z}, \quad \mathrm{Im}\;z > 0.
\end{align*}
We assume that $\mathtt a_j, \mathtt b_j$ and $\mathtt c_j$ are all known, or are well approximated.  The approach uses the Chebyshev polynomials of the second kind $(U_k)_{k \geq 0}$ \cite{DLMF} which are the orthogonal polynomials with respect to the semicircle distribution, scaled to $[-1,1]$:
\begin{align*}
  \int_{-1}^1 U_k(x) U_j(x)  \frac{2\sqrt{1 - x^2}}{\pi} \sd x = \delta_{jk}. 
\end{align*}
From \cite[Lemma 5.6]{TrogdonSOBook}
\begin{align*}
  &\int_{-1}^1 \frac{U_k(x)}{x - z} \frac{2\sqrt{1 - x^2}}{\pi} \sd x = - 2 \left[ z - \sqrt{z-1} \sqrt{z+1} \right]^{k+1} = c_k(z;\mu_{\mathrm{Cheb}}), \\
  &\mu_{\mathrm{Cheb}}(\sd x) = \frac{2\sqrt{1 - x^2}}{\pi} \one_{[-1,1]}(x)  \sd x.
\end{align*}
  
We then define the mapped polynomials for $a < b$
\begin{align*}
  U_k(x;a,b) = U_k(M_{a,b}^{-1}(x)), \quad M_{a,b}(x) = \frac{b - a}{2} x + \frac{b + a}{2}.
\end{align*}
It is then straightforward to see that
\begin{align*}
  \int_{a}^b \frac{U_k(x;a,b)}{x - z} \sqrt{(b-x)(x-a)}\sd x = \frac{\pi(b-a)}{4} c_k(M_{a,b}^{-1}(z);\mu_{\mathrm{Cheb}}).
  %&= \int_{a}^b \frac{U_k(M_{a,b}^{-1}(x))}{x - z} \sqrt{(b-x)(x-a)}\sd x\\
  %& = \left( \frac{b-a}{2} \right)^2 \int_{-1}^1 \frac{U_k(x)}{M_{a,b}(x) - z} \sqrt{1-x^2}\sd x\\
  %& = \left( \frac{b-a}{2} \right)^2 \int_{-1}^1 \frac{U_k(x)}{\frac{b - a}{2} x + \frac{b + a}{2} - z} \sqrt{1-x^2}\sd x\\
  %& =   \frac{b-a}{2}\int_{-1}^1 \frac{U_k(x)}{ x - M_{a,b}^{-1}(z)}\sqrt{1-x^2}\sd x\\
\end{align*}
So, given a (small) integer $\ell$ and unknown coefficients $d_{j,k}$  we can follow the idea of \cite{DEIFT1998388} to simply compute $\int \frac{\nu(\sd \lambda)}{\lambda - z}$ if $\nu$ is of the form \eqref{eq:mu} and
  \begin{align*}
    h_j(\lambda) = \sum_{k=0}^{\ell-1} d_{j,k} U_k(\lambda;\mathtt a_j, \mathtt b_j).
  \end{align*}
  Let $\vec x^{(k)} = (x_1^{(k)},\ldots,x_k^{(k)})= (x_1,\ldots,x_k)$ be the $k$ zeros of $U_k$ and define the $k \times \ell$ matrix $E_k = ( U_{j-1}(x_i))_{\substack{1 \leq i \leq k \\ 1 \leq j \leq \ell }}$.   This is defined so that
  \begin{align*}
    h_j(\vec x^{(k)}) = E_k \begin{bmatrix} d_{j,0} \\ d_{j,1} \\ \vdots \\ d_{j,\ell-1} \end{bmatrix}.
  \end{align*}
  For a vector $\vec z = [z_1,\ldots,z_m]$ of $m$ points in the upper-half plane define the $m \times \ell$ matrix $C_{\vec z} = (c_{j-1}(z_i;\mu_{\mathrm{Cheb}}))_{\substack{1 \leq i \leq m \\ 1 \leq j \leq \ell}}$.  
  
  In the non-spiked case, we seek a solution of the following constrained optimization problem
  \begin{align*}
    \mathrm{argmin}_{ \vec d_j : E_k \vec d_j \geq 0 } \left\| \sum_{j=1}^{g+1} \frac{ \pi}{4} (\mathtt b_j - \mathtt a_j) C_{M_{\mathtt a_j,\mathtt b_j}^{-1}(\vec z)} \vec d_j - r(\vec z) \right\|_2,
  \end{align*}
  where $\vec d_j = \begin{bmatrix} d_{j,0} & d_{j,1} & \cdots & d_{j,\ell-1} \end{bmatrix}$.  If there are spikes $\mathtt c_j$, one can approximate the weights $w_j$ using the trapezoidal rule around a small circle with center at $\mathtt c_j$.  Then the above constrained optimization problem applies to $r(z) - \sum_{j=1}^p \frac{ w_j}{\mathtt c_j - z}$.

  Once, the density is approximated, one would like to generate $\mathcal J(\mu)$.  The simplest way to do this is to use the Gaussian quadrature rule associated to the weight $\sqrt{1-x^2}$, i.e., consider the measure
  \begin{align*}
    \mu_K = \sum_{j=1}^K w_j \delta_{x_j^{(K)}},
  \end{align*}
  where the weights $\vec w_K = [w_1,\ldots,w_K]^T$ are chosen so that $\int p(x) \mu_K( \sd x) = \int_{-1}^1 p(x) \frac{2\sqrt{1 - x^2}}{\pi} \sd x$ whenever $p$ is a polynomial of degree at most $2K-1$. There are many ways to generate these weights, see \cite{Golub1969}.  Then define vectors of nodes and weights, respectively, by
  \begin{align*}
    \vec x = \begin{bmatrix} M_{\mathtt a_1, \mathtt b_1}(\vec x^{(K)}) \\ \vdots \\ M_{\mathtt a_{g+1}, \mathtt b_{g+1}}(\vec x^{(K)}) \end{bmatrix}, \quad  \vec W = \begin{bmatrix}  \frac{\mathtt b_1 - \mathtt a_1}{2} (E_{K} \vec d_1) \vec w_K  \\ \vdots \\ \frac{\mathtt b_{g+1} - \mathtt a_{g+1}}{2} (E_{K} \vec d_{g+1}) \vec w_K  \end{bmatrix}.
  \end{align*}
  If spikes are present, one needs to append $[\mathtt c_1,\ldots,\mathtt c_p]$ and $[\omega_1,\ldots, \omega_p]$ onto the end of $\vec x$ and $\vec W$, respectively.  Now, it follows, in the notation \eqref{eq:T_L} that
 $T(\mathrm{diag}(\vec x), \sqrt{\vec W}),$
  is a good approximation of $\mathcal J_K(\mu)$, see \cite{Brubeck2021a}, for example.  Indeed, if we ignore the errors induced by our approximations of each $h_j, \omega_j$, provided $K > K' + \ell/2$ one has that the upper-left $K' \times K'$ block of $T(\mathrm{diag}(\vec x), \sqrt{\vec W})$ coincides with that of $\mathcal J(\mu)$.

In practice, we generate 100 independent copies of a spiked sample covariance matrix and for each matrix we compute $r(z) = \langle \vec b , (W - z I)^{-1} \vec b \rangle$ and take set the points $\vec z$ to be the union of $M_{\mathtt a_j, \mathtt b_j}(\vec u) + \I/10$ where $\vec u$ is $m$ equally spaced points on $[-1,1]$. We take $\ell = 4, m = 200, k = 20$ in our computations.  The resulting 100 vectors $\vec d_j$ are averaged for each $j$.  We do not address the accuracy of this algorithm beyond noting that it suffices to identify the limiting curves in our computations.

Finally, we report the approximate density functions of the limiting VESD for the examples used in Section \ref{sec_simu}. Figure \ref{fig:single_gap_density_spiked} displays approximate density for the limiting VESD for the single gap example, and Figure \ref{fig:two_gap_density} displays that of the two gaps example. 

\begin{figure}[htbp]
  \centering
  \includegraphics[width=.76\linewidth]{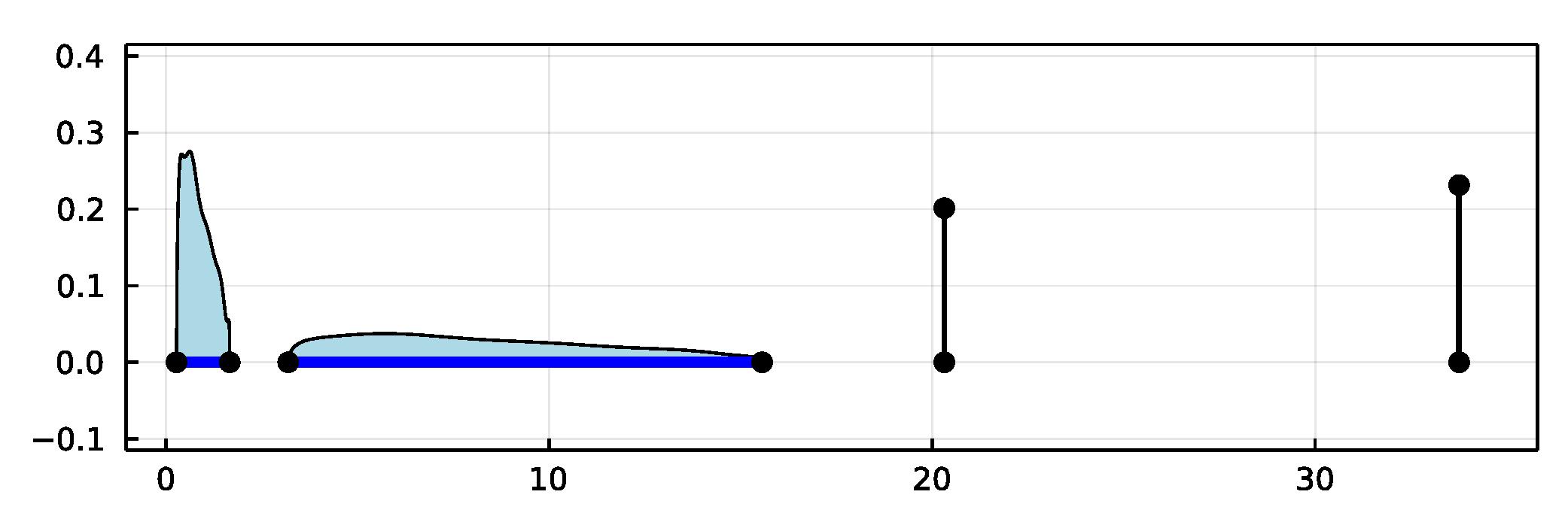}
  \caption{\label{fig:single_gap_density_spiked}  An approximation of the limiting density of the VESD for \eqref{eq:single_gap_spiked} with $2 \vec b = \vec f_1  + \vec f_2  + \vec f_3 + \vec f_N$ that display the presence of two spikes and their associated strengths.}
\end{figure}

\begin{figure}[htbp]
  \centering
  \includegraphics[width=.76\linewidth]{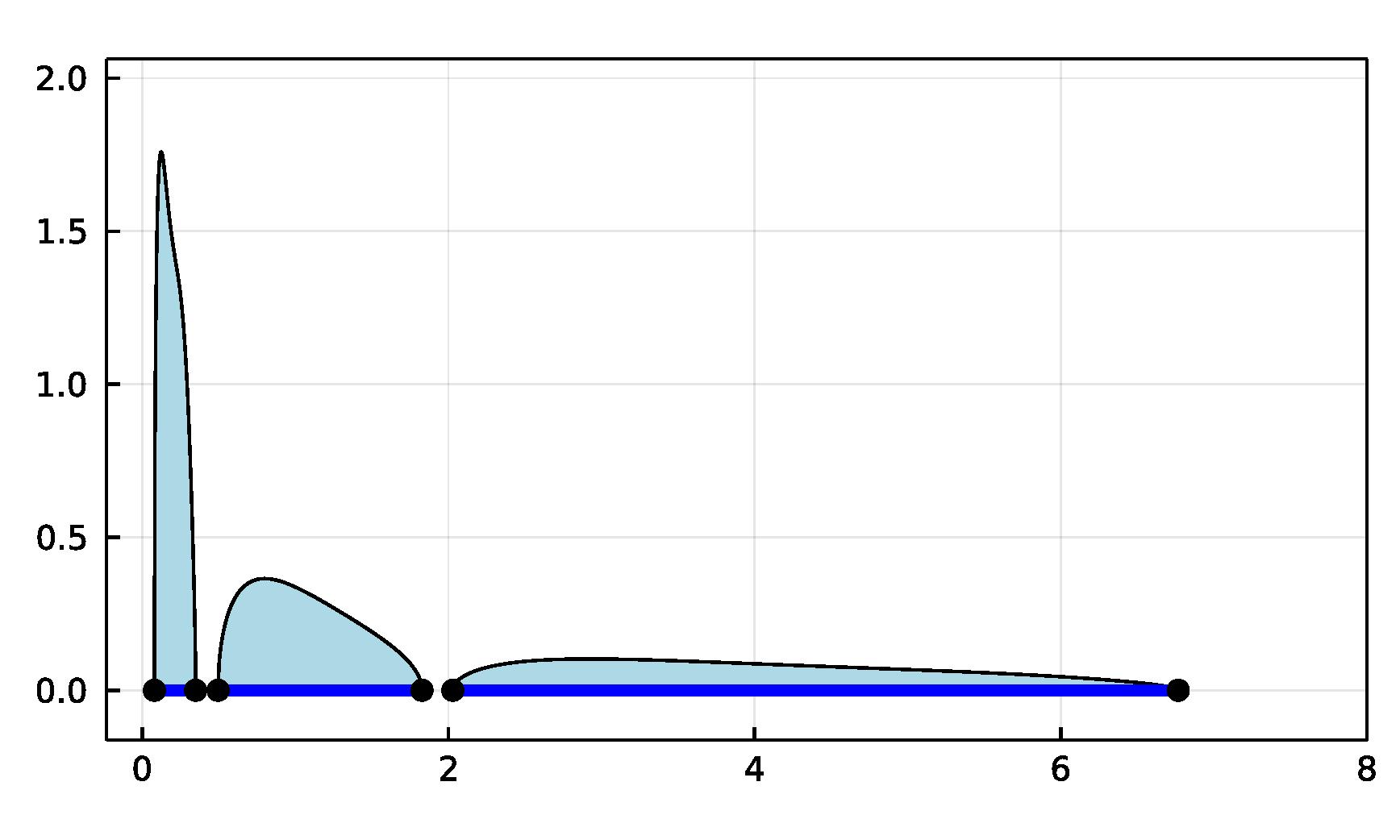}
  \caption{\label{fig:two_gap_density}  An approximation of the limiting density of the VESD for \eqref{eq:two_gap} with $\sqrt{2} \vec b = \vec f_1 + \vec f_N$.}
\end{figure}
   
\bibliographystyle{abbrv}
\bibliography{references,rmtref}

\begin{thebibliography}{10}

\bibitem{MR2567175}
Z.~Bai and J.~W. Silverstein.
\newblock {\em Spectral analysis of large dimensional random matrices}.
\newblock Springer Series in Statistics. Springer, New York, second edition,
  2010.

\bibitem{BMP}
Z.~D. Bai, B.~Q. Miao, and G.~M. Pan.
\newblock {On asymptotics of eigenvectors of large sample covariance matrix}.
\newblock {\em The Annals of Probability}, 35(4):1532 -- 1572, 2007.

\bibitem{MR1682248}
J.~Baik, P.~Deift, and K.~Johansson.
\newblock On the distribution of the length of the longest increasing
  subsequence of random permutations.
\newblock {\em J. Amer. Math. Soc.}, 12(4):1119--1178, 1999.

\bibitem{Baik2007}
J.~Baik, T.~Kriecherbauer, K.~T.-R. McLaughlin, and P.~D. Miller.
\newblock {\em {Discrete Orthogonal Polynomials}}.
\newblock Princeton University Press, Princeton, NJ, 2007.

\bibitem{10.1093/imrn/rnaa210}
Z.~Bao, K.~Schnelli, and Y.~Xu.
\newblock {Central Limit Theorem for Mesoscopic Eigenvalue Statistics of the
  Free Sum of Matrices}.
\newblock {\em International Mathematics Research Notices}, 2020.
\newblock rnaa210.

\bibitem{Beckermann2001}
B.~Beckermann and A.~B.~J. Kuijlaars.
\newblock {Superlinear Convergence of Conjugate Gradients}.
\newblock {\em SIAM Journal on Numerical Analysis}, 39(1):300--329, 1 2001.

\bibitem{Algebro}
E.~D. Belokolos, A.~I. Bobenko, V.~Z. Enol'skii, A.~R. Its, and V.~B. Matveev.
\newblock {\em {Algebro-Geometric Approach to Nonlinear Integrable Equations}}.
\newblock Springer, 1994.

\bibitem{MR2021905}
M.~Bertola, B.~Eynard, and J.~Harnad.
\newblock Differential systems for biorthogonal polynomials appearing in
  2-matrix models and the associated {R}iemann-{H}ilbert problem.
\newblock {\em Comm. Math. Phys.}, 243(2):193--240, 2003.

\bibitem{MR2486670}
M.~Bertola, M.~Gekhtman, and J.~Szmigielski.
\newblock The {C}auchy two-matrix model.
\newblock {\em Comm. Math. Phys.}, 287(3):983--1014, 2009.

\bibitem{BI}
P.~Bleher and A.~Its.
\newblock Semiclassical asymptotics of orthogonal polynomials,
  {R}iemann-{H}ilbert problem, and universality in the matrix model.
\newblock {\em Annals of Mathematics}, 150(1):185--266, 1999.

\bibitem{MR2103904}
P.~Bleher and A.~B.~J. Kuijlaars.
\newblock Large {$n$} limit of {G}aussian random matrices with external source.
  {I}.
\newblock {\em Comm. Math. Phys.}, 252(1-3):43--76, 2004.

\bibitem{Bleher2011}
P.~M. Bleher.
\newblock {\em Lectures on Random Matrix Models}, pages 251--349.
\newblock Springer New York, 2011.

\bibitem{Bobenko}
A.~I. Bobenko and L.~A. Bordag.
\newblock {Periodic multiphase solutions of the {K}adomsev-{P}etviashvili
  equation}.
\newblock {\em J. Phys. A}, 22:1259--1274, 1989.

\bibitem{Bottcher1997}
A.~B{\"{o}}ttcher and Y.~I. Karlovich.
\newblock {\em {Carleson Curves, Muckenhoupt Weights, and Toeplitz Operators}}.
\newblock Birkh{\"{a}}user Basel, Basel, 1997.

\bibitem{MR3459158}
J.~Breuer and M.~Duits.
\newblock Universality of mesoscopic fluctuations for orthogonal polynomial
  ensembles.
\newblock {\em Comm. Math. Phys.}, 342(2):491--531, 2016.

\bibitem{Brubeck2021a}
P.~D. Brubeck, Y.~Nakatsukasa, and L.~N. Trefethen.
\newblock {Vandermonde with Arnoldi}.
\newblock {\em SIAM Review}, 63(2):405--415, jan 2021.

\bibitem{MR2377682}
Y.~Chen and A.~R. Its.
\newblock A {R}iemann-{H}ilbert approach to the {A}khiezer polynomials.
\newblock {\em Philos. Trans. R. Soc. Lond. Ser. A Math. Phys. Eng. Sci.},
  366(1867):973--1003, 2008.

\bibitem{DeiftOrthogonalPolynomials}
P.~Deift.
\newblock {\em {Orthogonal Polynomials and Random Matrices: a Riemann-Hilbert
  Approach}}.
\newblock Amer. Math. Soc., Providence, RI, 2000.

\bibitem{MR2307753}
P.~Deift.
\newblock Riemann-{H}ilbert methods in the theory of orthogonal polynomials.
\newblock In {\em Spectral theory and mathematical physics: a {F}estschrift in
  honor of {B}arry {S}imon's 60th birthday}, volume~76 of {\em Proc. Sympos.
  Pure Math.}, pages 715--740. Amer. Math. Soc., Providence, RI, 2007.

\bibitem{MR2306224}
P.~Deift and D.~Gioev.
\newblock Universality at the edge of the spectrum for unitary, orthogonal, and
  symplectic ensembles of random matrices.
\newblock {\em Comm. Pure Appl. Math.}, 60, 2007.

\bibitem{MR2363388}
P.~Deift, D.~Gioev, T.~Kriecherbauer, and M.~Vanlessen.
\newblock Universality for orthogonal and symplectic {L}aguerre-type ensembles.
\newblock {\em J. Stat. Phys.}, 129(5-6):949--1053, 2007.

\bibitem{DEIFT1998388}
P.~Deift, T.~Kriecherbauer, and K.-R. McLaughlin.
\newblock New results on the equilibrium measure for logarithmic potentials in
  the presence of an external field.
\newblock {\em Journal of Approximation Theory}, 95(3):388--475, 1998.

\bibitem{DeiftWeights1}
P.~Deift, T.~Kriecherbauer, K.~T.-R. McLaughlin, S.~Venakides, and X.~Zhou.
\newblock {Strong asymptotics of orthogonal polynomials with respect to
  exponential weights}.
\newblock {\em Comm. Pure Appl. Math.}, 52(11):1491--1552, 1999.

\bibitem{MR1702716}
P.~Deift, T.~Kriecherbauer, K.~T.-R. McLaughlin, S.~Venakides, and X.~Zhou.
\newblock Uniform asymptotics for polynomials orthogonal with respect to
  varying exponential weights and applications to universality questions in
  random matrix theory.
\newblock {\em Comm. Pure Appl. Math.}, 52(11):1335--1425, 1999.

\bibitem{Deift2019a}
P.~Deift, S.~D. Miller, and T.~Trogdon.
\newblock {Stopping time signatures for some algorithms in cryptography}.
\newblock {\em arXiv preprint arXiv:1905.08408}, 5 2019.

\bibitem{DT17}
P.~Deift and T.~Trogdon.
\newblock Universality for eigenvalue algorithms on sample covariance matrices.
\newblock {\em SIAM Journal on Numerical Analysis}, 55(6):2835--2862, 2017.

\bibitem{DT18}
P.~Deift and T.~Trogdon.
\newblock Universality for the toda algorithm to compute the largest eigenvalue
  of a random matrix.
\newblock {\em Communications on Pure and Applied Mathematics}, 71(3):505--536,
  2018.

\bibitem{MR4188626}
P.~Deift and T.~Trogdon.
\newblock The conjugate gradient algorithm on well-conditioned {W}ishart
  matrices is almost deterministic.
\newblock {\em Quart. Appl. Math.}, 79(1):125--161, 2021.

\bibitem{DZS}
P.~Deift and X.~Zhou.
\newblock A steepest descent method for oscillatory {R}iemann--{H}ilbert
  problems. asymptotics for the {MK}d{V} equation.
\newblock {\em Annals of Mathematics}, 137(2):295--368, 1993.

\bibitem{MR1469319}
P.~A. Deift, A.~R. Its, and X.~Zhou.
\newblock A {R}iemann-{H}ilbert approach to asymptotic problems arising in the
  theory of random matrix models, and also in the theory of integrable
  statistical mechanics.
\newblock {\em Ann. of Math. (2)}, 146(1):149--235, 1997.

\bibitem{Deift14973}
P.~A. Deift, G.~Menon, S.~Olver, and T.~Trogdon.
\newblock Universality in numerical computations with random data.
\newblock {\em Proceedings of the National Academy of Sciences},
  111(42):14973--14978, 2014.

\bibitem{DRMTA}
X.~Ding.
\newblock Spiked sample covariance matrices with possibly multiple bulk
  components.
\newblock {\em Random Matrices: Theory and Applications}, 10(01):2150014, 2021.

\bibitem{DT1}
X.~{Ding} and T.~{Trogdon}.
\newblock {The conjugate gradient algorithm on a general class of spiked
  covariance matrices}.
\newblock {\em Quarterly of Applied Mathematics}, 80(1):99--155, 2022.

\bibitem{DY}
X.~Ding and F.~Yang.
\newblock {A necessary and sufficient condition for edge universality at the
  largest singular values of covariance matrices}.
\newblock {\em The Annals of Applied Probability}, 28(3):1679 -- 1738, 2018.

\bibitem{DubrovinTheta}
B.~A. Dubrovin.
\newblock {Theta functions and non-linear equations}.
\newblock {\em Russian Math. Surveys}, 36:11--92, 1981.

\bibitem{MR2531553}
M.~Duits and A.~B.~J. Kuijlaars.
\newblock Universality in the two-matrix model: a {R}iemann-{H}ilbert
  steepest-descent analysis.
\newblock {\em Comm. Pure Appl. Math.}, 62, 2009.

\bibitem{Dumitriu2002}
I.~Dumitriu and A.~Edelman.
\newblock {Matrix models for beta ensembles}.
\newblock {\em Journal of Mathematical Physics}, 43(11):5830, 10 2002.

\bibitem{MR278668}
F.~J. Dyson.
\newblock Correlations between eigenvalues of a random matrix.
\newblock {\em Comm. Math. Phys.}, 19:235--250, 1970.

\bibitem{MR3699468}
L.~Erd\H{o}s and H.-T. Yau.
\newblock {\em A dynamical approach to random matrix theory}, volume~28 of {\em
  Courant Lecture Notes in Mathematics}.
\newblock Courant Institute of Mathematical Sciences, New York; American
  Mathematical Society, Providence, RI, 2017.

\bibitem{FI}
Z.~Fan and I.~Johnstone.
\newblock {T}racy-{W}idom at each edge of real covariance and {MANOVA}
  estimators.
\newblock {\em The Annals of Applied Probability (to appear)}, 2021.

\bibitem{FokasOP}
A.~S. Fokas, A.~R. Its, and A.~V. Kitaev.
\newblock {The isomonodromy approach to matrix models in 2D quantum gravity}.
\newblock {\em Comm. Math. Phys.}, 147(2):395--430, 1992.

\bibitem{Garza2016}
L.~E. Garza and F.~Marcell{\'{a}}n.
\newblock {Orthogonal polynomials and perturbations on measures supported on
  the real line and on the unit circle. A matrix perspective}.
\newblock {\em Expositiones Mathematicae}, 34(3):287--326, 2016.

\bibitem{GautschiOP}
W.~Gautschi.
\newblock {\em {Orthogonal Polynomials: Applications and Computation}}.
\newblock Oxford University Press, 2004.

\bibitem{Golub1969}
G.~H. Golub and J.~H. Welsch.
\newblock {Calculation of Gauss quadrature rules}.
\newblock {\em Mathematics of Computation}, 23(106):221--221, 5 1969.

\bibitem{MR3678478}
Y.~He and A.~Knowles.
\newblock Mesoscopic eigenvalue statistics of {W}igner matrices.
\newblock {\em Ann. Appl. Probab.}, 27(3):1510--1550, 2017.

\bibitem{Hestenes1952}
M.~Hestenes and E.~Steifel.
\newblock {Method of Conjugate Gradients for Solving Linear Systems}.
\newblock {\em J. Research Nat. Bur. Standards}, 20:409--436, 1952.

\bibitem{MR2581882}
K.~Johansson.
\newblock Random matrices and determinantal processes.
\newblock In {\em Mathematical statistical physics}, pages 1--55. Elsevier B.
  V., Amsterdam, 2006.

\bibitem{Johnstone2001}
I.~M. Johnstone.
\newblock {On the distribution of the largest eigenvalue in principal
  components analysis}.
\newblock {\em The Annals of Statistics}, 29(2):295--327, 4 2001.

\bibitem{MR1985213}
A.~A. Kapaev.
\newblock Riemann-{H}ilbert problem for bi-orthogonal polynomials.
\newblock {\em J. Phys. A}, 36(16):4629--4640, 2003.

\bibitem{MR1411619}
A.~M. Khorunzhy, B.~A. Khoruzhenko, and L.~A. Pastur.
\newblock Asymptotic properties of large random matrices with independent
  entries.
\newblock {\em J. Math. Phys.}, 37(10):5033--5060, 1996.

\bibitem{Knowles2017}
A.~Knowles and J.~Yin.
\newblock {Anisotropic local laws for random matrices}.
\newblock {\em Probability Theory and Related Fields}, 169(1-2):257--352, 10
  2017.

\bibitem{Largegap}
I.~Krasovsky.
\newblock Large gap asymptotics for random matrices.
\newblock In V.~Sidoravi{\v{c}}ius, editor, {\em New Trends in Mathematical
  Physics}, pages 413--419. Springer Netherlands, 2009.

\bibitem{MR1680380}
T.~Kriecherbauer and K.~T.-R. McLaughlin.
\newblock Strong asymptotics of polynomials orthogonal with respect to {F}reud
  weights.
\newblock {\em Internat. Math. Res. Notices}, 1999(6):299--333, 1999.

\bibitem{MR2022855}
A.~B.~J. Kuijlaars.
\newblock Riemann-{H}ilbert analysis for orthogonal polynomials.
\newblock In {\em Orthogonal polynomials and special functions ({L}euven,
  2002)}, volume 1817 of {\em Lecture Notes in Math.}, pages 167--210.
  Springer, Berlin, 2003.

\bibitem{MR2127887}
A.~B.~J. Kuijlaars and K.~T.-R. McLaughlin.
\newblock A {R}iemann-{H}ilbert problem for biorthogonal polynomials.
\newblock {\em J. Comput. Appl. Math.}, 178(1-2):313--320, 2005.

\bibitem{MR2087231}
A.~B.~J. Kuijlaars, K.~T.-R. McLaughlin, W.~Van~Assche, and M.~Vanlessen.
\newblock The {R}iemann-{H}ilbert approach to strong asymptotics for orthogonal
  polynomials on {$[-1,1]$}.
\newblock {\em Adv. Math.}, 188(2):337--398, 2004.

\bibitem{MR1912278}
A.~B.~J. Kuijlaars and M.~Vanlessen.
\newblock Universality for eigenvalue correlations from the modified {J}acobi
  unitary ensemble.
\newblock {\em Int. Math. Res. Not.}, 2002(30):1575--1600, 2002.

\bibitem{MR3257662}
A.~B.~J. Kuijlaars and L.~Zhang.
\newblock Singular values of products of {G}inibre random matrices, multiple
  orthogonal polynomials and hard edge scaling limits.
\newblock {\em Comm. Math. Phys.}, 332(2):759--781, 2014.

\bibitem{MR4255183}
Y.~Li, K.~Schnelli, and Y.~Xu.
\newblock Central limit theorem for mesoscopic eigenvalue statistics of
  deformed {W}igner matrices and sample covariance matrices.
\newblock {\em Ann. Inst. Henri Poincar\'{e} Probab. Stat.}, 57(1):506--546,
  2021.

\bibitem{MR2561434}
A.~Lytova and L.~Pastur.
\newblock Central limit theorem for linear eigenvalue statistics of random
  matrices with independent entries.
\newblock {\em Ann. Probab.}, 37(5):1778--1840, 2009.

\bibitem{mehta2004random}
M.~Mehta.
\newblock {\em Random Matrices}.
\newblock Elsevier Science, 2004.

\bibitem{Menon2016}
G.~Menon and T.~Trogdon.
\newblock {Smoothed analysis for the conjugate gradient algorithm}.
\newblock {\em SIGMA}, 12:1--19, 11 2016.

\bibitem{DLMF}
F.~W.~J. Olver, D.~W. Lozier, R.~F. Boisvert, and C.~W. Clark.
\newblock {\em {NIST Handbook of Mathematical Functions}}.
\newblock Cambridge University Press, 2010.

\bibitem{Paquette2020}
E.~Paquette and T.~Trogdon.
\newblock {Universality for the conjugate gradient and MINRES algorithms on
  sample covariance matrices}.
\newblock {\em arXiv preprint arXiv:2007.00640}, 2020.

\bibitem{Peherstorfer}
F.~Peherstorfer.
\newblock {Orthogonal polynomials on several intervals: Accumulation points of
  recurrence coefficients and of zeros}.
\newblock {\em Journal of Approximation Theory}, 163(7):814--837, jul 2011.

\bibitem{DiagonalRMT}
C.~W. Pfrang, P.~Deift, and G.~Menon.
\newblock {How long does it take to compute the eigenvalues of a random
  symmetric matrix?}
\newblock {\em Random matrix theory, interacting particle systems, and
  integrable systems, MSRI Publications}, 65:411--442, 2014.

\bibitem{Sagun2015}
L.~Sagun, T.~Trogdon, and Y.~LeCun.
\newblock {Universal halting times in optimization and machine learning}.
\newblock {\em Quarterly of Applied Mathematics}, 76(2):289--301, 9 2017.

\bibitem{Sankar2006}
A.~Sankar, D.~A. Spielman, and S.-H. Teng.
\newblock {Smoothed Analysis of the Condition Numbers and Growth Factors of
  Matrices}.
\newblock {\em SIAM Journal on Matrix Analysis and Applications},
  28(2):446--476, 1 2006.

\bibitem{Silverstein1985}
J.~W. Silverstein.
\newblock {The Smallest Eigenvalue of a Large Dimensional Wishart Matrix}.
\newblock {\em The Annals of Probability}, 13(4):1364--1368, 1985.

\bibitem{Spielman2004}
D.~A. Spielman and S.-H. Teng.
\newblock {Smoothed analysis of algorithms}.
\newblock {\em Journal of the ACM}, 51(3):385--463, 5 2004.

\bibitem{MR1257246}
C.~A. Tracy and H.~Widom.
\newblock Level-spacing distributions and the {A}iry kernel.
\newblock {\em Comm. Math. Phys.}, 159(1):151--174, 1994.

\bibitem{TrefethenBau}
L.~N. Trefethen and D.~Bau~III.
\newblock {\em {Numerical linear algebra}}.
\newblock Society for Industrial and Applied Mathematics (SIAM), Philadelphia,
  PA, 1997.

\bibitem{TrogdonSOBook}
T.~Trogdon and S.~Olver.
\newblock {\em {Riemann--Hilbert Problems, Their Numerical Solution and the
  Computation of Nonlinear Special Functions}}.
\newblock SIAM, Philadelphia, PA, 2016.

\bibitem{MR2006283}
W.~Van~Assche, J.~S. Geronimo, and A.~B.~J. Kuijlaars.
\newblock Riemann-{H}ilbert problems for multiple orthogonal polynomials.
\newblock In {\em Special functions 2000: current perspective and future
  directions ({T}empe, {AZ})}, volume~30 of {\em NATO Sci. Ser. II Math. Phys.
  Chem.}, pages 23--59. Kluwer Acad. Publ., Dordrecht, 2001.

\bibitem{WZ}
D.~{Wang} and L.~{Zhang}.
\newblock {A vector {R}iemann-{H}ilbert approach to the {M}uttalib-{B}orodin
  ensembles}.
\newblock {\em arXiv preprint arXiv 2103.10327}, 2021.

\bibitem{YF}
F.~{Yang}.
\newblock {Linear spectral statistics of eigenvectors of anisotropic sample
  covariance matrices}.
\newblock {\em arXiv preprint arXiv:2005.00999}, 2020.

\bibitem{YATTSELEV201573}
M.~L. Yattselev.
\newblock Nuttall’s theorem with analytic weights on algebraic {S}-contours.
\newblock {\em Journal of Approximation Theory}, 190:73--90, 2015.

\bibitem{Zhedanov1997}
A.~Zhedanov.
\newblock {Rational spectral transformations and orthogonal polynomials}.
\newblock {\em Journal of Computational and Applied Mathematics}, 85(1):67--86,
  nov 1997.

\end{thebibliography}

\end{document}